\documentclass[10pt]{amsart}

\usepackage[left=2.8cm, right=2.8cm, top=2.2cm, bottom=2.2cm]{geometry}

\usepackage{amsmath,amssymb,amsthm,mathrsfs,stmaryrd,latexsym}
\usepackage{amsfonts}

\usepackage{tcolorbox}
\tcbuselibrary{skins, breakable, theorems}

\usepackage{enumerate} 
\usepackage{tikz}
\usepackage{etoolbox}

\def\Luoma#1{\uppercase\expandafter{\romannumeral#1}}
\def\luoma#1{\romannumeral#1}

\usepackage{url}

\newtheorem{mythm}{Theorem}[section]
\newtheorem{mylem}[mythm]{Lemma}
\newtheorem{myprop}[mythm]{Proposition}
\newtheorem{mycor}[mythm]{Corollary}

\theoremstyle{definition}
\newtheorem{mydefn}[mythm]{Definition}

\theoremstyle{remark}
\newtheorem{myrem}[mythm]{Remark}
\newtheorem{mypara}[mythm]{}

\usepackage[all]{xy}
\usepackage{graphicx}
\usepackage{subfigure}

\newcommand{\bb}{\mathbb}
\newcommand{\ca}{\mathcal}
\newcommand{\ak}{\mathfrak}
\newcommand{\scr}{\mathscr}

\newcommand{\mbf}{\mathbf}
\newcommand{\mrm}{\mathrm}

\newcommand{\trm}{\textrm}

\newcommand{\tit}{\textit}

\def\op#1{\mathop{\mathrm{#1}}}

\newcommand{\ho}{\mrm{Hom}}
\newcommand{\ke}{\mrm{Ker}}

\newcommand{\id}{\mrm{id}}
\newcommand{\spec}{\op{Spec}}
\newcommand{\colim}{\op{colim}}

\newcommand{\ob}{\mrm{Ob}}
\newcommand{\rr}{\mrm{R}}
\newcommand{\dl}{\mrm{L}}

\newcommand{\iso}{\stackrel{\sim}{\longrightarrow}}

\newcommand{\plim}{\varprojlim}

\newcommand{\ash}{\mrm{a}}
\newcommand{\psh}{\mrm{p}}
\newcommand{\ssh}{\mrm{s}}
\newcommand{\oppo}{\mrm{op}}

\newcommand{\al}{\mrm{al}}

\newcommand{\frob}{\mrm{Frob}}

\newcommand{\sch}{\mbf{Sch}}
\newcommand{\schqcqs}{\mbf{Sch}^\mrm{coh}}

\newcommand{\dd}{\mbf{D}}

\newcommand{\module}{\trm{-}\mbf{Mod}}
\newcommand{\alg}{\trm{-}\mbf{Alg}}
\newcommand{\perf}{\trm{-}\mbf{Perf}}
\newcommand{\locsys}{\mbf{LocSys}}

\newcommand{\et}{\mrm{\acute{e}t}}
\newcommand{\fet}{\mrm{f\acute{e}t}}
\newcommand{\proet}{\mrm{pro\acute{e}t}}
\newcommand{\profet}{\mrm{prof\acute{e}t}}
\newcommand{\fal}{\mbf{E}}
\newcommand{\falh}{\mbf{I}}
\newcommand{\falb}{\overline{\scr{B}}}
\newcommand{\falhb}{\scr{O}}

\title[Cohomological Descent for Faltings' $p$-adic Hodge Theory and Applications]{Cohomological Descent for Faltings' $p$-adic Hodge Theory and Applications}
\author{Tongmu He}
\date{\today}
\address{Tongmu He, Institut des Hautes \'Etudes Scientifiques, 35 route de Chartres, 91440 Bures-sur-Yvette, France}
\email{hetm15@ihes.fr}

\usepackage{hyperref}
\hypersetup{colorlinks,
citecolor=black,
filecolor=black,
linkcolor=black,
urlcolor=black,
pdftitle={Cohomological Descent for Faltings' $p$-adic Hodge Theory and Applications},
pdfauthor={Tongmu He},
pdftex}

\numberwithin{equation}{mythm}

\begin{document}
\maketitle

\begin{abstract}
	Faltings' approach in $p$-adic Hodge theory can be schematically divided into two main steps: firstly, a local reduction of the computation of the $p$-adic \'etale cohomology of a smooth variety over a $p$-adic local field to a Galois cohomology computation and then, the establishment of a link between the latter and differential forms. These relations are organized through Faltings ringed topos whose definition relies on the choice of an integral model of the variety, and whose good properties depend on the (logarithmic) smoothness of this model. Scholze's generalization for rigid analytic varieties has the advantage of depending only on the variety (i.e. the generic fibre). Inspired by Deligne's approach to classical Hodge theory for singular varieties, we establish a cohomological descent result for the structural sheaf of Faltings topos, which makes it possible to extend Faltings' approach to any integral model, i.e. without any smoothness assumption. An essential ingredient of our proof is a descent result of perfectoid algebras in the arc-topology due to Bhatt and Scholze. 
	
	As an application of our cohomological descent, using a variant of de Jong's alteration theorem for morphisms of schemes due to Gabber-Illusie-Temkin, we generalize Faltings' main $p$-adic comparison theorem to any proper and finitely presented morphism of coherent schemes over an absolute integral closure of $\mathbb{Z}_p$ (without any assumption of smoothness) for torsion \'etale sheaves (not necessarily finite locally constant). As a second application, we prove a local version of the relative Hodge-Tate filtration as a consequence of the global version constructed by Abbes-Gros.
\end{abstract}

\tableofcontents

\section{Introduction}

\begin{mypara}
	Faltings and Scholze's approaches to $p$-adic Hodge theory share several similarities. The most recent approach, that of Scholze, generalizes Faltings' main techniques from schemes to adic spaces. Nevertheless, beyond the analogies, there is no thread connecting the two. The main difficulty stems from the difference between the nature of their keystones, namely the Faltings topos for Faltings' approach and the pro-\'etale topos of an adic space for Scholze's approach. Faltings' approach has the advantage of only using schemes and their classical \'etale topoi. But it depends on the choice of an integral model of the $p$-adic variety, which intervenes in the very definition of Faltings topos and whose (log-)smoothness seems necessary for the good properties of this one. On the other hand, Scholze's approach which uses adic spaces and their pro-\'etale topoi, does not depend on any integral model. 
	
	The initial goal of this work is to make Faltings' approach ``free of integral models''. For this, we establish a cohomological descent result for Faltings ringed topos. Along the way, we introduce a variant for the v-topology which satisfies good cohomological descent properties and which can be regarded as a scheme theoretic analogue of the v-topos of an adic space. In particular, we establish a cohomological descent result from this topos to Faltings topos. It is an analogue of the cohomological descent from the v-topos to the pro-\'etale topos of an adic space established by Scholze \cite{scholze2021diamond}. We give two applications of our cohomological descent result. Firstly, we extend Faltings' main $p$-adic comparison theorem (which we refer to as ``Faltings' comparison theorem'' for short in the rest of the introduction), both in the absolute and the relative cases, to general integral models without any smoothness condition. Faltings' comparison theorem was generalized by Scholze to rigid analytic varieties, first in the smooth case and then in the general case. Our application is an analogue of this last generalization. Even in the smooth case, Faltings' comparison theorem and Scholze's generalization cannot be directly deduced from each other. Secondly, we prove a local version of the relative Hodge-Tate filtration as a consequence of the global version constructed by Abbes and Gros \cite{abbes2020suite} and our cohomological descent result. We would like to mention a third interesting application of our result to the $p$-adic Simpson correspondence given by Xu \cite{xu2022higgs}.
\end{mypara}

\begin{mypara}
	Faltings' proof of the Hodge-Tate decomposition illustrates his approach in $p$-adic Hodge theory and the role of his ringed topos. Let $K$ be a complete discrete valuation field of characteristic $0$ with algebraically closed residue field of characteristic $p>0$. We fix an algebraic closure $\overline{K}$ of $K$ and denote by $\widehat{\overline{K}}$ the $p$-adic completion of $\overline{K}$. For a proper smooth $K$-scheme $X$, Tate conjectured that there is a canonical $G_K=\mrm{Gal}(\overline{K}/K)$-equivariant decomposition, now called the Hodge-Tate decomposition (\cite[Remark, page 180]{tate1967p}),
	\begin{align}
		H^n_\et(X_{\overline{K}},\bb{Q}_p)\otimes_{\bb{Q}_p} \widehat{\overline{K}}=\bigoplus_{0\leq q\leq n} H^q(X,\Omega_{X/K}^{n-q})\otimes_{K}\widehat{\overline{K}}(q-n),
	\end{align}
	where $\widehat{\overline{K}}(q-n)$ is the $(q-n)$-th Tate twist of $\widehat{\overline{K}}$. This conjecture was settled by Faltings \cite{faltings1988p,faltings2002almost} and Tsuji \cite{tsuji1999p,tsuji2002semi} independently, and had been generalized to rigid analytic settings by Scholze \cite{scholze2013hodge}. There is also a version for non-proper smooth varieties showed by Faltings. Let $X^\circ$ be an open subset of $X$ whose complement is a normal crossings divisor $D$. Then, there is a canonical $G_K$-equivariant decomposition
	\begin{align}
		H^n_\et(X^\circ_{\overline{K}},\bb{Q}_p)\otimes_{\bb{Q}_p} \widehat{\overline{K}}=\bigoplus_{0\leq q\leq n} H^q(X,\Omega_{X/K}^{n-q}(\log D))\otimes_{K}\widehat{\overline{K}}(q-n).
	\end{align}
\end{mypara}

\begin{mypara}\label{para:notation-intro}
	One of the applications of our main result in this article is a generalization of the Hodge-Tate decomposition to the relative case. Let $(f,g):(X'^{\triangleright}\to X')\to (X^\circ\to X)$ be a morphism of open immersions of coherent schemes over $\spec(K)\to \spec(\ca{O}_K)$ (``coherent'' stands for ``quasi-compact and quasi-separated''). We assume that the following conditions hold:
	\begin{enumerate}
		\renewcommand{\labelenumi}{{\rm(\theenumi)}}
		\item The associated log schemes $(X',\scr{M}_{X'}), (X,\scr{M}_X)$ endowed with compactifying log structures are adequate (a technical condition which holds if the open immersions $X'^{\triangleright}\to X', X^\circ\to X$ are semi-stable over $\ca{O}_K$, cf. \ref{para:notation-log}).
		\item The morphism of log schemes $(X',\scr{M}_{X'})\to (X,\scr{M}_X)$ is smooth and saturated.
		\item The morphism of schemes $g:X'\to X$ is projective.
		\item The scheme $X=\spec(R)$ is affine and there exist finitely many nonzero divisors $f_1,\dots, f_r$ of $R[1/p]$ such that the divisor $D=\sum_{i=1}^r \mrm{div}(f_i)$ on $X_K$ has support $X_K\setminus X^\circ_K$ and that at each strict henselization of $X_K$ those elements $f_i$ contained in the maximal ideal form a subset of a regular system of parameters (in particular, $D$ is a normal crossings divisor on $X_K$).
	\end{enumerate}
	For any coherent $X^\circ_K$-scheme $Y$, we define a pro-finite \'etale $Y$-scheme
	\begin{align}
		Y_{\infty}&=\lim_{n} Y[T_1,\dots,T_r]/(T_1^{n}-f_1,\dots,T_r^{n}-f_r).
	\end{align}
\end{mypara}

\begin{mythm}[{cf. \ref{thm:main} and \ref{thm:acyclic}}]\label{thm:main-intro}
	Under the assumptions in {\rm\ref{para:notation-intro}}, let $U$ be an affine scheme pro-\'etale over $X$ and let $V$ be a pro-finite \'etale $U^\circ_{\overline{K},\infty}$-scheme (where $U^\circ=X^\circ\times_X U$) satisfying the following conditions:
	\begin{enumerate}
		\renewcommand{\labelenumi}{{\rm(\theenumi)}}
		\item The integral closure of $U$ in $V$ is the spectrum of an $\ca{O}_{\overline{K}}$-algebra $A$ which is almost pre-perfectoid in the sense of {\rm\ref{defn:pre-alg}}.
		\item For any integers $n\geq 0$ and $k\geq 0$, the pullback $(\rr^n f_{\et*}\bb{Z}/p^k\bb{Z})|_{V_{\et}}$ is a constant sheaf.
	\end{enumerate}
	Let $\overline{x}$ be a geometric point of $V$. Then, for any integer $n\geq 0$, there is a canonical finite decreasing filtration $(\mrm{fil}^q)_{q\in\bb{Z}}$ on $H^n_\et(X'^{\triangleright}_{\overline{x}},\bb{Z}_p)\otimes_{\bb{Z}_p} \widehat{A}[1/p]$ and a canonical isomorphism for each $q\in \bb{Z}$,
	\begin{align}\label{eq:1.4.1}
		\mrm{gr}^q(H^n_\et(X'^{\triangleright}_{\overline{x}},\bb{Z}_p)\otimes_{\bb{Z}_p} \widehat{A}[\frac{1}{p}])\cong H^q(X',\Omega^{n-q}_{(X',\scr{M}_{X'})/(X,\scr{M}_X)})\otimes_R \widehat{A}[\frac{1}{p}](q-n),
	\end{align}
	where $\mrm{gr}^q$ denotes the graded piece $\mrm{fil}^q/\mrm{fil}^{q+1}$.
	Moreover, if $U^\circ$ and $V$ are connected and if the function field of $V$ is a Galois extension of that of $U^\circ$ with Galois group $\Gamma$, then the filtration $(\mrm{fil}^q)_{q\in\bb{Z}}$ and the isomorphisms \eqref{eq:1.4.1} are $\Gamma$-equivariant.
\end{mythm}

\begin{myrem}\label{rem:main-intro}
	The objects $V\to U$ satisfying the conditions in \ref{thm:main-intro} form a topological generating family of the pro-\'etale Faltings site of $X^\circ_{\overline{K}}\to X$ (see the proof of \ref{thm:acyclic}).
\end{myrem}

\begin{mypara}
	This local \emph{relative Hodge-Tate filtration} stems from the global relative Hodge-Tate filtration constructed by Abbes-Gros \cite{abbes2020suite}. Their filtration takes place on the Faltings topos associated to $X^\circ_{\overline{K}}\to X$. In the first version of their work, they asked for an explicit local version. Scholze and Caraiani \cite{caraiani2017generic} constructed independently a relative Hodge-Tate filtration for proper smooth morphisms of smooth adic spaces, and Scholze announced that he can give a local version, answering the question of Abbes-Gros. Our construction is obtained by applying our cohomological descent result for Faltings ringed topos to the global relative Hodge-Tate filtration of Abbes-Gros. In a new version of their manuscript, Abbes-Gros gave a third construction of the local Hodge-Tate filtration in a slightly more restrictive framework, using a cohomological descent result which is a special case of ours.
\end{mypara}

\begin{mypara}
	Faltings ringed topos plays a central role in the proof of the Hodge-Tate decomposition. Let $X^\circ\to X$ be an open immersion of coherent schemes over $\spec(K)\to \spec(\ca{O}_K)$ such that the associated log scheme $(X,\scr{M}_X)$ endowed with compactifying log structure is adequate. We set $Y=X_{\overline{K}}$. The Faltings ringed site $(\fal_{Y \to X}^\et,\falb)$ was constructed by Faltings and developed by Abbes-Gros \cite[\Luoma{6}]{abbes2016p}. Faltings designed it as a bridge between the $p$-adic \'etale cohomology of $Y$ and differential forms of $X$. Concretely, these links are established through natural morphisms of sites
	\begin{align}
		Y_\et\stackrel{\psi}{\longrightarrow} \fal_{Y \to X}^\et \stackrel{\sigma}{\longrightarrow} X_\et
	\end{align}
	which satisfy the following properties:
	\begin{enumerate}
		\renewcommand{\labelenumi}{{\rm(\theenumi)}}
		\item (Faltings' comparison theorem, \cite[Thm.8, page 223]{faltings2002almost}, \cite[4.8.13]{abbes2020suite}). Assume that $X$ is proper over $\ca{O}_K$. For any finite locally constant abelian sheaf $\bb{F}$ on $Y_\et$, there exists a canonical morphism
		\begin{align}\label{eq:1.2.1}
			\rr\Gamma(Y_\et,\bb{F})\otimes^\dl_{\bb{Z}}\ca{O}_{\overline{K}}\longrightarrow \rr\Gamma(\fal_{Y \to X}^\et,\psi_*\bb{F}\otimes_{\bb{Z}}\falb),
		\end{align}
		which is an almost isomorphism, that is, the cohomology groups of its cone are killed by $p^r$ for any rational number $r>0$.
		\item (Faltings' computation of Galois cohomology, \cite[6.3.8]{abbes2020suite}). There exists a canonical homomorphism of $\ca{O}_X\otimes_{\ca{O}_K}\ca{O}_{\overline{K}}$-modules
		\begin{align}
			\Omega^q_{(X,\scr{M}_X)/(S,\scr{M}_S)}\otimes_{\ca{O}_K}\ca{O}_{\overline{K}}/p^n\ca{O}_{\overline{K}}\longrightarrow \rr^q\sigma_*(\falb/p^n\falb)
		\end{align}
		whose kernel and cokernel are killed by $p^r$ for any rational number $r>\frac{2\dim(Y)+1}{p-1}$.
	\end{enumerate}
Observing that $\bb{Z}/p^n\bb{Z}=\psi_*(\bb{Z}/p^n\bb{Z})$, Faltings deduced the Hodge-Tate decomposition from the degeneration and splitting of the Cartan-Leray spectral sequence for the composed functor $\rr\Gamma(X_\et,-)\circ \rr\sigma_*$, later named the Hodge-Tate spectral sequence by Scholze. Using de Jong's alteration theorem, one can deduce the Hodge-Tate decomposition for a general proper smooth $K$-scheme by reducing to the case where it admits a semi-stable model (cf. \cite[A5]{tsuji2002semi}).

%Abbes and Gros \cite[6.7.5]{abbes2020suite} generalized the Hodge-Tate spectral sequence to relative settings. Their work requires adequate models over $\ca{O}_K$. The starting point of this work is to see whether the relative Hodge-Tate spectral sequence can be made free of models. We give a partial positive answer to this question by proving that the structural sheaf on Faltings site satisfies cohomological descent along proper hypercoverings. Our approach is inspired by Deligne's generalization \cite{deligne1974hodge3} of classical Hodge theory to singular varieties using cohomological descent of \'etale cohomology and Hironaka's resolution of singularities.
\end{mypara}

\begin{mypara}
	In order to state our cohomological descent result, we recall now the definition of the Faltings site associated to a morphism of coherent schemes $Y\to X$ (cf. \ref{defn:falsite}). Let $\fal_{Y\to X}^\et$ be the category of morphisms of coherent schemes $V\to U$ over $Y\to X$, i.e. commutative diagrams
	\begin{align}
		\xymatrix{
			V\ar[r]\ar[d]& U\ar[d]\\
			Y\ar[r]& X
		}
	\end{align}
	such that $U$ is \'etale over $X$ and that $V$ is finite \'etale over $Y\times_X U$. We endow  $\fal_{Y\to X}^\et$ with the topology generated by the following types of families of morphisms
	\begin{enumerate}
		\item[\rm{(v)}] $\{(V_m \to U) \to (V \to U)\}_{m \in M}$, where $M$ is a finite set and $\coprod_{m\in M} V_m\to V$ is surjective;
		\item[\rm{(c)}] $\{(V\times_U{U_n} \to U_n) \to (V \to U)\}_{n \in N}$, where $N$ is a finite set and $\coprod_{n\in N} U_n\to U$ is surjective.
	\end{enumerate}
	Consider the presheaf $\falb$ on $\fal_{Y\to X}^\et$ defined by
	\begin{align}
		\falb(V \to U) = \Gamma(U^V , \ca{O}_{U^V}),
	\end{align}
	where $U^V$ is the integral closure of $U$ in $V$. It is indeed a sheaf of rings, the structural sheaf of $\fal_{Y\to X}^\et$ (cf. \ref{prop:falb-sheaf}).
\end{mypara}

\begin{mypara}
	Recall that the cohomological descent of \'etale cohomology along proper hypercoverings can be generalized as follows: for a coherent $S$-scheme, we endow the category of coherent $S$-schemes $\schqcqs_{/S}$ with Voevodsky's \emph{h-topology} which is generated by \'etale coverings and proper surjective morphisms of finite presentation. Then, for any torsion abelian sheaf $\ca{F}$ on $S_\et$, denoting by $a:(\schqcqs_{/S})_{\mrm{h}}\to S_\et$ the natural morphism of sites, the adjunction morphism $\ca{F}\to \rr a_*a^{-1}\ca{F}$ is an isomorphism. 
	
	This result remains true for a finer topology, the \emph{v-topology}. A morphism of coherent schemes $T\to S$ is called a v-covering if for any morphism $\spec(A)\to S$ with $A$ a valuation ring, there exists an extension of valuation rings $A\to B$ and a lifting $\spec(B)\to T$. In fact, a v-covering is a limit of h-coverings (cf. \ref{lem:descendv-cov}). We will describe the cohomological descent for $\falb$ using a new site built from the v-topology, which can be regarded as a scheme theoretic analogue of the v-site of adic spaces (cf. \cite[8.1, 14.1, 15.5]{scholze2021diamond}).
\end{mypara}

\begin{mydefn}[{cf. \ref{defn:falh-v-top}}]\label{defn:intro-falh-v-top}
	Let $S^\circ\to S$ be an open immersion of coherent schemes such that $S$ is integrally closed in $S^\circ$. We define a site $\falh_{S^\circ \to S}$ as follows: 
	\begin{enumerate}
		\renewcommand{\labelenumi}{{\rm(\theenumi)}}
		\item The underlying category is formed by coherent $S$-schemes $T$ which are integrally closed in $S^\circ\times_S T$.
		\item The topology is generated by covering families $\{T_i\to T\}_{i\in I}$ in the v-topology.
	\end{enumerate}
	We call $\falh_{S^\circ \to S}$ the \emph{v-site of $S^\circ$-integrally closed coherent $S$-schemes}, and we call the sheaf $\falhb$ on $\falh_{S^\circ \to S}$ associated to the presheaf $T\mapsto \Gamma(T,\ca{O}_T)$ the structural sheaf of $\falh_{S^\circ \to S}$.
\end{mydefn}

\begin{mypara}
	Let $p$ be a prime number, $\overline{\bb{Z}_p}$ the integral closure of $\bb{Z}_p$ in an algebraic closure $\overline{\bb{Q}_p}$ of $\bb{Q}_p$. We take $S^\circ=\spec(\overline{\bb{Q}_p})$ and $S=\spec(\overline{\bb{Z}_p})$.
	Consider a diagram of coherent schemes
	\begin{align}
		\xymatrix{
			Y\ar[r]\ar[d]& X^Y\ar[r]\ar[d]& X\\
			\spec(\overline{\bb{Q}_p})\ar[r]&\spec(\overline{\bb{Z}_p})&
		}
	\end{align}
	where $X^Y$ is the integral closure of $X$ in $Y$ and the square is Cartesian (we don't impose any condition on the regularity or finiteness of $Y$ or $X$). The functor $\varepsilon^+: \fal_{Y\to X}^\et\to \falh_{Y \to X^{Y}}$ sending $V\to U$ to $U^V$ defines a natural morphism of ringed sites
	\begin{align}
		\varepsilon: (\falh_{Y \to X^{Y}},\falhb)\longrightarrow (\fal_{Y\to X}^\et,\falb).
	\end{align}
	Our cohomological descent results are stated as follows, which can be regarded as a scheme theoretic analogue of the cohomological descent result for the pro-\'etale site of an adic space (cf. \cite[8.8, 14.7, 15.5]{scholze2021diamond}):
\end{mypara}

\begin{mythm}[{Cohomological descent for Faltings ringed sites, cf. \ref{cor:falh-et}}]\label{thm:intro-falh-et}
	For any finite locally constant abelian sheaf $\bb{L}$ on $\fal_{Y\to X}^\et$, the canonical morphism
	\begin{align}
		\bb{L}\otimes_{\bb{Z}} \falb\longrightarrow \rr\varepsilon_*(\varepsilon^{-1}\bb{L}\otimes_{\bb{Z}}\falhb)
	\end{align}
	is an almost isomorphism.
\end{mythm}

\begin{mycor}[{cf. \ref{cor:resolution}}]\label{cor:intro-resolution}
	For any proper hypercovering $X_\bullet\to X$, if $a:\fal_{Y_\bullet\to X_\bullet}^\et\to \fal_{Y\to X}^\et$ denotes the augmentation of simplicial site where $Y_\bullet=Y\times_X X_\bullet$, then the canonical morphism
	\begin{align}
		\bb{L}\otimes_{\bb{Z}} \falb\longrightarrow \rr a_*(a^{-1}\bb{L}\otimes_{\bb{Z}}\falb_\bullet)
	\end{align}
	is an almost isomorphism.
\end{mycor}

The key ingredient of our proof of \ref{thm:intro-falh-et} is the descent of perfectoid algebras in the arc-topology (a topology finer than the v-topology) due to Bhatt-Scholze \cite[8.10]{bhattscholze2019prisms} (cf. \ref{thm:arc-descent-perf}). The analogue in characteristic $p$ of \ref{thm:intro-falh-et} is Gabber's computation of the cohomology of the structural sheaf in the h-topology (cf. \ref{sec:arc-descent}). Theorem \ref{thm:intro-falh-et} allows us to descend important results for Faltings sites with nice models to Faltings sites associated to general models. On the other hand, its proof shows how to compute the cohomologies of Faltings ringed sites locally. Using Abhyankar's lemma, one can treat the open case which in the generic fibre is the complement of a normal crossings divisor.

\begin{mycor}[{cf. \rm\ref{thm:acyclic}}]\label{cor:intro-acyclic}
	Under the assumptions in {\rm\ref{thm:main-intro}} and with the same notation, for any integer $n>0$, the canonical morphism
	\begin{align}
		A/p^n A\longrightarrow \rr\Gamma(\fal_{V\to U}^\et,\falb/p^n\falb)
	\end{align}
	is an almost isomorphism.
\end{mycor}

Thus, we apply the derived functor $\rr\Gamma(\fal_{V\to U}^\et,-)$ to the global relative Hodge-Tate filtration defined on the Faltings ringed site by Abbes-Gros, and then we obtain the local version \ref{thm:main-intro}.

\begin{mypara}
	On the other hand, we use \ref{thm:intro-falh-et} to generalize Faltings' comparison theorem in the absolute case. Let $A$ be a valuation ring extension of $\overline{\bb{Z}_p}$ with algebraically closed fraction field. Consider a Cartesian square of coherent schemes
	\begin{align}
		\xymatrix{
			Y\ar[r]\ar[d]& X\ar[d]\\
			\spec(A[\frac{1}{p}])\ar[r]&\spec(A)
		}
	\end{align}
\end{mypara}

\begin{mythm}[{Faltings' comparison theorem in the absolute case, cf. \ref{thm:fal-comp-abs}}]\label{thm:intro-fal-comp-abs}
	Assume that $X$ is proper of finite presentation over $A$. Then, for any finite locally constant abelian sheaf $\bb{F}$ on $Y_\et$, there exists a canonical morphism
	\begin{align}\label{eq:1.6.1}
		\rr\Gamma (Y_\et, \bb{F})\otimes^{\dl}_{\bb{Z}}A \longrightarrow \rr\Gamma (\fal_{Y \to X}^\et, \psi_*\bb{F}\otimes_{\bb{Z}}\falb),
	\end{align}
	which is an almost isomorphism.
\end{mythm}

We remark that the natural morphism $\psi:Y_\et\to \fal_{Y \to X}^\et$ induces an equivalence of the categories of finite locally constant abelian sheaves on $Y_\et$ and $\fal_{Y \to X}^\et$ (cf. \ref{prop:loc-sys-trans}),
\begin{align}
	\xymatrix{
		\locsys(Y_\et)\ar@<0.5ex>[r]^-{\psi_*}&\locsys(\fal_{Y \to X}^\et)\ar@<0.5ex>[l]^-{\psi^{-1}}.
	}
\end{align}
As a continuation of the work of Abbes-Gros, the canonical morphism \eqref{eq:1.6.1} (refered as \emph{Faltings' comparison morphism}) is constructed using the acyclicity of $\psi$ for $\bb{F}$, i.e. $\psi_*\bb{F}=\rr\psi_*\bb{F}$ (so that $\rr\Gamma (Y_\et, \bb{F})=\rr\Gamma (\fal_{Y \to X}^\et, \psi_*\bb{F})$), which is a consequence of Achinger's result on $K(\pi,1)$-schemes (cf. \ref{defn:easy-fal-comp-mor} and \ref{cor:achinger}). We also propose a new way to construct Faltings' comparison morphism in the derived category of almost modules using our cohomological descent result \ref{thm:intro-falh-et}, which avoids using the acyclicity of $\psi$. Indeed, there are natural morphisms of sites 
\begin{align}\label{diam:1.7.1}
	\xymatrix{
		(\schqcqs_{/Y})_{\mrm{v}}\ar[r]^-{a}\ar[d]_-{\Psi} &Y_\et\ar[d]^-{\psi}\\
		\falh_{Y \to X^Y}\ar[r]^-{\varepsilon} & \fal_{Y \to X}^\et
	}
\end{align}
and $\Psi$ is acyclic for any torsion abelian sheaf $\ca{F}$ on $Y_\et$, i.e. $\Psi_*(a^{-1}\ca{F})=\rr\Psi_*(a^{-1}\ca{F})$, which allows more general coefficients and whose proof is much easier than that of $\psi$ (cf. \ref{prop:Psi-const-equal}). We remark that this new construction won't give us a ``real morphism'' \eqref{eq:1.6.1} but a canonical morphism in the derived category of almost modules (cf. \ref{para:comparison-mor-rel}).

We briefly explain the strategy for proving \ref{thm:intro-fal-comp-abs}:
\begin{enumerate}
	\renewcommand{\labelenumi}{{\rm(\theenumi)}}
	\item Firstly, we use de Jong-Gabber-Illusie-Temkin's alteration theorem for morphisms of schemes \cite[\Luoma{10}.3]{gabber2014travaux} to obtain a proper surjective morphism of finite presentation $X'\to X$ such that the morphism $X'\to \spec(A)$ is the cofiltered limit of a system of ``nice'' morphisms $X'_\lambda\to T_\lambda$ of ``nice'' models over $\ca{O}_{K_\lambda}$, where $K_\lambda$ is a finite extension of $\bb{Q}_p$ (cf. \ref{prop:val-lim-ss}).
	\item Then, we can apply Faltings' comparison theorem in the relative case to the ``nice'' morphisms $X'_\lambda\to T_\lambda$ (formulated by Faltings \cite[Thm.6, page 266]{faltings2002almost} and proved by Abbes-Gros \cite[5.7.3]{abbes2020suite}, cf. \ref{thm:abess-gros}). By a limit argument, we get the comparison theorem for $X'$.
	\item Finally, using our cohomological descent result \ref{cor:intro-resolution}, we deduce the comparison theorem for $X$.
\end{enumerate}

\begin{mypara}
	The site $\falh_{Y\to X^Y}$ is also appropriate to globalize Faltings' comparison theorem. Consider a Cartesian square of coherent schemes
	\begin{align}
		\xymatrix{
			Y'\ar[r]\ar[d]& X'\ar[d]\\
			Y\ar[r]&X
		}
	\end{align}
	where $Y\to X$ is Cartesian over $\spec(\overline{\bb{Q}_p})\to \spec(\overline{\bb{Z}_p})$. In particular, there is a natural morphism of ringed sites by the functoriality of \eqref{diam:1.7.1},
	\begin{align}
		f_{\falh}:(\falh_{Y' \to X'^{Y'}},\falhb')\longrightarrow (\falh_{Y \to X^Y},\falhb).
	\end{align}
\end{mypara}
\begin{mythm}[{cf. \ref{thm:hard-comp-rel}}]\label{thm:intro-hard-comp-rel}
	Assume that $X'\to X$ is proper of finite presentation. Let $\ca{F}'$ be a torsion abelian sheaf on $Y'_\et$ and $\scr{F}'=\Psi'_*a'^{-1}\ca{F}'$ {\rm(cf. \eqref{diam:1.7.1})}. Then, the canonical morphism
	\begin{align}\label{eq:intro-rel-comp}
		(\rr f_{\falh *}\scr{F}')\otimes_{\bb{Z}}^\dl \falhb \longrightarrow \rr f_{\falh *}(\scr{F}'\otimes_{\bb{Z}} \falhb')
	\end{align}
	is an almost isomorphism.
\end{mythm}

We remark that if $\ca{F}'=\bb{Z}/p^n\bb{Z}$ then $\scr{F}'=\bb{Z}/p^n\bb{Z}$ (cf. \ref{prop:Psi-const-equal}), and that $\rr^q f_{\falh *}\scr{F}'$ is the sheafification of \'etale cohomologies of $Y'$ over $Y$ with coefficient $\ca{F}'$ in the v-topology (cf. \ref{lem:falh-presheaf}). Very roughly speaking, objects of $\falh_{Y \to X^Y}$ are ``locally'' the spectrums of valuation rings, and the ``stalks'' of \eqref{eq:intro-rel-comp} are Faltings' comparison morphisms \eqref{eq:1.6.1} when $\ca{F}'$ is finite locally constant (cf. \ref{lem:comparison-mor-rel}). Theorem \ref{thm:intro-hard-comp-rel} can be regarded as a scheme theoretical analogue of Scholze's comparison theorem for $p$-adic \'etale cohomology of a morphism of rigid analytic varieties \cite[3.13]{scholze2013perfsurv}.

Finally, we generalize Faltings' comparison theorem in the relative case using \ref{thm:intro-falh-et} and \ref{thm:intro-hard-comp-rel}.
\begin{mythm}[{Faltings' comparison theorem in the relative case, cf. \ref{thm:fal-comp-rel} and \ref{rem:fal-comp-rel}}]\label{thm:intro-fal-comp-rel}
	Assume that $Y'\to Y$ is smooth and that $X'\to X$ is proper of finite presentation. Then, for any finite locally constant abelian sheaf $\bb{F}'$ on $Y'_\et$, there exists a canonical morphism
	\begin{align}
		\xymatrix{
		(\rr\psi_*\rr f_{\et *} \bb{F}')\otimes^{\dl}_{\bb{Z}}\falb \longrightarrow \rr f_{\fal *} (\psi'_*\bb{F}'\otimes_{\bb{Z}}\falb'),
		}
	\end{align}
	which is an almost isomorphism, and	where $f_\et:Y'_\et\to Y_\et$ and $f_\fal:\fal_{Y' \to X'}^\et\to \fal_{Y \to X}^\et$ are the natural morphisms of sites. In particular, there exists a canonical morphism
	\begin{align}
		(\psi_*\rr^q f_{\et *} \bb{F}')\otimes_{\bb{Z}}\falb \longrightarrow \rr^q f_{\fal *} (\psi'_*\bb{F}'\otimes_{\bb{Z}}\falb'),
	\end{align}
	which is an almost isomorphism, for any integer $q$.
\end{mythm}

\begin{mypara}
	The paper is structured as follows. In section \ref{sec:v-site}, we establish the foundation of the site $\falh_{S^\circ \to S}$, where proposition \ref{prop:Psi-const-equal} discussing the cohomological properties of $\Psi:(\schqcqs_{/S^\circ})_{\mrm{v}}\to\falh_{S^\circ \to S}$ is the key to our new construction of Faltings' comparison morphism (cf. \ref{para:comparison-mor-rel}). Sections \ref{sec:arc-descent} and \ref{sec:perfectoid} are devoted to a detailed proof of the arc-descent for perfectoid algebras following Bhatt-Scholze \cite[8.10]{bhattscholze2019prisms}. Since we use the language of schemes, the terminology ``pre-perfectoid'' is introduced for those algebras whose $p$-adic completions are perfectoid. In sections \ref{sec:covanishing} and \ref{sec:faltings-site}, we review the definition and some basic properties of Faltings sites and we introduce a pro-version of Faltings site to evaluate the structural sheaf on the spectrums of pre-perfectoid algebras. Then, we prove our cohomological descent result in section \ref{sec:coh-descent}. In section \ref{sec:log-geo}, we review de Jong-Gabber-Illusie-Temkin's alteration theorem and apply it to schemes over a valuation ring of height $1$. Section \ref{sec:abs-comp} is devoted to proving our generalization of Faltings' comparison theorem in the absolute case. In section \ref{sec:rel-comp}, we give a new construction of Faltings' comparison morphism and our generalization of Faltings' comparison theorem in the relative case. Finally, we deduce from the global relative Hodge-Tate filtration an explicit local version in section \ref{sec:loc-hodge-tate-fil}.
\end{mypara}

%\begin{mypara}
%	Our work suggests that the site $\falh_{S^\circ \to S}$ could play an important role in Faltings' $p$-adic Hodge theory, as it allows us to work with general models (at least in the ``comparison'' part). A recent work of Guo \cite{guo2019hodgetate} generalized the Hodge-Tate decomposition to singular rigid analytic varieties. Thus, it seems reasonable that we could use $\falh_{S^\circ \to S}$ to generalize the ``Galois cohomology'' part of Faltings' $p$-adic Hodge theory and relate it to Deligne-du Bois complex, and the final goal is to generalize Abbes-Gros' relative Hodge-Tate spectral sequence to morphisms of singular varieties. If one would like to be more optimistic about the site $\falh_{S^\circ \to S}$, then it is interesting to look for an intrinsic proof of Faltings' comparison theorems on $\falh_{S^\circ \to S}$ instead of doing alterations and cohomological descent. A further possible generalization would be over a general base $S^\circ\to S$ (not only Cartesian over $\spec(\overline{\bb{Q}_p})\to \spec(\overline{\bb{Z}_p})$). On the other hand, it seems that the site $\falh_{S^\circ \to S}$ is closely related to Bhatt-Scholze's perfectoidization in their prismatic cohomology theory \cite[8]{bhattscholze2019prisms}, and further concrete relations are waiting to be explored. The author is still studying these problems.
%\end{mypara}

\subsection*{Acknowledgements}
This work is part of my thesis prepared at Universit\'e Paris-Saclay and Institut des Hautes \'Etudes Scientifiques. 
I would like to express my sincere gratitude to my doctoral supervisor, Ahmed~ Abbes, for his guidance to this project, his thorough review of this work and his plenty of helpful suggestions on both research and writing.

\section{Notation and Conventions}

\begin{mypara}\label{para:notation-val}
	We fix a prime number $p$ throughout this paper. For any monoid $M$, we denote by $\frob:M \to M$ the map sending an element $x$ to $x^p$ and we call it the \emph{Frobenius} of $M$. For a ring $R$, we denote by $R^\times$ the group of units of $R$. 
	A ring $R$ is called \emph{absolutely integrally closed} if any monic polynomial $f\in R[T]$ has a root in $R$ (\cite[\href{https://stacks.math.columbia.edu/tag/0DCK}{0DCK}]{stacks-project}). We remark that quotients, localizations and products of absolutely integrally closed rings are still absolutely integrally closed.
	
	Recall that a valuation ring is a domain $V$ such that for any element $x$ in its fraction field, if $x\notin V$ then $x^{-1}\in V$. The family of ideals of $V$ is totally ordered by the inclusion relation (\cite[\Luoma{6}.\textsection1.2, Thm.1]{bourbaki2006commalg5-7}). In particular, a radical ideal of $V$ is a prime ideal. Moreover, any quotient of $V$ by a prime ideal and any localization of $V$ are still valuations rings (\cite[\href{https://stacks.math.columbia.edu/tag/088Y}{088Y}]{stacks-project}). We remark that $V$ is normal, and that $V$ is absolutely integrally closed if and only if its fraction field is algebraically closed.
	An \emph{extension of valuation rings} is an injective and local homomorphism of valuation rings.
\end{mypara}

\begin{mypara}\label{para:notation-intclos}
	Following \cite[\Luoma{6}.1.22]{sga4-2}, a \emph{coherent} scheme (resp. morphism of schemes) stands for a quasi-compact and quasi-separated scheme (resp. morphism of schemes). For a coherent morphism $Y \to X$ of schemes, we denote by $X^Y$ the integral closure of $X$ in $Y$ (\cite[\href{https://stacks.math.columbia.edu/tag/0BAK}{0BAK}]{stacks-project}). For an $X$-scheme $Z$, we say that $Z$ is \emph{$Y$-integrally closed} if $Z=Z^{Y \times_X Z}$.
\end{mypara}

\begin{mypara} 
	Throughout this paper, we fix two universes $\bb{U}$ and $\bb{V}$ such that the set of natural numbers $\bb{N}$ is an element of $\bb{U}$ and that $\bb{U}$ is an element of $\bb{V}$  (\cite[\Luoma{1}.0]{sga4-1}). In most cases, we won't emphasize this set theoretical issue. Unless stated otherwise, we only consider $\bb{U}$-small schemes and we denote by $\sch$ the category of $\bb{U}$-small schemes, which is a $\bb{V}$-small category.
\end{mypara}
	
\begin{mypara} \label{para:finite-complete}
	Let $C$ be a category. We denote by $\widehat{C}$ the category of presheaves of $\bb{V}$-small sets on $C$. If $C$ is a $\bb{V}$-site (\cite[\Luoma{2}.3.0.2]{sga4-1}), we denote by $\widetilde{C}$ the topos of sheaves of $\bb{V}$-small sets on $C$. We denote by $h^C : C \to \widehat{C}$, $x\mapsto h_x^C$ the Yoneda embbeding (\cite[\Luoma{1}.1.3]{sga4-1}), and by $\widehat{C} \to \widetilde{C}$, $\ca{F} \mapsto \ca{F}^\ash$ the sheafification functor (\cite[\Luoma{2}.3.4]{sga4-1}). Unless stated otherwise, a site in this paper stands for a site where \emph{all finite limits are representable}.
\end{mypara}

\begin{mypara}\label{para:topos}
	Let $u^+ : C \to D$ be a functor of categories. We denote by $u^{\psh} : \widehat{D} \to \widehat{C}$ the functor that associates to a presheaf $\ca{G}$ of $\bb{V}$-small sets on $D$ the presheaf $u^{\psh} \ca{G} = \ca{G} \circ u^+$. If $C$ is $\bb{V}$-small and $D$ is a $\bb{V}$-category, then $u^{\psh}$ admits a left adjoint $u_{\psh}$ \cite[\href{https://stacks.math.columbia.edu/tag/00VC}{00VC}]{stacks-project} and a right adjoint ${}_{\psh}u$ \cite[\href{https://stacks.math.columbia.edu/tag/00XF}{00XF}]{stacks-project} (cf. \cite[\Luoma{1}.5]{sga4-1}). So we have a sequence of adjoint functors
	\begin{align}\label{eq:threeadj}
		u_{\psh},\ u^{\psh},\ {}_{\psh}u .
	\end{align}
	If moreover $C$ and $D$ are $\bb{V}$-sites, then we denote by $u_{\ssh} , u^{\ssh} , {}_{\ssh}u$ the functors of the topoi $\widetilde{C}$ and $\widetilde{D}$ of sheaves of $\bb{V}$-small sets induced by composing the sheafification functor with the functors $u_{\psh} , u^{\psh} , {}_{\psh}u$ respectively. As we only consider finite complete sites, we say that the functor $u^+$ gives a \emph{morphism of sites}, if $u^+$ is left exact and preserves covering families (\cite[\Luoma{4}.4.9.2]{sga4-1}). Then, we denote by 
	\begin{align}
		u=(u^{-1}, u_*) : \widetilde{D} \to \widetilde{C}
	\end{align}
	the associated morphism of topoi, where $u^{-1}=u_{\ssh}$ and $u_*=u^{\ssh}=u^{\psh}|_{\widetilde{D}}$. If moreover $u$ is a morphism of ringed topoi, then we denote by $u^*=\ca{O}_{\widetilde{D}}\otimes_{u^{-1}\ca{O}_{\widetilde{C}}}u^{-1}$ the pullback functor of modules. We remark that the notation here, adopted by \cite{stacks-project}, is slightly different with that in \cite{sga4-1} (cf. \cite[\href{https://stacks.math.columbia.edu/tag/0CMZ}{0CMZ}]{stacks-project}).
\end{mypara}

\begin{mypara}\label{para:isogeny}
	Let $C$ be an additive category. The \emph{catgory of objects of $C$ up to isogeny} (cf. \cite[\Luoma{3}.6]{abbes2016p}) is the category $C_{\bb{Q}}$ with a functor $F:C\to C_{\bb{Q}}$ (called the localization functor) such that $\ob(C_{\bb{Q}})=\ob(C)$ and that
	\begin{align}
		\ho_{C_{\bb{Q}}}(M_{\bb{Q}},N_{\bb{Q}})=\ho_C(M,N)\otimes_{\bb{Z}}\bb{Q},
	\end{align}
	where we denote by $M_{\bb{Q}}$ the image of an object $M$ of $C$ via $F$ in $C_{\bb{Q}}$.
	
	For a ringed site $(C,\ca{O})$, we denote by $\ca{O}\module_{\bb{Q}}$ the category of $\ca{O}$-modules up to isogeny, whose objects are called \emph{$\ca{O}_{\bb{Q}}$-modules}. It is an abelian category and the localization functor $\ca{O}\module\to \ca{O}\module_{\bb{Q}}$ sends injective objects to injective objects. We remark that if $\ca{O}$ is a $\bb{Q}$-module, then $\ca{O}\module\to \ca{O}\module_{\bb{Q}}$ is an equivalence. A morphism of ringed sites $u:(C,\ca{O})\to (C', \ca{O}')$ induces a pair of adjoint functors 
	\begin{align}
		u^*:\ca{O}'\module_{\bb{Q}}\to \ca{O}\module_{\bb{Q}},\ u_*:\ca{O}\module_{\bb{Q}}\to \ca{O}'\module_{\bb{Q}}.
	\end{align}
	%We simply denote this pair of functors by $u_{\bb{Q}}=(u_{\bb{Q}}^*,u_{\bb{Q}*}):(E,A_{\bb{Q}})\to (E',A'_{\bb{Q}})$. 
	The derived functor $\rr u_{*}$ commutes with the localization functor.
\end{mypara}

\section{The v-site of Integrally Closed Schemes}\label{sec:v-site}

\begin{mydefn}\label{defn:h-v-arc}
	Let $X \to Y$ be a quasi-compact morphism of schemes. 
	\begin{enumerate}
		\item We say that $X \to Y$ is a \emph{v-covering}, if for any valuation ring $V$ and any morphism $\spec(V) \to Y$, there exists an extension of valuation rings $V\to W$ (\ref{para:notation-val}) and a commutative diagram (cf. \cite[\href{https://stacks.math.columbia.edu/tag/0ETN}{0ETN}]{stacks-project})
		\begin{align}\label{eq:1.1.1}
			\xymatrix{
				\spec(W) \ar[r]\ar[d]& X\ar[d]\\
				\spec(V) \ar[r] & Y
			}
		\end{align}
		\item Let $\pi$ be an element of $\Gamma(Y,\ca{O}_Y)$. We say that $X \to Y$ is an \emph{arc-covering} (resp. \emph{$\pi$-complete arc-covering}), if for any valuation ring (resp. $\pi$-adically complete valuation ring) $V$ of height $\leq 1$ and any morphism $\spec(V) \to Y$, there exists an extension of valuation rings (resp. $\pi$-adically complete valuation rings) $V\to W$ of height $\leq 1$ and a commutative diagram \eqref{eq:1.1.1} (cf. \cite[1.2]{bhattmathew2020arc}, \cite[2.2.1]{cesnaviciusscholze2019purity}). \label{defn:h-v-arc-arc}
		\item We say that $X \to Y$ is an \emph{h-covering}, if it is a v-covering and locally of finite presentation (cf. \cite[\href{https://stacks.math.columbia.edu/tag/0ETS}{0ETS}]{stacks-project}).
	\end{enumerate}	
\end{mydefn}
We note that an arc-covering is simply a $0$-complete arc-covering.

\begin{mylem}\label{lem:arc-parc-basic}
	Let $Z\stackrel{g}{\longrightarrow}Y \stackrel{f}{\longrightarrow}X$ be quasi-compact morphisms of schemes, $\pi \in \Gamma(X,\ca{O}_X)$, $\tau \in \{$h, v, $\pi$-complete arc$\}$.
	\begin{enumerate}
		\renewcommand{\labelenumi}{{\rm(\theenumi)}}
		\item If $f$ is a $\tau$-covering, then any base change of $f$ is also a $\tau$-covering.
		\item If $f$ and $g$ are $\tau$-coverings, then $f \circ g$ is also a $\tau$-covering.
		\item If $f\circ g$ is a $\tau$-covering (and if $f$ is locally of finite presentation when $\tau=\trm{h}$), then $f$ is also a $\tau$-covering.
	\end{enumerate}
\end{mylem}
\begin{proof}
	It follows directly from the definitions.
\end{proof}

\begin{mypara}\label{para:h-v-arc-topology}
	Let $\schqcqs$ be the category of coherent $\bb{U}$-small schemes, $\tau \in \{\trm{h, v, arc}\}$. We endow $\schqcqs$ with the $\tau$-topology generated by the pretopology formed by families of morphisms $\{X_i \to X\}_{i \in I}$ with $I$ finite such that $\coprod_{i \in I} X_i \to X$ is a $\tau$-covering, and we denote the corresponding site by $\schqcqs_\tau$. It is clear that a morphism $Y \to X$ (which is locally of finite presentation if $\tau=\trm{h}$) is a $\tau$-covering if and only if $\{Y \to X\}$ is a covering family in $\schqcqs_\tau$ by \ref{lem:arc-parc-basic} and \cite[\Luoma{2}.1.4]{sga4-1}.
	
	For any coherent $\bb{U}$-small scheme $X$, we endow the category $\schqcqs_{/X}$ of objects of $\schqcqs$ over $X$ with the topology induced by the $\tau$-topology of $\schqcqs$, i.e. the topology generated by the pretopology formed by families of $X$-morphisms $\{Y_i \to Y\}_{i \in I}$ with $I$ finite such that $\coprod_{i \in I} Y_i \to Y$ is a $\tau$-covering (\cite[\Luoma{3}.5.2]{sga4-1}). For any sheaf $\ca{F}$ of $\bb{V}$-small abelian groups on the site $(\schqcqs_{/X})_\tau$, we denote its $q$-th cohomology by $H^q_\tau(X, \ca{F})$.
	
\end{mypara}

\begin{mylem}\label{lem:arc-parc}
	Let $f: X \to Y$ be a quasi-compact morphism of schemes, $\pi \in \Gamma(Y,\ca{O}_Y)$.
	\begin{enumerate}
		\renewcommand{\labelenumi}{{\rm(\theenumi)}}
		\item If $f$ is proper surjective or faithfully flat, then $f$ is a v-covering.\label{lem:arc-parc-v}
		\item If $f$ is an h-covering and $Y$ is affine, then there exists a proper surjective morphism $Y' \to Y$ of finite presentation and a finite affine open covering $Y' = \bigcup_{r=1}^n Y'_r$ such that $Y'_r \to Y$ factors through $f$ for each $r$. \label{lem:arc-parc-h}
		\item If $f$ is an h-covering and if there exists a directed inverse system $(f_\lambda: X_\lambda \to Y_\lambda)_{\lambda \in \Lambda}$ of finitely presented morphisms of coherent schemes with affine transition morphisms $\psi_{\lambda'\lambda}: X_{\lambda'} \to X_\lambda$ and $\phi_{\lambda'\lambda}: Y_{\lambda'} \to Y_\lambda$ such that $X=\lim X_\lambda$, $Y=\lim Y_\lambda$ and that $f_\lambda$ is the base change of $f_{\lambda_0}$ by $\phi_{\lambda\lambda_0}$ for some index $\lambda_0 \in \Lambda$ and any $\lambda \geq \lambda_0$, then there exists an index $\lambda_1\geq \lambda_0$ such that $f_\lambda$ is an h-covering for any $\lambda \geq \lambda_1$.\label{lem:arc-parc-limit-h}
		\item If $f$ is a v-covering, then it is a $\pi$-complete arc-covering.\label{lem:arc-parc-arc}
		\item Let $\pi'$ be another element of $\Gamma(Y,\ca{O}_Y)$ which divides $\pi$. If $f$ is a $\pi$-complete arc-covering, then it is a $\pi'$-complete arc-covering.\label{lem:arc-parc-pi}
		\item If $\spec(B) \to \spec(A)$ is a $\pi$-complete arc-covering, then the morphism $\spec(\widehat{B}) \to \spec(\widehat{A})$ between the spectrums of their $\pi$-adic completions is also a $\pi$-complete arc-covering.\label{lem:arc-parc-complete}
	\end{enumerate}
\end{mylem}
\begin{proof}
	(\ref{lem:arc-parc-v}), (\ref{lem:arc-parc-h}) are proved in \cite[\href{https://stacks.math.columbia.edu/tag/0ETK}{0ETK}, \href{https://stacks.math.columbia.edu/tag/0ETU}{0ETU}]{stacks-project} respectively.
	
	(\ref{lem:arc-parc-limit-h}) To show that one can take $\lambda_1\geq \lambda_0$ such that $f_{\lambda_1}$ is an h-covering, we may assume that $Y_{\lambda_0}$ is affine by replacing it by a finite affine open covering by \ref{lem:arc-parc-basic} and (\ref{lem:arc-parc-v}). Thus, applying (\ref{lem:arc-parc-h}) to the h-covering $f$ and using \cite[8.8.2, 8.10.5]{ega4-3}, there exists an index $\lambda_1\geq \lambda_0$, a proper surjective morphism $Y'_{\lambda_1} \to Y_{\lambda_1}$ and a finite affine open covering $Y'_{\lambda_1}=\bigcup_{r=1}^n Y'_{r\lambda_1}$ such that the morphisms $Y'_r\to Y' \to Y$ are the base changes of the morphisms $Y'_{r\lambda_1}\to Y'_{\lambda_1} \to Y_{\lambda_1}$ by the transition morphism $Y\to Y_{\lambda_1}$, and that $Y'_{r\lambda_1} \to Y_{\lambda_1}$ factors through $X_{\lambda_1}$. This shows that $f_{\lambda_1}$ is an h-covering by \ref{lem:arc-parc-basic} and (\ref{lem:arc-parc-v}).
	
	(\ref{lem:arc-parc-arc}) With the notation in \eqref{eq:1.1.1}, if $V$ is a $\pi$-adically complete valuation ring of height $\leq 1$ with maximal ideal $\ak{m}$, then since the family of prime ideals of $W$ is totally ordered by the inclusion relation (\ref{para:notation-val}), we take the maximal prime ideal $\ak{p}\subseteq W$ over $0 \subseteq V$ and the minimal prime ideal $\ak{q} \subseteq W$ over $\ak{m} \subseteq V$. Then, $\ak{p} \subseteq \ak{q}$ and $W' = (W/\ak{p})_{\ak{q}}$ over $V$ is an extension of valuation rings of height $\leq 1$. Since $\pi \in \ak{m}$ and $W'$ is of height $\leq 1$, the $\pi$-adic completion $\widehat{W'}$ is still a valuation ring extension of $V$ of height $\leq 1$ (cf. \cite[\Luoma{6}.\textsection5.3, Prop.5]{bourbaki2006commalg5-7}), which proves (\ref{lem:arc-parc-arc}).
	
	(\ref{lem:arc-parc-pi})  Since a $\pi'$-adically complete valuation ring $V$ is also $\pi$-adically complete  (\cite[\href{https://stacks.math.columbia.edu/tag/090T}{090T}]{stacks-project}), there exists a lifting $\spec(W) \to X$ for any morphism $\spec(V) \to Y$. After replacing $W$ by its $\pi'$-adic completion, the conclusion follows.
	
	(\ref{lem:arc-parc-complete}) Let $V$ be a $\pi$-adically complete valuation ring of height $\leq 1$. Given a morphism $\widehat{A} \to V$, there exists a lifting $B \to W$ where $V\to W$ is an extension of $\pi$-adically complete valuation rings of height $\leq 1$. It is clear that $B \to W$ factors through $\widehat{B}$, which proves (\ref{lem:arc-parc-complete}).
\end{proof}

\begin{mypara}\label{para:fp-h-site}
	Let $X$ be a coherent scheme, $\sch^{\mrm{fp}}_{/X}$ the full subcategory of $\schqcqs_{/X}$ formed by finitely presented $X$-schemes. We endow it with the topology generated by the pretopology formed by families of morphisms $\{Y_i \to Y\}_{i \in I}$ with $I$ finite such that $\coprod_{i \in I} Y_i \to Y$ is an h-covering, and we denote the corresponding site by $(\sch^{\mrm{fp}}_{/X})_{\mrm{h}}$. It is clear that this topology coincides with the topologies induced from $(\schqcqs_{/X})_{\mrm{v}}$ and from $(\sch^{\mrm{coh}}_{/X})_{\mrm{h}}$. The inclusion functors $(\sch^{\mrm{fp}}_{/X})_{\mrm{h}}\stackrel{\xi^+}{\longrightarrow}(\schqcqs_{/X})_{\mrm{h}}\stackrel{\zeta^+}{\longrightarrow}(\schqcqs_{/X})_{\mrm{v}}$ define morphisms of sites (\ref{para:topos})
	\begin{align}
		(\schqcqs_{/X})_{\mrm{v}}\stackrel{\zeta}{\longrightarrow}(\schqcqs_{/X})_{\mrm{h}} \stackrel{\xi}{\longrightarrow} (\sch^{\mrm{fp}}_{/X})_{\mrm{h}}.
	\end{align}
\end{mypara}

\begin{mylem}\label{lem:descendv-cov}
	Let $X$ be a coherent scheme. Then, for any covering family $\ak{U}=\{Y_i \to Y\}_{i \in I}$ in $(\schqcqs_{/X})_{\mrm{v}}$ with $I$ finite, 
	\begin{enumerate}
		\renewcommand{\theenumi}{\roman{enumi}}
		\renewcommand{\labelenumi}{{\rm(\theenumi)}}
		\item there exists a directed inverse system $(Y_\lambda)_{\lambda\in \Lambda}$ of finitely presented $X$-schemes with affine transition morphisms such that $Y=\lim Y_\lambda$, and
		\item for each $i \in I$, there exists a directed inverse system $(Y_{i\lambda})_{\lambda\in \Lambda}$ of finitely presented $X$-schemes with affine transition morphisms over the inverse system $(Y_\lambda)_{\lambda\in \Lambda}$ such that $Y_i = \lim Y_{i \lambda}$, and
		\item for each $\lambda\in \Lambda$, the family $\ak{U}_\lambda=\{Y_{i\lambda} \to Y_\lambda\}_{i \in I}$ is a covering in $(\sch^{\mrm{fp}}_{/X})_{\mrm{h}}$.
	\end{enumerate}
\end{mylem}
\begin{proof}
	We take a directed set $A$ such that for each $i\in I$, we can write $Y_i$ as a cofiltered limit of finitely presented $Y$-schemes $Y_i=\lim_{\alpha\in A} Y_{i\alpha}$ with affine transition morphisms (\cite[\href{https://stacks.math.columbia.edu/tag/09MV}{09MV}]{stacks-project}). We see that $\coprod_{i \in I} Y_{i\alpha} \to Y$ is an h-covering for each $\alpha \in A$ by \ref{lem:arc-parc-basic}.
	
	We write $Y$ as a cofiltered limit of finitely presented $X$-schemes $Y=\lim_{\beta \in B} Y_\beta$ with affine transition morphisms (\cite[\href{https://stacks.math.columbia.edu/tag/09MV}{09MV}]{stacks-project}). By \cite[8.8.2, 8.10.5]{ega4-3} and \ref{lem:arc-parc}.(\ref{lem:arc-parc-limit-h}), for each $\alpha \in A$, there exists an index $\beta_\alpha \in B$ such that the morphism $Y_{i\alpha}\to Y$ is the base change of a finitely presented morphism $Y_{i\alpha\beta_\alpha}\to Y_{\beta_\alpha}$ by the transition morphism $Y\to Y_{\beta_\alpha}$ for each $i\in I$, and that $\coprod_{i\in I} Y_{i\alpha\beta_\alpha}\to Y_{\beta_\alpha}$ is an h-covering. For each $\beta\geq \beta_\alpha$, let $Y_{i\alpha\beta}$ be the base change of $Y_{i\alpha\beta_\alpha}$ by $Y_\beta\to Y_{\beta_\alpha}$. 
	
	We define a category $\Lambda^\oppo$, whose set of objects is $\{(\alpha,\beta)\in A\times B\ |\ \beta \geq \beta_\alpha\}$, and for any two objects $\lambda'=(\alpha',\beta')$, $\lambda=(\alpha,\beta)$, the set $\ho_{\Lambda^\oppo}(\lambda',\lambda)$ is 
	\begin{enumerate}
		\renewcommand{\theenumi}{\roman{enumi}}
		\renewcommand{\labelenumi}{{\rm(\theenumi)}}
		\item the subset of $\prod_{i\in I}\ho_{Y_{\beta'}}(Y_{i\alpha'\beta'},Y_{i\alpha\beta'})$ formed by elements $f=(f_i)_{i\in I}$ such that for each $i\in I$, $f_i:Y_{i\alpha'\beta'}\to Y_{i\alpha\beta'}$ is affine and the base change of $f_i$ by $Y\to Y_{\beta'}$ is the transition morphism $Y_{i\alpha'}\to Y_{i\alpha}$, if $\alpha'\geq \alpha$ and $\beta'\geq\beta$;
		\item empty, if else.
	\end{enumerate}
	The composition of morphisms $(g_i:Y_{i\alpha''\beta''}\to Y_{i\alpha'\beta''})_{i\in I}$ with $(f_i:Y_{i\alpha'\beta'}\to Y_{i\alpha\beta'})_{i\in I}$ in $\Lambda^\oppo$ is $(g_i\circ f'_i:
	Y_{i\alpha''\beta''}\to Y_{i\alpha\beta''})$, where $f'_i$ is the base change of $f_i$ by the transition morphism $Y_{\beta''}\to Y_{\beta'}$. We see that $\Lambda^\oppo$ is cofiltered by \cite[8.8.2]{ega4-3}. Let $\Lambda$ be the opposite category of $\Lambda^\oppo$. For each $i\in I$ and $\lambda=(\alpha,\beta)\in \Lambda$, we set $Y_\lambda=Y_\beta$ and $Y_{i\lambda}=Y_{i\alpha\beta}$. It is clear that the natural functors $\Lambda\to A$ and $\Lambda\to B$ are cofinal (\cite[\Luoma{1}.8.1.3]{sga4-1}). After replacing $\Lambda$ by a directed set (\cite[\href{https://stacks.math.columbia.edu/tag/0032}{0032}]{stacks-project}), the families $\ak{U}_\lambda=\{Y_{i\lambda} \to Y_\lambda\}_{i \in I}$ satisfy the required conditions.
\end{proof}

\begin{mylem}\label{lem:mu-inv}
	With the notation in {\rm\ref{para:fp-h-site}}, let $\ca{F}$ be a presheaf on $(\sch^{\mrm{fp}}_{/X})_{\mrm{h}}$, $(Y_\lambda)$ a directed inverse system of finitely presented $X$-schemes with affine transition morphisms, $Y=\lim Y_\lambda$. Then, we have $\nu_\psh\ca{F}(Y) = \colim \ca{F}(Y_\lambda)$, where $\nu^+=\xi^+$ (resp. $\nu^+=\zeta^+\circ\xi^+$).
\end{mylem}
\begin{proof}
	Notice that the presheaf $\ca{F}$ is a filtered colimit of representable presheaves by \cite[\Luoma{1}.3.4]{sga4-1}
	\begin{align}
		\ca{F} = \colim_{Y' \in (\sch^{\mrm{fp}}_{/X})_{/\ca{F}}} h_{Y'}.
	\end{align}
	Thus, we may assume that $\ca{F}$ is representable by a finitely presented $X$-scheme $Y'$ since the section functor $\Gamma(Y,-)$ commutes with colimits of presheaves (\cite[\href{https://stacks.math.columbia.edu/tag/00VB}{00VB}]{stacks-project}). Then, we have
	\begin{align}
		\nu_\psh h_{Y'}(Y) = h_{\nu^+(Y')}(Y) =\ho_{X}(Y , Y') =\colim \ho_{X}(Y_\lambda , Y') =\colim h_{Y'}(Y_\lambda)
	\end{align}
	where the first equality follows from \cite[\href{https://stacks.math.columbia.edu/tag/04D2}{04D2}]{stacks-project}, and the third equality follows from \cite[8.14.2]{ega4-3}.
\end{proof}

\begin{myprop}\label{prop:v-h-coh}
	With the notation in {\rm\ref{para:fp-h-site}}, let $\ca{F}$ be an abelian sheaf on $(\sch^{\mrm{fp}}_{/X})_{\mrm{h}}$, $(Y_\lambda)$ a directed inverse system of finitely presented $X$-schemes with affine transition morphisms, $Y=\lim Y_\lambda$. Let $\tau = \mrm{h}$ and $\nu^+=\xi^+$ (resp.  $\tau = \mrm{v}$ and $\nu^+=\zeta^+\circ\xi^+$). Then, for any integer $q$, we have
	\begin{align}\label{eq:3.8.1}
		H^q_\tau(Y, \nu^{-1}\ca{F})=\colim H^q((\sch^{\mrm{fp}}_{/Y_\lambda})_{\mrm{h}},\ca{F}).
	\end{align}
	In particular, the canonical morphism $\ca{F} \longrightarrow \rr\nu_*\nu^{-1}\ca{F}$ is an isomorphism.
\end{myprop}
\begin{proof}
	For the second assertion, the sheaf $\rr^q\nu_*\nu^{-1}\ca{F}$ is the sheaf associated to the presheaf $Y \mapsto H^q_\tau(Y, \nu^{-1}\ca{F}) = H^q((\sch^{\mrm{fp}}_{/Y})_{\mrm{h}}, \ca{F})$ by the first assertion, which is $\ca{F}$ if $q=0$ and vanishes otherwise.
		
	We claim that it suffices to show that \eqref{eq:3.8.1} holds for any injective abelian sheaf $\ca{F}=\ca{I}$ on $(\sch^{\mrm{fp}}_{/X})_{\mrm{h}}$. Indeed, if so, then we prove by induction on $q$ that \eqref{eq:3.8.1} holds in general. The case where $q\leq -1$ is trivial. We set $H^q_1(\ca{F})=	H^q_\tau(Y, \nu^{-1}\ca{F})$ and $H^q_2(\ca{F})= \colim H^q((\sch^{\mrm{fp}}_{/Y_\lambda})_{\mrm{h}},\ca{F})$. We embed an abelian sheaf $\ca{F}$ to an injective abelian sheaf $\ca{I}$. Consider the exact sequence $0\to \ca{F}\to \ca{I}\to \ca{G}\to 0$ and the morphism of long exact sequences
	\begin{align}
		\xymatrix{
			H^{q-1}_1(\ca{I})\ar[r]\ar[d]^-{\gamma_1}& H^{q-1}_1(\ca{G})\ar[r]\ar[d]^-{\gamma_2}& H^{q}_1(\ca{F})\ar[r]\ar[d]^-{\gamma_3}& H^{q}_1(\ca{I})\ar[r]\ar[d]^-{\gamma_4}& H^{q}_1(\ca{G})\ar[d]^-{\gamma_5}\\
			H^{q-1}_2(\ca{I})\ar[r]&H^{q-1}_2(\ca{G})\ar[r]&H^q_2(\ca{F})\ar[r]&H^q_2(\ca{I})\ar[r]&H^q_2(\ca{G})
		}
	\end{align}
	If \eqref{eq:3.8.1} holds for any abelian sheaf $\ca{F}$ for degree $q-1$, then $\gamma_1$, $\gamma_2$, $\gamma_4$ are isomorphisms and thus $\gamma_3$ is injective by the 5-lemma (\cite[\href{https://stacks.math.columbia.edu/tag/05QA}{05QA}]{stacks-project}). Thus, $\gamma_5$ is also injective since $\ca{F}$ is an arbitrary abelian sheaf. Then, we see that $\gamma_3$ is an isomorphism, which completes the induction procedure. 
	
	For an injective abelian sheaf $\ca{I}$ on $(\sch^{\mrm{fp}}_{/X})_{\mrm{h}}$, we claim that for any covering family $\ak{U}=\{(Y_i \to Y)\}_{i \in I}$ in $(\schqcqs_{/X})_\tau$ with $I$ finite, the augmented \v Cech complex associated to the presheaf $\nu_{\psh}\ca{I}$
	\begin{align}\label{eq:3.5.3}
		\nu_{\psh}\ca{I}(Y) \to \prod_{i \in I}\nu_{\psh}\ca{I}(Y_i) \to \prod_{i,j \in I} \nu_\psh\ca{I}(Y_i \times_{Y} Y_j)\to \cdots
	\end{align}
	is exact. Admitting this claim, we see that $\nu_\psh\ca{I}$ is indeed a sheaf, i.e. $\nu^{-1}\ca{I}=\nu_\psh\ca{I}$, and the vanishing of higher \v Cech cohomologies implies that $H^q_\tau(Y,\nu^{-1}\ca{I})=0$ for $q>0$ by \ref{lem:descendv-cov} (\cite[\href{https://stacks.math.columbia.edu/tag/03F9}{03F9}]{stacks-project}), which completes the proof together with \ref{lem:mu-inv}. For the claim, we take the covering families $\ak{U}_\lambda=\{Y_{i\lambda} \to Y_\lambda\}_{i \in I}$ in $(\sch^{\mrm{fp}}_{/X})_{\mrm{h}}$ constructed by \ref{lem:descendv-cov}. By \ref{lem:mu-inv}, the sequence \eqref{eq:3.5.3} is the filtered colimit of the augmented \v Cech complexes
	\begin{align}
		\ca{I}(Y_\lambda) \to \prod_{i \in I}\ca{I}(Y_{i\lambda}) \to \prod_{i,j \in I} \ca{I}(Y_{i\lambda} \times_{Y_\lambda} Y_{j\lambda})\to \cdots,
	\end{align}
	which are exact since $\ca{I}$ is an injective abelian sheaf on $(\sch^{\mrm{fp}}_{/X})_{\mrm{h}}$.
\end{proof}

\begin{mycor}\label{cor:v-etale-coh}
	Let $X$ be a coherent scheme, $\ca{F}$ a torsion abelian sheaf on the site $X_\et$ formed by coherent \'etale $X$-schemes endowed with the \'etale topology, $a: (\schqcqs_{/X})_{\mrm{v}} \to X_\et$ the morphism of sites defined by the inclusion functor. Then, the canonical morphism $\ca{F} \to \rr a_*a^{-1}\ca{F}$ is an isomorphism.
\end{mycor}
\begin{proof}
	Consider the morphisms of sites defined by inclusion functors
	\begin{align}\label{eq:3.9.1}
		(\schqcqs_{/X})_{\mrm{v}}\stackrel{\zeta}{\longrightarrow}(\schqcqs_{/X})_{\mrm{h}} \stackrel{\xi}{\longrightarrow} (\sch^{\mrm{fp}}_{/X})_{\mrm{h}} \stackrel{\mu}{\longrightarrow} X_\et.
	\end{align}
	Notice that the morphism $\ca{F} \to \rr (\mu\circ\xi)_*(\mu\circ\xi)^{-1}\ca{F}$ is an isomorphism by \cite[\href{https://stacks.math.columbia.edu/tag/0EWG}{0EWG}]{stacks-project}.  Hence, $\ca{F} \to \rr \mu_*\mu^{-1}\ca{F}$ is an isomorphism by \ref{prop:v-h-coh}, and thus so is $\ca{F} \to \rr a_*a^{-1}\ca{F}$ by \ref{prop:v-h-coh}.
\end{proof}

\begin{mycor}\label{cor:v-etale-bc}
	Let $f:X \to Y$ be a proper morphism of coherent schemes, $\ca{F}$ a torsion abelian sheaf on $X_\et$. Consider the commutative diagram
	\begin{align}
		\xymatrix{
			(\schqcqs_{/X})_{\mrm{v}} \ar[r]^-{a_X} \ar[d]_-{f_{\mrm{v}}}& X_\et\ar[d]^-{f_{\et}}\\
			(\schqcqs_{/Y})_{\mrm{v}} \ar[r]^-{a_Y}& Y_\et
		}
	\end{align}
	where $f_{\mrm{v}}$ and $f_{\et}$ are defined by the base change by $f$. Then, the canonical morphism
	\begin{align}
		a_Y^{-1}\rr f_{\et *} \ca{F} \longrightarrow \rr f_{\mrm{v}*} a_X^{-1}\ca{F}
	\end{align}
	is an isomorphism.
\end{mycor}
\begin{proof}
	Consider the commutative diagram
	\begin{align}
		\xymatrix{
			(\schqcqs_{/X})_{\mrm{v}} \ar[r]^-{\zeta_X} \ar[d]_-{f_{\mrm{v}}}& (\schqcqs_{/X})_{\mrm{h}} \ar[r]^-{b_X} \ar[d]_-{f_{\mrm{h}}} & X_\et\ar[d]^-{f_{\et}}\\
			(\schqcqs_{/Y})_{\mrm{v}} \ar[r]^-{\zeta_Y}& (\schqcqs_{/Y})_{\mrm{h}} \ar[r]^-{b_Y} & Y_\et
		}
	\end{align}
	The canonical morphism $b_Y^{-1}\rr f_{\et *} \ca{F} \longrightarrow \rr f_{\mrm{h}*} b_X^{-1}\ca{F}$ is an isomorphism by \cite[\href{https://stacks.math.columbia.edu/tag/0EWF}{0EWF}]{stacks-project}. It remains to show that the canonical morphism $\zeta_Y^{-1}\rr f_{\mrm{h} *} b_X^{-1}\ca{F} \longrightarrow \rr f_{\mrm{v}*} a_X^{-1}\ca{F}$ is an isomorphism. Let $Y'$ be a coherent $Y$-scheme and we set $g:X'=Y'\times_Y X \to X$. For each integer $q$, $\zeta_Y^{-1}\rr^q f_{\mrm{h} *} b_X^{-1}\ca{F}$ is the sheaf associated to the presheaf $Y' \mapsto H^q_{\mrm{h}}(X', b_{X'}^{-1}g^{-1}_\et\ca{F})=H^q_{\et}(X', g^{-1}_\et\ca{F})$ by \cite[\href{https://stacks.math.columbia.edu/tag/0EWH}{0EWH}]{stacks-project}, and $\rr^q f_{\mrm{v}*} a_X^{-1}\ca{F}$ is the sheaf associated to the presheaf $Y'\mapsto H^q_{\mrm{v}}(X', a_{X'}^{-1}g^{-1}_\et\ca{F})=H^q_{\et}(X', g^{-1}_\et\ca{F})$ by \ref{cor:v-etale-coh}. 
\end{proof}

\begin{mylem}\label{lem:prod-val}
Let $A$ be a product of (resp. absolutely integrally closed) valuation rings {\rm(\ref{para:notation-val})}. 
\begin{enumerate}
	\renewcommand{\labelenumi}{{\rm(\theenumi)}}
	\item Any finitely generated ideal of $A$ is principal.\label{lem:prod-val-prin}
	\item Any connected component of $\spec(A)$ with the reduced closed subscheme structure is isomorphic to the spectrum of a (resp. absolutely integrally closed) valuation ring.\label{lem:prod-val-conn}
\end{enumerate}
\end{mylem}
\begin{proof}
(\ref{lem:prod-val-prin}) is proved in \cite[\href{https://stacks.math.columbia.edu/tag/092T}{092T}]{stacks-project}, and (\ref{lem:prod-val-conn}) follows from the proof of \cite[6.2]{bhattscholze2017witt}.
\end{proof}

\begin{mylem}\label{lem:small-val}
	 Let $X$ be a $\bb{U}$-small scheme, $y\rightsquigarrow x$ a specialization of points of $X$. Then, there exists a $\bb{U}$-small family $\{f_\lambda:\spec(V_\lambda) \to X\}_{\lambda \in \Lambda_{y\rightsquigarrow x}}$ of morphisms of schemes such that
	 \begin{enumerate}
	 	\renewcommand{\theenumi}{\roman{enumi}}
	 	\renewcommand{\labelenumi}{{\rm(\theenumi)}}
	 	\item the ring $V_\lambda$ is a $\bb{U}$-small (resp. absolutely integrally closed) valuation ring, and that
	 	\item the morphism $f_\lambda$ maps the generic point and closed point of $\spec(V_\lambda)$ to $y$ and $x$ respectively, and that
	 	\item for any morphism of schemes $f:\spec(V) \to X$ where $V$ is a (resp. absolutely integrally closed) valuation ring such that $f$ maps the generic point and closed point of $V$ to $y$ and $x$ respectively, there exists an element $\lambda \in \Lambda_{y\rightsquigarrow x}$ such that $f$ factors through $f_\lambda$ and that $V_\lambda\to V$ is an extension of valuation rings.
	 \end{enumerate}
\end{mylem}
\begin{proof}
	Let $K_{y}$ be the residue field $\kappa(y)$ of $y$ (resp. an algebraic closure of $\kappa(y)$). Let $\ak{p}_y$ be the prime ideal of the local ring $\ca{O}_{X,x}$ corresponding to the point $y$, and let $\{V_\lambda\}_{\lambda \in \Lambda_{y\rightsquigarrow x}}$ be the set of all valuation rings with fraction field $K_{y}$ which contain $\ca{O}_{X,x}/\ak{p}_y$ such that the injective homomorphism $\ca{O}_{X,x}/\ak{p}_y \to V_\lambda$ is local. The smallness of $\Lambda_{y\rightsquigarrow x}$ and $V_\lambda$ is clear, and the inclusion $\ca{O}_{X,x}/\ak{p}_y \to V_\lambda$ induces a morphism $f_\lambda:\spec(V_\lambda) \to X$ satisfying (\luoma{2}). It remains to check (\luoma{3}). The morphism $f$ induces an injective and local homomorphism $\ca{O}_{X,x}/\ak{p}_y \to V$. Notice that $\ca{O}_{X,x}/\ak{p}_y \to \mrm{Frac}(V)$ factors through $K_{y}$ and that $K_{y}\cap V$ is a valuation ring with fraction field $K_y$. It is clear that $K_{y}\cap V\to V$ is local and injective, which shows that $K_{y}\cap V$ belongs to the set $\{V_\lambda\}_{\lambda\in \Lambda_{y\rightsquigarrow x}}$ constructed before.
\end{proof}

\begin{mylem}\label{lem:specialization}
	Let $f:\spec(V) \to X$ be a morphism of schemes where $V$ is a valuation ring. We denote by $a$ and $b$ the closed point and generic point of $\spec(V)$ respectively. If $c \in X$ is a generalization of $f(b)$, then there exists an absolutely integrally closed valuation ring $W$, a prime ideal $\ak{p}$ of $W$, and a morphism $g : \spec(W) \to X$ satisfying the following conditions:
	\begin{enumerate}
		\item[\rm(i)] If $z$, $y$, $x$ denote respectively the generic point, the point $\ak{p}$ and the closed point of $\spec(W)$, then $g(z)=c$, $g(y)=f(b)$ and $g(x) = f(a)$.
		\item[\rm(ii)] The induced morphism $\spec(W/\ak{p}) \to X$ factors through $f$, and the induced morphism $V \to W/ \ak{p}$ is an extension of valuation rings.
	\end{enumerate}
\end{mylem}
\begin{proof}
	According to \cite[7.1.4]{ega2}, there exists an absolutely integrally closed valuation ring $U$ and a morphism $\spec (U) \to X$ which maps the generic point $z$ and the closed point $y$ of $\spec (U)$ to $c$ and $f(b)$ respectively. After extending $U$, we may assume that the morphism $y\to f(b)$ factors through $b$ (\cite[7.1.2]{ega2}). We denote by $\kappa(y)$ the residue field of the point $y$. Let $V'$ be a valuation ring extension of $V$ with fraction field $\kappa(y)$, and let $W$ be the preimage of $V'$ by the surjection $U \to \kappa(y)$. Then, the maximal ideal $\ak{p}=\ke(U \to \kappa(y))$ of $U$ is a prime ideal of $W$, and $W/\ak{p} = V'$. We claim that $W$ is an absolutely integrally closed valuation ring such that $W_\ak{p} = U$. Indeed, firstly note that the fraction fields of $U$ and $W$ are equal as $\ak{p}\subseteq W$. Let $\gamma$ be an element of $\mrm{Frac}(W)\setminus W$. If $\gamma\in U$, then $\gamma^{-1}\in W\setminus \ak{p}$ by definition since $\gamma^{-1}\in U\setminus \ak{p}$ and $V$ is a valuation ring, and then $\gamma\in W_{\ak{p}}$. If $\gamma\notin U$, then $\gamma^{-1}\in \ak{p}$ since $U$ is a valuation ring, and then $\gamma\notin W_{\ak{p}}$. Thus, we have proved the claim, which shows that $W$ satisfies the required conditions.
\end{proof}

\begin{myprop}\label{prop:prod-val-cov}
	Let $X$ be a coherent $\bb{U}$-small scheme, $X^\circ$ a quasi-compact dense open subset of $X$. Then, there exists a $\bb{U}$-small product $A$ of absolutely integrally closed $\bb{U}$-small valuation rings and a v-covering $\spec(A) \to X$ such that $\spec(A)$ is $X^\circ$-integrally closed {\rm(\ref{para:notation-intclos})}.
\end{myprop}
\begin{proof}
	After replacing $X$ by a finite affine open covering, we may assume that $X=\spec(R)$. For a specialization ${y\rightsquigarrow x}$ of points of $X$, let $\{R \to V_\lambda\}_{\lambda \in \Lambda_{y\rightsquigarrow x}}$ be the $\bb{U}$-small set constructed in \ref{lem:small-val}. Let $\Lambda=\coprod_{y\in X^\circ} \Lambda_{y\rightsquigarrow x}$ where ${y\rightsquigarrow x}$ runs through all specializations in $X$ such that $y \in X^\circ$. We take $A=\prod_{\lambda\in \Lambda}V_\lambda$ and $R\to A$ the natural homomorphism. As a quasi-compact open subscheme of $\spec(A)$, $X^\circ\times_X \spec(A)$ is the spectrum of $A[1/\pi]$ for an element $\pi=(\pi_\lambda)_{\lambda\in \Lambda}\in A$ by \ref{lem:prod-val}.(\ref{lem:prod-val-prin}) (\cite[\href{https://stacks.math.columbia.edu/tag/01PH}{01PH}]{stacks-project}). Notice that $\pi_\lambda\neq 0$ for any $\lambda\in \Lambda$. We see that $A$ is integrally closed in $A[1/\pi]$. It remains to check that $\spec(A) \to X$ is a v-covering. For any morphism $f:\spec(V) \to X$ where $V$ is a valuation ring, by \ref{lem:specialization}, there exists an absolutely integrally closed valuation ring $W$, a prime ideal $\ak{p}$ of $W$ and a morphism $g:\spec(W) \to X$ such that $g$ maps the generic point of $W$ into $X^\circ$ and that $W/\ak{p}$ is a valuation ring extension of $V$. By construction, there exists $\lambda\in \Lambda$ such that $g$ factors through $\spec(V_\lambda) \to X$. We see that $f$ lifts to the composition of $\spec(W/\ak{p}) \to \spec(V_\lambda) \to \spec(A)$.
\end{proof}

\begin{myprop}\label{prop:normal-v-cov}
	Consider a commutative diagram of schemes
	\begin{align}
		\xymatrix{
			Y' \ar[r] \ar[d] & Z' \ar[r] \ar[d] & X'\ar[d]\\
			Y \ar[r] & Z \ar[r] & X
		}
	\end{align}
	where $Z' \to Z$ and $X' \to X$ are quasi-compact. Assume that $Y' \to Y \times_X X'$ is surjective, $Y \to Z$ is dominant, $Z \to X$ is separated and $Z' \to X'$ is integral. If $X' \to X$ is a v-covering, then $Z' \to Z$ is also a v-covering.
\end{myprop}
\begin{proof}
	Notice that $Z'\to Z\times_X X'$ is still integral as $Z \to X$ is separated. After replacing $X' \to X$ by $Z\times_X X' \to Z$, we may assume that $Z=X$. Let $\spec(V) \to Z$ be a morphism of schemes where $V$ is a valuation ring. Since $Y \to Z$ is dominant, by \ref{lem:specialization}, there exists a morphism $\spec(W) \to Z$ where $W$ is an absolutely integrally closed valuation ring, a prime ideal $\ak{p}$ of $W$ such that $W/\ak{p}$ is a valuation ring extension of $V$ and that the generic point $\xi$ of $\spec(W)$ is over the image of $Y \to Z$. After extending $W$ (\cite[\href{https://stacks.math.columbia.edu/tag/00IA}{00IA}]{stacks-project}), we may assume that there exists a lifting $\xi \to Y$ of $\xi \to Z$. The morphism $\spec(W) \to Z = X $ admits a lifting $\spec(W') \to X'$ where $W\to W'$ is an extension of valuation rings. We claim that after extending $W'$, $\spec(W') \to X'$ factors through $Z'$. Indeed, if $\xi'$ denotes the generic point of $\spec(W')$, as $Y' \to Y \times_XX'$ is surjective, after extending $W'$, we may assume that there exists an $X'$-morphism $\xi' \to Y'$ which is over $\xi \to Y$. Since $\spec(W')$ is  integrally closed in $\xi'$ and $Z'$ is integral over $X'$, the morphism $\spec(W') \to X'$ factors through $Z'$ (\cite[\href{https://stacks.math.columbia.edu/tag/035I}{035I}]{stacks-project}). Finally, let $\ak{q} \in \spec(W')$ which lies over $\ak{p} \in \spec(W)$, then we get a lifting $\spec(W'/\ak{q}) \to Z'$ of $\spec(V) \to Z$, which shows that $Z' \to Z$ is a v-covering.
\end{proof}

\begin{mypara}\label{defn:relative-normal}
	Let $S^\circ \to S$ be an open immersion of coherent schemes such that $S$ is $S^\circ$-integrally closed (\ref{para:notation-intclos}). For any $S$-scheme $X$, we set $X^\circ=S^\circ \times_S X$. We denote by $\falh_{S^\circ \to S}$ the category formed by coherent $S$-schemes which are $S^\circ$-integrally closed. Note that any $S^\circ$-integrally closed coherent $S$-scheme $X$ is also $X^\circ$-integrally closed by definition. It is clear that the category $(\falh_{S^\circ \to S})_{/X}$ of objects of $\falh_{S^\circ \to S}$ over $X$ is canonically equivalent to the category $\falh_{X^\circ \to X}$.
\end{mypara}

\begin{mylem}[{\cite[\href{https://stacks.math.columbia.edu/tag/03GV}{03GV}]{stacks-project}}]\label{lem:rel-norm-commute}
	Let $Y\to X$ be a coherent morphism of schemes, $X'\to X$ a smooth morphism of schemes, $Y'=Y\times_X X'$. Then, we have $X'^{Y'}=X^Y\times_X X'$.
\end{mylem}

\begin{mylem}\label{lem:rel-norm-fil-limit}
	Let $(Y_\lambda\to X_\lambda)_{\lambda\in \Lambda}$ be a  directed inverse system of morphisms of coherent schemes with affine transition morphisms $Y_{\lambda'}\to Y_\lambda$ and $X_{\lambda'}\to X_\lambda$ ($\lambda'\geq \lambda$). We set $Y=\lim Y_\lambda$ and $X=\lim X_\lambda$. Then, $(X_\lambda^{Y_\lambda})_{\lambda\in \Lambda}$ is a directed inverse system of coherent schemes with affine transition morphisms and we have $X^Y=\lim X_\lambda^{Y_\lambda}$.
\end{mylem}
\begin{proof}
	We fix an index $\lambda_0\in \Lambda$. After replacing $X_{\lambda_0}$ by an affine open covering, we may assume that $X_{\lambda_0}$ is affine (\ref{lem:rel-norm-commute}). We write $X_\lambda=\spec(A_\lambda)$ and $B_\lambda=\Gamma(Y_\lambda, \ca{O}_{Y_\lambda})$ for each $\lambda\geq \lambda_0$, and we set $A=\colim A_\lambda$ and $B=\colim B_\lambda$. Then, we have $X=\spec(A)$ and $B=\Gamma(Y, \ca{O}_{Y})$ (\cite[\href{https://stacks.math.columbia.edu/tag/009F}{009F}]{stacks-project}). Let $R_\lambda$ (resp. $R$) be the integral closure of $A_\lambda$ in $B_\lambda$ (resp. $A$ in $B$). By definition, we have $X_\lambda^{Y_\lambda}=\spec(R_\lambda)$ and $X^Y=\spec(R)$. The conclusion follows from the fact that $R=\colim R_\lambda$.
\end{proof}

\begin{mylem}\label{lem:relative-normal}
	Let $S^\circ \to S$ be an open immersion of coherent schemes.
	\begin{enumerate}
		\renewcommand{\labelenumi}{{\rm(\theenumi)}}
		\item If $X$ is an $S^\circ$-integrally closed coherent $S$-scheme, then the open subscheme $X^\circ$ is scheme theoretically dense in $X$. \label{lem:relative-normal-open}
		\item If $X$ is an $S^\circ$-integrally closed coherent $S$-scheme and $X'$ is a coherent smooth $X$-scheme, then $X'$ is also $S^\circ$-integrally closed. \label{lem:relative-normal-et}
		\item If $(X_\lambda)_{\lambda\in \Lambda}$ is a directed inverse system of $S^\circ$-integrally closed coherent $S$-scheme with affine transition morphisms, then $X=\lim_{\lambda\in \Lambda}X_\lambda$ is also $S^\circ$-integrally closed. \label{lem:relative-normal-fil-limit}
		\item If $Y \to X$ is a morphism of coherent schemes over $S^\circ \to S$ such that $Y$ is integral over $X^\circ$, then the integral closure $X^Y$ is $S^\circ$-integrally closed with $(X^Y)^\circ=Y$.\label{lem:relative-normal-norm}
	\end{enumerate}
\end{mylem}
\begin{proof}
	(\ref{lem:relative-normal-open}), (\ref{lem:relative-normal-et}), (\ref{lem:relative-normal-fil-limit}) follow from \cite[\href{https://stacks.math.columbia.edu/tag/035I}{035I}]{stacks-project}, \ref{lem:rel-norm-commute} and \ref{lem:rel-norm-fil-limit} respectively. For (\ref{lem:relative-normal-norm}), $(X^Y)^\circ = X^\circ \times_{X} X^Y$ is the integral closure of $X^\circ$ in $X^\circ \times_{X} Y = Y$ by \ref{lem:rel-norm-commute}, which is $Y$ itself.
\end{proof}

\begin{mypara}\label{para:fun-sigma-psi-N}
	We take the notation in {\rm\ref{defn:relative-normal}}. The inclusion functor
	\begin{align}
		\Phi^+: \falh_{S^\circ \to S} \longrightarrow \schqcqs_{/S},\ X \longmapsto X,
	\end{align}
	admits a right adjoint
	\begin{align}\label{eq:3.19.2}
		\sigma^+: \schqcqs_{/S} \longrightarrow \falh_{S^\circ \to S},\ X \longmapsto \overline{X}=X^{X^\circ}.
	\end{align}
	Indeed, $\sigma^+$ is well-defined by \ref{lem:relative-normal}.(\ref{lem:relative-normal-norm}), and the adjointness follows from the functoriality of taking integral closures. We remark that $\overline{X}^\circ=X^\circ$. On the other hand, the functor
	\begin{align}
		\Psi^+: \falh_{S^\circ \to S} \longrightarrow \schqcqs_{/S^\circ},\ X \longmapsto X^\circ,
	\end{align}
	admits a left adjoint
	\begin{align}
		\alpha^+: \schqcqs_{/S^\circ} \longrightarrow \falh_{S^\circ \to S},\ Y \longmapsto Y.
	\end{align}
\end{mypara}

\begin{mylem}\label{lem:relative-normal-limit}
	With the notation in {\rm\ref{defn:relative-normal}}, let $\varphi : I \to \falh_{S^\circ \to S}$ be a functor sending $i$ to $X_i$. If $X = \lim X_i$ represents the limit of $\varphi$ in the category of coherent $S$-schemes, then the integral closure $\overline{X}=X^{X^\circ}$ represents the limit of $\varphi$ in $\falh_{S^\circ \to S}$ with $\overline{X}^\circ=X^\circ$.
\end{mylem}
\begin{proof}
	It follows directly from the adjoint pair $(\Phi^+,\sigma^+)$ (\ref{para:fun-sigma-psi-N}).
\end{proof}

	It follows from \ref{lem:relative-normal-limit} that for a diagram $X_1\to X_0 \leftarrow X_2$ in $\falh_{S^\circ \to S}$, the fibred product is representable by
	\begin{align}
		X_1 \overline{\times}_{X_0} X_2 = (X_1 \times_{X_0} X_2)^{X_1^\circ \times_{X_0^\circ}X_2^\circ}.
	\end{align}

\begin{myprop}\label{prop:falh-v-top}
	With the notation in {\rm\ref{defn:relative-normal}}, let $\scr{C}$ be the set of families of morphisms $\{X_i \to X\}_{i \in I}$ of $\falh_{S^\circ \to S}$ with $I$ finite such that $\coprod_{i \in I} X_i \to X$ is a v-covering. Then, $\scr{C}$ forms a pretopology of $\falh_{S^\circ \to S}$.
\end{myprop}
\begin{proof}	
	Let $\{X_i \to X\}_{i \in I}$ be an element of $\scr{C}$. Firstly, we check that for a morphism $X' \to X$ of $\falh_{S^\circ \to S}$, the base change $\{X_i' \to X'\}_{i \in I}$ also lies in $\scr{C}$, where $Z_i = X_i \times_X X'$ and $X_i'=Z_i^{Z_i^\circ}$ by \ref{lem:relative-normal-limit}. Applying \ref{prop:normal-v-cov} to the following diagram
	\begin{align}
		\xymatrix{
			\coprod_{i \in I} Z_{i}^\circ \ar[r] \ar[d] & \coprod_{i \in I} X_i' \ar[r] \ar[d] & \coprod_{i \in I} Z_i\ar[d]\\
			X'^\circ \ar[r] & X' \ar[r] & X'
		}
	\end{align}
	we deduce that $\coprod_{i \in I} X_i' \to X'$ is also a v-covering, which shows the stability of $\scr{C}$ under base change.
	
	For each $i\in I$, let $\{X_{ij} \to X_i\}_{j \in J_i}$ be an element of $\scr{C}$. We need to show that the composition $\{X_{ij} \to X\}_{i \in I,j \in J_i}$ also lies in $\scr{C}$. This follows immediately from the stability of v-coverings under composition. We conclude that $\scr{C}$ forms a pretopology.
\end{proof}

\begin{mydefn}\label{defn:falh-v-top}
	With the notation in {\rm\ref{defn:relative-normal}}, we endow the category $\falh_{S^\circ \to S}$ with the topology generated by the pretopology defined in \ref{prop:falh-v-top}, and we call $\falh_{S^\circ \to S}$ the \emph{v-site of $S^\circ$-integrally closed coherent $S$-schemes}.
\end{mydefn}

	By definition, any object in $\falh_{S^\circ \to S}$ is quasi-compact. Let $\falhb$ be the sheaf on $\falh_{S^\circ \to S}$ associated to the presheaf $X\mapsto  \Gamma(X,\ca{O}_X)$. We call $\falhb$ the structural sheaf of $\falh_{S^\circ \to S}$.

\begin{myprop}\label{prop:sheaf-isom}
	With the notation in {\rm\ref{defn:relative-normal}}, let $f:X'\to X$ be a covering in $\falh_{S^\circ \to S}$ such that $f$ is separated and that the diagonal morphism $X'^\circ\to X'^\circ\times_{X^\circ}X'^\circ$ is surjective. Then, the morphism of representable sheaves $h_{X'}^\ash \to h_{X}^\ash$ is an isomorphism.
\end{myprop}
\begin{proof}
	We need to show that for any sheaf $\ca{F}$ on $\falh_{S^\circ \to S}$, $\ca{F}(X) \to \ca{F}(X')$ is an isomorphism. Since the composition of $X'^\circ \to X'^\circ\times_{X^\circ}X'^\circ\to X'\overline{\times}_X X'$ factors through the closed immersion $X' \to X'\overline{\times}_X X'$ (as $f$ is separated), the closed immersion $X' \to X'\overline{\times}_X X'$ is surjective (\ref{lem:relative-normal}.(\ref{lem:relative-normal-open})). Consider the following sequence
	\begin{align}\label{eq:5.1.1}
		\ca{F}(X) \to \ca{F}(X') \rightrightarrows \ca{F}(X'\overline{\times}_X X')\to \ca{F}(X').
	\end{align}
	The right arrow is injective as $X' \to X'\overline{\times}_X X'$ is a v-covering. Thus, the middle two arrows are actually the same. Thus, the first arrow is an isomorphism by the sheaf property of $\ca{F}$.
\end{proof}

\begin{myprop}\label{prop:v-site-stalk}
	With the notation in {\rm\ref{defn:relative-normal}}, let $\alpha:\ca{F}_1\to \ca{F}_2$ be a morphism of presheaves on $\falh_{S^\circ \to S}$. Assume that 
	\begin{enumerate}
		\renewcommand{\theenumi}{\roman{enumi}}
		\renewcommand{\labelenumi}{{\rm(\theenumi)}}
		\item the morphism $\ca{F}_1(\spec(V))\to \ca{F}_2(\spec(V))$ is an isomorphism for any $S^\circ$-integrally closed $S$-scheme $\spec(V)$ where $V$ is an absolutely integrally closed valuation ring, and that
		\item for any directed inverse system of $S^\circ$-integrally closed affine schemes $(\spec(A_\lambda))_{\lambda\in \Lambda}$ over $S$ the natural morphism $\colim \ca{F}_i(\spec(A_\lambda))\to \ca{F}_i(\spec(\colim A_\lambda))$ is an isomorphism for $i=1,2$ {\rm(cf. \ref{lem:relative-normal}.(\ref{lem:relative-normal-fil-limit}))}.\label{prop:v-site-stalk-limit}
	\end{enumerate}
	Then, the morphism of the associated sheaves $\ca{F}_1^\ash\to \ca{F}_2^\ash$ is an isomorphism.
\end{myprop}
\begin{proof}
	Let $A$ be a product of absolutely integrally closed valuation rings such that $X=\spec(A)$ is an $S^\circ$-integrally closed $S$-scheme. Let $\spec(V)$ be a connected component of $\spec(A)$ with the reduced closed subscheme structure. Then, $V$ is an absolutely integrally closed valuation ring by \ref{lem:prod-val}.(\ref{lem:prod-val-conn}), and $\spec(V)$ is also $S^\circ$-integrally closed since it has nonempty intersection with the dense open subset $X^\circ$ of $X$. Notice that each connected component of an affine scheme is the intersection of some open and closed subsets (\cite[\href{https://stacks.math.columbia.edu/tag/04PP}{04PP}]{stacks-project}). Moreover, since $A$ is reduced, we have $V = \colim A'$, where the colimit is taken over all the open and closed subschemes $X'=\spec(A')$ of $X$ which contain $\spec(V)$. By assumption, we have an isomorphism
	\begin{align}\label{eq:3.24.1}
		\colim \ca{F}_1(X') \iso \colim \ca{F}_2(X').
	\end{align}	
	For two elements $\xi_1,\xi_1'\in \ca{F}_1(X)$ with $\alpha_X(\xi_1)=\alpha_X(\xi_1')$ in $\ca{F}_2(X)$, by \eqref{eq:3.24.1} and a limit argument, there exists a finite disjoin union $X=\coprod_{i=1}^r X_i'$ such that the images of $\xi_1$ and $\xi_1'$ in $\ca{F}_1(X_i')$ are the same. Therefore, $\ca{F}_1^\ash\to \ca{F}_2^\ash$ is injective by \ref{prop:prod-val-cov}. On the other hand, for an element $\xi_2 \in \ca{F}_2(X)$, by \eqref{eq:3.24.1} and a limit argument, there exists a finite disjoin union $X=\coprod_{i=1}^r X_i'$ such that there exists an element $\xi_{1,i}\in \ca{F}_1(X_i')$ for each $i$ such that the image of $\xi_2$ in $\ca{F}_2(X_i')$ is equal to $\alpha_{X_i'}(\xi_{1,i})$. Therefore, $\ca{F}_1^\ash\to \ca{F}_2^\ash$ is surjective by \ref{prop:prod-val-cov}.
\end{proof}

\begin{mypara}\label{para:mor-sigma-psi-N}
	We take the notation in {\rm\ref{defn:relative-normal}}. Endowing $\schqcqs$ with the v-topology (\ref{para:h-v-arc-topology}), we see that the functors $\sigma^+$ and $\Psi^+$ defined in \ref{para:fun-sigma-psi-N} are left exact (as they have left adjoints) and continuous by \ref{prop:normal-v-cov} and \ref{prop:falh-v-top}. Therefore, they define morphisms of sites (\ref{para:topos})
	\begin{align}
		(\schqcqs_{/S^\circ})_{\mrm{v}}\stackrel{\Psi}{\longrightarrow}\falh_{S^\circ \to S} \stackrel{\sigma}{\longrightarrow} (\schqcqs_{/S})_\mrm{v}.
	\end{align}
\end{mypara}

\begin{myprop}\label{prop:Psi-const-equal}
	With the notation in {\rm\ref{para:mor-sigma-psi-N}}, let $a: (\schqcqs_{/S^\circ})_{\mrm{v}} \to S^\circ_\et$ be the morphism of site defined by the inclusion functor {\rm(\ref{cor:v-etale-coh})}.
	\begin{enumerate}
		\renewcommand{\labelenumi}{{\rm(\theenumi)}}
		\item For any torsion abelian sheaf $\ca{F}$ on $S^\circ_\et$, the canonical morphism $\Psi_*(a^{-1}\ca{F}) \to \rr\Psi_*(a^{-1}\ca{F})$ is an isomorphism.\label{prop:Psi-const-equal-torsion}
		\item For any locally constant torsion abelian sheaf $\bb{L}$ on $\falh_{S^\circ \to S}$, the canonical morphism $\bb{L} \to \rr\Psi_*\Psi^{-1}\bb{L}$ is an isomorphism.\label{prop:Psi-const-equal-const}
	\end{enumerate}
\end{myprop}
\begin{proof}
	(\ref{prop:Psi-const-equal-torsion}) For each integer $q$, the sheaf $\rr^q \Psi_*(a^{-1}\ca{F})$ is the sheaf associated to the presheaf $X\mapsto H^q_{\mrm{v}}(X^\circ,a^{-1}\ca{F})=H^q_{\et}(X^\circ,f^{-1}_\et\ca{F})$ by \ref{cor:v-etale-coh}, where $f_\et:X^\circ_\et \to S^\circ_\et$ is the natural morphism. If $X$ is the spectrum of a nonzero absolutely integrally closed valuation ring $V$, then $X^\circ=\spec(V[1/\pi])$ for a nonzero element $\pi \in V$ by \ref{lem:prod-val}.(\ref{lem:prod-val-prin}) and \ref{lem:relative-normal}.(\ref{lem:relative-normal-open}), which is also the spectrum of an absolutely integrally closed valuation ring (\ref{para:notation-val}). In this case, $H^q_{\et}(X^\circ,f^{-1}_\et\ca{F})=0$ for $q>0$, which proves (\ref{prop:Psi-const-equal-torsion}) by \ref{prop:v-site-stalk} and \cite[\Luoma{7}.5.8]{sga4-2}.
	
	(\ref{prop:Psi-const-equal-const}) The problem is local on $\falh_{S^\circ \to S}$. We may assume that $\bb{L}$ is the constant sheaf with value $L$. Then, $\rr^q\Psi_*\Psi^{-1}\bb{L}=0$ for $q>0$ by applying (\ref{prop:Psi-const-equal-torsion}) on the constant sheaf with value $L$ on $S^\circ_\et$. For $q=0$, notice that $\bb{L}$ is also the sheaf associated to the presheaf $X\mapsto H^0_\et(X,L)$, while $\Psi_*\Psi^{-1}\bb{L}$ is the sheaf $X\mapsto H^0_\et(X^\circ,L)$ by the discussion in (\ref{prop:Psi-const-equal-torsion}). If $X$ is the spectrum of a nonzero absolutely integrally closed valuation ring, then so is $X^\circ$ and so that $H^0_\et(X,L)=H^0_\et(X^\circ,L)=L$. The conclusion follows from \ref{prop:v-site-stalk} and \cite[\Luoma{7}.5.8]{sga4-2}.
\end{proof}

\section{The arc-Descent of Perfect Algebras}\label{sec:arc-descent}

\begin{mydefn}\label{defn:A_perf}
	For any $\bb{F}_p$-algebra $R$, we denote by $R_{\mrm{perf}}$ the filtered colimit
	\begin{align}
		R_{\mrm{perf}} = \colim_{\frob} R
	\end{align}
	 indexed by $(\bb{N},\leq)$, where the transition map associated to $i \leq (i+1)$ is the Frobenius of $R$.
\end{mydefn}
It is clear that the endo-functor of the category of $\bb{F}_p$-algebras, $R \mapsto R_{\mrm{perf}}$, commutes with colimits.

\begin{mypara}
	We define a presheaf $\ca{O}_{\mrm{perf}}$ on the category $\schqcqs_{\bb{F}_p}$ of coherent $\bb{U}$-small $\bb{F}_p$-schemes $X$ by
	\begin{align}
	\ca{O}_{\mrm{perf}}(X) = \Gamma(X,\ca{O}_X)_{\mrm{perf}}.
	\end{align}
	For any morphism $\spec(B) \to \spec(A)$ of affine $\bb{F}_p$-schemes, we consider the augmented \v Cech complex of the presheaf $\ca{O}_{\mrm{perf}}$,
	\begin{align}\label{eq:arc-descent}
		0 \to A_{\mrm{perf}} \to B_{\mrm{perf}} \to B_{\mrm{perf}} \otimes_{A_{\mrm{perf}}} B_{\mrm{perf}} \to \cdots.
	\end{align}
\end{mypara}

\begin{mylem}[{\cite[\href{https://stacks.math.columbia.edu/tag/0EWT}{0EWT}]{stacks-project}}]\label{lem:fppf-descent}
	The presheaf $\ca{O}_{\mrm{perf}}$ is a sheaf on $\schqcqs_{\bb{F}_p}$ with respect to the fppf topology {\rm(\cite[\href{https://stacks.math.columbia.edu/tag/021L}{021L}]{stacks-project})}. Moreover, for any coherent $\bb{F}_p$-scheme $X$ and any integer $q$,
	\begin{align}
		H^q_{\mrm{fppf}}(X, \ca{O}_{\mrm{perf}}) = \colim_{\frob} H^q(X,\ca{O}_X).
	\end{align}
\end{mylem}
\begin{proof}
	Firstly, we remark that for any integer $q$, the functor $H^q_{\mrm{fppf}}(X,-)$ commutes with filtered colimit of abelian sheaves on $(\schqcqs_{/X})_{\mrm{fppf}}$ for any coherent scheme $X$ (\cite[\href{https://stacks.math.columbia.edu/tag/0739}{0739}]{stacks-project}). Since the presheaf $\ca{O}$ sending $X$ to $\Gamma(X,\ca{O}_X)$ on $\schqcqs_{\bb{F}_p}$ is an fppf-sheaf, we have $H^0_{\mrm{fppf}}(X, \ca{O}_{\mrm{perf}}^\ash)=\colim_{\frob}H^0_{\mrm{fppf}}(X, \ca{O})= \ca{O}_{\mrm{perf}}(X)$. Thus, $\ca{O}_{\mrm{perf}}$ is an fppf-sheaf. Moreover, $H^q_{\mrm{fppf}}(X, \ca{O}_{\mrm{perf}}) = \colim_{\frob} H^q_{\mrm{fppf}}(X, \ca{O}) = \colim_{\frob} H^q(X,\ca{O}_X)$ by faithfully flat descent (\cite[\href{https://stacks.math.columbia.edu/tag/03DW}{03DW}]{stacks-project}).
\end{proof}

\begin{mylem}\label{lem:star-descent}
	Let $\tau \in \{\trm{fppf, h, v, arc}\}$. The following propositions are equivalent:
	\begin{enumerate}
		\renewcommand{\labelenumi}{{\rm(\theenumi)}}
		\item The presheaf $\ca{O}_{\mrm{perf}}$ on $\schqcqs_{\bb{F}_p}$ is a $\tau$-sheaf and $H^q_\tau(X, \ca{O}_{\mrm{perf}}) = \colim_{\frob} H^q(X,\ca{O}_X)$ for any coherent $\bb{F}_p$-scheme $X$ and any integer $q$.
		\item  For any $\tau$-covering $\spec(B) \to \spec(A)$ of affine $\bb{F}_p$-schemes, the augmented \v Cech complex \eqref{eq:arc-descent} is exact.
	\end{enumerate}
\end{mylem}
\begin{proof}
	For an affine scheme $X = \spec(A)$, $H^q(X,\ca{O}_X)$ vanishes for $q>0$ and $H^0(X,\ca{O}_X)=A$. For (1) $\Rightarrow$ (2), the exactness of \eqref{eq:arc-descent} follows from the \v Cech-cohomology-to-cohomology spectral sequence associated to the $\tau$-covering $\spec(B) \to \spec(A)$ \cite[\href{https://stacks.math.columbia.edu/tag/03AZ}{03AZ}]{stacks-project}. Therefore, (1) and (2) hold for $\tau=\trm{fppf}$ by \ref{lem:fppf-descent}. Conversely, the exactness of \eqref{eq:arc-descent} shows the sheaf property for any $\tau$-covering of an affine scheme by affine schemes, which implies the fppf-sheaf $\ca{O}_{\mrm{perf}}$ is a $\tau$-sheaf (cf. \cite[\href{https://stacks.math.columbia.edu/tag/0ETM}{0ETM}]{stacks-project}). The vanishing of higher \v Cech cohomologies implies that $H^q_\tau(X, \ca{O}_{\mrm{perf}})=0$ for any affine $\bb{F}_p$-scheme $X$ and any $q>0$  (\cite[\href{https://stacks.math.columbia.edu/tag/03F9}{03F9}]{stacks-project}). Therefore, for a coherent $\bb{F}_p$-scheme $X$, $H^q_\tau(X, \ca{O}_{\mrm{perf}})$ can be computed by the hyper-\v Cech cohomology of a hypercovering of $X$ formed by affine open subschemes (\cite[\href{https://stacks.math.columbia.edu/tag/01GY}{01GY}]{stacks-project}). In particular, we have $H^q_\tau(X, \ca{O}_{\mrm{perf}})=H^q_{\mrm{fppf}}(X, \ca{O}_{\mrm{perf}})$ for any integer $q$, which completes the proof by \ref{lem:fppf-descent}.
\end{proof}

\begin{mylem}[Gabber]\label{lem:h-descent}
	The augmented \v Cech complex \eqref{eq:arc-descent} is exact for any h-covering $\spec(B) \to \spec(A)$ of affine $\bb{F}_p$-schemes.
\end{mylem}
\begin{proof}
	This is a result of Gabber, cf. \cite[3.3]{bhatt2017gabber} or \cite[\href{https://stacks.math.columbia.edu/tag/0EWU}{0EWU}]{stacks-project}, and \ref{lem:star-descent}.
\end{proof}

\begin{mylem}[{\cite[4.1]{bhattscholze2017witt}}]\label{lem:v-descent}
	The augmented \v Cech complex \eqref{eq:arc-descent} is exact for any v-covering $\spec(B) \to \spec(A)$ of affine $\bb{F}_p$-schemes.
\end{mylem}
\begin{proof}
	We write $B$ as a filtered colimit of finitely presented $A$-algebras $B = \colim B_\lambda$. Then, $\spec(B_\lambda) \to \spec(A)$ is an h-covering for each $\lambda$ by \ref{lem:arc-parc-basic}. Notice that $B_{\mrm{perf}} = \colim B_{\lambda,\mrm{perf}}$, then the conclusion follows from applying \ref{lem:h-descent} on $\spec(B_\lambda) \to \spec(A)$ and taking colimit.
\end{proof}

\begin{mylem}[{\cite[6.3]{bhattscholze2017witt}}]\label{lem:tech1}
	For any valuation ring $V$ and any prime ideal $\ak{p}$ of $V$, the sequence
	\begin{align}\label{eq:excision-valu}
	0 \longrightarrow V \stackrel{\alpha}{\longrightarrow} V/\ak{p} \oplus V_{\ak{p}} \stackrel{\beta}{\longrightarrow} V_{\ak{p}}/\ak{p}V_{\ak{p}} \longrightarrow 0
	\end{align}
	is exact, where $\alpha(a) = (a,a)$ and $\beta(a,b) = a-b$. If moreover $V$ is a perfect $\bb{F}_p$-algebra, then for any perfect $V$-algebra $R$, the base change of \eqref{eq:excision-valu} by $V\to R$,
	\begin{align}\label{eq:excision-valu-tensor}
	0 \longrightarrow R \longrightarrow R/\ak{p}R \oplus R_{\ak{p}} \longrightarrow R_{\ak{p}}/\ak{p}R_{\ak{p}} \longrightarrow 0
	\end{align}
	is exact.
\end{mylem}
\begin{proof}
	The sequence \eqref{eq:excision-valu} is exact if and only if $\ak{p}=\ak{p}V_{\ak{p}}$. Let $a \in \ak{p}$ and $s \in V \setminus \ak{p}$. Since $\ak{p}$ is an ideal, $s/a \notin V$, thus $a/s \in V$ as $V$ is a valuation ring. Moreover, we must have $a/s \in \ak{p}$ as $\ak{p}$ is a prime ideal. This shows the equality $\ak{p}=\ak{p}V_{\ak{p}}$.
	
	The second assertion follows directly from the fact that $\op{Tor}_q^A(B,C) =0 $ for any $q>0$ and any diagram $B\leftarrow A \to C$ of perfect $\bb{F}_p$-algebras (\cite[3.16]{bhattscholze2017witt}).
\end{proof}

\begin{mylem}[{\cite[4.8]{bhattmathew2020arc}}]\label{lem:tech2}
	The augmented \v Cech complex \eqref{eq:arc-descent} is exact for any arc-covering $\spec(B) \to \spec(A)$ of affine $\bb{F}_p$-schemes with $A$ a valuation ring.
\end{mylem}
\begin{proof}
	We follow the proof of Bhatt-Mathew \cite[4.8]{bhattmathew2020arc}. Let $B = \colim B_\lambda$ be a filtered colimit of finitely presented $A$-algebras. Then, $\spec(B_\lambda)  \to \spec(A) $ is also an arc-covering by \ref{lem:arc-parc-basic}. Thus, we may assume that $B$ is a finitely presented $A$-algebra.
	
	An interval $I=[\ak{p}, \ak{q}]$ of a valuation ring $A$ is a pair of prime ideals $\ak{p} \subseteq \ak{q}$ of $A$. We denote by $A_I = (A/\ak{p})_{\ak{q}}$. The set $\ca{I}$ of intervals of $A$ is partially ordered under inclusion. Let $\ca{P}$ be the subset consisting of intervals $I$ such that the lemma holds for $\spec(B\otimes_A A_I) \to \spec(A_I)$. It suffices to show that $\ca{P} = \ca{I}$.
	
	\begin{enumerate}
		\item If the valuation ring $A_I$ is of height $\leq 1$, we claim that $\spec(B\otimes_A A_I)  \to \spec(A_I) $ is automatically a v-covering. Indeed, there is an extension of valuation rings $A_I \to V$ of height $\leq 1$ which factors through $B\otimes_A A_I$. As $A_I \to V$ is faithfully flat, $\spec(B\otimes_A A_I) \to \spec(A_I)$ is a v-covering by \ref{lem:arc-parc-basic} and \ref{lem:arc-parc}.(\ref{lem:arc-parc-v}).  Therefore, $I \in \ca{P}$ by \ref{lem:v-descent}. 
		\item For any interval $J \subseteq I$, if $I \in \ca{P}$ then $J \in \ca{P}$. Indeed, applying $\otimes_{\bb{F}_p} (A_J)_{\mrm{perf}}$ to the exact sequence \eqref{eq:arc-descent} for $\spec(B\otimes_A A_I) \to \spec(A_I)$, we still get an exact sequence by the Tor-independence of perfect $\bb{F}_p$-algebras (\cite[3.16]{bhattscholze2017witt}).
		\item If $\ak{p} \subseteq A$ is not maximal, then there exists $\ak{q} \supsetneq \ak{p}$ with $I=[\ak{p}, \ak{q}] \in \ca{P}$. Indeed, if there is no such $I$ with the height of $A_I$ no more than $1$, then $\ak{p} = \bigcap_{\ak{q} \supsetneq \ak{p}} \ak{q}$, and thus,
		\begin{align}
		A_{\ak{p}}/\ak{p}A_{\ak{p}} = \colim_{I=[\ak{p}, \ak{q}],\ak{q} \supsetneq \ak{p}} A_I.
		\end{align}
		Since $\spec(B\otimes_A {A_{\ak{p}}/\ak{p}A_{\ak{p}}}) \to \spec(A_{\ak{p}}/\ak{p}A_{\ak{p}})$ is an h-covering as $A_{\ak{p}}/\ak{p}A_{\ak{p}}$ is a field (and we have assumed that $B$ is of finite presentation over $A$), there exists an interval $I$ in the above colimit, such that $\spec(B\otimes_A A_I) \to \spec(A_I)$ is also an h-covering by \ref{lem:arc-parc}.(\ref{lem:arc-parc-limit-h}). Therefore, this $I$ lies in $\ca{P}$ by \ref{lem:v-descent}.
		\item If $\ak{p} \subseteq A$ is nonzero, then there exists $\ak{q} \subsetneq \ak{p}$ with $I=[\ak{q}, \ak{p}] \in \ca{P}$. This is similar to (3).
		\item If $I, J \in \ca{P}$ are overlapping, then $I \cup J \in \ca{P}$. Indeed, by (2) and replacing $A$ by $A_{I \cup J}$, we may assume that $I = [0, \ak{p}]$, $J = [\ak{p}, \ak{m}]$ with $\ak{m}$ the maximal ideal. In particular, $A_I = A_{\ak{p}}$, $A_J = A/ \ak{p}$, and $A_{I \cap J} = A_{\ak{p}}/\ak{p}A_{\ak{p}}$. Since for each $R = \otimes_{A_{\mrm{perf}}}^n B_{\mrm{perf}}$ we have the short exact sequence \eqref{eq:excision-valu-tensor}, we get $I \cup J \in \ca{P}$.
	\end{enumerate}
	In general, by Zorn's lemma, the above five properties of $\ca{P}$ guarantee that $\ca{P} = \ca{I}$ (cf. \cite[4.7]{bhattmathew2020arc}).
\end{proof}

\begin{mylem}[{cf. \cite[3.30]{bhattmathew2020arc}}]\label{lem:tech3}
	The augmented \v Cech complex \eqref{eq:arc-descent} is exact for any arc-covering $\spec(B) \to \spec(A)$ of affine $\bb{F}_p$-schemes with $A$ a product of valuation rings.
\end{mylem}
\begin{proof}
	We follow closely the proof of \ref{prop:v-site-stalk}. Let $\spec(V)$ be a connected component of $\spec(A)$ with the reduced closed subscheme structure. Then, $V$ is a valuation ring by \ref{lem:prod-val}.(\ref{lem:prod-val-conn}). By \ref{lem:tech2}, the augmented \v Cech complex
	\begin{align}\label{eq:1.8.1}
		0 \to V_{\mrm{perf}} \to (B\otimes_A V)_{\mrm{perf}} \to (B\otimes_A V)_{\mrm{perf}}\otimes_{V_{\mrm{perf}}}(B\otimes_A V)_{\mrm{perf}} \to \cdots
	\end{align}
	is exact. Notice that each connected component of an affine scheme is the intersection of some open and closed subsets (\cite[\href{https://stacks.math.columbia.edu/tag/04PP}{04PP}]{stacks-project}). Moreover, since $A$ is reduced, we have $V = \colim A'$, where the colimit is taken over all the open and closed subschemes $\spec(A')$ which contain $\spec(V)$. 
	
	Therefore, by a limit argument, for an element $f \in \otimes^n_{A_{\mrm{perf}}}B_{\mrm{perf}}$ which maps to zero in $\otimes^{n+1}_{A_{\mrm{perf}}}B_{\mrm{perf}}$, as $\spec(A)$ is quasi-compact, we can decompose $\spec(A)$ into a finite disjoint union $\coprod_{i=1}^N \spec(A_i)$ such that there exists $g_i \in \otimes^{n-1}_{A_{i,\mrm{perf}}}(B \otimes_A {A_i})_{\mrm{perf}}$ which maps to the image $f_i$ of $f$ in $\otimes^{n}_{A_{i,\mrm{perf}}}(B \otimes_A {A_i})_{\mrm{perf}}$. Since we have
	\begin{align}
	\otimes^{n}_{A_{\mrm{perf}}}B_{\mrm{perf}} = \prod_{i=1}^N \otimes^{n}_{A_{i,\mrm{perf}}}(B \otimes_A {A_i})_{\mrm{perf}},
	\end{align} 
	the element $g=(g_i)_{i=1}^N$ maps to $f$, which shows the exactness of \eqref{eq:arc-descent}.
\end{proof}

\begin{myprop}[{\cite[8.10]{bhattscholze2019prisms}}]\label{prop:arc-descent}
	Let $\tau \in \{\trm{fppf, h, v, arc}\}$.
	\begin{enumerate}
		\renewcommand{\labelenumi}{{\rm(\theenumi)}}
		\item The presheaf $\ca{O}_{\mrm{perf}}$ is a $\tau$-sheaf over $\schqcqs_{\bb{F}_p}$, and for any coherent $\bb{F}_p$-scheme $X$ and any integer $q$,
		\begin{align}
			H^q_\tau(X, \ca{O}_{\mrm{perf}}) = \colim_{\frob} H^q(X,\ca{O}_X).
		\end{align}
		\item  For any $\tau$-covering $\spec(B) \to \spec(A)$ of affine $\bb{F}_p$-schemes, the augmented \v Cech complex 
		\begin{align}\label{eq:cech}
			0 \to A_{\mrm{perf}} \to B_{\mrm{perf}} \to B_{\mrm{perf}} \otimes_{A_{\mrm{perf}}} B_{\mrm{perf}} \to \cdots
		\end{align}
		is exact.
	\end{enumerate}
\end{myprop}
\begin{proof}
	We follow closely the proof of Bhatt-Scholze \cite[8.10]{bhattscholze2019prisms}. (1) and (2) are equivalent by \ref{lem:star-descent}, and they hold for $\tau \in \{\trm{fppf, h, v}\}$ by \ref{lem:fppf-descent}, \ref{lem:h-descent} and \ref{lem:v-descent}. In particular,
	\begin{align}
		H^0_{\mrm{v}}(\spec(A),\ca{O}_{\mrm{perf}}) = A_{\mrm{perf}} \trm{ and } H^q_{\mrm{v}}(\spec(A),\ca{O}_{\mrm{perf}}) = 0, \ \forall q >0.
	\end{align}
	We take a hypercovering in the v-topology $\spec(A_\bullet) \to \spec(A)$ such that $A_n$ is a product of valuation rings for each degree $n$ by \ref{prop:prod-val-cov} and \cite[\href{https://stacks.math.columbia.edu/tag/094K}{094K} and \href{https://stacks.math.columbia.edu/tag/0DB1}{0DB1}]{stacks-project}. The associated sequence
	\begin{align}\label{eq:arc-hyperdescent}
	0 \to A_{\mrm{perf}} \to A_{0,\mrm{perf}} \to A_{1,\mrm{perf}} \to \cdots
	\end{align}
	is exact by the hyper-\v Cech-cohomology-to-cohomology spectral sequence \cite[\href{https://stacks.math.columbia.edu/tag/01GY}{01GY}]{stacks-project}.

	Consider the double complex $(A_i^j)$ where the $i$-th row $A_{i}^\bullet$ is the base change of \eqref{eq:cech} by $A_{\mrm{perf}} \to A_{i,\mrm{perf}}$, i.e. the augmented \v Cech complex \eqref{eq:arc-descent} associated to $\spec(B \otimes_A {A_i}) \to \spec(A_i)$ (we set $A_{-1}=A$). On the other hand, the $j$-th column $A_\bullet^j$ is the associated sequence \eqref{eq:arc-hyperdescent} to the hypercovering $\spec({A_{\bullet}} \otimes_A (\otimes_A^j B)) \to \spec(\otimes_A^j B)$, which is exact by the previous discussion. Since $A_{-1}^\bullet \to \mrm{Tot}(A_i^j)_{i\geq 0}^{j\geq 0}$ is a quasi-isomorphism (\cite[\href{https://stacks.math.columbia.edu/tag/0133}{0133}]{stacks-project}), for the exactness of the $(-1)$-row $A_{-1}^\bullet$, we only need to show the exactness of the $i$-th row $A_i^\bullet$ for any $i\geq 0$ but this has been proved in \ref{lem:tech3} thanks to our choice of the hypercovering, which completes the proof.
\end{proof}

\section{Almost Pre-perfectoid Algebras}\label{sec:perfectoid}
%In this section, we review the theory of perfectoid algebras following Scholze \cite{scholze2012perfectoid} in terms of "uncompleted" commutative algebras. We will give a detailed proof of the cohomological arc-descent of perfectoid algebras (cf. \ref{thm:arc-descent-perf}) and state the almost purity theorem (cf. \ref{thm:almost-purity}).

\begin{mydefn}\label{defn:perfectoid-field}
	\mbox{ }
	\begin{enumerate}
		\item A \emph{pre-perfectoid field} $K$ is a valuation field whose valuation ring $\ca{O}_K$ is non-discrete, of height $1$ and of residue characteristic $p$, and such that the Frobenius map on $\ca{O}_K/p\ca{O}_K$ is surjective. 
		\item A \emph{perfectoid field} $K$ is a pre-perfectoid field which is complete for the topology defined by its valuation (cf. \cite[3.1]{scholze2012perfectoid}).
		\item A \emph{pseudo-uniformizer} $\pi$ of a pre-perfectoid field $K$ is a nonzero element of the maximal ideal $\ak{m}_K$ of $\ca{O}_K$.
	\end{enumerate}
	A \emph{morphism of pre-perfectoid fields} $K \to L$ is a homomorphism of fields which induces an extension of valuation rings $\ca{O}_K\to \ca{O}_L$.
\end{mydefn}

\begin{mylem}\label{lem:pre-perf-field}
	Let $K$ be a pre-perfectoid field with a pseudo-uniformizer $\pi$. Then, the fraction field $\widehat{K}$ of the $\pi$-adic completion of $\ca{O}_K$ is a perfectoid field.
\end{mylem}
\begin{proof}
	The $\pi$-adic completion $\widehat{\ca{O}_K}$ of $\ca{O}_K$ is still a non-discrete valuation ring of height $1$ with residue characteristic $p$ (cf. \cite[\Luoma{6}.\textsection5.3, Prop.5]{bourbaki2006commalg5-7}). If $p\neq 0$, then it is canonically isomorphic to the $p$-adic completion of $\ca{O}_K$, so that there is a canonical isomorphism $\ca{O}_K/p\ca{O}_K \iso \widehat{\ca{O}_K}/p\widehat{\ca{O}_K}$, from which we see that $\widehat{K}$ is a perfectoid field. If $p=0$, then the Frobenius induces a surjection $\ca{O}_K \to \ca{O}_K$ if and only if $\ca{O}_K$ is perfect. Thus, $\widehat{\ca{O}_K}$ is also perfect, and we see that $\widehat{K}$ is a perfectoid field.
\end{proof}

\begin{mypara}\label{para:valuation}
	Let $K$ be a pre-perfectoid field. There is a unique (up to scalar) ordered group homomorphism $v_K: K^\times \to \bb{R}$ such that $v_K^{-1}(0) = \ca{O}_K^\times$, where the group structure on $\bb{R}$ is given by the addition. In particular, $\ca{O}_K\setminus 0= v_K^{-1}(\bb{R}_{\geq 0})$ and $\ak{m}_K\setminus 0= v_K^{-1}(\bb{R}_{> 0})$ (cf. \cite[\Luoma{6}.\textsection4.5 Prop.7]{bourbaki2006commalg5-7} and \cite[\Luoma{5}.\textsection2 Prop.1, Rem.2]{bourbaki2007top5-10}). The non-discrete assumption on $\ca{O}_K$ implies that the image $v_K(K^\times) \subseteq \bb{R}$ is dense. We set $v_K(0)=+\infty$.
\end{mypara}

\begin{mylem}[{\cite[3.2]{scholze2012perfectoid}}]\label{lem:adm-uniformizer}
	Let $K$ be a pre-perfectoid field. Then, for any pseudo-uniformizer $\pi$ of $K$, there exists $\pi_n \in \ak{m}_K$ for each integer $n \geq 0$ such that $\pi_0 = \pi$ and $\pi_n = u_n\cdot \pi_{n+1}^p$ for some unit $u_n \in \ca{O}_K^\times$, and $\ak{m}_K$ is generated by $\{\pi_n\}_{n\geq 0}$.
\end{mylem}
\begin{proof}
	If $v_K(\pi)<v_K(p)$, since the Frobenius is surjective on $\ca{O}_K/p$, there exists $\pi_1 \in \ca{O}_K$ such that $v_K(\pi-\pi_1^p)\geq v_K(p)$. Then, $v_K(\pi)=v_K(\pi_1^p)$ and thus $\pi = u \cdot \pi_1^p$ with $u \in \ca{O}_K^\times$. In general, since $v_K(K^\times) \subseteq \bb{R}$ is dense, any pseudo-uniformizer $\pi$ is a finite product of pseudo-uniformizers whose valuation values are strictly less than $v_K(p)$, from which we get a $p$-th root $\pi_1$ of $\pi$ up to a unit. Since $\pi_1$ is also a pseudo-uniformizer, we get $\pi_n$ inductively. As $v_K(\pi_n)$ tends to zero when $n$ tends to infinity, $\ak{m}_K$ is generated by $\{\pi_n\}_{n\geq 0}$.
\end{proof}

\begin{mypara}\label{para:almost-math}
	Let $K$ be a pre-perfectoid field. We briefly review almost algebra over $(\ca{O}_K, \ak{m}_K)$ for which we mainly refer to \cite[2.6]{abbes2020suite}, \cite[\Luoma{5}]{abbes2016p} and \cite{gabber2003almost}. Remark that $\ak{m}_K \otimes_{\ca{O}_K} \ak{m}_K \cong \ak{m}_K^2 = \ak{m}_K$ is flat over $\ca{O}_K$. 
	
	An $\ca{O}_K$-module $M$ is called \emph{almost zero} if $\ak{m}_K M =0$. A morphism of $\ca{O}_K$-modules $M \to N$ is called an \emph{almost isomorphism} if its kernel and cokernel are almost zero. Let $\scr{N}$ be the full subcategory of the category $\ca{O}_K\module$ of $\ca{O}_K$-modules formed by almost zero objects. It is clear that $\scr{N}$ is a Serre subcategory of $\ca{O}_K\module$ (\cite[\href{https://stacks.math.columbia.edu/tag/02MO}{02MO}]{stacks-project}). Let $\ca{S}$ be the set of almost isomorphisms in $\ca{O}_K\module$. Since $\scr{N}$ is a Serre subcategory, $\ca{S}$ is a multiplicative system, and moreover the quotient abelian category $\ca{O}_K\module/\scr{N}$ is representable by the localized category $\ca{S}^{-1}\ca{O}_K\module$ (cf. \cite[\href{https://stacks.math.columbia.edu/tag/02MS}{02MS}]{stacks-project}). We denote $\ca{S}^{-1}\ca{O}_K\module$ by $\ca{O}_K^{\al}\module$, whose objects are called \emph{almost $\ca{O}_K$-modules} or simply \emph{$\ca{O}_K^\al$-modules} (cf. \cite[2.6.2]{abbes2020suite}). We denote by
	\begin{align}
		\alpha^*: \ca{O}_K\module \longrightarrow \ca{O}_K^{\al}\module,\ M \longmapsto M^\al
	\end{align}
	the localization functor. 
	It induces an $\ca{O}_K$-linear structure on $\ca{O}_K^{\al}\module$. For any two $\ca{O}_K$-modules $M$ and $N$, we have a natural $\ca{O}_K$-linear isomorphism (\cite[2.6.7.1]{abbes2020suite})
	\begin{align}\label{eq:alhom}
		\ho_{\ca{O}_K^{\al}\module}(M^\al, N^\al) = \ho_{\ca{O}_K\module}(\ak{m}_K\otimes_{\ca{O}_K} M , N).
	\end{align} 
	
	The localization functor $\alpha^*$ admits a right adjoint 
	\begin{align}
	\alpha_* : \ca{O}_K^{\al}\module \longrightarrow \ca{O}_K\module,\ M \longmapsto M_*=\ho_{\ca{O}_K^{\al}\module}(\ca{O}_K^{\al}, M),
	\end{align}
	and a left adjoint
	\begin{align}
	\alpha_! : \ca{O}_K^{\al}\module \longrightarrow \ca{O}_K\module,\ M \longmapsto M_!=\ak{m}_K \otimes_{\ca{O}_K} M_*.
	\end{align}
	Moreover, the natural morphisms
	\begin{align}\label{eq:star-al}
	(M_*)^\al \iso M,\ M \iso (M_!)^\al
	\end{align}
	are isomorphisms for any $\ca{O}_K^\al$-module $M$ (cf. \cite[2.6.8]{abbes2020suite}). In particular, for any functor $\varphi: I \to \ca{O}_K^\al\module$ sending $i$ to $M_i$, the colimit and limit of $\varphi$ are representable by
	\begin{align}\label{eq:lim-colim-mod}
		\colim M_i = (\colim M_{i*})^\al,\ \lim M_i = (\lim M_{i*})^\al.
	\end{align}
	
	The tensor product in $\ca{O}_K\module$ induces a tensor product in $\ca{O}_K^{\al}\module$ by
	\begin{align}\label{eq:altensor}
		M^\al \otimes_{\ca{O}_K^\al} N^\al = (M\otimes_{\ca{O}_K} N)^\al
	\end{align}
	making $\ca{O}_K^{\al}\module$ an abelian tensor category (\cite[2.6.4]{abbes2020suite}). 
	We denote by $\ca{O}_K^\al\alg$	the category of commutative unitary monoids in $\ca{O}_K^\al\module$ induced by the tensor structure, whose objects are called \emph{almost $\ca{O}_K$-algebras} or simply \emph{$\ca{O}_K^\al$-algebras} (cf. \cite[2.6.11]{abbes2020suite}). Notice that $R^\al$ (resp. $R_*$) admits a canonical algebra structure for any $\ca{O}_K$-algebra (resp. $\ca{O}_K^\al$-algebra) $R$. Moreover, $\alpha^*$ and $\alpha_*$ induce adjoint functors between $\ca{O}_K\alg$ and $\ca{O}_K^\al\alg$ (\cite[2.6.12]{abbes2020suite}). Combining with \eqref{eq:star-al} and \eqref{eq:lim-colim-mod}, we see that for any functor $\varphi: I \to \ca{O}_K^\al\alg$ sending $i$ to $R_i$, the colimit and limit of $\varphi$ are representable by (cf. \cite[2.2.16]{gabber2003almost})
	\begin{align}\label{eq:lim-colim-alg}
		\colim R_i = (\colim R_{i*})^\al,\ \lim R_i = (\lim R_{i*})^\al.
	\end{align}
	In particular, for any diagram $B\leftarrow A \to C$ of $\ca{O}_K^\al$-algebras, we denote its colimit by
	\begin{align}\label{eq:tensor-alg}
		B\otimes_A C = (B_* \otimes_{A_*} C_*)^\al,
	\end{align}
	which is clearly compatible with the tensor products of modules. We remark that $\alpha^*$ commutes with arbitrary colimits (resp. limits), since it has a right adjoint $\alpha_*$ (resp. since the forgetful functor $\ca{O}_K^\al\alg\to \ca{O}_K^\al\module$ and the localization functor $\alpha^*:\ca{O}_K\module\to \ca{O}_K^\al\module$ commute with arbitrary limits).
\end{mypara}
	
\begin{mypara}\label{para:almost-frob}
	For an element $a$ of $\ca{O}_K$, we denote by $(\ca{O}_K/a\ca{O}_K)^\al\module$ the full subcategory of $\ca{O}_K^\al\module$ formed by the objects on which the morphism induced by multiplication by $a$ is zero. Notice that for an $(\ca{O}_K/a\ca{O}_K)^\al$-module $M$, $M_*$ is an $\ca{O}_K/a\ca{O}_K$-module. Thus, the localization functor $\alpha^*$ induces an essentially surjective exact functor $(\ca{O}_K/a\ca{O}_K)\module\to (\ca{O}_K/a\ca{O}_K)^\al\module$, which identifies the latter with the quotient abelian category $(\ca{O}_K/a\ca{O}_K)\module/\scr{N}\cap(\ca{O}_K/a\ca{O}_K)\module$.
	
	Let $\pi$ be a pseudo-uniformizer of $K$ dividing $p$ with a $p$-th root $\pi_1$ up to a unit. The Frobenius on $\ca{O}_K/\pi\ca{O}_K$ induces an isomorphism $\ca{O}_K/\pi_1\ca{O}_K \iso \ca{O}_K/\pi\ca{O}_K$. The Frobenius on $(\ca{O}_K/\pi)$-algebras and the localization functor $\alpha^*$ induce a natural transformation from the base change functor $(\ca{O}_K/\pi)^\al\alg \to (\ca{O}_K/\pi)^\al\alg$, $R \mapsto (\ca{O}_K/\pi)\otimes_{\frob,(\ca{O}_K/\pi)} R$ to the identity functor.
	\begin{align}
		\xymatrix{
			(\ca{O}_K/\pi)^\al\alg\ar[r]\ar@/_2pc/[rr]^-{\Downarrow\frob}_-{\id}& (\ca{O}_K/\pi_1)^\al\alg \ar[r]^-{\sim} & (\ca{O}_K/\pi)^\al\alg
		}
	\end{align}
	For an $(\ca{O}_K/\pi)^\al$-algebra $R$, we usually identify the $(\ca{O}_K/\pi_1)^\al$-algebra $R/\pi_1R$ with the $(\ca{O}_K/\pi)^\al$-algebra $(\ca{O}_K/\pi)\otimes_{\frob,(\ca{O}_K/\pi)} R$, and we denote by $R/\pi_1R\to R$ the natural morphism  $(\ca{O}_K/\pi)\otimes_{\frob,(\ca{O}_K/\pi)} R\to R$ induced by the Frobenius (cf. \cite[3.5.6]{gabber2003almost}).
	Moreover, the natural transformations induced by Frobenius for $(\ca{O}_K/\pi)\alg$ and $(\ca{O}_K/\pi)^\al\alg$ are also compatible with the functor $\alpha_*$. Indeed,
	it follows from the fact that for any $(\ca{O}_K/\pi)$-algebra $R$, the composition of
	\begin{align}
		\xymatrix{
		(\ca{O}_K/\pi)\otimes_{(\ca{O}_K/\pi)}\ho(\ak{m}_K, R) \ar[r]&\ho(\ak{m}_K, (\ca{O}_K/\pi)\otimes_{(\ca{O}_K/\pi)}R) \ar[rr]^-{\ho(\ak{m}_K,\frob)}&& \ho(\ak{m}_K, R)
		}
	\end{align}
	is the relative Frobenius on $(R^\al)_*=\ho_{\ca{O}_K\module}(\ak{m}_K, R)$.
\end{mypara}

\begin{mypara}\label{para:almost-derived}	
	 Let $C$ be a site. A presheaf $\ca{F}$ of $\ca{O}_K$-modules on $C$ is called \emph{almost zero} if $\ca{F}(U)$ is almost zero for any object $U$ of $C$. A morphism of presheaves $\ca{F} \to \ca{G}$ of $\ca{O}_K$-modules on $C$ is called an \emph{almost isomorphism} if $\ca{F}(U) \to \ca{G}(U)$ is an almost isomorphism for any object $U$ of $C$ (cf. \cite[2.6.23]{abbes2020suite}). Let $\scr{N}$ be the full subcategory of the category $\ca{O}_K\module_C$ of sheaves of $\ca{O}_K$-modules on $C$ formed by almost zero objects. Similarly, $\scr{N}$ is a Serre subcategory of $\ca{O}_K\module_C$. Let $\dd_{\scr{N}}(\ca{O}_K\module_C)$ be the full subcategory of the derived category $\dd (\ca{O}_K\module_C)$ formed by complexes with almost zero cohomologies. It is a strictly full saturated triangulated subcategory (\cite[\href{https://stacks.math.columbia.edu/tag/06UQ}{06UQ}]{stacks-project}). We also say that the objects of $\dd_{\scr{N}}(\ca{O}_K\module_C)$ are \emph{almost zero}. Let $\ca{S}$ be the set of arrows in $\dd (\ca{O}_K\module_C)$ which induce almost isomorphisms on cohomologies. We also call the elements of $\ca{S}$ \emph{almost isomorphisms}. Then, $\ca{S}$ is a saturated multiplicative system (\cite[\href{https://stacks.math.columbia.edu/tag/05RG}{05RG}]{stacks-project}), and moreover the quotient triangulated category $\dd(\ca{O}_K\module_C)/\dd_{\scr{N}}(\ca{O}_K\module_C)$ is representable by the localized triangulated category $\ca{S}^{-1}\dd(\ca{O}_K\module_C)$ (\cite[\href{https://stacks.math.columbia.edu/tag/05RI}{05RI}]{stacks-project}). The natural functor
	 \begin{align}
	 	\ca{S}^{-1}\dd(\ca{O}_K\module_C) \longrightarrow \dd(\ca{O}_K^\al\module_C)
	 \end{align}
 	is an equivalence by \cite[\href{https://stacks.math.columbia.edu/tag/06XM}{06XM}]{stacks-project} and \eqref{eq:star-al} (cf. \cite[2.4.9]{gabber2003almost}).
\end{mypara}

\begin{mylem}\label{lem:al-star}
	Let $K$ be a pre-perfectoid field with a pseudo-uniformizer $\pi$, $M$ a flat $\ca{O}_K$-module. We fix a system of $p^n$-th roots $(\pi_n)_{n\geq 0}$ of $\pi$ up to units \eqref{lem:adm-uniformizer}, then the map
	\begin{align}
	\bigcap_{n\geq 0} \pi_n^{-1}M \to (M^\al)_* = \ho_{\ca{O}_K\module}(\ak{m}_K, M),\ a \mapsto (x \mapsto xa)
	\end{align} 
	where $\pi_n^{-1}M \subseteq M[1/\pi]$, is an isomorphism of $\ca{O}_K$-modules. Moreover, for an extension of valuation rings $\ca{O}_K\to R$ of height $1$, we have $R = \bigcap_{n\geq 0} \pi_n^{-1}R$ and the above isomorphism coincides with the unit map $R \to (R^\al)_*$.
\end{mylem}
\begin{proof}
	Since $\ak{m}_K$ is generated by $\{\pi_n\}_{n\geq 0}$, any $\ca{O}_K$-linear morphism $f: \ak{m}_K \to M$ is determined by its values $f(\pi_n) \in M$. Notice that $(\pi/\pi_n)\cdot f(\pi_n) = f(\pi)$ and $M$ is $\pi$-torsion free, so that $f$ must be given by the multiplication by an element $a = f(\pi)/\pi \in M[1/\pi]$. It is clear that such a multiplication sends $\ak{m}_K$ to $M$ if and only if $a \in \bigcap_{n\geq 0} \pi_n^{-1}M$, which shows the first assertion. If $\ca{O}_K\to R$ is an extension of valuation rings of height $1$, then we directly deduce from the valuation map $v : R[1/\pi]\setminus 0 \to \bb{R}$ (\ref{para:valuation}) the equality $R = \bigcap_{n\geq 0} \pi_n^{-1}R$.
\end{proof}

\begin{mylem}\label{lem:R-V}
	Let $K$ be a pre-perfectoid field, $R$ an $\ca{O}_K$-algebra, $\ca{O}_K\to V$ an extension of valuation rings of height $1$. Then, the canonical map
	\begin{align}\label{eq:R-V}
	\ho_{\ca{O}_K\alg} (R, V) \longrightarrow \ho_{\ca{O}_K^{\al}\alg}(R^\al, V^\al)
	\end{align}
	is bijective.
\end{mylem}
\begin{proof}
	There are natural maps
	\begin{align}
		\ho_{\ca{O}_K\alg} (R, V) \to \ho_{\ca{O}_K^{\al}\alg}(R^\al, V^\al) \iso \ho_{\ca{O}_K\alg}(R, (V^\al)_*) \iso \ho_{\ca{O}_K\alg}(R, V),
	\end{align}
	where the middle isomorphism is given by adjunction and the last isomorphism is induced by the inverse of the unit map $V \to (V^\al)_*$ by \ref{lem:al-star}. The composition is the identity map, which completes the proof.
\end{proof}

\begin{mydefn}\label{defn:flat}
	Let $K$ be a pre-perfectoid field. We say that an $\ca{O}_K^\al$-module $M$ (resp. an $\ca{O}_K$-module $M$) is \emph{flat} (resp. \emph{almost flat}) if the functor $\ca{O}_K^\al\module \to \ca{O}_K^\al\module$ given by tensoring with $M$ is exact (resp. $M^\al$ is flat). 
\end{mydefn}

\begin{myrem}\label{rem:flat}
	In general, one can define the flatness of a morphism of $\ca{O}_K^\al$-algebras (cf. \cite[3.1.1.(\luoma{1})]{gabber2003almost}). We say that a morphism of $\ca{O}_K$-algebras $A\to B$ is \emph{almost flat} if $A^\al\to B^\al$ is flat.
\end{myrem}

\begin{mylem}\label{lem:flat}
	Let $K$ be a pre-perfectoid field with a pseudo-uniformizer $\pi$. Then, an $\ca{O}_K^\al$-module $M$ is flat if and only if $M_*$ is $\pi$-torsion free. In particular, an $\ca{O}_K$-module $N$ is almost flat if and only if the submodule of $\pi$-torsion elements of $N$ is almost zero.
\end{mylem}
\begin{proof}
	First of all, for any $\ca{O}_K^\al$-modules $L_1$ and $L_2$, we have a canonical isomorphism
	\begin{align}
		\ho_{\ca{O}_K^\al\module}(M\otimes_{\ca{O}_K^\al} L_1, L_2) = \ho_{\ca{O}_K^\al\module}(L_1,\ho_{\ca{O}_K\module} (M_*, L_{2*})^\al)
	\end{align}
	by \eqref{eq:alhom}, \eqref{eq:star-al} and \eqref{eq:altensor}. Therefore, the functor defined by tensoring with $M$ admits a right adjoint, and thus it is right exact.
	Consider the sequence 
	\begin{align}
		0 \longrightarrow \ca{O}_K^\al \stackrel{\cdot \pi}{\longrightarrow} \ca{O}_K^\al \longrightarrow (\ca{O}_K/\pi\ca{O}_K)^\al \longrightarrow 0,
	\end{align}
 	which is exact since the localization functor $\alpha^*$ is exact. If $M$ is flat, tensoring the above sequence with $M$ and applying $\alpha_*$, we deduce that $M_*$ is $\pi$-torsion free since $\alpha_*$ is left exact (as a right adjoint to $\alpha^*$). Conversely, if $M_*$ is $\pi$-torsion free, then it is flat over $\ca{O}_K$. For any injective morphism $L_1 \to L_2$ of $\ca{O}_K^\al$-modules, $L_{1*} \to L_{2*}$ is also injective, and it remains injective after tensoring with $M_*$. Therefore, $L_1 \to L_2$ also remains injective after tensoring with $M$ since $\alpha^*$ is exact. This shows that $M$ is flat.
 	
 	The second assertion follows from the almost isomorphism $N \to (N^\al)_*$ and the fact that $(N^\al)_* = \ho_{\ca{O}_K\module}(\ak{m}_K, N)$ has no nonzero almost zero element.
\end{proof}

\begin{mylem}\label{lem:almost-complete}
	Let $K$ be a pre-perfectoid field with a pseudo-uniformizer $\pi$, $M$ a flat $\ca{O}_K^\al$-module, $x$ an element of $\ca{O}_K$. Then, the canonical morphism $M_*/xM_*\to (M/xM)_*$ is injective, and for any $\epsilon\in \ak{m}_K$, the image of $\varphi_\epsilon:(M/\epsilon xM)_*\to (M/xM)_*$ is $M_*/xM_*$. In particular, the canonical morphism
	\begin{align}
		\plim_n M_*/\pi^n M_*\longrightarrow (\plim_n M/\pi^n M)_*
	\end{align}
	is an isomorphism of $\ca{O}_K$-modules.
\end{mylem}
\begin{proof}
	We follow the proof of \cite[5.3]{scholze2012perfectoid}. Applying the left exact functor $\alpha_*$ to the exact sequence
	\begin{align}
		\xymatrix{
			0\ar[r]& M\ar[r]^-{\cdot x} & M\ar[r] &M/xM\ar[r]& 0,
		}
	\end{align}
	we see that $M_*/xM_*\to (M/xM)_*$ is injective. 
	
	To show that the image of $\varphi_\epsilon$ is $M_*/xM_*$, it suffices to show that $\varphi_\epsilon$ factors through $M_*/xM_*$. We identify $(M/xM)_*$ with $\ho_{\ca{O}_K\module}(\ak{m}_K, M_*/xM_*)$ by \eqref{eq:star-al} and \eqref{eq:alhom} so that $M_*/xM_*$ identifies with the subset consisting of the $\ca{O}_K$-morphisms $\ak{m}_K \to M_*/xM_*$ sending $y$ to $ya$ for some element $a\in M_*/xM_*$. For an $\ca{O}_K$-morphism $f:\ak{m}_K \to M_*/\epsilon xM_*$, let $b$ be an element of $M_*$ which lifts $f(\epsilon)$. Notice that $M_*$ is $\pi$-torsion free by \ref{lem:flat}. With notation in \ref{lem:al-star}, we have $b\equiv (\epsilon/\pi_n)\cdot f(\pi_n)\ \mod \epsilon x M_*$ for $n$ big enough so that the element $b/\epsilon\in M_*[1/\pi]$ lies in $\bigcap_{n\geq 0} \pi_n^{-1}M_*=M_*$. Moreover, $\pi_n\cdot(b/\epsilon) \equiv f(\pi_n)\ \mod x M_*$ for $n$ big enough. As $\varphi_\epsilon(f)$ is determined by its values on $\pi_n$ for $n$ big enough, it follows that $\varphi_\epsilon(f)=a$, where $a$ is the image of $b/\epsilon$ in $M_*/xM_*$.
	
	Finally, the previous result implies that the inverse system $((M/\pi^n M)_*)_{n\geq 1}$ is Mittag-Leffler so that the ``in particular'' part follows immediately from the fact that $\alpha_*$ commutes with arbitrary limits (as a right adjoint to $\alpha^*$) (\cite[\href{https://stacks.math.columbia.edu/tag/0596}{0596}]{stacks-project}).
\end{proof}

\begin{mydefn}\label{defn:tilt}
	Let $K$ be a pre-perfectoid field. For any $\ca{O}_K$-algebra $R$, we define a perfect ring $R^\flat$ as the projective limit
	\begin{align}
		R^\flat = \varprojlim_{\frob} R/p R
	\end{align}
	indexed by $(\bb{N},\leq)$, where transition map associated to $i \leq (i+1)$ is the Frobenius on $R/pR$. We call $R^\flat$ the \emph{tilt} of $R$. 
\end{mydefn}

\begin{mylem}[{\cite[3.4]{scholze2012perfectoid}}]\label{lem:tilt-perf-field}
	Let $K$ be a perfectoid field with a pseudo-uniformizer $\pi$ dividing $p$.
	\begin{enumerate}
		\item[\rm(1)] The projection induces an isomorphism of multiplicative monoids
		\begin{align}\label{eq:mulmonoids}
			\plim_{\frob} \ca{O}_K \longrightarrow \plim_{\frob} \ca{O}_K/\pi \ca{O}_K.
		\end{align}
		In particular, the right hand side is canonically isomorphic to $(\ca{O}_K)^\flat$ as a ring.
		\item[\rm(2)] We denote by 
		\begin{align}
			\sharp: (\ca{O}_K)^\flat \longrightarrow \ca{O}_K, \ x \mapsto x^\sharp,
		\end{align}
		the composition of the inverse of \eqref{eq:mulmonoids} and the projection onto the first component. Then $v_K \circ \sharp : (\ca{O}_K)^\flat\setminus 0 \to \bb{R}_{\geq 0}$ defines a valuation of height $1$ on $(\ca{O}_K)^\flat$.
		\item[\rm(3)] The fraction field $K^\flat$ of $(\ca{O}_K)^\flat$ is a perfectoid field of characteristic $p$ and the element
		\begin{align}
			\pi^\flat = (\cdots, \pi_1^{1/p^2},\pi_1^{1/p}, \pi_1,0) \in (\ca{O}_K)^\flat
		\end{align}
		is a pseudo-uniformizer of $K^\flat$, where $\pi = u \cdot \pi_1^p$ with $\pi_1 \in \ak{m}_K$ and $u \in \ca{O}_K^\times$.
		\item[\rm(4)] We have $\ca{O}_{K^\flat} = (\ca{O}_K)^\flat$, and there is a canonical isomorphism
		\begin{align}\label{eq:2.23.4}
			\ca{O}_{K^\flat}/\pi^\flat \ca{O}_{K^\flat} \iso \ca{O}_K/\pi\ca{O}_K
		\end{align}
		induced by {\rm(1)} and the projection onto the first component.
	\end{enumerate}
\end{mylem}

\begin{mypara}
	We see that the tilt defines a functor $\ca{O}_K\alg \to \ca{O}_{K^\flat}\alg$, $R\mapsto R^\flat$, which preserves almost zero objects and almost isomorphisms. For an $\ca{O}_K^\al$-algebra $R$, we set $R^\flat=((R_*)^\flat)^\al$ and call it the \emph{tilt} of $R$, which induces a functor $\ca{O}_K^\al\alg \to \ca{O}_{K^\flat}^\al\alg$, $R \mapsto R^\flat$. Note that the tilt functor commutes with the localization functor $\alpha^*$ up to a canonical isomorphism, and commutes with the functor $\alpha_*$ up to a canonical almost isomorphism.
\end{mypara}

\begin{mydefn}[{\cite[5.1]{scholze2012perfectoid}}]\label{defn:pre-perf-alg}
	Let $K$ be a perfectoid field, $\pi$ a pseudo-uniformizer of $K$ dividing $p$ with a $p$-th root $\pi_1$ up to a unit.
	\begin{enumerate}
		\item A \emph{perfectoid $\ca{O}_K^\al$-algebra} is an $\ca{O}_K^\al$-algebra $R$ such that 
		\begin{enumerate}
			\renewcommand{\theenumii}{\roman{enumii}}
			\renewcommand{\labelenumii}{{\rm(\theenumii)}}
			\item $R$ is flat over $\ca{O}_K^\al$;
			\item the Frobenius of $R/\pi R$ induces an isomorphism $R/\pi_1R \to R/\pi R$ of $\ca{O}_K^\al$-algebras (\ref{para:almost-frob});
			\item the canonical morphism $R \to \varprojlim_n R/\pi^n R$ is an isomorphism in $\ca{O}_K^\al\alg$.
		\end{enumerate}
		We denote by $\ca{O}_K^\al\perf$ the full subcategory of $\ca{O}_K^\al\alg$ formed by perfectoid $\ca{O}_K^\al$-algebras.
		\item A \emph{perfectoid $(\ca{O}_K/\pi)^\al$-algebra} is a flat $(\ca{O}_K/\pi)^\al$-algebra $R$ such that the Frobenius map induces an isomorphism $R/\pi_1R \iso R$. We denote by $(\ca{O}_K/\pi)^\al\perf$ the full subcategory of $(\ca{O}_K/\pi)^\al\alg$ formed by perfectoid $(\ca{O}_K/\pi)^\al$-algebras.
	\end{enumerate}
\end{mydefn}

\begin{mylem}\label{lem:pre-perf-alg}
	Let $K$ be a pre-perfectoid field, $\pi$ a pseudo-uniformizer of $K$ dividing $p$ with a $p$-th root $\pi_1$ up to a unit. Then, for an $\ca{O}_K$-algebra $R$, the following conditions are equivalent:
	\begin{enumerate}
		\renewcommand{\labelenumi}{{\rm(\theenumi)}}
		\item The almost algebra $\widehat{R}^\al$ associated to the $\pi$-adic completion $\widehat{R}$ of $R$ is a perfectoid $\ca{O}_{\widehat{K}}^\al$-algebra.
		\item The $\ca{O}_{\widehat{K}}$-module $\widehat{R}$ is almost flat, and the Frobenius of $R/\pi R$ induces an almost isomorphism $R/\pi_1R \to R/\pi R$.
	\end{enumerate}
\end{mylem}
\begin{proof}	
	We have seen that $\widehat{K}$ is a perfectoid field in \ref{lem:pre-perf-field} and $\pi$ is obviously a pseudo-uniformizer of $\widehat{K}$.
	Since the localization functor $\alpha^*: \ca{O}_K\alg \to \ca{O}_K^\al\alg$ commutes with arbitrary limits and colimits (\ref{para:almost-math}), we have a canonical isomorphism $\widehat{R}^\al \iso \varprojlim_n \widehat{R}^\al/\pi^n \widehat{R}^\al$. Thus, the third condition in \ref{defn:pre-perf-alg}.(1) holds for $\widehat{R}^\al$. Since there are canonical isomorphisms
	\begin{align}
		R/\pi_1 R \iso \widehat{R}/\pi_1\widehat{R}, \ R/\pi R \iso \widehat{R}/\pi\widehat{R},
	\end{align}
	the conditions (1) and (2) are clearly equivalent.
\end{proof}

\begin{mydefn}\label{defn:pre-alg}
	Let $K$ be a pre-perfectoid field, $\pi$ a pseudo-uniformizer of $K$ dividing $p$ with a $p$-th root $\pi_1$ up to a unit. We say that an $\ca{O}_K$-algebra is \emph{almost pre-perfectoid} if it satisfies the equivalent conditions in \ref{lem:pre-perf-alg}.
\end{mydefn}

We remark that in \ref{defn:pre-alg}, if a morphism of $\ca{O}_K$-algebras $R\to R'$ induces an almost isomorphism $R/\pi^n R \to R'/\pi^n R'$ for each $n\geq 1$, then the morphism of the $\pi$-adic completions $\widehat{R}\to \widehat{R'}$ is an almost isomorphism since $\alpha^*$ commutes with limits. In particular, $R$ is almost pre-perfectoid if and only if $R'$ is almost pre-perfectoid.

\begin{mylem}\label{lem:completion-flat}
	Let $K$ be a pre-perfectoid field with a pseudo-uniformizer $\pi$, $R$ an $\ca{O}_K$-algebra. If $R$ is almost flat (resp. flat) over $\ca{O}_K$, then the $\pi$-adic completion $\widehat{R}$ is almost flat (resp. flat) over $\ca{O}_{\widehat{K}}$.
\end{mylem}
\begin{proof}
	For any integer $n>0$, there is a canonical isomorphism
	\begin{align}\label{eq:1.29.1}
		R/\pi^n R \iso \widehat{R}/\pi^n\widehat{R}.
	\end{align}
	Let $x \in \widehat{R}$ be a $\pi$-torsion element. Since any $\pi$-torsion element of $R$ is almost zero (resp. zero) by \ref{lem:flat}, for any $\epsilon \in \ak{m}_K$ (resp. $\epsilon=1$), the image of $\epsilon x$ in $\widehat{R}/\pi^n\widehat{R}$ lies in $\pi^{n-1}\widehat{R}/\pi^n\widehat{R}$. Therefore, $\epsilon x \in \bigcap_{n> 0} \pi^{n-1}\widehat{R} =0 $, which amounts to say that $\widehat{R}$ is almost flat (resp. flat) over $\ca{O}_{\widehat{K}}$.
\end{proof}

\begin{mylem}\label{lem:p-clos}
	Let $K$ be a pre-perfectoid field, $\pi$ a pseudo-uniformizer of $K$ dividing $p$ with a $p$-th root $\pi_1$ up to a unit, $R$ a flat $\ca{O}_K$-algebra. Then, the following conditions are equivalent:
	\begin{enumerate}
		\renewcommand{\labelenumi}{{\rm(\theenumi)}}
		\item The Frobenius induces an injection $R/ \pi_1 R \to R/\pi R$.
		\item For any $x \in R[1/\pi]$, if $x^p \in R$, then $x \in R$.
	\end{enumerate}
\end{mylem}
\begin{proof}
	We follow the proof of \cite[5.7]{scholze2012perfectoid}. Assume first that $R/ \pi_1 R \to R/\pi R$ is injective. Let $x \in R[1/\pi]$ with $x^p \in R$, $k$ the minimal natural number such that $y = \pi_1^{k} x \in R$. If $k\geq 1$, then $y^p = \pi_1^{pk} x^p \in \pi R$. Therefore, $y \in \pi_1 R$ by the injectivity of the Frobenius. However, as $R$ is $\pi$-torsion free, we have $y' = y/\pi_1 = \pi_1^{k-1} x \in R$ which contradicts the minimality of $k$.
	
	Conversely, for any $x \in R$ with $x^p \in \pi R$, we have $(x/\pi_1)^p \in R$. Thus, $x/\pi_1 \in R$ by assumption, i.e. $x \in \pi_1 R$, which implies the injectivity of the Frobenius.
\end{proof}

\begin{mylem}\label{lem:almost-inj}
	Let $K$ be a pre-perfectoid field, $\pi$ a pseudo-uniformizer of $K$ dividing $p$ with a $p$-th root $\pi_1$ up to a unit, $R$ an $\ca{O}_K$-algebra which is almost flat. Then, the following conditions are equivalent:
	\begin{enumerate}
		\renewcommand{\labelenumi}{{\rm(\theenumi)}}
		\item The Frobenius induces an almost injection (resp. almost isomorphism) $R/\pi_1R \to R/\pi R$.
		\item The Frobenius induces an injection (resp. isomorphism) $(R^\al)_*/\pi_1 (R^\al)_* \to (R^\al)_*/\pi (R^\al)_*$.
	\end{enumerate}
\end{mylem}
\begin{proof}
	We follow the proof of \cite[5.6]{scholze2012perfectoid}.
	Notice that the Frobenius is compatible with the functors $\alpha^*$ and $\alpha_*$ (\ref{para:almost-frob}). (2) $\Rightarrow$ (1) follows from the almost isomorphism $R \to (R^\al)_*$. The ``injection'' part of (1) $\Rightarrow$ (2) follows from the inclusions (\ref{lem:almost-complete})
	\begin{align}
		(R^\al)_*/\pi_1 (R^\al)_* \subseteq((R/\pi_1 R)^\al)_*,\ (R^\al)_*/\pi (R^\al)_* \subseteq ((R/\pi R)^\al)_*.
	\end{align}
	For the ``isomorphism'' part of (1) $\Rightarrow$ (2), notice that $(R^\al)_*/\pi_1 (R^\al)_* \to (R^\al)_*/\pi (R^\al)_*$ is almost surjective. Let $\pi_2$ be a $p$-th root of $\pi_1$ up to a unit (\ref{lem:adm-uniformizer}). Then, for an element $x$ of $(R^\al)_*$, there exist elements $y$ and $x'$ of $(R^\al)_*$ such that $\pi_2^p x=y^p+\pi_2^{p^2} x'$. Thus, $x=y'^p+\pi_2^{p^2-p} x'$ where $y'=y/\pi_2\in (R^\al)_*[1/\pi]$ (as $(R^\al)_*$ is flat over $\ca{O}_K$ by \ref{lem:flat}). In fact, $y'$ lies in $(R^\al)_*$ by \ref{lem:p-clos} and the ``injection'' part of (1) $\Rightarrow$ (2). By applying this process to $x'$, there exist elements $y''$ and $x''$ of $(R^\al)_*$ such that $x'=y''^p+\pi_2^{p^2-p}x''$. In conclusion, we have $x=y'^p+\pi_2^{p^2-p}(y''^p+\pi_2^{p^2-p}x'')\equiv (y'+\pi_2^{p-1}y'')^p\ \mod\ \pi(R^\al)_*$, which shows the surjectivity of $(R^\al)_*/\pi_1 (R^\al)_* \to (R^\al)_*/\pi (R^\al)_*$.
\end{proof}

\begin{mylem}\label{lem:anypsuniv}
	Let $K$ be a pre-perfectoid field, $R$ an almost flat $\ca{O}_K$-algebra, $\pi, \pi'$ pseudo-uniformizers dividing $p$ with $p$-th roots $\pi_1, \pi_1'$ respectively up to units. Then, the following conditions are equivalent:
	\begin{enumerate}
		\renewcommand{\labelenumi}{{\rm(\theenumi)}}
		\item The Frobenius induces an almost injection (resp. almost surjection) $R/\pi_1R \to R/\pi R$.
		\item The Frobenius induces an almost injection (resp. almost surjection) $R/\pi_1'R \to R/\pi' R$.
	\end{enumerate}
	In particular, the definitions {\rm\ref{defn:pre-perf-alg}.(1)} and {\rm\ref{defn:pre-alg}} do not depend on the choice of the pseudo-uniformizer.
\end{mylem}
\begin{proof}
	 Notice that $(R^\al)_*$ is flat over $\ca{O}_K$ by \ref{lem:flat}. The ``injection'' part follows from \ref{lem:p-clos} and \ref{lem:almost-inj}.
	For the ``surjection'' part, we assume that $R/\pi_1 R \to R/\pi R$ is almost surjective. Let $\epsilon \in \ak{m}_K$.
	We can take a pseudo-uniformizer $\widetilde{\pi}$ of $K$ dividing $p$ with $\widetilde{\pi}_1^p = \widetilde{\pi}$ and $v_K(\pi)/3 < v_K(\widetilde{\pi}) < v_K(\pi)/2$. For any $x\in R$, by the almost surjectivity, we have $\epsilon x = y^p + \widetilde{\pi}^2 z$ for some $y, z \in R$. We also have $\widetilde{\pi}z= v^p + \pi w$ for some $v,w \in R$, then $\epsilon x = y^p +\widetilde{\pi} v^p + \widetilde{\pi}\pi w$. Since $y^p+\widetilde{\pi} v^p  \equiv (y+ \widetilde{\pi}_1 v)^p \mod pR$, $R'/\pi_1'R \to R/\pi' R$ is almost surjective for any pseudo-uniformizer $\pi'$ dividing $p$ with $v_K(\pi') < 4v_K(\pi)/3$. By induction, we see that $R'/\pi_1'R \to R/\pi' R$ is almost surjective in general.
\end{proof}

\begin{myprop}\label{prop:pre-perf-perf}
	Let $K$ be a pre-perfectoid field of characteristic $p$ with a pseudo-uniformizer $\pi$, $R$ an $\ca{O}_K$-algebra, $\widehat{R}$ the $\pi$-adic completion of $R$. Then, $R$ is almost pre-perfectoid if and only if $(\widehat{R}^\al)_*$ is perfect.
\end{myprop}
\begin{proof}
	Note that $\ca{O}_K$ is perfect by definition.
	If $R$ is almost pre-perfectoid, then $\widehat{R}$ is almost flat so that $(\widehat{R}^\al)_*$ is $\pi$-adically complete by taking $M=\widehat{R}^\al$ in \ref{lem:almost-complete}. Moreover, the Frobenius induces an isomorphism $(\widehat{R}^\al)_*/\pi^n(\widehat{R}^\al)_*\to (\widehat{R}^\al)_*/\pi^{pn}(\widehat{R}^\al)_*$ for any integer $n\geq 1$ by \ref{lem:almost-inj} and \ref{lem:anypsuniv}, which implies that $(\widehat{R}^\al)_*$ is perfect. Conversely, assume that $(\widehat{R}^\al)_*$ is perfect. For any $\pi$-torsion element $f\in (\widehat{R}^\al)_*$, we have $\pi^{1/p^n}f=0$ for any integer $n\geq 0$, which shows that $\widehat{R}$ is almost flat by \ref{lem:flat}. Moreover, it is clear that the Frobenius induces an isomorphism $(\widehat{R}^\al)_*/\pi(\widehat{R}^\al)_*\to (\widehat{R}^\al)_*/\pi^{p}(\widehat{R}^\al)_*$, which shows that $R$ is almost pre-perfectoid by \ref{lem:almost-inj} and \ref{lem:anypsuniv}.
\end{proof}

\begin{myprop}\label{prop:int-closed}
	Let $K$ be a pre-perfectoid field with a pseudo-uniformizer $\pi$, $R$ an $\ca{O}_K$-algebra which is almost flat, $R'$ the integral closure of $R$ in $R[1/\pi]$. If the Frobenius induces an almost injection $R/\pi_1R \to R/\pi R$, then $R \to R'$ is an almost isomorphism.
\end{myprop}
\begin{proof}
	Since $R \to (R^\al)_*$ is an almost isomorphism, we may replace $R$ by $(R^\al)_*$ so that we may assume that $R = (R^\al)_*$, $R \subseteq R[1/\pi]$ by \ref{lem:flat} and for any $x \in R[1/\pi]$ such that $x^p \in R$, then $x \in R$ by \ref{lem:p-clos} and \ref{lem:almost-inj}. It suffices to show that $R$ is integrally closed in $R[1/\pi]$. Suppose that $x \in R[1/\pi]$ is integral over $R$. There is an integer $N>0$ such that $x^r$ is an $R$-linear combination of $1, x, \dots,x^N$ for any $r>0$. Therefore, there exists an integer $k>0$ such that $\pi^kx^r \in R$ for any $r>0$. Taking $r=p^n$, we get $x \in \bigcap_{n\geq 0} \pi_n^{-1} R = (R^\al)_* = R$ by \ref{lem:al-star}, which completes our proof.
\end{proof}

\begin{mylem}\label{lem:completion-filtration}
	Let $R$ be a ring, $\pi$ a nonzero divisor of $R$, $\widehat{R}$ the $\pi$-adic completion of $R$, $\varphi:R[1/\pi]\to \widehat{R}[1/\pi]$ the canonical morphism. Then, $\varphi^{-1}(\pi^n\widehat{R})=\pi^nR$ for any integer $n$.
\end{mylem}
\begin{proof}
	Remark that $\widehat{R}$ is also $\pi$-torsion free by \ref{lem:completion-flat}. For an element $x/\pi^k\in R[1/\pi]$ (where $x\in R$, $k\geq 0$) such that $\varphi(x/\pi^k)=\pi^ny$ for some $y\in \widehat{R}$. After enlarging $k$, we may assume that $k+n>0$. Thus, we deduce from the canonical isomorphism $R/\pi^{k+n}R\to \widehat{R}/\pi^{k+n}\widehat{R}$ that $x\in \pi^{k+n}R$, which completes the proof.
\end{proof}

\begin{mylem}\label{lem:almost-flat}
	Let $K$ be a pre-perfectoid field with a pseudo-uniformizer $\pi$, $R$ an $\ca{O}_K$-algebra such that its $\pi$-adic completion $\widehat{R}$ is almost flat (resp. flat) over $\ca{O}_{\widehat{K}}$, $R[\pi^\infty]$ the $R$-submodule of elements of $R$ killed by some power of $\pi$. Then, $(R[\pi^\infty])^\wedge$ is almost zero (resp. zero) and the canonical morphism $\widehat{R}\to (R/R[\pi^\infty])^\wedge$ is surjective and is an almost isomorphism (resp. an isomorphism).
\end{mylem}
\begin{proof}
	The exact sequence $0\to R[\pi^\infty]\to R\to R/R[\pi^\infty]\to 0$ induces an exact sequence of the $\pi$-adic completions
	\begin{align}\label{eq:5.26.1}
		\xymatrix{
			0\ar[r]& (R[\pi^\infty])^\wedge\ar[r]& \widehat{R}\ar[r]& (R/R[\pi^\infty])^\wedge\ar[r]&0,
		}
	\end{align}
	since $R/R[\pi^\infty]$ is flat over $\ca{O}_K$ (\cite[\href{https://stacks.math.columbia.edu/tag/0315}{0315}]{stacks-project}). As $\widehat{R}[\pi^\infty]$ is almost zero (resp. zero) by assumption (\ref{lem:flat}), the canonical morphism $R[\pi^\infty]^{\al}\to\widehat{R}^{\al}$ (resp. $R[\pi^\infty]\to \widehat{R}$) factors through $0$, and thus so is the morphism $(R[\pi^\infty])^{\wedge\al}\to\widehat{R}^{\al}$ (resp. $(R[\pi^\infty])^{\wedge}\to \widehat{R}$). The conclusion follows from the exactness of \eqref{eq:5.26.1}.
\end{proof}

\begin{mylem}\label{lem:int-clos-aliso}
	Let $K$ be a pre-perfectoid field. Given a commutative diagram of $\ca{O}_K$-algebras 
	\begin{align}
		\xymatrix{
			B\ar[r]^-{g} & B'\\
			A\ar[r]^-{f}\ar[u]& A'\ar[u]
		}
	\end{align}
	we denote by $C$ (resp. $C'$) the integral closure of $A$ in $B$ (resp. of $A'$ in $B'$). Assume that $f$ and $g$ are almost isomorphisms. Then, the morphism $C\to C'$ is an almost isomorphism.
\end{mylem}
\begin{proof}
	Since $C\to C'$ is almost injective as $g$ is so, it remains to show the almost surjectivity. For any $\epsilon\in \ak{m}_K$ and $x'\in C'$ with identity $x'^n+a_{n-1}'x'^{n-1}+\cdots +a_1'x'+a_0'=0$ in $B'$ where $a_{n-1}',\dots,a_0'\in A'$, there exist $a_{n-1},\dots,a_0\in A$ and $x\in B$ such that $f(a_i)=\epsilon^{n-i} a_i'$ ($0\leq i<n$) and $g(x)=\epsilon x'$. Thus, $g(x^n+a_{n-1}x^{n-1}+\cdots a_1x+a_0)=0$. Since $g$ is almost injective, we see that $\epsilon x\in C$. It follows that $C\to C'$ is almost surjective.
\end{proof}

\begin{myprop}\label{prop:intclos-completion}
	Let $K$ be a pre-perfectoid field with a pseudo-uniformizer $\pi$, $A$ an $\ca{O}_K$-algebra such that its $\pi$-adic completion $\widehat{A}$ is almost flat over $\ca{O}_{\widehat{K}}$. We denote by $B$ (resp. $B'$) the integral closure of $A$ in $A[1/\pi]$ (resp. of $\widehat{A}$ in $\widehat{A}[1/\pi]$). Then, the canonical morphism of $\pi$-adic completions $\widehat{B}\to \widehat{B'}$ is an almost isomorphism of $\ca{O}_K$-algebras.
\end{myprop}
\begin{proof}
	We take a system of $p^k$-th roots $(\pi_k)_{k\geq 0}$ of $\pi$ up to units (\ref{lem:adm-uniformizer}). By \ref{lem:almost-flat} and \ref{lem:int-clos-aliso}, we can replace $A$ by its image $A/A[\pi^\infty]$ in $A[1/\pi]$, so that we may assume that $A$ is $\pi$-torsion free (and thus so is $\widehat{A}$). Let $\varphi:A[1/\pi]\to\widehat{A}[1/\pi]$ be the canonical morphism. It suffices to show that $\varphi$ induces an almost isomorphism $B/\pi^n B\to B'/\pi^n B'$ for any $n>0$. 
	
	For any element $x'\in B'$, there exists $r>0$ such that $\pi^rx'^{p^k}\in \widehat{A}$ for any $k>0$. We take an element $x_{ki} \in A$ such that $\varphi(x_{ki})- \pi^rx'^{p^i}\in\pi^{rp^k}\widehat{A}$ for $i=0,k$. Thus, $\varphi(x_{k0}^{p^k})-\varphi(\pi^{r(p^k-1)}x_{kk})\in \pi^{rp^k}\widehat{A}$. By \ref{lem:completion-filtration}, we see that $x_{k0}^{p^k}/\pi^{r(p^k-1)}-x_{kk}\in \pi^{r}A$. In particular, $(x_{k0}/\pi_k^{r(p^k-1)})^{p^k}\in A$, which implies that $x_{k0}/\pi_k^{r(p^k-1)}\in B$. Notice that $\varphi(x_{k0}/\pi_k^{r(p^k-1)})-(\pi/\pi_k^{p^k-1})^rx'\in \pi^{r(p^k-1)}\widehat{A}$. Since $k$ is an arbitrary positive integer, we see that $B/\pi^n B\to B'/\pi^n B'$ is almost surjective.
	
	For any element $x\in B$ such that $\varphi(x/\pi^n)\in B'$, there exists $r>0$ such that $\pi^r\varphi(x/\pi^n)^{p^k}\in \widehat{A}$ for any $k>0$. We take $y\in A$ such that $\pi^r\varphi(x/\pi^n)^{p^k}-\varphi(y)\in \pi\widehat{A}$, and then we see that $\pi^r(x/\pi^n)^{p^k}-y\in \pi A$ by \ref{lem:completion-filtration}. In particular, $(x/\pi_k^{np^k-r})^{p^k}\in A$, which implies that $x/\pi_k^{np^k-r}\in B$. Since $k$ is an arbitrary positive integer, we see that $B/\pi^n B\to B'/\pi^n B'$ is almost injective.
\end{proof}

\begin{mycor}\label{cor:pre-perf-intclos}
	Let $K$ be a pre-perfectoid field with a pseudo-uniformizer $\pi$, $R$ an $\ca{O}_K$-algebra which is almost pre-perfectoid, $R'$ the integral closure of $R$ in $R[1/\pi]$. Then, the morphism of $\pi$-adic completions $\widehat{R} \to \widehat{R'}$ is an almost isomorphism. In particular, $R'$ is also almost pre-perfectoid.
\end{mycor}
\begin{proof}
	We consider the following commutative diagram
	\begin{align}
		\xymatrix{
			R \ar[r] \ar[d] & R' \ar[r] \ar[d] & R[\frac{1}{\pi}] \ar[d]\\
			\widehat{R} \ar[r]  & R'' \ar[r] & \widehat{R}[\frac{1}{\pi}]
		}
	\end{align}
	where $R''$ is the integral closure of $\widehat{R}$ in $\widehat{R}[1/\pi]$. Since $\widehat{R}\to R''$ is an almost isomorphism by \ref{prop:int-closed}, $R''$ is also perfectoid. The conlusion follows from the fact that $\widehat{R'}\to \widehat{R''}$ is an almost isomorphism by \ref{prop:intclos-completion}. 
\end{proof}

\begin{mythm}[Tilting correspondence {\cite[5.2, 5.21]{scholze2012perfectoid}}]\label{thm:tilt-corre}
	Let $K$ be a perfectoid field, $\pi$ a pseudo-uniformizer of $K$ dividing $p$ with a $p$-th root $\pi_1$ up to a unit.
	\begin{enumerate}
		\renewcommand{\labelenumi}{{\rm(\theenumi)}}
		\item The functor $\ca{O}_K^\al\perf \to (\ca{O}_K/\pi)^\al\perf,\ R \mapsto R/\pi R$, is an equivalence of categories.
		\item The functor $\ca{O}_{K^\flat}^\al\perf \to (\ca{O}_{K^\flat}/\pi^\flat)^\al\perf,\ R \mapsto R/\pi^\flat R$ is an equivalence of categories, and the functor $ (\ca{O}_{K^\flat}/\pi^\flat)^\al\perf \to \ca{O}_{K^\flat}^\al\perf,\ R \mapsto  R^\flat$ is a quasi-inverse. 
		\item Let $R$ be a perfectoid $\ca{O}_K^\al$-algebra with tilt $R^\flat$. Then, $R$ is isomorphic to $\ca{O}_L^\al$ for some perfectoid field $L$ over $K$ if and only if $R^\flat$ is isomorphic to $\ca{O}_{L'}^\al$ for some perfectoid field $L'$ over $K^\flat$.
	\end{enumerate}
\end{mythm}
In conclusion, we have natural equivalences
\begin{align}
	\ca{O}_K^\al\perf \iso (\ca{O}_K/\pi)^\al\perf \iso (\ca{O}_{K^\flat}/\pi^\flat)^\al\perf \stackrel{\sim}{\longleftarrow} \ca{O}_{K^\flat}^\al\perf,
\end{align}
where the middle equivalence is given by the isomorphism \eqref{eq:2.23.4} $ \ca{O}_{K^\flat}/\pi^\flat\ca{O}_{K^\flat} \iso \ca{O}_K/\pi\ca{O}_K$.
We remark that the natural isomorphisms of the equivalence in (2) are defined as follows: for a perfectoid $\ca{O}_{K^\flat}^\al$-algebra $R$, the natural isomorphism $R\iso (R/\pi^\flat R)^\flat$ is induced by the morphism $R_*\to (R_*/\pi^\flat R_*)^\flat$ sending $x$ to $(\cdots,x^{1/p^2},x^{1/p},x)$ (notice that $R_*$ is perfect by \ref{prop:pre-perf-perf}); for a perfectoid $(\ca{O}_{K^\flat}/\pi^\flat)$-algebra $R$, the natural isomorphism $R^{\flat}/\pi^\flat R^{\flat}\iso R$ is induced by the projection on the first component $(R_*)^{\flat}\to R_*$ (cf. \cite[5.17]{scholze2012perfectoid}). Consequently, for a perfectoid $\ca{O}_K^\al$-algebra $R$, the morphism
\begin{align}
	R^\flat/\pi^\flat R^\flat \longrightarrow R/\pi R
\end{align}
induced by the projection on the first component is an isomorphism.

\begin{myprop}\label{prop:tensor-perf}
	Let $K$ be a perfectoid field with a pseudo-uniformizer $\pi$ of $K$ dividing $p$, $B \leftarrow A \to C$ a diagram of perfectoid $\ca{O}_K^\al$-algebras. Then, the $\pi$-adically completed tensor product $B \widehat{\otimes}_A C$ is also perfectoid.
\end{myprop}
\begin{proof}
	We follow closely the proof of \cite[6.18]{scholze2012perfectoid}.
	Firstly, we claim that $(B \otimes_A C)/\pi$ is flat over $(\ca{O}_K/\pi)^\al$. Since $(B \otimes_A C)/\pi= (B^\flat \otimes_{A^\flat} C^\flat)/\pi^\flat$, it suffices to show the flatness of $B^\flat \otimes_{A^\flat} C^\flat$ over $\ca{O}_{K^\flat}^\al$, which amounts to say that the submodule of $\pi^\flat$-torsion elements of $(B_*)^\flat \otimes_{(A_*)^\flat} (C_*)^\flat$ is almost zero as $B^\flat \otimes_{A^\flat} C^\flat = ((B_*)^\flat \otimes_{(A_*)^\flat} (C_*)^\flat)^\al$. If $f \in (B_*)^\flat \otimes_{(A_*)^\flat} (C_*)^\flat$ is a $\pi^\flat$-torsion element, then by perfectness of $(B_*)^\flat \otimes_{(A_*)^\flat} (C_*)^\flat$, we have $(\pi^\flat)^{1/p^n} f =0$ for any $n>0$, which proves the claim.
	
	Thus, $(B \otimes_A C)/\pi$ is a perfectoid $(\ca{O}_K/\pi)^\al$-algebra. It corresponds to a perfectoid $\ca{O}_K^\al$-algebra $D$ by \ref{thm:tilt-corre} and it induces a morphism $B \widehat{\otimes}_A C \to D$ by universal property of $\pi$-adically completed tensor product. We use d\'evissage to show that $(B\otimes_A C)/\pi^n \to D/\pi^n$ is an isomorphism for any integer $n>0$. By induction, 
	\begin{align}
		\xymatrix{
			& (B\otimes_A C)/\pi^n \ar[r]^-{\cdot \pi} \ar[d] & (B\otimes_A C)/\pi^{n+1} \ar[r] \ar[d]& (B\otimes_A C)/\pi \ar[r] \ar[d]& 0\\
			0 \ar[r]	& D/\pi^n \ar[r]^-{\cdot \pi} & D/\pi^{n+1} \ar[r] & D/\pi \ar[r] & 0
		}
	\end{align}
	the vertical arrows on the left and right are isomorphisms. By snake's lemma in the abelian category $\ca{O}_K^\al\module$ (\cite[\href{https://stacks.math.columbia.edu/tag/010H}{010H}]{stacks-project}), we know that the vertical arrow in the middle is also an isomorphism. In conclusion, $B \widehat{\otimes}_A C \to D$ is an isomorphism, which completes the proof.
\end{proof}

\begin{mycor}\label{cor:tensor-pre-perf}
	Let $K$ be a pre-perfectoid field, $B \leftarrow A \to C$ a diagram of $\ca{O}_K$-algebras which are almost pre-perfectoid. Then, the tensor product $B \otimes_A C$ is also almost pre-perfectoid.
\end{mycor}
\begin{proof}
	Since $\alpha^*$ commutes with arbitrary limits and colimits (\ref{para:almost-math}), we have $(B \widehat{\otimes}_A C)^\al = \widehat{B}^\al \widehat{\otimes}_{\widehat{A}^\al} \widehat{C}^\al$, which is perfectoid by \ref{prop:tensor-perf}.
\end{proof}

\begin{mylem}\label{lem:replace}
	Let $K$ be a perfectoid field, $\ca{O}_K\to V$ an extension of valuation rings of height $1$. Then, there exists an extension of perfectoid fields $K\to L$ and an extension of valuation rings $V\to \ca{O}_L$ over $\ca{O}_K$.
\end{mylem}
\begin{proof}
	Let $\pi$ be a pseudo-uniformizer of $K$, $E$ the fraction field of $V$, $\overline{E}$ an algebraic closure of $E$, $\overline{V}$ the integral closure of $V$ in $\overline{E}$. Let $\ak{m}$ be a maximal ideal of $\overline{V}$. It lies over the unique maximal ideal of $V$ as $V\to \overline{V}$ is integral. Setting $W = \overline{V}_{\ak{m}}$, according to \cite[\Luoma{6}.\textsection8.6, Prop.6]{bourbaki2006commalg5-7}, $V\to W$ is an extension of valuation rings of height $1$. Since $W$ is integrally closed in the algebraically closed fraction field $\overline{E}$, the Frobenius is surjective on $W/pW$. Thus, the fraction field of $W$ is a pre-perfectoid field over $K$. Passing to completion, we get an extension of perfectoid fields $K\to L$ by \ref{lem:pre-perf-field}.
\end{proof}

\begin{mythm}[{\cite[8.10]{bhattscholze2019prisms}}]\label{thm:arc-descent-perf}
	Let $K$ be a pre-perfectoid field with a pseudo-uniformizer $\pi$ dividing $p$, $R \to R'$ a homomorphism of $\ca{O}_K$-algebras which are almost pre-perfectoid. If $\spec(R') \to \spec(R)$ is a $\pi$-complete arc-covering, then for any integer $n \geq 1$, the augmented \v Cech complex 
\begin{align}
	0 \to R/\pi^n \to R'/\pi^n \to (R'\otimes_R R')/\pi^n \to \cdots
\end{align}
is almost exact.
\end{mythm}
\begin{proof}
	We follow Bhatt-Scholze's proof \cite[8.10]{bhattscholze2019prisms}.	After replacing $\ca{O}_K$, $R$, $R'$ by their $\pi$-adic completions, we may assume that $K$ is a perfectoid field and that $R^\al$ and $R'^\al$ are perfectoid $\ca{O}_K^\al$-algebras such that $\spec(R') \to \spec(R)$ is a $\pi$-complete arc-covering by \ref{lem:arc-parc}.(\ref{lem:arc-parc-complete}). Since the localization functor $\alpha^*$ commutes with arbitrary limits and colimits (\ref{para:almost-math}), $(\widehat{\otimes}^k_RR')^\al = \widehat{\otimes}^k_{R^\al}R'^\al$ is still a perfectoid $\ca{O}_K^\al$-algebra by \ref{prop:tensor-perf} for any $k \geq 0$. In particular, $\widehat{\otimes}^k_RR'$ is almost flat over $\ca{O}_K$. Then, by d\'evissage, it suffices to show the almost exactness of the augmented \v Cech complex when $n=1$, i.e. the almost exactness of
	\begin{align}\label{eq:arc-descent-flat}
	0 \to R^\flat/\pi^\flat \to R'^\flat/\pi^\flat \to (R'^\flat\otimes_{R^\flat} R'^\flat)/\pi^\flat \to \cdots.
	\end{align}
	
	We claim that the natural morphism $X=\spec(R'^\flat) \coprod \spec(R^\flat[1/\pi^\flat]) \to Y=\spec (R^\flat)$ is an arc-covering. 
	Since $\spec (R'/\pi) \to \spec (R/\pi)$ is an arc-covering, $X \to Y$ is surjective. Therefore, we only need to consider the test map $\spec(V) \to Y$ where $V$ is a valuation ring of height $1$. There are three cases:
	\begin{enumerate}
		\item If $\pi^\flat$ is invertible in $V$, then we get a natural lifting $R^\flat[1/\pi^\flat] \to V$.
		\item If $\pi^\flat$ is zero in $V$, then we have $R/\pi=R^\flat/\pi^\flat \to V$, and there is a lifting $R'/\pi=R'^\flat/\pi^\flat \to W$.
		\item Otherwise, $\ca{O}_{K^\flat}\to V$ is an extension of valuation rings. After replacing $V$ by an extension (\ref{lem:replace}), we may assume that $V[1/\pi^\flat]$ is a perfectoid field over $K^\flat$ with valuation ring $V$. By tilting correspondence \ref{thm:tilt-corre}, it corresponds to a perfectoid field over $K$ with valuation ring $V^\sharp$, together with an $\ca{O}_K$-morphism $R \to V^\sharp$ by \ref{lem:R-V}. Since $R \to R'$ gives a $\pi$-complete arc-covering, there is an extension $V^\sharp\to W$ of valuation rings of height $1$ and a lifting $R' \to W$. After replacing $W$ by an extension (\ref{lem:replace}), we may assume that $W[1/\pi]$ is a perfectoid field over $K$ with valuation ring $W$. By tilting correspondence \ref{thm:tilt-corre} and \ref{lem:R-V}, we get a lifting $R'^\flat \to W^\flat$ of $R^\flat \to V$.
	\end{enumerate}   
	Now we apply \ref{prop:arc-descent} to the arc-covering $X \to Y$ of perfect affine $\bb{F}_p$-schemes. We get an exact augmented \v Cech complex
	\begin{align}
	0 \to R^\flat \to R'^\flat \times R^\flat[\frac{1}{\pi^\flat}] \to (R'^\flat \times R^\flat[\frac{1}{\pi^\flat}]) \otimes_{R^\flat} (R'^\flat \times R^\flat[\frac{1}{\pi^\flat}]) \to \cdots.
	\end{align}
	Since each term is a perfect $\bb{F}_p$-algebra, the submodule of $\pi^\flat$-torsion elements is almost zero, in other words, each term is almost flat over $\ca{O}_{K^\flat}$. Modulo $\pi^\flat$, we get the almost exactness of \eqref{eq:arc-descent-flat}, which completes the proof.
\end{proof}

\begin{mydefn}\label{defn:almost-et}
	Let $K$ be a pre-perfectoid field, $A\to B$ a morphism of $\ca{O}_K$-algebras. 
	\begin{enumerate}
		\renewcommand{\labelenumi}{{\rm(\theenumi)}}
		\item We say that $A\to B$ is \emph{almost \'etale} if $A^\al\to B^\al$ is an \'etale morphism of $\ca{O}_K^\al$-algebras in the sense of \cite[3.1.1.(\luoma{4})]{gabber2003almost}.
		\item We say that $A\to B$ is \emph{almost finite \'etale} if it is almost \'etale and if $B^\al$ is an almost finitely presented $A^\al$-module in the sense of \cite[2.3.10]{gabber2003almost} (cf. \cite[4.13]{scholze2012perfectoid}, \cite[\Luoma{5}.7.1]{abbes2016p}).
	\end{enumerate}
\end{mydefn}
We remark that in \ref{defn:almost-et} if $A\to B$ is a morphism of $K$-algebras, then it is almost \'etale (resp. almost finite \'etale) if and only if it is \'etale (resp. finite \'etale).

\begin{myprop}\label{lem:pre-perf-alg-basic}
	Let $K$ be a pre-perfectoid field, $\scr{C}$ the full subcategory of the category of $\ca{O}_K$-algebras formed by those $\ca{O}_K$-algebras which are almost pre-perfectoid.
	\begin{enumerate}
		\renewcommand{\labelenumi}{{\rm(\theenumi)}}
		\item The subcategory $\scr{C}$ is stable under taking colimits and products.\label{item:lem-pre-perf-alg-basic1}
		\item Let $A\to B$ be an almost \'etale morphism of $\ca{O}_K$-algebras. If $A\in \ob(\scr{C})$, then $B\in \ob(\scr{C})$.\label{item:lem-pre-perf-alg-basic2}
	\end{enumerate}
\end{myprop}
\begin{proof}
	Let $\pi$ be a pseudo-uniformizer of $K$ dividing $p$ with a $p$-th root $\pi_1$ up to a unit. 
	
	(\ref{item:lem-pre-perf-alg-basic1}) The subcategory $\scr{C}$ is stable under taking tensor products by \ref{cor:tensor-pre-perf}. Let $(R_\lambda)_{\lambda\in\Lambda}$ be a directed system of objects in $\scr{C}$ and $R=\colim_{\lambda\in\Lambda} R_\lambda$. It is clear that the Frobenuis induces an almost isomorphism $R/\pi_1 R\to R/\pi R$. On the other hand, $\widehat{R}$ is the $\pi$-adic completion of $\colim_{\lambda\in\Lambda} \widehat{R_\lambda}$. Since the latter is almost flat over $\ca{O}_{\widehat{K}}$, so is $\widehat{R}$ (\ref{lem:completion-flat}). Thus, $\scr{C}$ is stable under taking colimits. 
	
	Let $(R_\lambda)_{\lambda\in\Lambda}$ be a set of objects in $\scr{C}$. Since $R/\pi R=\prod_{\lambda\in \Lambda} R_\lambda/\pi R_\lambda$, the Frobenius induces an almost isomorphism $R/\pi_1 R\to R/\pi R$. Moreover, the submodule of $\pi$-torsion elements of $\widehat{R}=\prod_{\lambda\in \Lambda} \widehat{R_\lambda}$ is almost zero, which implies that $\widehat{R}$ is almost flat over $\ca{O}_{\widehat{K}}$ (\ref{lem:flat}). We conclude that $\scr{C}$ is stable under taking products.
	
	(\ref{item:lem-pre-perf-alg-basic2}) Since $B$ is almost flat over $A$, it is almost flat over $\ca{O}_{K}$ and thus $\widehat{B}$ is almost flat over $\ca{O}_{\widehat{K}}$ (\ref{lem:completion-flat}). Since $B$ is almost \'etale over $A$, the map $B/\pi_1 B\to B/\pi B$ induced by the Frobenius is almost isomorphic to the base change of the map $A/\pi_1 A\to A/\pi A$ by $A\to B$ (\cite[3.5.13]{gabber2003almost}), which completes the proof.
\end{proof}

\begin{mylem}\label{lem:almost-fet}
	Let $K$ be a pre-perfectoid field with a pseudo-uniformizer $\pi$, $R$ an $\ca{O}_K$-algebra which is almost flat and almost pre-perfectoid, $R'$ an $R$-algebra which is almost finite \'etale. Then, the integral closure of $R$ in $R'$ is almost isomorphic to both $R'$ and the integral closure of $R$ in $R'[1/\pi]$.
\end{mylem}
\begin{proof}
	Notice that $R'$ is also almost flat and almost pre-perfectoid by \ref{lem:pre-perf-alg-basic}. Since $R'$ is almost finitely generated over $R$ as an $R$-module, the elements of $\ak{m}_KR'$ are integral over $R$ (cf. \cite[2.3.10]{gabber2003almost}). Thus, the integral closure of $R$ in $R'$ is almost isomorphic to $R'$. On the other hand, since $R'$ is almost isomorphic to its integral closure in $R'[1/\pi]$ by \ref{prop:int-closed}, the integral closure of $R$ in $R'$ is almost isomorphic to the integral closure of $R$ in $R'[1/\pi]$ by \ref{lem:int-clos-aliso}.
\end{proof}

\begin{mypara}\label{para:affinoid}
	We recall some basic definitions about affinoid algebras used in \cite{scholze2012perfectoid} in order to prove the almost purity theorem \ref{thm:almost-purity} by reducing to \tit{loc.cit}. Let $K$ be a complete valuation field of height $1$. A \emph{Tate $K$-algebra} is a topological $K$-algebra $\ca{R}$ whose topology is generated by the open subsets $a\ca{R}_0$ for a subring $\ca{R}_0\subset \ca{R}$ and any $a\in K^\times$. We denote by $\ca{R}^\circ$ the subring of power-bounded elements of $\ca{R}$, which is thus an $\ca{O}_K$-algebra. An \emph{affinoid $K$-algebra} is a pair $(\ca{R},\ca{R}^+)$ consisting of a Tate $K$-algebra $\ca{R}$ and a subring $\ca{R}^+$ of $\ca{R}^\circ$ which is open and integrally closed in $\ca{R}$. A morphism of affinoid $K$-algebras $(\ca{R},\ca{R}^+)\to(\ca{R}',\ca{R}'^+)$ is a morphism of topological $K$-algebras $f:\ca{R}\to \ca{R}'$ with $f(\ca{R}^+)\subseteq \ca{R}'^+$. Such a morphism is called \emph{finite \'etale} in the sense of \cite[7.1.(\luoma{1})]{scholze2012perfectoid} if $\ca{R}'$ is finite \'etale over $\ca{R}$ endowed with the canonical topology as a finitely generated $\ca{R}$-module and if $\ca{R}'^+$ is the integral closure of $\ca{R}^+$ in $\ca{R}'$.
	
	For a perfectoid field $K$ and an affinoid $K$-algebra $(\ca{R},\ca{R}^+)$, the inclusion $\ca{R}^+\subseteq \ca{R}^\circ$ is an almost isomorphism. Indeed, for any $\epsilon\in \ak{m}_K$ and any power-bounded element $x\in \ca{R}^\circ$, we have $(\epsilon x)^n\in \ca{R}^+$ for $n\in \bb{N}$ large enough as $\ca{R}^+$ is open. Thus, $\epsilon x\in \ca{R}^+$ as $\ca{R}^+$ is integrally closed. We remark that $(\ca{R},\ca{R}^+)$ is perfectoid in the sense of \cite[6.1]{scholze2012perfectoid} if and only if $\ca{R}^\circ$ is bounded and almost perfectoid over $\ca{O}_K$ (\cite[5.5, 5.6]{scholze2012perfectoid}).
\end{mypara}

\begin{mypara}\label{para:commalg-affinoid}
	There is a typical example for constructing affinoid algebras from commutative algebras (cf. \cite[Sorite 2.3.1]{andre2018abhyankar}). Let $K$ be a complete valuation field of height $1$ with a pseudo-uniformizer $\pi$, $R$ a flat $\ca{O}_K$-algebra. The $K$-algebra $R[1/\pi]$ endowed with the $\pi$-adic topology defined by $R$ is a Tate $K$-algebra. Let $\overline{R}$ be the integral closure of $R$ in $R[1/\pi]$. It is clear that any element of $\overline{R}$ is power-bounded. Thus, $(R[1/\pi],\overline{R})$ is an affinoid $K$-algebra. %Moreover, we can define a semi-norm on $R[1/\pi]$ by 
	%	\begin{align}
		%		|x|_{R}=|\pi|_K^{v(x)}, \ x\in R[1/\pi],
		%	\end{align}
	%	where $v(x)=\sup\{s\in \bb{R}\ |\ \exists a\in K^\times,\ |a|_K=|\pi|_K^s,\ x\in aR\}\in \bb{R}\cup\{+\infty\}$. It is a norm if $R$ is $\pi$-adically separated. We remark that $R$ is $\pi$-adically complete if and only if the normed $K$-algebra $R[1/\pi]$ is a Banach $K$-algebra. 
	
	Let $S$ be a finite $R[1/\pi]$-algebra endowed with the canonical topology. More precisely, the topology can be defined as follows: we take a finite $R$-algebra $R'$ contained in $S$ which contains a family of generators of the $R[1/\pi]$-algebra $S$; then the canonical topology of $S=R'[1/\pi]$ is the $\pi$-adic topology defined by $R'$ (which is independent of the choice of $R'$). Let $\overline{R'}$ be the integral closure of $R'$ in $R'[1/\pi]$, which is also the integral closure of $R$ in $R'[1/\pi]$. We remark that $(R[1/\pi],\overline{R})\to (R'[1/\pi],\overline{R'})$ is a finite \'etale morphism of affinoid $K$-algebras if and only if $R[1/\pi]\to R'[1/\pi]$ is finite \'etale.
\end{mypara}

\begin{mythm}[Almost purity, {\cite[7.9]{scholze2012perfectoid}}]\label{thm:almost-purity}
	Let $K$ be a pre-perfectoid field with a pseudo-uniformizer $\pi$, $R$ an $\ca{O}_K$-algebra which is almost pre-perfectoid, $R'$ the integral closure of $R$ in a finite \'etale $R[1/\pi]$-algebra. Then, $R'$ is almost pre-perfectoid and the $\pi$-adic completion $\widehat{R'}$ is almost finite \'etale over $\widehat{R}$.
	
	Moreover, if $R$ is $\pi$-torsion free and if $(R,\pi R)$ is a henselian pair, then $R'$ is almost finite \'etale over $R$.
\end{mythm}
\begin{proof} 
	For the first statement, by \ref{lem:almost-flat}, we can replace $R$ by its image $R/R[\pi^\infty]$ in $R[1/\pi]$ (which does not change $R'$), so that we may assume that $R$ is $\pi$-torsion free (and thus so is $\widehat{R}$). Let $S$ (resp. $S'$) be the integral closure of $\widehat{R}$ in $\widehat{R}[1/\pi]$ (resp. of $R'\otimes_R \widehat{R}$ in $R'\otimes_R \widehat{R}[1/\pi]$). Then, we obtain a finite \'etale morphism of affinoid $\widehat{K}$-algebras $(\widehat{R}[1/\pi],S)\to(R'\otimes_R \widehat{R}[1/\pi],S')$ by \ref{para:commalg-affinoid}.
	
	Since $\widehat{R}$ is almost perfectoid, $\widehat{R}\to S$ is an almost isomorphism (\ref{prop:int-closed}). Thus, $S$ is bounded and almost perfectoid over $\ca{O}_{\widehat{K}}$. In other words, $(\widehat{R}[1/\pi],S)$ is a perfectoid affinoid $\widehat{K}$-algebra. Then, by almost purity (\cite[7.9.(\luoma{3})]{scholze2012perfectoid}), the $\ca{O}_{\widehat{K}}$-algebra $S'$ is almost perfectoid (thus $S'\to \widehat{S'}$ is an almost isomorphism by definition) and almost finite \'etale over $S$. 
	
	On the other hand, the two $\ca{O}_{\widehat{K}}$-algebras $R'$ and $R'\otimes_R \widehat{R}$ have the same $\pi$-adic completion $\widehat{R'}$. Thus, the $\pi$-adic completions of the integral closures of $R'$ and $R'\otimes_R \widehat{R}$ in $R'[1/\pi]$ and $R'\otimes_R \widehat{R}[1/\pi]$ respectively are almost isomorphic to that of $\widehat{R'}$ in $\widehat{R'}[1/\pi]$ by \ref{prop:intclos-completion}. In other words, $\widehat{R'}\to \widehat{S'}$ is an almost isomorphism. In conclusion, $R'$ is almost pre-perfectoid, and $\widehat{R'}$ is almost finite \'etale over $\widehat{R}$.
	
	We assume moreover that $R$ is $\pi$-torsion free and $(R,\pi R)$ is a henselian pair. Recall that the category of almost $\ca{O}_K$-algebras finite \'etale over $R^\al$ (resp. over $\widehat{R}^\al$) is equivalent to that over $(R/\pi R)^\al$ via the base change functor (\cite[5.5.7.(\luoma{3})]{gabber2003almost}). Hence, there exists an $R$-algebra $R''$ which is almost finite \'etale over $R$ such that $(R''\otimes_R \widehat{R})^\al$ is isomorphic to $\widehat{R'}^{\al}$. On the other hand, recall that the category of finite \'etale $R[1/\pi]$-algebras is equivalent to the category of finite \'etale $\widehat{R}[1/\pi]$-algebras via the base change functor (\cite[5.4.53]{gabber2003almost}). Notice that $R''[1/\pi]\otimes_{R}\widehat{R}\cong \widehat{R'}[1/\pi]$ by the construction of $R''$ and that $R'[1/\pi]\otimes_{R}\widehat{R}\cong\widehat{R'}[1/\pi]$ by the almost isomorphisms $\widehat{R'}\to \widehat{S'}\leftarrow S'$. Hence, there is an isomorphism $R''[1/\pi]\cong R'[1/\pi]$. By \ref{lem:almost-fet}, we see that $R''$ is almost isomorphic to $R'$, which completes the proof.
\end{proof}

\section{Brief Review on Covanishing Fibred Sites}\label{sec:covanishing}
We give a brief review on covanishing fibred sites, which are developed by Abbes and Gros \cite[\Luoma{6}]{abbes2016p}. We remark that \cite[\Luoma{6}]{abbes2016p} does not require the sites to admit finite limits (\ref{para:finite-complete}).

\begin{mypara}\label{par:cova-def}
	A \emph{fibred site} $E/C$ is a fibred category $\pi : E \to C$ whose fibres are sites such that for a cleavage and for every morphism $f : \beta \to \alpha$ in $C$, the inverse image functor $f^+ : E_\alpha \to E_\beta$ gives a morphism of sites (so that the same holds for any cleavage) (cf. \cite[\Luoma{6}.7.2]{sga4-2}).
	
	Let $x$ be an object of $E$ over $\alpha \in \ob(C)$. We denote by
	\begin{align}
		\iota_\alpha^+ : E_\alpha \to E
	\end{align} the inclusion functor of the fibre category $E_{\alpha}$ over $\alpha$ into the whole category $E$. A \emph{vertical covering} of $x$ is the image by $\iota_\alpha^+$ of a covering family $\{ x_m \to x \}_{m \in M}$ in $E_\alpha$. We call the topology generated by all vertical coverings the \emph{total topology} on $E$ (cf. \cite[\Luoma{6}.7.4.2]{sga4-2}). 

	Assume further that $C$ is a site. A \emph{Cartesian covering} of $x$ is a family $\{ x_n \to x \}_{n \in N}$ of morphisms of $E$ such that there exists a covering family $\{ \alpha_n \to \alpha \}_{n \in N}$ in $C$ with $x_n$ isomorphic to the pullback of $x$ by $\alpha_n\to \alpha$ for each $n$. 
\end{mypara}

\begin{mydefn}[{\cite[\Luoma{6}.5.3]{abbes2016p}}]\label{defn:covtop}
	A \emph{covanishing fibred site} is a fibred site $E/C$ where $C$ is a site. We associate to $E$ the  \emph{covanishing topology} which is generated by all vertical coverings and Cartesian coverings. We simply call a covering family for the covanishing topology a \emph{covanishing covering}.
\end{mydefn}

\begin{mydefn}\label{exmp:stdcovcov}
	Let $E/C$ be a covanishing fibred site. We call a composition of a Cartesian covering followed by vertical coverings a \emph{standard covanishing covering}. More precisely, a standard covanishing covering is a family of morphisms of $E$
	\begin{align}
		\{x_{nm} \to x\}_{n \in N, m \in M_n}
	\end{align}
	such that there is a Cartesian covering $\{x_n \to x\}_{n \in N}$ and for each $n \in N$ a vertical covering $\{x_{nm} \to x_n\}_{m \in M_n}$.
\end{mydefn}

\begin{myprop}[{\cite[\Luoma{6}.5.9]{abbes2016p}}]\label{prop:covcov}
	Let $E/C$ be a covanishing fibred site. Assume that in each fibre any object is quasi-compact, then a family of morphisms $\{x_i \to x\}_{i \in I}$ of $E$ is a covanishing covering if and only if it can be refined by a standard covanishing covering.
\end{myprop}

\begin{mypara}\label{para:presheafonE}
	Let $E/C$ be a fibred category. Fixing a cleavage of $E/C$, to give a presheaf $\ca{F}$ on $E$ is equivalent to give a presheaf $\ca{F}_\alpha$ on each fibre category $E_\alpha$ and transition morphisms $\ca{F}_\alpha\to f^\psh\ca{F}_{\beta}$ for each morphism $f:\beta\to \alpha$ in $C$ satisfying a cocycle relation (cf. \cite[\Luoma{6}.7.4.7]{sga4-2}). Thus, we simply denote a presheaf $\ca{F}$ on $E$ by
	\begin{align}
		\ca{F} = (\ca{F}_\alpha )_{\alpha \in \ob(C)},
	\end{align}
	where $\ca{F}_\alpha = \iota_\alpha^{\psh} \ca{F}$ is the restriction of $\ca{F}$ on the fibre category $E_\alpha$. If $E/C$ is a fibred site, then $\ca{F}$ is a sheaf with respect to the total topology on $E$ if and only if $\ca{F}_\alpha$ is a sheaf on $E_\alpha$ for each $\alpha$ (\cite[\Luoma{6}.7.4.7]{sga4-2}).
	Moreover, we have the following description of a covanishing sheaf.
\end{mypara}

\begin{myprop}[{\cite[\Luoma{6}.5.10]{abbes2016p}}]\label{prop:sheaf-covfibsite}
	Let $E/C$ be a covanishing fibred site. Then, a presheaf $\ca{F}$ on $E$ is a sheaf if and only if the following conditions hold:
	\begin{enumerate}
		\item[\rm{(v)}] The presheaf $\ca{F}_\alpha = \iota_\alpha^{\psh} \ca{F}$ on $E_\alpha$ is a sheaf for any $\alpha \in \ob(C)$.
		\item[\rm{(c)}] For any covering family $\{f_i:\alpha_i \to \alpha\}_{i \in I}$ of $C$, if we set $\alpha_{ij}=\alpha_i \times_\alpha \alpha_j$ and $f_{ij}: \alpha_{ij} \to \alpha$, then the sequence of sheaves on $E_\alpha$,
		\begin{align}
			\ca{F}_\alpha \to \prod_{i \in I} f_{i*}\ca{F}_{\alpha_i} \rightrightarrows \prod_{i,j \in I} f_{ij*}\ca{F}_{\alpha_{ij}},
		\end{align}
		is exact.
	\end{enumerate}
\end{myprop}

\section{Faltings Ringed Sites}\label{sec:faltings-site}

\begin{mypara}\label{para:E}
	Let $Y \to X$ be a morphism of $\bb{U}$-small coherent schemes, and let $\fal_{Y\to X}$ be the category of morphisms $V\to U$ of $\bb{U}$-small coherent schemes over the morphism $Y \to X$, namely, the category of commutative diagrams of coherent schemes 
	\begin{align}
		\xymatrix{
		V \ar[r] \ar[d] &U \ar[d]\\
		Y \ar[r] & X
		}
	\end{align}
	Given a functor $I \to \fal_{Y\to X}$ sending $i$ to $V_i \to U_i$, if $\lim V_i $ and $\lim U_i$ are representable in the category of coherent schemes, then $\lim (V_i \to U_i)$ is representable by $\lim V_i \to \lim U_i$. We say that a morphism $(V'\to U')\to (V\to U)$ of objects of $\fal_{Y \to X}$ is \emph{Cartesian} if $V'\to V\times_U U'$ is an isomorphism. It is clear that the Cartesian morphisms in $\fal_{Y \to X}$ are stable under base change.
	
	Consider the functor
	\begin{align}\label{eq:sigma-fal}
		\phi^+ : \fal_{Y\to X} \longrightarrow \schqcqs_{/X},\ (V \to U) \longmapsto U.
	\end{align} 
	The fibre category over $U$ is canonically equivalent to the category $\schqcqs_{/U_Y}$ of coherent $U_Y$-schemes, where $U_Y=Y\times_X U$. The base change by $U' \to U$ gives an inverse image functor $\schqcqs_{/U_Y} \to \schqcqs_{/U'_Y}$, which endows $\fal_{Y\to X}/\schqcqs_{/X}$ with a structure of fibred category. We define a presheaf on $\fal_{Y\to X}$ by
	\begin{align}\label{eq:falb}
		\falb(V \to U) = \Gamma(U^V , \ca{O}_{U^V}),
	\end{align}
	where $U^V$ is the integral closure of $U$ in $V$.
\end{mypara}

\begin{mydefn}\label{defn:Eet}
	Let $Y\to X$ be a morphism of coherent schemes. A morphism $(V' \to U') \to (V \to U)$ in $\fal_{Y\to X}$ is called \emph{\'etale}, if $U' \to U$ is \'etale and $V' \to V \times_U U'$ is finite \'etale. 
\end{mydefn}

\begin{mylem}\label{lem:etale-basic}
	Let $Y\to X$ be a morphism of coherent schemes, $(V'' \to U'') \stackrel{g}{\longrightarrow} (V' \to U') \stackrel{f}{\longrightarrow} (V \to U)$ morphisms in $\fal_{Y \to X}$.
	\begin{enumerate}
		\renewcommand{\labelenumi}{{\rm(\theenumi)}}
		\item If $f$ is \'etale, then any base change of $f$ is also \'etale.
		\item If $f$ and $g$ are \'etale, then $f \circ g$ is also \'etale.
		\item If $f$ and $f\circ g$ are \'etale, then $g$ is also \'etale.
	\end{enumerate}
\end{mylem}
\begin{proof}
	It follows directly from the definitions.
\end{proof}

\begin{mypara}
	Let $Y\to X$ be a morphism of coherent schemes. We still denote by $X_\et$ (resp. $X_\fet$) the site formed by coherent \'etale (resp. finite \'etale) $X$-schemes endowed with the \'etale topology. 
	Let $\fal_{Y\to X}^\et$ be the full subcategory of $\fal_{Y\to X}$ formed by $V \to U$ \'etale over the final object $Y \to X$. It is clear that $\fal_{Y\to X}^\et$ is stable under finite limits in $\fal_{Y\to X}$. Then, the functor \eqref{eq:sigma-fal} induces a functor
	\begin{align}
		\phi^+: \fal_{Y\to X}^\et \longrightarrow X_\et,\ (V \to U) \longmapsto U,
	\end{align}
	which endows $\fal_{Y\to X}^\et / X_\et$ with a structure of fibred sites, whose fibre over $U$ is the finite \'etale site $U_{Y,\fet}$. We endow $\fal_{Y\to X}^\et$ with the associated covanishing topology, that is, the topology generated by the following types of families of morphisms
	\begin{enumerate}
		\item[\rm{(v)}] $\{(V_m \to U) \to (V \to U)\}_{m \in M}$, where $M$ is a finite set and $\coprod_{m\in M} V_m\to V$ is surjective;
		\item[\rm{(c)}] $\{(V\times_U{U_n} \to U_n) \to (V \to U)\}_{n \in N}$, where $N$ is a finite set and $\coprod_{n\in N} U_n\to U$ is surjective.
	\end{enumerate}	
	It is clear that any object of $\fal_{Y\to X}^\et$ is quasi-compact by \ref{prop:covcov}.
	We still denote by $\falb$ the restriction of the presheaf $\falb$ on $\fal_{Y\to X}$ to $\fal_{Y\to X}^\et$ if there is no ambiguity.
\end{mypara}

\begin{mylem}\label{lem:falb-sheaf}
	Let $Y \to X$ be a morphism of coherent schemes. Then, the presheaf on $\schqcqs_{/Y}$ sending $Y'$ to $\Gamma(X^{Y'},\ca{O}_{X^{Y'}})$ is a sheaf with respect to the fpqc topology {\rm(\cite[\href{https://stacks.math.columbia.edu/tag/022A}{022A}]{stacks-project})}.
\end{mylem}
\begin{proof}
	We may regard $\ca{O}_{X^{Y'}}$ as a quasi-coherent $\ca{O}_X$-algebra over $X$. It suffices to show that for a finite family of morphisms $\{Y_i\to Y\}_{i\in I}$ with $Y'=\coprod_{i\in I} Y_i$ faithfully flat over $Y$, the sequence of quasi-coherent $\ca{O}_X$-algebras
	\begin{align}
		\xymatrix{
			0\ar[r]& \ca{O}_{X^Y}\ar[r]& \ca{O}_{X^{Y'}}\ar@<0.5ex>[r]\ar@<-0.5ex>[r]& \ca{O}_{X^{Y'\times_Y Y'}}
		}
	\end{align}
	is exact. Thus, we may assume that $X=\spec(R)$ is affine. We set $A_0=\Gamma(Y,\ca{O}_Y)$, $A_1=\Gamma(Y',\ca{O}_{Y'})$, $A_2=\Gamma(Y'\times_YY',\ca{O}_{Y'\times_YY'})$, $R_0=\Gamma(X^Y,\ca{O}_{X^Y})$, $R_1=\Gamma(X^{Y'},\ca{O}_{X^{Y'}})$, $R_2=\Gamma(X^{Y'\times_Y Y'},\ca{O}_{X^{Y'\times_Y Y'}})$. Notice that $R_i$ is the integral closure of $R$ in $A_i$ for $i=0,1,2$ (\cite[\href{https://stacks.math.columbia.edu/tag/035F}{035F}]{stacks-project}). Consider the diagram
	\begin{align}
		\xymatrix{
			0\ar[r]& R_0\ar[r]\ar[d]& R_1\ar@<0.5ex>[r]\ar@<-0.5ex>[r]\ar[d]& R_2\ar[d]\\
			0\ar[r]& A_0\ar[r]& A_1\ar@<0.5ex>[r]\ar@<-0.5ex>[r]& A_2
		}
	\end{align}
	We see that the vertical arrows are injective and the second row is exact by faithfully flat descent. Notice that $R_0=A_0\cap R_1$, since they are both the integral closure of $R$ in $A_0$ as $A_0\subseteq A_1$. Thus, the first row is also exact, which completes the proof.
\end{proof}

\begin{myprop}\label{prop:falb-sheaf}
	Let $Y \to X$ be a morphism of coherent schemes. Then, the presheaf $\falb$ on $\fal_{Y\to X}^\et$ is a sheaf.
\end{myprop}
\begin{proof}
	It follows directly from \ref{prop:sheaf-covfibsite}, whose first condition holds by \ref{lem:falb-sheaf}, and whose second condition holds by \ref{lem:rel-norm-commute} (cf. \cite[\Luoma{3}.8.16]{abbes2016p}).
\end{proof}

\begin{mydefn}[{\cite[page 214]{faltings2002almost}, \cite[\Luoma{6}.10.1]{abbes2016p}}]\label{defn:falsite}
	We call $(\fal_{Y\to X}^\et,\falb)$ the \emph{Faltings ringed site} of the morphism of coherent schemes $Y \to X$.
\end{mydefn}
It is clear that the localization $(\fal_{Y\to X}^\et)_{/ (V \to U)}$ of $\fal_{Y\to X}^\et$ at an object $V \to U$ is canonically equivalent to the Faltings ringed site $\fal_{V \to U}^\et$ of the morphism $V \to U$ by \ref{prop:covcov} (cf. \cite[\Luoma{6}.10.14]{abbes2016p}).

\begin{mypara}\label{para:notation-psi-beta-sigma}
	Let $Y \to X$ be a morphism of coherent schemes. Consider the natural functors
	\begin{align}
		\psi^+ &: \fal_{Y\to X}^\et \longrightarrow Y_{\et},\ (V \to U) \longmapsto V,\\
		\beta^+ &: Y_{\fet} \longrightarrow \fal_{Y\to X}^\et,\ V \longmapsto (V \to X),\\
		\sigma^+ &: X_\et \longrightarrow \fal_{Y\to X}^\et,\ U\longmapsto (Y\times_X U\to U) .
	\end{align}
	They are left exact and continuous (cf. \cite[\Luoma{6} 10.6, 10.7]{abbes2016p}). Then, we obtain a commutative diagram of sites associated functorially to the morphism $Y\to X$ by \ref{para:topos},
	\begin{align}\label{diag:beta-psi-sigma}
		\xymatrix{
			&Y_\et\ar[dl]_-{\rho}\ar[dr]\ar[d]^-{\psi}&\\
			Y_\fet&\fal_{Y \to X}^\et \ar[r]_-{\sigma}\ar[l]^-{\beta}& X_\et
		}
	\end{align}
	where $\rho: Y_\et \to Y_\fet$ is defined by the inclusion functor, and the unlabelled arrow $Y_\et \to X_\et$ is induced by the morphism $Y\to X$. Moreover, if $\ca{O}_{X_\et}$ denotes the structural sheaf on $X_\et$ sending $U$ to $\Gamma(U,\ca{O}_U)$, then $\sigma^+$ actually defines a morphism of ringed sites
	\begin{align}
		\sigma: (\fal_{Y \to X}^\et,\falb) \longrightarrow (X_\et,\ca{O}_{X_\et}).
	\end{align}
	We will study more properties of these morphisms in the remaining sections.
\end{mypara}

\begin{mylem}\label{lem:fal-localtopos}
	Let $X$ be the spectrum of an absolutely integrally closed valuation ring, $Y$ a quasi-compact open subscheme of $X$. Then, for any presheaf $\ca{F}$ on $\fal_{Y \to X}^\et$, we have $\ca{F}^\ash(Y\to X)=\ca{F}(Y \to X)$. In particular, the associated topos of $\fal_{Y \to X}^\et$ is local {\rm(\cite[\Luoma{6}.8.4.6]{sga4-2})}.
\end{mylem}
\begin{proof}
	Notice that $Y$ is also the spectrum of an absolutely integrally closed valuation ring by \ref{lem:prod-val}.(\ref{lem:prod-val-prin}) and that absolutely integrally closed valuation rings are strictly Henselian. Thus, any covering of $Y\to X$ in $\fal_{Y \to X}^\et$ can be refined by the identity covering by \ref{prop:covcov}. We see that $\ca{F}^\ash(Y\to X)=\ca{F}(Y \to X)$ for any presheaf $\ca{F}$. For the last assertion, it suffices to show that the section functor $\Gamma(Y\to X, -)$ commutes with colimits of sheaves. For a colimit of sheaves $\ca{F}=\colim \ca{F}_i$, let $\ca{G}$ be the colimit of presheaves $\ca{G}=\colim \ca{F}_i$. Then, we have $\ca{F}=\ca{G}^\ash$ and $\Gamma(Y\to X, \ca{F})=\Gamma(Y\to X, \ca{G})=\colim \Gamma(Y\to X, \ca{F}_i)$.
\end{proof}

\begin{mypara}\label{para:limit-falet}
	Let $(Y_\lambda\to X_\lambda)_{\lambda\in \Lambda}$ be a $\bb{U}$-small directed inverse system of morphisms of $\bb{U}$-small coherent schemes with affine transition morphisms $Y_{\lambda'} \to Y_\lambda$ and $X_{\lambda'} \to X_\lambda$ ($\lambda'\geq \lambda$). We set $(Y\to X)=\lim_{\lambda\in \Lambda} (Y_\lambda\to X_\lambda)$. We regard the directed set $\Lambda$ as a filtered category and regard the inverse system $(Y_\lambda\to X_\lambda)_{\lambda\in \Lambda}$ as a functor $\varphi: \Lambda^{\oppo} \to \fal$ from the opposite category of $\Lambda$ to the category of morphisms of $\bb{U}$-small coherent schemes. Consider the fibred category $\fal_{\varphi}^\et\to \Lambda^{\oppo}$ defined by $\varphi$ whose fibre category over $\lambda$ is $\fal_{Y_\lambda\to X_\lambda}^\et$ and whose inverse image functor $\varphi_{\lambda'\lambda}^+:\fal_{Y_{\lambda}\to X_{\lambda}}^\et\to \fal_{Y_\lambda'\to X_\lambda'}^\et$ associated to a morphism $\lambda'\to\lambda$ in $\Lambda^{\oppo}$ is given by the base change by the transition morphism $(Y_{\lambda'}\to X_{\lambda'})\to (Y_\lambda\to X_\lambda)$ (cf. \cite[\Luoma{6}.11.2]{abbes2016p}). Let $\varphi_\lambda^+:\fal_{Y_\lambda\to X_\lambda}^\et\to \fal_{Y \to X}^\et$ be the functor defined by the base change by the transition morphism $(Y\to X)\to (Y_\lambda\to X_\lambda)$. 
	
	Recall that the filtered colimit of categories $(\fal_{Y_\lambda\to X_\lambda}^\et)_{\lambda\in \Lambda}$ is representable by the category $\underrightarrow{\fal}_{\varphi}^\et$ whose objects are those of $\fal_{\varphi}^\et$ and whose morphisms are given by (\cite[\Luoma{6} 6.3, 6.5]{sga4-2})
	\begin{align}
		\ho_{\underrightarrow{\fal}_{\varphi}^\et}((V\to U),(V'\to U'))=\colim_{\stackrel{(V''\to U'')\to (V\to U)}{\trm{Cartesian}}} \ho_{\fal_{\varphi}^\et}((V''\to U''),(V'\to U')),
	\end{align}
	where the colimit is taken over the opposite category of the cofiltered category of Cartesian morphisms with target $V\to U$ of the fibred category $\fal_{\varphi}^\et$ over $\Lambda^{\oppo}$ (distinguish with the Cartesian morphisms defined in \ref{para:E}).
	We see that the functors $\varphi_\lambda^+$ induces an equivalence of categories by \cite[8.8.2, 8.10.5]{ega4-3} and \cite[17.7.8]{ega4-4}
	\begin{align}\label{eq:7.9.2}
		\underrightarrow{\fal}_{\varphi}^\et \iso \fal_{Y \to X}^\et.
	\end{align}

	Recall that the cofiltered limit of sites $(\fal_{Y_\lambda\to X_\lambda}^\et)_{\lambda\in \Lambda}$ is representable by $\underrightarrow{\fal}_{\varphi}^\et$ endowed with the coarsest topology such that the natural functors $\fal_{Y_\lambda\to X_\lambda}^\et\to \underrightarrow{\fal}_{\varphi}^\et$ are continuous (\cite[\Luoma{6}.8.2.3]{sga4-2}).
\end{mypara}

\begin{mylem}\label{lem:descendcov0}
	With the notation in {\rm\ref{para:limit-falet}}, for any covering family $\ak{U}=\{f_k:(V_k \to U_k) \to (V \to U)\}_{k \in K}$ in $\fal_{Y\to X}^\et$ with $K$ finite, there exists an index $\lambda_0\in\Lambda$ and a covering family $\ak{U}_{\lambda_0}=\{f_{k\lambda_0}:(V_{k\lambda_0} \to U_{k\lambda_0}) \to (V_{\lambda_0} \to U_{\lambda_0})\}_{k \in K}$ in $\fal_{Y_{\lambda_0}\to X_{\lambda_0}}^\et$ such that $f_k$ is the base change of $f_{k\lambda_0}$ by the transition morphism $(Y\to X)\to (Y_{\lambda_0}\to X_{\lambda_0})$.
\end{mylem}
\begin{proof}
	There is a standard covanishing covering $\ak{U}'=\{g_{nm}:(V'_{nm} \to U'_n) \to (V \to U) \}_{n \in N, m \in M_n}$ in $\fal_{Y\to X}^\et$ with $N$, $M_n$ finite, which refines $\ak{U}$ by \ref{prop:covcov}. The equivalence \eqref{eq:7.9.2} implies that there exists an index $\lambda_1\in\Lambda$ and families of morphisms $\ak{U}'_{\lambda_1}=\{g_{nm\lambda_1}:(V'_{nm\lambda_1} \to U'_{n\lambda_1}) \to (V_{\lambda_1} \to U_{\lambda_1}) \}_{n \in N, m \in M_n}$ (resp. $\ak{U}_{\lambda_1}=\{f_{k\lambda_1}:(V_{k\lambda_1} \to U_{k\lambda_1}) \to (V_{\lambda_1} \to U_{\lambda_1})\}_{k \in K}$) in $\fal_{Y_{\lambda_1}\to X_{\lambda_1}}^\et$ such that $g_{nm}$ (resp. $f_k$) is the base change of $g_{nm\lambda_1}$ (resp. $f_{k\lambda_1}$) by the transition morphism $(Y\to X)\to (Y_{\lambda_1}\to X_{\lambda_1})$ and that $\ak{U}'_{\lambda_1}$ refines $\ak{U}_{\lambda_1}$. For each $\lambda\geq\lambda_1$, let $g_{nm\lambda}:(V'_{nm\lambda} \to U'_{n\lambda}) \to (V_{\lambda} \to U_{\lambda})$ (resp. $f_{k\lambda}:(V_{k\lambda} \to U_{k\lambda}) \to (V_{\lambda} \to U_{\lambda})$) be the base change of $g_{nm\lambda_1}$ (resp. $f_{k\lambda_1}$) by the transition morphism $(Y_\lambda\to X_\lambda)\to (Y_{\lambda_1}\to X_{\lambda_1})$. Since the morphisms $\coprod_{n\in N} U'_{n}\to U$ and $\coprod_{m\in M_n} V'_{nm}\to V\times_U U'_{n}$ are surjective, there exists an index $\lambda_0\geq \lambda_1$ such that the morphisms $\coprod_{n\in N} U'_{n\lambda_0}\to U_{\lambda_0}$ and $\coprod_{m\in M_n} V'_{nm\lambda_0}\to V_{\lambda_0}\times_{U_{\lambda_0}} U'_{n\lambda_0}$ are also surjective by \cite[8.10.5]{ega4-3}, i.e. $\ak{U}'_{\lambda_0}=\{g_{nm\lambda_0}\}_{n \in N, m \in M_n}$ is a standard covanishing covering in $\fal_{Y_{\lambda_0}\to X_{\lambda_0}}^\et$. Thus, $\ak{U}_{\lambda_0}=\{f_{k\lambda_0}\}_{k \in K}$ is a covering family in $\fal_{Y_{\lambda_0}\to X_{\lambda_0}}^\et$.
\end{proof}

\begin{myprop}[{\cite[\Luoma{6}.11]{abbes2016p}}]\label{prop:limit-fal-sites}
	With the notation in {\rm\ref{para:limit-falet}}, $\fal_{Y \to X}^\et$ represents the limit of sites $(\fal_{Y_\lambda\to X_\lambda}^\et)_{\lambda\in \Lambda}$, and $\falb=\colim_{\lambda\in \Lambda} \varphi_\lambda^{-1} \falb$.
\end{myprop}
\begin{proof}
	The first statement is proved in \cite[\Luoma{6}.11.3]{abbes2016p}. It also follows directly from the discussion in \ref{para:limit-falet} and \ref{lem:descendcov0}. For the second statement, notice that $\colim_{\lambda\in \Lambda} \varphi_\lambda^{-1} \falb=(\colim_{\lambda\in \Lambda} \varphi_{\lambda,\psh} \falb)^\ash$ (\cite[\href{https://stacks.math.columbia.edu/tag/00WI}{00WI}]{stacks-project}). It suffices to show that $\falb(V \to U)=\colim_{\lambda\in \Lambda} (\varphi_{\lambda,\psh} \falb)(V \to U)$ for each object $V\to U$ of $\fal_{Y \to X}^\et$. It follows from the equivalence \eqref{eq:7.9.2} that there exists an index $\lambda_0\in \Lambda$ and an object $V_{\lambda_0} \to U_{\lambda_0}$ of $\fal_{Y_{\lambda_0} \to X_{\lambda_0}}^\et$ such that $V \to U$ is the base change of $V_{\lambda_0} \to U_{\lambda_0}$ by the transition morphism. For each $\lambda\geq \lambda_0$, let $V_\lambda \to U_\lambda$ be the base change of $V_{\lambda_0} \to U_{\lambda_0}$ by the transition morphism $(Y_\lambda\to X_\lambda)\to (Y_{\lambda_0}\to X_{\lambda_0})$. Then, we have $\colim_{\lambda\in \Lambda}(\varphi_{\lambda,\psh} \falb)(V \to U)=\colim_{\lambda\in \Lambda}\falb(V_\lambda \to U_\lambda)$ by \cite[\Luoma{6} 8.5.2, 8.5.7]{sga4-2}. The conclusion follows from $\falb(V \to U)=\colim_{\lambda\in \Lambda} \falb(V_\lambda \to U_\lambda)$ by \ref{lem:rel-norm-fil-limit}.
\end{proof}

\begin{mydefn}\label{defn:profet-mor}
	A morphism $X \to S$ of coherent schemes is called \emph{pro-\'etale} (resp. \emph{pro-finite \'etale}), if there is a directed inverse system of \'etale (resp. finite \'etale) $S$-schemes $(X_\lambda)_{\lambda\in \Lambda}$ with affine transition morphisms such that there is an isomorphism of $S$-schemes $X \cong \lim_{\lambda\in \Lambda} X_\lambda$. We call such an inverse system $(X_\lambda)_{\lambda\in \Lambda}$ a \emph{pro-\'etale presentation} (resp. \emph{pro-finite \'etale presentation}) of $X$ over $S$.
\end{mydefn}

\begin{mylem}\label{lem:proet-profet}
	Let $X \stackrel{g}{\longrightarrow} Y \stackrel{f}{\longrightarrow} S$ be morphisms of coherent schemes.
	\begin{enumerate}
		\setcounter{enumi}{0}
		\renewcommand{\labelenumi}{{\rm(\theenumi)}}
		\item If $f$ is pro-\'etale (resp. pro-finite \'etale), then $f$ is flat (resp. flat and integral).\label{lem:proet-profet-flat}
		\item Any base change of a pro-\'etale (resp. pro-finite \'etale) morphism is pro-\'etale (resp. pro-finite \'etale).\label{lem:proet-profet-bc}
		\item If $f$ and $g$ are pro-\'etale (resp. pro-finite \'etale), then so is $f \circ g$.\label{lem:proet-profet-compsition}
		\item If $f$ and $f \circ g$ are pro-\'etale (resp. pro-finite \'etale), then so is $g$.\label{lem:proet-profet-cancel}
		\item If $f$ is pro-\'etale with a pro-\'etale presentation $Y=\lim Y_\beta$, and if $g$ is \'etale (resp. finite \'etale), then there is an index $\beta_0$ and an \'etale (resp. finite \'etale) $S$-morphism $g_{\beta_0}: X_{\beta_0} \to Y_{\beta_0}$ such that $g$ is the base change of $g_{\beta_0}$ by $Y \to Y_{\beta_0}$.\label{lem:proet-profet-limit}
		\item Let $Z$ and $Z'$ be coherent schemes pro-\'etale over $S$ with pro-\'etale presentations $Z=\lim Z_\alpha$, $Z'=\lim Z'_\beta$, then 
		\begin{align}\label{eq:hom-profet}
			\ho_S(Z, Z') = \lim_\beta \colim_\alpha \ho_S(Z_\alpha, Z'_\beta).
		\end{align}\label{lem:proet-profet-mor}
	\end{enumerate}
\end{mylem}
\begin{proof}	
	(\ref{lem:proet-profet-flat}) and (\ref{lem:proet-profet-bc}) follow directly from the definition. 
	
	(\ref{lem:proet-profet-compsition}) We follow closely the proof of \ref{lem:descendv-cov}. Let $X = \lim X_\alpha$ and $Y=\lim Y_\beta$ be pro-\'etale (resp. pro-finite \'etale) presentations over $Y$ and over $S$ respectively. As $Y_\beta$ are coherent, for each $\alpha$, there is an index $\beta_\alpha$ and an \'etale (resp. finite \'etale) $Y_{\beta_\alpha}$-scheme $X_{\alpha \beta_\alpha}$ such that $X_\alpha\to Y$ is the base change of $X_{\alpha \beta_\alpha}\to Y_{\beta_\alpha}$ (\cite[8.8.2, 8.10.5]{ega4-3}, \cite[17.7.8]{ega4-4}). For each $\beta\geq \beta_\alpha$, let $X_{\alpha \beta}\to Y_{\beta}$ be the base change of $X_{\alpha \beta_\alpha}\to Y_{\beta_\alpha}$ by $Y_\beta\to Y_{\beta_\alpha}$. Then, we have $X = \lim_{\alpha,\beta\geq \beta_\alpha} X_{\alpha \beta}$ by \cite[8.8.2]{ega4-3} (cf. \ref{lem:descendv-cov}), which is pro-finite \'etale over $S$. For (\ref{lem:proet-profet-limit}), one can take $X =X_\alpha$.
	
	(\ref{lem:proet-profet-mor}) We have
	\begin{align}
		\ho_S(Z, Z') = \lim_\beta \ho_S(Z,Z'_\beta) = \lim_\beta \colim_\alpha \ho_S(Z_\alpha, Z'_\beta)
	\end{align}
	where the first equality follows from the universal property of limits of schemes, and the second follows from the fact that $Z'_\beta\to S$ is locally of finite presentation (\cite[8.14.2]{ega4-3}). For (\ref{lem:proet-profet-cancel}), we take $Z=X$ and $Z'=Y$. Then, for each index $\beta$, we have an $S$-morphism $X_\alpha \to Y_\beta$ for $\alpha$ big enough, which is also \'etale (resp. finite \'etale). Then, $X = \lim_{\alpha} X_\alpha = \lim_{\alpha,\beta} X_\alpha \times_{Y_\beta} Y$ is pro-\'etale (resp. pro-finite \'etale) over $Y$.
\end{proof}

\begin{myrem}\label{rem:small-presentation}
	A pro-\'etale (resp. pro-finite \'etale) morphism of $\bb{U}$-small coherent schemes $X\to S$ admits a $\bb{U}$-small pro-\'etale (resp. pro-finite \'etale) presentation. Indeed, let $X=\lim_{\lambda\in\Lambda} X_\lambda$ be a presentation of $X\to S$. We may regard $\Lambda$ as a filtered category with an initial object $0$. Consider the category $\scr{C}={}_{X\backslash}X_{0,\et,\mrm{aff}}$ (resp. $\scr{C}={}_{X\backslash}X_{0,\fet}$) of affine (resp. finite) \'etale $X_0$-schemes which are under $X$. Notice that $\scr{C}$ is essentially $\bb{U}$-small and that the natual functor $\Lambda\to \scr{C}^{\oppo}$ is cofinal by \ref{lem:proet-profet}.(\ref{lem:proet-profet-mor}) (\cite[\Luoma{1}.8.1.3]{sga4-1}). Therefore, after replacing $\scr{C}^{\oppo}$ by a $\bb{U}$-small directed set $\Lambda'$, we obtain a $\bb{U}$-small presentation  $X=\lim_{X'\in \Lambda'} X'$ (\cite[\Luoma{1}.8.1.6]{sga4-1}).
\end{myrem}

\begin{mydefn}\label{defn:profet-site}
	For any $\bb{U}$-small coherent scheme $X$, we endow the category of $\bb{U}$-small coherent pro-\'etale (resp. pro-finite \'etale) $X$-schemes with the topology generated by the pretopology formed by families of morphisms
	\begin{align}
		\{f_i : U_i \to U\}_{i \in I}
	\end{align}
	such that $I$ is finite and that $U = \bigcup f_i(U_i)$. 
	This defines a site $X_\proet$ (resp. $X_\profet$), called the \emph{pro-\'etale site} (resp. \emph{pro-finite \'etale site}) of $X$.
\end{mydefn}
It is clear that the localization $X_{\proet /U}$ (resp. $X_{\profet /U}$) of $X_\proet$ (resp. $X_\profet$) at an object $U$ is canonically equivalent to the pro-\'etale (resp. pro-finite \'etale) site $U_\proet$ (resp. $U_\profet$) of $U$. By definition, any object in $X_\proet$ (resp. $X_\profet$) is quasi-compact.

\begin{mypara}\label{para:compare}
	We compare our definitions of pro-\'etale site and pro-finite \'etale site with some other definitions existing in the literature. But we don't use the comparison result in this paper.
	
	Let $X$ be a $\bb{U}$-small Noetherian scheme. Consider the category of pro-objects $\mrm{pro}\trm{-}X_\fet$ of $X_\fet$, i.e. the category whose objects are functors $F:\ca{A}\to X_\fet$ with $\ca{A}$ a $\bb{U}$-small cofiltered category and whose morphisms are given by $\ho(F,G)=\lim_{\beta\in \ca{B}}\colim_{\alpha\in \ca{A}} \ho(F(\alpha),G(\beta))$ for any $F:\ca{A}\to X_\fet$ and $G:\ca{B}\to X_\fet$ (\cite[3.2]{scholze2013hodge}). We may simply denote such a functor $F$ by $(X_\alpha)_{\alpha\in A}$. Remark that $\lim_{\alpha\in \ca{A}} X_\alpha$ exists which is pro-finite \'etale over $X$. Consider the functor
	\begin{align}
		\mrm{pro}\trm{-}X_\fet\longrightarrow X_\profet,\ (X_\alpha)_{\alpha\in A}\mapsto \lim_{\alpha\in \ca{A}} X_\alpha,
	\end{align}
	which is well-defined and fully faithful by \ref{lem:proet-profet}.(\ref{lem:proet-profet-mor}) and essentially surjective by \ref{rem:small-presentation}. Thus, according to \cite[3.3]{scholze2013hodge} and its corrigendum \cite{scholze2016erratum}, Scholze's pro-finite \'etale site $X_\profet^{\mrm{S}}$ has the underlying category $X_\profet$ and its topology is generated by the families of morphisms 
	\begin{align}
		\{U_i\stackrel{f_i}{\longrightarrow}U'\stackrel{f}{\longrightarrow} U\}_{i\in I}
	\end{align}
	where $I$ is finite and $\coprod_{i\in I} U_i \to U'$ is finite \'etale surjective, and there exists a well-ordered directed set $\Lambda$ with a least index $0$ and a pro-finite \'etale presentation $(U'_\lambda)_{\lambda\in \Lambda}$ of $f$ such that $U'_0=U$ and that for each $\lambda\in \Lambda$ the natural morphism $U'_\lambda\to \lim_{\mu<\lambda} U'_\mu$ is finite \'etale surjective (cf. \cite[5.5]{kerz2016transfinite}, \ref{lem:proet-profet} and \cite[8.10.5.(\luoma{6})]{ega4-3}). It is clear that the topology of our pro-finite \'etale site $X_\profet$ is finer than that of $X_\profet^{\mrm{S}}$. We remark that if $X$ is connected, then $X_\profet^{\mrm{S}}$ gives a site-theoretic interpretation of the continuous group cohomology of the fundamental group of $X$ (\cite[3.7]{scholze2013hodge}). For simplicity, we don't consider $X_\profet^{\mrm{S}}$ in the rest of the paper, but we can replace $X_\profet$ by $X_\profet^{\mrm{S}}$ for most of the statements in this paper (cf. \cite[6]{kerz2016transfinite}).
\end{mypara}

\begin{mypara}
	Let $X$ be a $\bb{U}$-small scheme. Bhatt-Scholze's pro-\'etale site $X_\proet^{\mrm{BS}}$ has the underlying category of $\bb{U}$-small weakly \'etale $X$-schemes and a family of morphisms $\{f_i:Y_i \to Y\}_{i \in I}$ in $X_\proet^{\mrm{BS}}$ is a covering if and only if for any affine open subscheme $U$ of $Y$, there exists a map $a:\{1,\dots,n\}\to I$ and affine open subschemes $U_j$ of $Y_{a(j)}$ ($j=1,\dots,n$) such that $U=\bigcup_{j=1}^n f_{a(j)}(U_j)$ (\cite[4.1.1]{bhattscholze2015proet}, cf. \cite[\href{https://stacks.math.columbia.edu/tag/0989}{0989}]{stacks-project}). Remark that a pro-\'etale morphism of coherent schemes is weakly \'etale by \cite[2.3.3.1]{bhattscholze2015proet}. Thus, for a coherent scheme $X$, $X_\proet$ is a full subcategory of $X_\proet^{\mrm{BS}}$.
\end{mypara}
\begin{mylem}
	Let $X$ be a coherent scheme. The full subcategory $X_\proet$ of $X_\proet^{\mrm{BS}}$ is a topologically generating family, and the induced topology on $X_\proet$ coincides the topology defined in {\rm\ref{defn:profet-site}}. In particular, the topoi of sheaves of $\bb{V}$-small sets associated to the two sites are naturally equivalent.
\end{mylem}
\begin{proof}
	For a weakly \'etale $X$-scheme $Y$, we show that it can be covered by pro-\'etale $X$-schemes. After replacing $X$ by a finite affine open covering and replacing $Y$ by an affine open covering, we may assume that $X$ and $Y$ are affine. Then, the result follows from the fact that for any weakly \'etale morphism of rings $A\to B$ there exists a faithfully flat ind-\'etale morphism $B\to C$ such that $A\to C$ is ind-\'etale by \cite[2.3.4]{bhattscholze2015proet} (cf. \cite[4.1.3]{bhattscholze2015proet}). Thus, $X_\proet$ is a topologically generating family of $X_\proet^{\mrm{BS}}$. A family of morphisms $\{f_i:Y_i\to Y\}_{i\in I}$ in $X_\proet$ is a covering with respect to the induced topology if and only if for any affine open subscheme $U$ of $Y$, there exists a map $a:\{1,\dots,n\}\to I$ and affine open subschemes $U_j$ of $Y_{a(j)}$ ($j=1,\dots,n$) such that $U=\bigcup_{j=1}^n f_{a(j)}(U_j)$ (\cite[\Luoma{3}.3.3]{sga4-1}). Notice that $Y_i$ and $Y$ are coherent, thus $\{f_i\}_{i\in I}$ is a covering if and only if there exists a finite subset $I_0\subseteq I$ such that $Y=\bigcup_{i\in I_0} f_i(Y_i)$, which shows that the induced topology on $X_\proet$ coincides the topology defined in {\rm\ref{defn:profet-site}}. Finally, the ``in particular'' part follows from \cite[\Luoma{3}.4.1]{sga4-1}.
\end{proof}

\begin{mydefn}\label{defn:Eproet}
	Let $Y\to X$ be a morphism of coherent schemes. A morphism $(V' \to U') \to (V \to U)$ in $\fal_{Y\to X}$ is called \emph{pro-\'etale} if $U' \to U$ is pro-\'etale and $V' \to V \times_U U'$ is pro-finite \'etale. A \emph{pro-\'etale presentation} of such a morphism is a directed inverse system $(V_\lambda \to U_\lambda)_{\lambda\in \Lambda}$ \'etale over $V \to U$ with affine transition morphisms $U_{\lambda'}\to U_\lambda$ and $V_{\lambda'}\to V_\lambda$ ($\lambda'\geq \lambda$) such that $(V' \to U')=\lim_{\lambda\in \Lambda} (V_\lambda \to U_\lambda)$.
\end{mydefn}

\begin{mylem}\label{lem:proetale-basic}
	Let $Y\to X$ be a morphism of coherent schemes, $(V'' \to U'') \stackrel{g}{\longrightarrow} (V' \to U') \stackrel{f}{\longrightarrow} (V \to U)$ morphisms in $\fal_{Y \to X}$.
	\begin{enumerate}
		\renewcommand{\labelenumi}{{\rm(\theenumi)}}
		\item If $f$ is pro-\'etale, then it admits a pro-\'etale presentation.
		\item If $f$ is pro-\'etale, then any base change of $f$ is also pro-\'etale.
		\item If $f$ and $g$ are pro-\'etale, then $f \circ g$ is also pro-\'etale.
		\item If $f$ and $f\circ g$ are pro-\'etale, then $g$ is also pro-\'etale.
	\end{enumerate}
\end{mylem}
\begin{proof}
	It follows directly from \ref{lem:proet-profet} and its arguments.
\end{proof}

\begin{myrem}\label{rem:small-presentation-proet}
	Similar to \ref{rem:small-presentation}, a pro-\'etale morphism in $\fal_{Y\to X}$ admits a $\bb{U}$-small presentation.
\end{myrem}

\begin{mypara}\label{para:proet-falsite}
	Let $Y\to X$ be a morphism of coherent schemes, $\fal_{Y\to X}^\proet$ the full subcategory of $\fal_{Y\to X}$ formed by objects which are pro-\'etale over the final object $Y \to X$. It is clear that $\fal_{Y\to X}^\proet$ is stable under finite limits in $\fal_{Y\to X}$. Then, the functor \eqref{eq:sigma-fal} induces a functor
	\begin{align}
		\phi^+: \fal_{Y\to X}^\proet \longrightarrow X_\proet,\ (V \to U) \longmapsto U,
	\end{align}
	which endows $\fal_{Y\to X}^\proet / X_\proet$ with a structure of fibred sites, whose fibre over $U$ is the pro-finite \'etale site $U_{Y,\profet}$. We endow $\fal_{Y\to X}^\proet$ with the associated covanishing topology, that is, the topology generated by the following types of families of morphisms
	\begin{enumerate}
		\item[\rm{(v)}] $\{(V_m \to U) \to (V \to U)\}_{m \in M}$, where $M$ is a finite set and $\coprod_{m\in M} V_m\to V$ is surjective;
		\item[\rm{(c)}] $\{(V\times_U{U_n} \to U_n) \to (V \to U)\}_{n \in N}$, where $N$ is a finite set and $\coprod_{n\in N} U_n\to U$ is surjective.
	\end{enumerate}	 
	It is clear that any object in $\fal_{Y\to X}^\proet$ is quasi-compact by \ref{prop:covcov}. We still denote by $\falb$ the restriction of the presheaf $\falb$ on $\fal_{Y\to X}$ to $\fal_{Y\to X}^\proet$ if there is no ambiguity. We will show in \ref{cor:nu-falb} that $\falb$ is a sheaf on $\fal_{Y\to X}^\proet$. 
\end{mypara}

\begin{mydefn}\label{defn:profalsite}
We call $(\fal_{Y\to X}^\proet,\falb)$ the \emph{pro-\'etale Faltings ringed site} of the morphism of coherent schemes $Y \to X$.
\end{mydefn}
It is clear that the localization $(\fal_{Y \to X}^\proet)_{/ (V \to U)}$ of $\fal_{Y\to X}^\proet$ at an object $V \to U$ is canonically equivalent to the pro-\'etale Faltings ringed site $\fal_{V \to U}^\proet$ of the morphism $V \to U$ by \ref{prop:covcov}.

\begin{myrem}
	The categories $X_\proet$, $X_\profet$ and $\fal_{Y\to X}^\proet$ are essentially $\bb{V}$-small categories.
\end{myrem}

\begin{mypara}\label{para:notation-psi-beta-sigma-proet}
	Let $Y \to X$ be a morphism of coherent schemes. Consider the natural functors
	\begin{align}
		\psi^+ &: \fal_{Y\to X}^\proet \longrightarrow Y_{\proet},\ (V \to U) \longmapsto V,\\
		\beta^+ &: Y_{\profet} \longrightarrow \fal_{Y\to X}^\proet,\ V \longmapsto (V \to X),\\
		\sigma^+ &: X_\proet \longrightarrow \fal_{Y\to X}^\proet,\ U\longmapsto (Y\times_X U\to U) .
	\end{align}
	They are left exact and continuous (cf. \ref{para:notation-psi-beta-sigma}). Then, we obtain a commutative diagram of sites associated functorially to the morphism $Y\to X$ by \ref{para:topos},
	\begin{align}\label{diag:beta-psi-sigma-proet}
		\xymatrix{
			&Y_\proet\ar[dl]_-{\rho}\ar[dr]\ar[d]^-{\psi}&\\
			Y_\profet&\fal_{Y \to X}^\proet \ar[r]_-{\sigma}\ar[l]^-{\beta}& X_\proet
		}
	\end{align}
	where $\rho: Y_\proet \to Y_\profet$ is defined by the inclusion functor, and the unlabelled arrow $Y_\proet \to X_\proet$ is induced by the morphism $Y\to X$. Moreover, if $\ca{O}_{X_\proet}$ denotes the structural sheaf on $X_\proet$ sending $U$ to $\Gamma(U,\ca{O}_U)$, then $\sigma^+$ actually defines a morphism of ringed sites
	\begin{align}
		\sigma: (\fal_{Y \to X}^\proet,\falb) \longrightarrow (X_\proet,\ca{O}_{X_\proet}).
	\end{align}
\end{mypara}

\begin{mylem}\label{lem:morphism-nu}
	Let $Y \to X$ be a morphism of coherent schemes. Then, the inclusion functor
	\begin{align}\label{eq:3.6.1}
		\nu^+ : \fal_{Y\to X}^\et \longrightarrow \fal_{Y\to X}^\proet,\ (V\to U) \longmapsto (V \to U)
	\end{align}
	 defines a morphism of sites $\nu :\fal_{Y\to X}^\proet \to \fal_{Y\to X}^\et$.
\end{mylem}
\begin{proof}
	It is clear that $\nu^+$ commutes with finite limits and sends a standard covanishing covering in $\fal_{Y\to X}^\et$ to a standard covanishing covering in $\fal_{Y\to X}^\proet$ (\ref{exmp:stdcovcov}). Therefore, $\nu^+$ is continuous by \ref{prop:covcov} and defines a morphism of sites.
\end{proof}

\begin{mylem}\label{lem:inducetop-proet-et}
	Let $Y \to X$ be a morphism of coherent schemes. Then, the topology on $\fal_{Y\to X}^\et$ is the topology induced from $\fal_{Y\to X}^\proet$.
\end{mylem}
\begin{proof}
	After \ref{lem:morphism-nu}, it suffices to show that for a family of morphisms $\ak{U}=\{(V_k \to U_k) \to (V \to U)\}_{k \in K}$ in $\fal_{Y\to X}^\et$, if $\nu^+(\ak{U})$ is a covering in $\fal_{Y\to X}^\proet$, then $\ak{U}$ is a covering in $\fal_{Y\to X}^\et$. We may assume that $K$ is finite. There is a standard covanishing covering $\ak{U}'=\{(V'_{nm} \to U'_n) \to (V \to U) \}_{n \in N, m \in M_n}$ in $\fal_{Y \to X}^\proet$ with $N$, $M_n$ finite, which refines $\nu^+(\ak{U})$ by \ref{prop:covcov}. We take a directed set $\Xi$ such that for each $n \in N$, we can take a pro-\'etale presentation $U'_n = \lim_{\xi\in \Xi} U'_{n\xi}$ over $U$, and we take a directed set $\Sigma$ such that for each $n\in N$ and $m \in M_n$, we can take a pro-finite \'etale presentation $V'_{nm} = \lim_{\sigma\in\Sigma} V'_{nm\sigma}$ over $V\times_U U'_n$. By \ref{lem:proet-profet} (\ref{lem:proet-profet-limit}), for each $\sigma\in\Sigma$, there exists an index $\xi_\sigma\in \Xi$ and a finite \'etale morphism $V'_{nm\sigma\xi_\sigma}\to V\times_U U'_{n\xi_\sigma}$ for each $n$ and $m$, whose base change by $U'_{n} \to U'_{n\xi_\sigma}$ is $V'_{nm\sigma}\to V \times_U U'_n$. Let $V'_{nm\sigma\xi}\to V\times_U U'_{n\xi}$ be the base change of $V'_{nm\sigma\xi_\sigma}\to V\times_U U'_{n\xi_\sigma}$ by the transition morphism $U'_{n\xi}\to U'_{n\xi_\sigma}$ for each $\xi\geq \xi_\sigma$. Since $\coprod_{m \in M_n} V'_{nm\sigma} \to V \times_U U'_n$ is surjective, after enlarging $\xi_\sigma$, we may assume that $\coprod_{m \in M_n} V'_{nm\sigma \xi} \to V \times_U U'_{n\xi}$ is also surjective for $\xi\geq \xi_\sigma$ by \cite[8.10.5.(\luoma{6})]{ega4-3}. It is clear that $\coprod_{n\in N}U'_{n\xi} \to U$ is surjective for each $\xi\in \Xi$. Therefore, $\ak{U}'_{\sigma\xi}=\{(V'_{nm\sigma \xi} \to U'_{n\xi}) \to (V \to U)\}_{n \in N, m \in M_n}$ is a standard covanishing covering in $\fal_{V \to U}^\et$ for each $\sigma\in \Sigma$ and $\xi\geq \xi_\sigma$. Notice that for each $n\in N$ and $m\in M_n$, there exists $k\in K$ such that the morphism $(V'_{nm\sigma \xi} \to U'_{n\xi}) \to (V \to U)$ factors through $(V_{k} \to U_{k})$ for $\sigma,\xi$ big enough by \ref{lem:proet-profet} (\ref{lem:proet-profet-mor}), which shows that $\ak{U}$ is a covering in $\fal_{Y\to X}^\et$.
\end{proof}

\begin{mylem}\label{lem:descendcov}
	Let $Y \to X$ be a morphism of coherent schemes, $\ak{U}=\{(V_k \to U_k) \to (V \to U)\}_{k \in K}$ a covering in $\fal_{Y\to X}^\proet$ with $K$ finite. Then, there exist pro-\'etale presentations $(V\to U)=\lim_{\lambda \in \Lambda} (V_{\lambda} \to U_{\lambda})$, $(V_k\to U_k)=\lim_{\lambda \in \Lambda} (V_{k\lambda} \to U_{k\lambda})$ over $Y\to X$ and compatible \'etale morphisms $(V_{k\lambda} \to U_{k\lambda}) \to (V_{\lambda} \to U_{\lambda})$ such that the family $\ak{U}_\lambda=\{(V_{k\lambda} \to U_{k\lambda}) \to (V_\lambda \to U_\lambda)\}_{k \in K}$ is a covering in $\fal_{Y\to X}^\et$.
\end{mylem}
\begin{proof}
	We follow closely the proof of \ref{lem:descendv-cov}. We take a directed set $A$ such that for each $k\in K$ we can take a pro-\'etale presentation $(V_k \to U_k)=\lim_{\alpha \in A}(V_{k\alpha} \to U_{k\alpha})$ over $(V \to U)$. Then, $\ak{U}_\alpha=\{(f_{k\alpha}:V_{k\alpha} \to U_{k\alpha})\to (V\to U)\}_{k\in K}$ is a covering family in $\fal_{V\to U}^\et$ for each $\alpha \in A$ by \ref{lem:inducetop-proet-et}.
	
	Let $(V \to U)=\lim_{\beta \in B}(V_\beta \to U_\beta)$ be a pro-\'etale presentation over $Y\to X$. For each $\alpha\in A$, there exists an index $\beta_\alpha\in B$ and a covering family $\ak{U}_{\alpha\beta_\alpha}=\{f_{k\alpha\beta_\alpha}:(V_{k\alpha\beta_\alpha} \to U_{k\alpha\beta_\alpha}) \to (V_{\beta_\alpha} \to U_{\beta_\alpha})\}_{k \in K}$ such that $f_{k\alpha}$ is the base change of $f_{k\alpha\beta_\alpha}$ by the transition morphism $(V \to U)\to (V_{\beta_\alpha} \to U_{\beta_\alpha})$ (\ref{lem:descendcov0}). For each $\beta\geq \beta_\alpha$, let $f_{k\alpha\beta}:(V_{k\alpha\beta} \to U_{k\alpha\beta}) \to (V_{\beta} \to U_{\beta})$ be the base change of $f_{k\alpha\beta_\alpha}$ by the transition morphism $(V_\beta \to U_\beta)\to (V_{\beta_\alpha} \to U_{\beta_\alpha})$. We take $\Lambda=\{(\alpha,\beta)\in A\times B\ |\ \beta\geq \beta_\alpha\}$, $(V_\lambda\to U_\lambda)=(V_\beta\to U_\beta)$ and $(V_{k\lambda}\to U_{k\lambda})=(V_{k\alpha\beta} \to U_{k\alpha\beta})$ for each $\lambda=(\alpha,\beta)\in \Lambda$. Then, the families $\ak{U}_\lambda=\{(V_{k\lambda} \to U_{k\lambda}) \to (V_\lambda \to U_\lambda)\}_{k \in K}$ meet the requirements in the lemma (cf. \ref{lem:descendv-cov}).
\end{proof}

\begin{mylem}\label{lem:nu-inv}
	Let $Y \to X$ be a morphism of coherent schemes, $\ca{F}$ a presheaf on $\fal_{Y\to X}^\et$, $V \to U$ an object of $\fal_{Y\to X}^\proet$ with a pro-\'etale presentation $(V\to U) = \lim (V_\lambda \to U_\lambda)$. Then, we have $\nu_\psh\ca{F}(V \to U) = \colim \ca{F}(V_\lambda \to U_\lambda)$, where $\nu^+:\fal_{Y\to X}^\et \to \fal_{Y\to X}^\proet$ is the inclusion functor.
\end{mylem}
\begin{proof}
	Notice that the presheaf $\ca{F}$ is a filtered colimit of representable presheaves by \cite[\Luoma{1}.3.4]{sga4-1}
	\begin{align}
		\ca{F} = \colim_{(V' \to U') \in (\fal_{Y \to X}^\et)_{/\ca{F}}} h_{V' \to U'}^{\et}.
	\end{align}
	Thus, we may assume that $\ca{F}$ is representable by $V' \to U'$ since the section functor $\Gamma(V\to U,-)$ commutes with colimits of presheaves (\cite[\href{https://stacks.math.columbia.edu/tag/00VB}{00VB}]{stacks-project}). Then, we have
	\begin{align}
		\nu_\psh h_{V' \to U'}^\et(V \to U) =& h_{V' \to U'}^\proet(V \to U)\\
		=& \ho_{\fal_{Y\to X}^\proet}((V \to U) , (V' \to U'))\nonumber\\
		=& \colim \ho_{\fal_{Y\to X}^\et}((V_\lambda \to U_\lambda) , (V' \to U'))\nonumber\\
		=& \colim h_{V' \to U'}^\et(V_\lambda \to U_\lambda)\nonumber
	\end{align}
	where the first equality follows from \cite[\href{https://stacks.math.columbia.edu/tag/04D2}{04D2}]{stacks-project}, and the third equality follows from \cite[8.14.2]{ega4-3} since $U'$ and $V'$ are locally of finite presentation over $X$ and $Y\times_X U'$ respectively.
\end{proof}

\begin{myprop}\label{prop:et-proet}
	Let $Y \to X$ be a morphism of coherent schemes, $\ca{F}$ an abelian sheaf on $\fal_{Y\to X}^\et$, $V \to U$ an object of $\fal_{Y\to X}^\proet$ with a pro-\'etale presentation $(V\to U) = \lim (V_\lambda \to U_\lambda)$. Then, for any integer $q$, we have
	\begin{align}
		H^q(\fal_{V \to U}^\proet, \nu^{-1}\ca{F}) = \colim H^q(\fal_{V_\lambda \to U_\lambda}^\et, \ca{F}),
	\end{align} 
	where $\nu :\fal_{Y\to X}^\proet \to \fal_{Y\to X}^\et$ is the morphism of sites defined by the inclusion functor {\rm(\ref{lem:morphism-nu})}.
	In particular, the canonical morphism $\ca{F} \longrightarrow \rr\nu_*\nu^{-1}\ca{F}$	is an isomorphism.
\end{myprop}
\begin{proof}
	We follow closely the proof of \ref{prop:v-h-coh}.
	For the second assertion, since $\rr^q\nu_*\nu^{-1}\ca{F}$ is the sheaf associated to the presheaf $(V \to U) \mapsto H^q(\fal_{V \to U}^\proet, \nu^{-1}\ca{F}) = H^q(\fal_{V \to U}^\et, \ca{F})$ by the first assertion, which is $\ca{F}$ if $q=0$ and vanishes otherwise.
	
	For the first assertion, we may assume that $\ca{F}=\ca{I}$ is an abelian injective sheaf on $\fal_{Y\to X}^\et$ (cf. \ref{prop:v-h-coh}). We claim that for any covering in $\fal_{Y\to X}^\proet$, $\ak{U}=\{(V_k \to U_k) \to (V \to U)\}_{k \in K}$ with $K$ finite, the augmented \v Cech complex associated to the presheaf $\nu_\psh \ca{I}$,
	\begin{align}\label{eq:3.7.4}
		0 \to \nu_{\psh}\ca{I}(V \to U) \to \prod_k \nu_{\psh}\ca{I}(V_k\to U_k) \to \prod_{k,k'} \nu_{\psh}\ca{I}(V_k \times_V V_{k'} \to U_k \times_U U_{k'}) \to \cdots
	\end{align}
	is exact. Admitting this claim, we see that $\nu_\psh \ca{I}$ is indeed a sheaf, i.e. $\nu^{-1}\ca{I} = \nu_\psh \ca{I}$, and the vanishing of higher \v Cech cohomologies implies that $H^q(\fal_{V \to U}^\proet,\nu^{-1}\ca{I})=0$ for any $q>0$, which completes the proof together with \ref{lem:nu-inv}. For the claim, let $(V\to U)=\lim_{\lambda \in \Lambda} (V_{\lambda} \to U_{\lambda})$ and $(V_k\to U_k)=\lim_{\lambda \in \Lambda} (V_{k\lambda} \to U_{k\lambda})$ be the pro-\'etale presentations constructed in \ref{lem:descendcov}. The family $\ak{U}_\lambda=\{(V_{k\lambda} \to U_{k\lambda}) \to (V_\lambda \to U_\lambda)\}_{k \in K}$ is a covering in $\fal_{Y\to X}^\et$. By \ref{lem:nu-inv}, the sequence \eqref{eq:3.7.4} is the filtered colimit of the augmented \v Cech complexes
	\begin{align}
		0 \to \ca{I}(V_\lambda \to U_\lambda) \to \prod_k \ca{I}(V_{k\lambda}\to U_{k\lambda}) \to \prod_{k,k'} \ca{I}(V_{k\lambda} \times_{V_\lambda} V_{k'\lambda} \to U_{k\lambda} \times_{U_\lambda} U_{k'\lambda}) \to \cdots
	\end{align}
	which are exact since $\ca{I}$ is an injective abelian sheaf on $\fal_{Y\to X}^\et$.
\end{proof}

\begin{mycor}\label{cor:nu-falb}
	With the notation in {\rm\ref{prop:et-proet}}, the presheaf $\falb$ on $\fal_{Y\to X}^\proet$ is a sheaf, and the canonical morphisms $\nu^{-1}\falb \to \falb$ and $\falb\to\rr\nu_*\falb$ are isomorphisms. If moreover $p$ is invertible on $Y$, then for each integer $n\geq 0$, the canonical morphisms $\nu^{-1}(\falb/p^n\falb) \to \falb/p^n\falb$ and $\falb/p^n\falb\to\rr\nu_*(\falb/p^n\falb)$ are isomorphisms.
\end{mycor}
\begin{proof}
	For any pro-\'etale presentation $(V\to U) = \lim (V_\lambda \to U_\lambda)$, we have $\nu^{-1}\falb(V \to U) = \colim \falb(V_\lambda \to U_\lambda) =\falb(V \to U)$ by \ref{lem:nu-inv} and \ref{lem:rel-norm-fil-limit}. This verifies that $\falb$ is a sheaf  on $\fal_{Y\to X}^\proet$ and that $\nu^{-1}\falb \to \falb$ is an isomorphism. The second isomorphism follows from the first and \ref{prop:et-proet}. For the last assertion, notice that the multiplication by $p^n$ is injective on $\falb$, so that the conclusion follows from the exact sequence
	\begin{align}
		\xymatrix{
			0\ar[r]& \falb\ar[r]^-{\cdot p^n} &\falb\ar[r]&\falb/p^n\falb\ar[r]&0.
		}
	\end{align}
\end{proof}

\begin{mypara}\label{para:fibred-site-N}
	We regard the ordered set $\bb{N}$ of natural numbers as a filtered category (there is an arrow $i\to j$ if $i\leq j$). Let $E$ be a site. We denote by $E^{\bb{N}}$ the fibred site $E\times \bb{N}$ over $\bb{N}$, and we endow the category $E^{\bb{N}}$ with the total topology which makes it into a site without final objects (\ref{par:cova-def}, cf. \cite[\Luoma{3}.7, \Luoma{6}.7]{abbes2016p}).  Giving a presheaf $\ca{F}$ on $E^{\bb{N}}$ is equivalent to giving a directed inverse system of presheaves $(\ca{F}_n )_{n\geq 0}$ on $E$. We write $\ca{F}=(\ca{F}_n )_{n\geq 0}$. Moreover, $\ca{F}$ is a sheaf on $E^{\bb{N}}$ if and only if each $\ca{F}_n$ is a sheaf on $E$.
\end{mypara}

\begin{mypara}\label{para:psi-sigma-nu-breve}
	Let $Y\to X$ be a morphism of coherent schemes. We obtain a fibred site $\fal_{Y\to X}^{\et,\bb{N}}$ (resp. $\fal_{Y\to X}^{\proet,\bb{N}}$) over $\bb{N}$ by \ref{para:fibred-site-N}. We define a sheaf $\breve{\falb}$ on $\fal_{Y\to X}^{\et,\bb{N}}$ (resp. $\fal_{Y\to X}^{\proet,\bb{N}}$) by
	\begin{align}
		\breve{\falb}=(\falb/p^n\falb)_{n\geq 0}.
	\end{align}
	The inclusion functor $\nu^+:\fal_{Y\to X}^\et\to \fal_{Y\to X}^\proet$ defines a morphism of ringed fibred sites (\cite[\Luoma{6}.7.2.2]{sga4-2})
	\begin{align}
		\breve{\nu}: (\fal_{Y\to X}^{\proet,\bb{N}},\breve{\falb})\longrightarrow (\fal_{Y\to X}^{\et,\bb{N}},\breve{\falb}),\label{eq:breve-nu}	
	\end{align}
	which induces a morphism of the associated ringed topoi with respect to the total topology (\cite[\Luoma{6}.7.4.13.1]{sga4-2}, cf. \cite[\Luoma{3}.7.18]{abbes2016p}). 
	If moreover the prime number $p$ is invertible on $Y$, then the canonical morphisms
	\begin{align}\label{eq:5.3.6}
		\breve{\nu}^{-1}\breve{\falb}\longrightarrow \breve{\falb}\ \trm{ and }\ \breve{\falb}\longrightarrow \rr\breve{\nu}_*\breve{\falb}
	\end{align}
	are isomorphisms by \ref{cor:nu-falb} ( \cite[\Luoma{6}.7.7]{abbes2016p}). Combining with the canonical diagrams \eqref{diag:beta-psi-sigma} and \eqref{diag:beta-psi-sigma-proet}, we obtain a canonical commutative diagram of fibred sites over $\bb{N}$,
	\begin{align}\label{diam:breve-psi-sigma}
		\xymatrix{
			Y_\proet^{\bb{N}}\ar[r]^-{\breve{\psi}}\ar[d]^-{\breve{\nu}}& \fal_{Y\to X}^{\proet,\bb{N}}\ar[r]^-{\breve{\sigma}}\ar[d]^-{\breve{\nu}}& X_\proet^{\bb{N}}\ar[d]^-{\breve{\nu}}\\
			Y_\et^{\bb{N}}\ar[r]^-{\breve{\psi}}& \fal_{Y\to X}^{\et,\bb{N}}\ar[r]^-{\breve{\sigma}}& X_\et^{\bb{N}}
		}
	\end{align}
	where we ambiguously denote by $\breve{\nu}$ the morphisms induced by the inclusion functors of the underlying categories. Moreover, if $\breve{\ca{O}}_{X_\et}$ (resp. $\breve{\ca{O}}_{X_\proet}$) denotes the sheaf $(\ca{O}_{X_\et}/p^n\ca{O}_{X_\et})_{n\geq 0}$ (resp. $(\ca{O}_{X_\proet}/p^n\ca{O}_{X_\proet})_{n\geq 0}$), then the diagram \eqref{diam:breve-psi-sigma} induces a commutative diagram of ringed fibred sites
	\begin{align}\label{diam:breve-psi-sigma-ringed}
		\xymatrix{
			(\fal_{Y\to X}^{\proet,\bb{N}},\breve{\falb})\ar[r]^-{\breve{\sigma}}\ar[d]^-{\breve{\nu}}& (X_\proet^{\bb{N}},\breve{\ca{O}}_{X_\proet})\ar[d]^-{\breve{\nu}}\\
			(\fal_{Y\to X}^{\et,\bb{N}},\breve{\falb})\ar[r]^-{\breve{\sigma}}& (X_\et^{\bb{N}},\breve{\ca{O}}_{X_\et})
		}
	\end{align}
\end{mypara}

\begin{mydefn}\label{defn:p-adic-falsites}
	We call $(\fal_{Y\to X}^{\et,\bb{N}},\breve{\falb})$ (resp. $(\fal_{Y\to X}^{\proet,\bb{N}},\breve{\falb})$) the \emph{$p$-adic Faltings ringed site} (resp. \emph{$p$-adic pro-\'etale Faltings ringed site}) of the morphism of coherent schemes $Y \to X$. It is a ringed site without final objects, which can be regarded as a fibred site over $\bb{N}$.
\end{mydefn}

\section{Cohomological Descent of the Structural Sheaves}\label{sec:coh-descent}
%In this section, we will establish a cohomological descent result for Faltings ringed sites (\ref{thm:falh-proet}) when there is no divisor on the generic fibre (\ref{para:notation}). In the end, using Abhyankar's lemma, we compute the cohomologies of Faltings ringed sites (\ref{thm:acyclic}) when there are divisors on the generic fibre (\ref{para:notation-open}).

\begin{mydefn}\label{defn:faltings-acyclic}
	Let $K$ be a pre-perfectoid field of mixed characteristic $(0,p)$, $Y\to X$ a morphism of coherent schemes such that $Y\to X^Y$ is over $\spec(K)\to \spec(\ca{O}_K)$, where $X^Y$ denotes the integral closure of $X$ in $Y$. We say that $Y\to X$ is \emph{Faltings acyclic} if $X$ is affine and if for any integer $n>0$ the canonical morphism
	\begin{align}\label{eq:5.5.1}
		A/p^n A\longrightarrow \rr\Gamma(\fal_{Y\to X}^\proet,\falb/p^n\falb)
	\end{align}
	is an almost isomorphism (cf. \ref{para:almost-derived}), where $A$ denotes the $\ca{O}_K$-algebra $\falb(Y\to X)$ (i.e. $X^Y=\spec(A)$).
\end{mydefn}

\begin{myrem}
	In \ref{defn:faltings-acyclic}, the canonical morphism $\rr\Gamma(\fal_{Y\to X}^\et,\falb/p^n\falb)\to \rr\Gamma(\fal_{Y\to X}^\proet,\falb/p^n\falb)$ is an isomorphism as \eqref{eq:5.3.6} are isomorphisms.
\end{myrem}

\begin{mylem}\label{lem:acyclic-falsite}
	Let $K$ be a pre-perfectoid field of mixed characteristic $(0,p)$, $Y\to X$ a morphism of coherent schemes such that $Y\to X^Y$ is over $\spec(K)\to \spec(\ca{O}_K)$,  $A=\falb(Y\to X)$. Assume that $Y\to X$ is Faltings acyclic. Then, we have:
	\begin{enumerate}
		\renewcommand{\labelenumi}{{\rm(\theenumi)}}
		\item For any pseudo-uniformizer $\pi$ of $K$, the canonical morphism $A/\pi A\to \rr\Gamma(\fal_{Y\to X}^\proet,\falb/\pi\falb)$ is an almost isomorphism.\label{item:lem-acyclic-falsite1}
		\item Let $\widehat{A}$ be the $p$-adic completion of $A$. Then, the canonical morphism $\widehat{A}\to \rr\Gamma(\fal^{\proet,\bb{N}}_{Y\to X},\breve{\falb})$ is an almost isomorphism.\label{item:lem-acyclic-falsite2}
	\end{enumerate}
\end{mylem}
\begin{proof}
	(\ref{item:lem-acyclic-falsite1}) There exists an integer $n>0$ such that $\pi'=p^n/\pi$ is a pseudo-uniformizer of $K$. Since $A$ and $\falb$ are flat over $\ca{O}_K$, we have a natural morphism of exact sequences
	\begin{align}
		\xymatrix{
			0\ar[r]& A/\pi'A\ar[d]^-{\alpha_1}\ar[r]& A/p^nA\ar[r]\ar[d]^-{\alpha_2} & A/\pi A\ar[r]\ar[d]^-{\alpha_3}&0\\
			0\ar[r]& H^0(\fal_{Y\to X}^\proet,\falb/\pi'\falb)\ar[r]& H^0(\fal_{Y\to X}^\proet,\falb/p^n\falb)\ar[r] & H^0(\fal_{Y\to X}^\proet,\falb/\pi\falb)&
		}
	\end{align}
	By definition, $\alpha_2$ is an almost isomorphism. Thus, $\alpha_1$ is almost injective. Since any pseudo-uniformizer of $K$ is of the form $\pi'=p^n/\pi$ for some pseudo-uniformizer $\pi$ of $K$ and $n>0$, $\alpha_3$ is almost injective. By diagram chasing, we see that $\alpha_1$ is an almost isomorphism (and thus so is $\alpha_3$). It remains to show that $H^q(\fal_{Y\to X}^\proet,\falb/\pi\falb)$ is almost zero for $q>0$. Since $H^q(\fal_{Y\to X}^\proet,\falb/p^n\falb)$ is almost zero. By the long exact sequence associated to the short exact sequence $0\to \falb/\pi'\falb\to \falb/p^n\falb\to \falb/\pi \falb\to 0$, we see that $H^1(\fal_{Y\to X}^\proet,\falb/\pi'\falb)$ is almost zero and that $H^q(\fal_{Y\to X}^\proet,\falb/\pi\falb)\to H^{q+1}(\fal_{Y\to X}^\proet,\falb/\pi'\falb)$ is an almost isomorphism. By induction, we complete the proof.

	(\ref{item:lem-acyclic-falsite2}) Recall that for any integer $q\geq 0$ there exists a canonical exact sequence (\cite[\Luoma{6}.7.10]{abbes2016p})
	\begin{align}
		0\to \rr^1\lim_{n\to \infty} H^{q-1}(\fal^\proet_{Y\to X},\falb/p^n\falb)\to H^q(\fal^{\proet,\bb{N}}_{Y\to X},\breve{\falb})\to \lim_{n\to \infty} H^{q}(\fal^\proet_{Y\to X},\falb/p^n\falb)\to 0.
	\end{align}
	The conclusion follows from the almost isomorphisms \eqref{eq:5.5.1}.
\end{proof}

\begin{myprop}\label{prop:acyclic-diff}
	Let $K$ be a pre-perfectoid field of mixed characteristic $(0,p)$, $Y\to X$ a morphism of coherent schemes such that $Y\to X^Y$ is over $\spec(K)\to \spec(\ca{O}_K)$, $A=\falb(Y\to X)$. Assume that $Y\to X$ is Faltings acyclic and that $X=\spec(R)$ with $R$ being $p$-torsion free. Let $\ca{M}$ be an $\ca{O}_X$-module of finite presentation with $M=\ca{M}(X)$ such that $M[1/p]$ is a projective $R[1/p]$-module. Then, the canonical morphism
	\begin{align}
		M\otimes_R \widehat{A}[\frac{1}{p}]\longrightarrow \rr\Gamma(\fal^{\proet,\bb{N}}_{Y\to X},\breve{\sigma}^*\breve{\ca{M}})[\frac{1}{p}]
	\end{align}
	is an isomorphism, where $\breve{\sigma}: (\fal^{\proet,\bb{N}}_{Y\to X},\breve{\falb})\to (X_\proet^{\bb{N}},\breve{\ca{O}}_{X_\proet})$ is defined in {\rm\ref{para:psi-sigma-nu-breve}} and $\breve{\ca{M}}=\ca{M}\otimes_{\ca{O}_{X_\proet}} \breve{\ca{O}}_{X_\proet}=(\ca{M}/p^{n}\ca{M})_{n\geq 0}$.
\end{myprop}
\begin{proof}
	Let $N$ be the kernel of a surjective $R$-linear homomorphism $\phi:R^{\oplus n}\to M$. We take a splitting $R^{\oplus n}[1/p]=M[1/p]\oplus N[1/p]$. Composing with the inclusion $(M\oplus N)\subseteq (M[1/p]\oplus N[1/p])$, we get an injective map $\widetilde{\phi}: M\oplus N\to R^{\oplus n}[1/p]$. Since $M$ and $N$ are finitely generated, there exists an integer $k\geq 1$ such that $\varphi=p^k \widetilde{\phi}:M\oplus N\to R^{\oplus n}\subseteq R^{\oplus n}[1/p]$. It is clear that $\varphi$ is injective. We claim that the cokernel of $\varphi$ is killed by $p^{2k}$. Indeed, the composition of the maps $M\stackrel{\varphi}{\longrightarrow}R^{\oplus n}\stackrel{\phi}{\longrightarrow} M$ is $p^k\id_M$. Thus, for any $x\in R^{\oplus n}$, the element $y=p^kx-\varphi(\phi(x))$ of $R^{\oplus n}$ lies in $N$. Thus, $p^{2k}x=\varphi(p^k\phi(x)+y)$, which proves the claim. Then, there exists an $R$-linear homomorphism $\psi:R^n\to M\oplus N$ such that $\varphi\circ \psi= p^{4k}\id_{R^n}$ and $\psi\circ \varphi=p^{4k}\id_{M\oplus N}$ (\cite[2.6.3]{abbes2020suite}). Let $\ca{N}$ be the quasi-coherent $\ca{O}_X$-module associated to $N$. Then, for each integer $q$, the morphism $\varphi$ induces an $\widehat{A}$-linear homomorphism
	\begin{align}\label{eq:2.5.2}
		H^q(\fal^{\proet,\bb{N}}_{Y\to X},\breve{\sigma}^*\breve{\ca{M}})\oplus H^q(\fal^{\proet,\bb{N}}_{Y\to X},\breve{\sigma}^*\breve{\ca{N}})\longrightarrow H^q(\fal^{\proet,\bb{N}}_{Y\to X},\breve{\falb}^{\oplus n}),
	\end{align} 
	whose kernel and cokernel are killed by $p^{4k}$ by the existence of $\psi$. For $q\neq 0$, we have $H^q(\fal^{\proet,\bb{N}}_{Y\to X},\breve{\falb}^{\oplus n})[1/p]=0$ by \ref{lem:acyclic-falsite}.(\ref{item:lem-acyclic-falsite2}), thus $H^q(\fal^{\proet,\bb{N}}_{Y\to X},\breve{\sigma}^*\breve{\ca{M}})[1/p]=0$. For $q=0$, we have $H^0(\fal^{\proet,\bb{N}}_{Y\to X},\breve{\falb}^{\oplus n})[1/p]=\widehat{A}^{\oplus n}[1/p]$ by \ref{lem:acyclic-falsite}.(\ref{item:lem-acyclic-falsite2}). On the other hand, there is a canonical morphism
	\begin{align}
		M\otimes_R \widehat{A} \oplus N\otimes_R \widehat{A} \to H^0(\fal^{\proet,\bb{N}}_{Y\to X},\breve{\sigma}^*\breve{\ca{M}})\oplus H^0(\fal^{\proet,\bb{N}}_{Y\to X},\breve{\sigma}^*\breve{\ca{N}})
	\end{align}
	whose composition with \eqref{eq:2.5.2} is compatible with $\varphi\otimes_R \id_{\widehat{A}}:M\otimes_R \widehat{A}\oplus N\otimes_R \widehat{A}\to \widehat{A}^{\oplus n}$. Thus, $H^0(\fal^{\proet,\bb{N}}_{Y\to X},\breve{\sigma}^*\ca{M})[1/p]=M\otimes_R \widehat{A}[1/p]$.
\end{proof}

\begin{mylem}\label{lem:rel-norm-proet-map}
	Let $Y \to X$ be a morphism of coherent schemes such that $Y \to X^Y$ is an open immersion. Then, the functor
	\begin{align}\label{eq:epsilon}
		\epsilon^+: \fal_{Y \to X}^\proet \longrightarrow \falh_{Y \to X^Y},\ (V \to U) \longmapsto U^V,
	\end{align}
	is well-defined, left exact and continuous. Moreover, we have $Y\times_{X^Y} U^V=V$.
\end{mylem}
\begin{proof}
	Since $U'=X^Y \times_X U$ is integral over $U$, we have $U^V = U'^V$. Applying \ref{lem:relative-normal}.(\ref{lem:relative-normal-norm}) to $V \to U'$ over $Y \to X^Y$, we see that the $X^Y$-scheme $U^V$ is $Y$-integrally closed with $Y \times_{X^Y} U^V = V$, and thus the functor $\epsilon^+$ is well-defined. Let $(V_1 \to U_1) \to (V_0 \to U_0) \leftarrow (V_2 \to U_2)$ be a diagram in $\fal_{Y \to X}^\proet$. By \ref{lem:relative-normal-limit}, $U_1^{V_1} \overline{\times}_{U_0^{V_0}} U_2^{V_2} = (U_1^{V_1} \times_{U_0^{V_0}} U_2^{V_2})^{V_1\times_{V_0} V_2} = (U_1 \times_{U_0} U_2)^{V_1\times_{V_0} V_2}$ which shows the left exactness of $\epsilon^+$. For the continuity, notice that any covering in $\fal_{Y \to X}^\proet$ can be refined by a standard covanishing covering (\ref{prop:covcov}). For a Cartesian covering family $\ak{U}=\{(V \times_U U_n \to U_n) \to (V \to U)\}_{n \in N}$ with $N$ finite, we apply \ref{prop:normal-v-cov} to the commutative diagram  
	\begin{align}
		\xymatrix{
			\coprod_{n\in N} V \times_U U_n\ar[r]\ar[d]& \coprod_{n\in N}U_n^{V \times_U U_n}\ar[r]\ar[d]& \coprod_{n\in N}U_n\ar[d]\\
			V\ar[r]& U^V\ar[r] & U
		}
	\end{align}
	then we see that $\epsilon^+(\ak{U})$ is a covering family in $\falh_{Y \to X^Y}$. For a vertical covering family $\ak{U}=\{(V_{m} \to U) \to (V \to U) \}_{m \in M}$ with $M$ finite, we apply \ref{prop:normal-v-cov} to the commutative diagram  
	\begin{align}
		\xymatrix{
			\coprod_{m\in M} V_m\ar[r]\ar[d]& \coprod_{m\in M}U^{V_m}\ar[r]\ar[d]& U\ar[d]\\
			V\ar[r]& U^V\ar[r] & U
		}
	\end{align}
	then we see that $\epsilon^+(\ak{U})$ is also a covering family in $\falh_{Y \to X^Y}$.
\end{proof}

\begin{mypara}\label{par:epsilon-mor}
	Let $Y \to X$ be a morphism of coherent schemes such that $Y \to X^Y$ is an open immersion. Then, there are morphisms of sites
	\begin{align}
		\epsilon : \falh_{Y \to X^Y} &\longrightarrow \fal_{Y \to X}^\proet,\\
		\varepsilon : \falh_{Y \to X^Y} &\longrightarrow \fal_{Y \to X}^\et
	\end{align}
	defined by \eqref{eq:epsilon} and the composition of \eqref{eq:epsilon} with \eqref{eq:3.6.1} respectively. We temporarily denote by $\falhb^{\mrm{pre}}$ the presheaf on $\falh_{Y \to X^Y}$ sending $W$ to $\Gamma(W,\ca{O}_W)$ (thus $\falhb=(\falhb^{\mrm{pre}})^\ash$). Notice that we have $\falb=\epsilon^\psh\falhb^{\mrm{pre}}$ (resp. $\falb=\varepsilon^\psh\falhb^{\mrm{pre}}$). The canonical morphism $\epsilon^\psh\falhb^{\mrm{pre}}\to \epsilon^\psh\falhb$ (resp. $\varepsilon^\psh\falhb^{\mrm{pre}}\to \varepsilon^\psh\falhb$) induces a canonical morphism $\falb \to \epsilon_*\falhb$ (resp. $\falb \to \varepsilon_*\falhb$).
\end{mypara}

\begin{mypara}\label{para:notation}
	Let $K$ be a pre-perfectoid field (\ref{defn:perfectoid-field}) of mixed characteristic $(0, p)$, $\eta = \spec(K)$, $S = \spec (\ca{O}_K)$, $Y \to X$ a morphism of coherent schemes such that $X^Y$ is an $S$-scheme with generic fibre $(X^Y)_\eta=Y$. In particular, $X^Y$ is an object of $\falh_{\eta\to S}$.
\end{mypara}

\begin{mylem}\label{lem:profet-tower}
	For any ring $R$, there is an $R$-algebra $R_\infty$ satisfying the following conditions:
	\begin{enumerate}
		\item[\rm(i)] The scheme $\spec(R_\infty[1/p])$ is pro-finite \'etale and  faithfully flat over $\spec(R[1/p])$.
		\item[\rm(ii)] The $R$-algebra $R_\infty$ is the integral closure of $R$ in $R_\infty[1/p]$.
		\item[\rm(iii)] Any unit $t$ of $R_\infty$ admits a $p$-th root $t^{1/p}$ in $R_\infty$.
	\end{enumerate}
	Moreover, if $p$ lies in the Jacobson radical $J(R)$ of $R$, and if there is a $p^2$-th root $p_2 \in R$ of $p$ up to a unit, and we write $p_1=p_2^p$, then we may require further that
	\begin{enumerate}
		\item[\rm(iv)] the Frobenius of $R_\infty/p R_\infty$ induces an isomorphism $R_\infty/p_1R_\infty \to R_\infty/p R_\infty$.
	\end{enumerate} 
\end{mylem}
\begin{proof}
	Setting $B_0 = R[1/p]$, we construct inductively a ring $B_{n+1}$ ind-finite \'etale over $B_n$ and we denote by $R_{n}$ the integral closure of $R$ in $B_{n}$. For $n\geq 0$, we set
	\begin{align}
		B_{n+1} = \colim_{T\subseteq R_n^\times} \bigotimes\nolimits_{B_n}^{t \in T} B_n[X]/(X^p-t)
	\end{align}
	where the colimit runs through all finite subsets $T$ of the subset $R_n^\times$ of units of $R_n$ and the transition maps are given by the inclusion relation of these finite subsets $T$. Notice that $B_n[X]/(X^p-t)$ is finite \'etale and faithfully flat over $B_n$, thus $B_{n+1}$ is ind-finite \'etale and faithfully flat over $B_n$. Now we take $B_\infty = \colim_n B_n$. The integral closure $R_\infty$ of $R$ in $B_\infty$ is equal to $\colim_n R_n$ by \ref{lem:rel-norm-fil-limit}. We claim that $R_\infty$ satisfies the first three conditions. Firstly, we see inductively that $B_n = R_n[1/p]$ ($0\leq n\leq \infty$) by \ref{lem:rel-norm-commute}. Thus, (i), (ii) follow immediately. For (iii), notice that we have $R_\infty^\times = \colim_n R_n^\times$. For an unit $t\in R_\infty^\times$, we suppose that it is the image of $t_n \in R_n^\times$. By construction, there exists an element $x_{n+1} \in R_{n+1}$ such that $x_{n+1}^p=t_n$. Thus, $t$ admits a $p$-th root in $R_\infty$.
	
	For (iv), the injectivity follows from the fact that $R_\infty$ is integrally closed in $R_\infty[1/p]$ (cf. \ref{lem:p-clos}). For the surjectivity, let $a \in R_\infty$. Firstly, since $R_\infty$ is integral over $R$, $p$ also lies in the Jacobson radical $J(R_\infty)$ of $R_\infty$. Thus, $1+p_1a \in R_\infty^\times$, and then by (iii) there is $b \in R_\infty$ such that $b^p = 1+p_1a$. We write $(b-1)^p=p_1a'$ for some $a' \in a+ p_1R_\infty$. Thus, $1+a'-a \in R_\infty^\times$, and then by (iii) there is $c \in R_\infty$ such that $c^p = 1+a'-a$. On the other hand, since $R_\infty$ is integrally closed in $R_\infty[1/p]$, we have $x= (b-1)/p_2 \in R_\infty$. Now we have $(x-c+1)^p\equiv x^p-c^p+1 \equiv a\ (\trm{mod } pR_\infty)$, which completes the proof. 
\end{proof}
\begin{myrem}
	In \ref{lem:profet-tower}, it follows from the construction that $\spec(R_\infty[1/p])\to \spec(R[1/p])$ is a covering in $\spec(R[1/p])_\profet^{\mrm{S}}$ (\ref{para:compare}).
\end{myrem}

\begin{myprop}\label{prop:pre-perf-base}
	With the notation in {\rm\ref{para:notation}}, for any object $V \to U$ in $\fal_{Y \to X}^\proet$, there exists a covering $\{(V_i \to U_i) \to (V \to U)\}_{i \in I}$ with $I$ finite such that for each $i \in I$, $U_i^{V_i}$ is the spectrum of an $\ca{O}_K$-algebra which is almost pre-perfectoid {\rm(\ref{defn:pre-alg})}.
\end{myprop}
\begin{proof}	
	After replacing $U$ by an affine open covering, we may assume that $U=\spec(A)$. Consider the category $\scr{C}$ of \'etale $A$-algebras $B$ such that $A/pA \to B/pB$ is an isomorphism, and the colimit $A^\mrm{h} = \colim B$ over $\scr{C}$. In fact, $\scr{C}$ is filtered and $(A^\mrm{h},pA^\mrm{h})$ is the Henselization of the pair $(A,pA)$ (cf. \cite[\href{https://stacks.math.columbia.edu/tag/0A02}{0A02}]{stacks-project}). It is clear that $\spec(A^\mrm{h}) \coprod \spec(A[1/p]) \to \spec(A)$ is a covering in $U_\proet$. So we reduce to the situation where $p \in J(A)$ or $p \in A^\times$. The latter case is trivial, since the $p$-adic completion of $R= \Gamma(U^V,\ca{O}_{U^V})$ is zero if $p$ is invertible in $A$. Therefore, we may assume that $p \in J(A)$ in the following.
	
	Since $R= \Gamma(U^V,\ca{O}_{U^V})$ is integral over $A$, we also have $p \in J(R)$. Applying \ref{lem:profet-tower} to the $\ca{O}_K$-algebra $R$, we obtain a covering $V_\infty=\spec (R_\infty[1/p]) \to V=\spec (R[1/p])$ in $V_\profet$ such that $R_\infty = \Gamma(U^{V_\infty},\ca{O}_{U^{V_\infty}})$ is an $\ca{O}_K$-algebra which is almost pre-perfectoid by \ref{lem:adm-uniformizer} and \ref{lem:completion-flat}.
\end{proof}

\begin{myprop}\label{prop:coh-vanish-preperf}
	With the notation in {\rm\ref{para:notation}}, if $W$ is an object of $\falh_{\eta \to S}$ such that $W$ is the spectrum of an $\ca{O}_K$-algebra which is almost pre-perfectoid, then for any integer $n>0$, the canonical morphism
	\begin{align}
		\Gamma(W,\ca{O}_W)/p^n\Gamma(W,\ca{O}_W) \to \rr\Gamma(\falh_{W_\eta\to W}, \falhb/p^n\falhb)
	\end{align}
	is an almost isomorphism {\rm(\ref{para:almost-derived})}.
\end{myprop}
\begin{proof}
	Let $\scr{C}$ be the full-subcategory of $\falh_{\eta \to S}$ formed by the spectrums of $\ca{O}_K$-algebras which are almost pre-perfectoid. It is stable under fibred product by \ref{cor:tensor-pre-perf}, \ref{cor:pre-perf-intclos} and \ref{lem:relative-normal-limit}, and it forms a topologically generating family for the site $\falh_{\eta \to S}$ by \ref{lem:rel-norm-proet-map} and \ref{prop:pre-perf-base}. It suffices to show that for any covering in $\falh_{\eta \to S}$, $\ak{U}=\{W_k \to W\}_{k \in K}$ consisting of objects of $\scr{C}$ with $K$ finite, the augmented \v Cech complex associated to the presheaf $W\mapsto \Gamma(W,\ca{O}_W)/p^n\Gamma(W,\ca{O}_W)$ on $\falh_{\eta \to S}$ (whose associated sheaf is just $\falhb/p^n\falhb$),
	\begin{align}\label{eq:4.28.1}
		0 \to \Gamma(W,\ca{O}_W)/p^n \to \prod_k \Gamma(W_k,\ca{O}_{W_k})/p^n \to \prod_{k,k'} \Gamma(W_k \overline{\times}_W W_{k'},\ca{O}_{W_k \overline{\times}_W W_{k'}})/p^n \to \cdots
	\end{align}
	is almost exact. Indeed, the almost exactness shows firstly that $\Gamma(W,\ca{O}_W)/p^n \to H^0(\falh_{W_\eta\to W}, \falhb/p^n\falhb)$ is an almost isomorphism (cf. \cite[\href{https://stacks.math.columbia.edu/tag/00W1}{00W1}]{stacks-project}), so that the augmented \v Cech complex associated to the sheaf $\falhb/p^n\falhb$ is also almost exact. Then, the conclusion follows from the almost vanishing of the higher \v Cech cohomologies of $\falhb/p^n\falhb$ by \cite[\href{https://stacks.math.columbia.edu/tag/03F9}{03F9}]{stacks-project}. 
	
	We set $R=\Gamma(W,\ca{O}_W)$ and $R'=\prod_{k \in K} \Gamma(W_k,\ca{O}_{W_k})$. They are almost pre-perfectoid, and $\spec (R') \to \spec (R)$ is a v-covering by definition. Thus, the almost exactness of \eqref{eq:4.28.1} follows from \ref{cor:tensor-pre-perf}, \ref{cor:pre-perf-intclos} and \ref{thm:arc-descent-perf}.
\end{proof}

\begin{mythm}\label{thm:falh-proet}
	With the notation in {\rm\ref{para:notation}}, let $\epsilon : \falh_{Y \to X^Y} \to \fal_{Y \to X}^\proet$ be the morphism of sites defined in {\rm\ref{par:epsilon-mor}}. Then, for any integer $n>0$, the canonical morphism
	\begin{align}
		\falb/p^n\falb \to \rr\epsilon_*(\falhb/p^n\falhb)
	\end{align}
	is an almost isomorphism in the derived category $\dd (\ca{O}_K\module_{\fal_{Y \to X}^\proet})$ {\rm(\ref{para:almost-derived})}.
\end{mythm}
\begin{proof}
	Since $\rr^q\epsilon_*(\falhb/p^n\falhb)$ is the sheaf associated to the presheaf $(V \to U) \mapsto H^q(\falh_{V \to U^V}, \falhb/p^n\falhb)$ and any object in $\fal_{Y \to X}^\proet$ can be covered by those objects whose image under $\epsilon^+$ are the spectrums of $\ca{O}_K$-algebras which are almost pre-perfectoid by \ref{prop:pre-perf-base}, the conclusion follows from \ref{prop:coh-vanish-preperf}. 
\end{proof}

\begin{mycor}\label{prop:fal-coh}
	With the notation in {\rm\ref{para:notation}}, let $V\to U$ be an object of $\fal_{Y\to X}^\proet$ such that $U$ is affine and that the integral closure $U^V=\spec(A)$ is the spectrum of an $\ca{O}_K$-algebra $A$ which is almost pre-perfectoid. Then, $V\to U$ is Faltings acyclic.
\end{mycor}
\begin{proof}
	It follows directly from \ref{thm:falh-proet} and \ref{prop:coh-vanish-preperf}. 
\end{proof}

\begin{mycor}\label{cor:falh-et}
	With the notation in {\rm\ref{para:notation}}, let $\varepsilon: \falh_{Y \to X^Y} \to \fal_{Y \to X}^\et$ be the morphism of sites defined in {\rm\ref{par:epsilon-mor}}. Then, for any finite locally constant abelian sheaf $\bb{L}$ on $\fal_{Y \to X}^\et$, the canonical morphism 
	\begin{align}
		\bb{L}\otimes_{\bb{Z}} \falb \to \rr\varepsilon_*(\varepsilon^{-1}\bb{L}\otimes_{\bb{Z}}\falhb)
	\end{align}
	is an almost isomorphism in the derived category $\dd (\ca{O}_K\module_{\fal_{Y \to X}^\et})$ {\rm(\ref{para:almost-derived})}.
\end{mycor}
\begin{proof}
	The problem is local on $\fal_{Y \to X}^\et$, thus we may assume that $\bb{L}$ is the constant sheaf with value $\bb{Z}/p^n\bb{Z}$. Then, the conclusion follows from \ref{thm:falh-proet} and \ref{cor:nu-falb}.
\end{proof}

\begin{myrem}\label{rem:falh-et}
	In \ref{cor:falh-et}, if $\bb{L}$ is a bounded complex of abelian sheaves on $\fal_{Y \to X}^\et$ with finite locally constant cohomology sheaves, then the canonical morphism $\bb{L}\otimes_{\bb{Z}}^\dl \falb \to \rr\varepsilon_*(\varepsilon^{-1}\bb{L}\otimes_{\bb{Z}}^\dl\falhb)$ is also an almost isomorphism. Indeed, after replacing $\bb{L}$ by $\bb{L}\otimes_{\bb{Z}}^\dl \bb{Z}_p$, we may assume that $\bb{L}$ is a complex of $\bb{Z}/p^n\bb{Z}$-modules for some integer $n$ (\cite[\href{https://stacks.math.columbia.edu/tag/0DD7}{0DD7}]{stacks-project}). Then, there exists a covering family $\{(Y_i\to X_i)\to (Y\to X)\}_{i \in I}$ in $\fal_{Y \to X}^\et$ such that the restriction of $\bb{L}$ on $\fal_{Y_i \to X_i}^\et$ is represented by a bounded complex of finite locally constant $\bb{Z}/p^n\bb{Z}$-modules (\cite[\href{https://stacks.math.columbia.edu/tag/094G}{094G}]{stacks-project}). Then, the conclusion follows directly from \ref{cor:falh-et}.
\end{myrem}

\begin{mycor}\label{cor:modelless}
	With the notation in {\rm\ref{para:notation}}, let $Y \to X_i$ ($i=1,2$) be a morphism of coherent schemes such that $X_i^Y$ is an $S$-scheme with generic fibre $(X_i^Y)_\eta=Y$, $\bb{L}$ a finite locally constant abelian sheaf on $\fal_{Y\to X_2}^\et$. If there is a morphism $f:X_1 \to X_2$ under $Y$ such that the natural morphism $g:X_1^Y \to X_2^Y$ is a separated v-covering and that $g^{-1}(Y)=Y$, and if we denote by $u:\fal_{Y \to X_1}^\et \to \fal_{Y\to X_2}^\et$ the corresponding morphism of sites, then the natural morphism
	\begin{align}
		\bb{L}\otimes_{\bb{Z}}\falb \to \rr u_*(u^{-1}\bb{L}\otimes_{\bb{Z}}\falb)
	\end{align}
	is an almost isomorphism.
\end{mycor}
\begin{proof}
	The morphism $u$ is defined by the functor $u^+ : \fal_{Y\to X_2}^\et \to \fal_{Y \to X_1}^\et$ sending $(V\to U_2)$ to $(V\to U_1)=(V\to X_1 \times_{X_2} U_2)$. We set $V_0 = Y\times_{X_1} U_{1} = Y\times_{X_2} U_{2}$. According to \ref{lem:rel-norm-commute}, $U_1^{V_0} \to U_2^{V_0}$ is the base change of $X_1^Y \to X_2^Y$ by $U_2 \to X_2$, and thus it is a separated v-covering. Notice that $V_0$ is an open subscheme in both $U_1^{V_0}$ and $U_2^{V_0}$, and moreover $V_0=V_0\times_{U_2^{V_0}}U_1^{V_0}$. Applying \ref{prop:normal-v-cov} to the commutative diagram
	\begin{align}
		\xymatrix{
			V \ar[r] \ar[d] & U_1^V \ar[r] \ar[d] & U_1^{V_0}\ar[d]\\
			V \ar[r] & U_2^V \ar[r] & U_2^{V_0}
		}
	\end{align}
	it follows that $U_1^V \to U_2^V$ is also a separated v-covering. Let $\varepsilon_i: \falh_{Y \to X_i^{Y}} \to \fal_{Y \to X_i}^\et$ ($i=1,2$) be the morphisms of sites defined in \ref{par:epsilon-mor}. The sheaf $\rr^q u_*(u^{-1}\bb{L}\otimes_{\bb{Z}}\falb)$ is associated to the presheaf $(V \to U_2) \mapsto H^q(\fal_{V \to U_1}^\et, u^{-1}\bb{L}\otimes_{\bb{Z}}\falb)$. We have
	\begin{align}
		H^q(\fal_{V \to U_1}^\et, u^{-1}\bb{L}\otimes_{\bb{Z}}\falb)&\to H^q(\falh_{V \to U_1^V}, \varepsilon_1^{-1}u^{-1}\bb{L}\otimes_{\bb{Z}}\falhb)\\
		&= H^q(\falh_{V \to U_2^V}, \varepsilon_2^{-1}\bb{L}\otimes_{\bb{Z}}\falhb)\leftarrow H^q(\fal_{V \to U_2}^\et, \bb{L}\otimes_{\bb{Z}}\falb),\nonumber
	\end{align}
	where the equality follows from the fact that the morphism of representable sheaves associated to $U_1^V \to U_2^V$ on $\falh_{\eta \to S}$ is an isomorphism by \ref{prop:sheaf-isom}, and where the two arrows are almost isomorphisms by \ref{cor:falh-et}, which completes the proof.
\end{proof}

\begin{mypara}\label{para:hypercov}
	Let $\Delta$ be the category formed by finite ordered sets $[n]=\{0,1,\dots,n\}$ ($n\geq 0$) with non-decreasing maps (\cite[\href{https://stacks.math.columbia.edu/tag/0164}{0164}]{stacks-project}). For a functor from its opposite category $\Delta^{\oppo}$ to the category $\fal$ of morphisms of coherent schemes sending $[n]$ to $Y_n \to X_n$, we simply denote it by $Y_\bullet \to X_\bullet$. Then, we obtain a fibred site $\fal_{Y_\bullet \to X_\bullet}^\et$ over $\Delta^{\oppo}$ whose fibre category over $[n]$ is $\fal_{Y_n\to X_n}^\et$ and the inverse image functor $f^+:\fal_{Y_n\to X_n}^\et\to \fal_{Y_m\to X_m}^\et$ associated to a morphism $f:[m] \to [n]$ in $\Delta^{\oppo}$ is induced the base change by the morphism $(Y_m \to X_m)\to (Y_n\to X_n)$ associated to $f$. We endow $\fal_{Y_\bullet \to X_\bullet}^\et$ with the total topology (\ref{par:cova-def}) and call it the \emph{simplicial Faltings site} associated to $Y_\bullet \to X_\bullet$ (\cite[\href{https://stacks.math.columbia.edu/tag/09WE}{09WE.(A)}]{stacks-project}). The sheaf $\falb$ on each $\fal_{Y_n \to X_n}^\et$ induces a sheaf $\falb_{\bullet}=(\falb)_{[n]\in\ob(\Delta)}$ on $\fal_{Y_\bullet \to X_\bullet}^\et$ with the notation in \ref{para:presheafonE}.
	
	For an augmentation $(Y_\bullet\to X_\bullet)\to (Y\to X)$ in $\fal$ (\cite[\href{https://stacks.math.columbia.edu/tag/018F}{018F}]{stacks-project}), we obtain an augmentation of simplicial site $a:\fal_{Y_\bullet \to X_\bullet}^\et\to \fal_{Y\to X}^\et$ (\cite[\href{https://stacks.math.columbia.edu/tag/0D6Z}{0D6Z.(A)}]{stacks-project}). We denote by $a_n: \fal_{Y_n \to X_n}^\et\to \fal_{Y\to X}^\et$ the natural morphism induced by $(Y_n\to X_n)\to (Y\to X)$. Notice that for any sheaf $\ca{F}$ on $\fal_{Y\to X}^\et$, we have $a^{-1}\ca{F}=(a_n^{-1}\ca{F})_{[n]\in\ob(\Delta)}$ with the notation in \ref{para:presheafonE} (\cite[\href{https://stacks.math.columbia.edu/tag/0D70}{0D70}]{stacks-project}).
\end{mypara}

\begin{mycor}\label{cor:resolution}
	With the notation in {\rm\ref{para:notation}}, let $\bb{L}$ a finite locally constant abelian sheaf on $\fal_{Y \to X}^\et$, $X_\bullet \to X$ an augmentation of simplicial coherent scheme. If we set $Y_\bullet=Y \times_X X_\bullet$ and denote by $a: \fal_{Y_\bullet \to X_\bullet}^\et \to \fal_{Y \to X}^\et$ the augmentation of simplicial site, assuming that $X_\bullet^{Y_\bullet} \to X^Y$ is a hypercovering in $\falh_{\eta \to S}$, then the canonical morphism 
	\begin{align}
		\bb{L}\otimes_{\bb{Z}} \falb \to \rr a_*(a^{-1}\bb{L}\otimes_{\bb{Z}}\falb_\bullet)
	\end{align}
	is an almost isomorphism.
\end{mycor}
\begin{proof}
	Let $b: \falh_{Y_\bullet \to X_\bullet^{Y_\bullet}} \to \falh_{Y \to X^Y}$ be the augmentation of simplicial site associated to the augmentation of simplicial object $X_\bullet^{Y_\bullet} \to X^Y$ of $\falh_{\eta \to S}$ (\cite[\href{https://stacks.math.columbia.edu/tag/09X8}{09X8}]{stacks-project}). The functorial morphism of sites $\varepsilon: \falh_{Y \to X^Y} \to \fal_{Y \to X}^\et$ defined in {\rm\ref{par:epsilon-mor}} induces a commutative diagram of topoi (\cite[\href{https://stacks.math.columbia.edu/tag/0D99}{0D99}]{stacks-project})
	\begin{align}
		\xymatrix{
			\falh_{Y_\bullet \to X_\bullet^{Y_\bullet}}^\sim\ar[r]^-{\varepsilon_\bullet}\ar[d]_-{b}&\fal_{Y_\bullet \to X_\bullet}^{\et\sim}\ar[d]^-{a}\\
			\falh_{Y \to X^Y}^\sim \ar[r]^-{\varepsilon} & \fal_{Y \to X}^{\et\sim}
		}
	\end{align}
	We denote by $a_n: \fal_{Y_n \to X_n}^\et \to \fal_{Y \to X}^\et$ and $b_n : \falh_{Y_n \to X_n^{Y_n}} \to \falh_{Y \to X^Y}$ the natural morphisms of sites. Consider the commutative diagram
	\begin{align}
		\xymatrix{
			\rr a_* (a^{-1}\bb{L} \otimes \falb_\bullet)\ar[d]_-{\alpha_3}&\bb{L}\otimes \falb \ar[r]^-{\alpha_2}\ar[l]_-{\alpha_1}& \rr\varepsilon_*(\varepsilon^{-1}\bb{L} \otimes \falhb)\ar[d]^-{\alpha_4}\\
			\rr a_* \rr\varepsilon_{\bullet*}\varepsilon_{\bullet}^{-1}(a^{-1}\bb{L} \otimes \falb_\bullet)\ar[r]^-{\alpha_5}& \rr c_*(c^{-1}\bb{L}\otimes \falhb_\bullet) &\rr\varepsilon_*\rr b_* b^{-1}(\varepsilon^{-1}\bb{L} \otimes \falhb)\ar@{=}[l]
		}
	\end{align}
	where $c=a\circ\varepsilon_{\bullet}=\varepsilon\circ b$, and $\alpha_2$ (resp. $\alpha_5$) is induced by the canonical morphism $\varepsilon^{-1}\falb\to \falhb$ (resp. $\varepsilon_{\bullet}^{-1}\falb_\bullet\to \falhb_\bullet$), and other arrows are the canonical morphisms.
	
	Notice that $\alpha_2$ is an almost isomorphism by \ref{cor:falh-et}, and that $\alpha_4$ is an isomorphism by \cite[\href{https://stacks.math.columbia.edu/tag/0D8N}{0D8N}]{stacks-project} as $X_\bullet^{Y_\bullet} \to X^Y$ is a hypercovering in $\falh_{\eta \to S}$. It remains to show that $\alpha_5\circ\alpha_3$ is an almost isomorphism. With the notation in \ref{para:presheafonE}, we have
	\begin{align}
		a^{-1}\bb{L} \otimes \falb_\bullet&=(a_n^{-1}\bb{L} \otimes \falb)_{[n]\in\ob(\Delta)} \trm{ and }\\
		c^{-1}\bb{L} \otimes \falhb_\bullet&=( \varepsilon_{n}^{-1}a_n^{-1}\bb{L} \otimes \falhb)_{[n]\in\ob(\Delta)}.
	\end{align}
	Moreover, by \cite[\href{https://stacks.math.columbia.edu/tag/0D97}{0D97}]{stacks-project} we have
	\begin{align}
		\rr^q\varepsilon_{\bullet*} (c^{-1}\bb{L} \otimes \falhb_\bullet)=(\rr^q \varepsilon_{n*}(\varepsilon_{n}^{-1}a_n^{-1}\bb{L} \otimes \falhb))_{[n]\in\ob(\Delta)}
	\end{align}
	for each integer $q$. Therefore, $a^{-1}\bb{L} \otimes \falb_\bullet\to \rr\varepsilon_{\bullet*} (c^{-1}\bb{L} \otimes \falhb_\bullet)$ is an almost isomorphism by \ref{cor:falh-et} and so is $\alpha_5\circ\alpha_3$. 
\end{proof}

\begin{mylem}\label{lem:pre-perf-falsite}
	With the notation in {\rm\ref{para:notation}}, assume that $X^Y$ is the spectrum of an $\ca{O}_K$-algebra which is almost pre-perfectoid. Let $V\to U$ be an object of $\fal_{Y\to X}^\proet$ with $U$ affine. Then, $U^V$ is the spectrum of an $\ca{O}_K$-algebra which is almost pre-perfectoid.
\end{mylem}
\begin{proof}
	Consider the following commutative diagram:
	\begin{align}\label{diam:5.5.1}
		\xymatrix@R=16pt{
			V\ar[r]\ar[d]&U^V\ar[d]&\\
			Y\times_X U\ar[r]\ar[d]&U^{Y\times_X U}\ar[r]\ar[d]&U\ar[d]\\
			Y\ar[r]\ar[d]&X^Y\ar[r]\ar[d]&X\\
			\spec(K)\ar[r]&\spec(\ca{O}_K)&
		}
	\end{align}
	Since taking integral closures commutes with \'etale base change and filtered colimits (\ref{lem:rel-norm-commute}, \ref{lem:rel-norm-fil-limit}), all the squares in \eqref{diam:5.5.1} are Cartesian (\ref{lem:relative-normal}). Notice that $U^{Y\times_X U}$ is integral over $U$ and thus affine. Since $U^{Y\times_X U}$ is pro-\'etale over $X^Y$, it is the spectrum of an $\ca{O}_K$-algebra which is almost pre-perfectoid by \ref{lem:pre-perf-alg-basic}. As $V$ is pro-finite \'etale over $Y\times_X U$, by almost purity \ref{thm:almost-purity} and \ref{lem:pre-perf-alg-basic}, we see that $U^V$ is the spectrum of an $\ca{O}_K$-algebra which is almost pre-perfectoid.
\end{proof}

\begin{mypara}\label{para:notation-open}
	Let $K$ be a pre-perfectoid field of mixed characteristic $(0, p)$ which contains a compatible system $\{\zeta_n\}_{n\geq 1}$ of primitive roots of unity, $\eta = \spec(K)$, $S = \spec (\ca{O}_K)$, $Y \to X$ a morphism of coherent schemes such that $X^Y$ is an $S$-scheme and that the induced morphism $Y\to X^Y$ is an open immersion over $\eta\to S$. Remark that the morphism $X^Y_\eta\to X$ over $\eta\to S$ is in the situation \ref{para:notation}. We assume further that there exist finitely many nonzero divisors $f_1,\dots, f_r$ of $\Gamma(X^Y_\eta,\ca{O}_{X^Y_\eta})$ such that the divisor $D=\sum_{i=1}^r \mrm{div}(f_i)$ on $X^Y_\eta$ has support $X^Y_\eta\setminus Y$ and that at each strict henselization of $X^Y_\eta$ those elements $f_i$ contained in the maximal ideal form a subset of a regular system of parameters (in particular, $D$ is a normal crossings divisor on $X^Y_\eta$, and we allow $D$ to be empty, i.e. $r=0$). We set
	\begin{align}
		Y_{\infty}&=\lim_{n} Y[T_1,\dots,T_r]/(T_1^{n}-f_1,\dots,T_r^{n}-f_r),
	\end{align}
	where the limit is taken over $\bb{N}$ with the division relation. We see that $Y_\infty$ is faithfully flat and pro-finite \'etale over $Y$.
\end{mypara}

\begin{myprop}[Abhyankar's lemma]\label{prop:abhyankar}
	Under the assumptions in {\rm\ref{para:notation-open}} and with the same notation, for any finite \'etale $Y_\infty$-scheme $V_\infty$, the integral closure $X^{V_\infty}_\eta$ is finite \'etale over $X^{Y_\infty}_\eta$.
\end{myprop}
\begin{proof}
	We set $Z=X^Y_\eta$. Passing to a strict henselization of $Z$ where $D$ is non-empty, we may assume that $Z$ is local and regular and that $f_1,\dots, f_r$ ($r\geq 1$) are all contained in the maximal ideal. We set $Y_n=Y[T_1,\dots,T_r]/(T_1^{n}-f_1,\dots,T_r^{n}-f_r)$ and $Z_n=Z[T_1,\dots,T_r]/(T_1^{n}-f_1,\dots,T_r^{n}-f_r)$ for any integer $n>0$. Notice that $Z_{n}$ is still local and regular (thus isomorphic to $X^{Y_{n}}_\eta$) and that $g_0=f_0^{1/n},\dots, g_r=f_r^{1/n}$ form a subset of a regular system of parameters for $Z_{n}$ (see the proof of \cite[\Luoma{13}.5.1]{sga1}). Using \cite[8.8.2, 8.10.5]{ega4-3} and \cite[17.7.8]{ega4-4}, there exists an integer $n_0>0$ and a finite \'etale $Y_{n_0}$-scheme $V_{n_0}$ such that $V_\infty=Y_\infty\times_{Y_{n_0}} V_{n_0}$. We set $V_n=Y_n\times_{Y_{n_0}} V_{n_0}$ for any $n\geq n_0$. According to \cite[\Luoma{13}.5.2]{sga4-3}, there exists a multiple $n_1$ of $n_0$ such that $Z_{n_1}^{V_{n_1}}$ is finite \'etale over $Z_{n_1}$. As $Z_\infty=\lim Z_n$ is normal, $Z_\infty \times_{Z_{n_1}}Z_{n_1}^{V_{n_1}}$ is also normal and thus isomorphic to $Z_\infty^{V_\infty}=X^{V_\infty}_\eta$, which shows that the latter is also finite \'etale over $Z_\infty=X^{Y_\infty}_\eta$.
\end{proof}

\begin{mycor}\label{cor:abhyankar}
	Under the assumptions in {\rm\ref{para:notation-open}} and with the same notation, the natural functor $\fal^\proet_{X^{Y_\infty}_\eta\to X}\to \fal^\proet_{Y_\infty\to X}$ sending $V\to U$ to $Y_\infty\times_{X^{Y_\infty}_\eta}V\to U$ induces an equivalence of ringed sites $(\fal^\proet_{Y_\infty\to X},\falb)\to (\fal^\proet_{X^{Y_\infty}_\eta\to X},\falb)$.
\end{mycor}
\begin{proof}
	For the equivalence of categories, it suffices to show that the induced functor $u^+:\fal^\et_{X^{Y_\infty}_\eta\to X}\to \fal^\et_{Y_\infty\to X}$ is an equivalence by \ref{lem:proet-profet}.(\ref{lem:proet-profet-mor}). Since
	$u^+$ is a morphism of fibred categories over $X_\et$, it suffices to show that for each object $U$ of $X_\et$, the fibre functor $u^+_U: U^{Y_\infty\times_XU}_{\eta,\fet}\to (Y_\infty\times_XU)_{\fet}$ induced by $u^+$ is an equivalence of categories. Notice that if we replace $Y\to X$ in \ref{para:notation-open} by $Y\times_XU\to U$, then $(Y\times_XU)_\infty=Y_\infty\times_XU$. Therefore, the equivalence of categories follows from applying \ref{prop:abhyankar} to $Y\times_XU\to U$. 
	
	To show the equivalence of categories identifies their topologies, it suffices to show that it identifies the vertical coverings and Cartesian coverings in \ref{para:proet-falsite}. For a finite family $\{(V_m \to U) \to (V \to U)\}_{m \in M}$ in $\fal^\proet_{X^{Y_\infty}_\eta\to X}$, its image in $\fal^\proet_{Y_\infty\to X}$ is $\{(Y_\infty\times_{X^{Y_\infty}_\eta}V_m \to U) \to ((Y_\infty\times_{X^{Y_\infty}_\eta}V \to U)\}_{m \in M}$. Notice that $Y_\infty\times_{X^{Y_\infty}_\eta}V$ is a dense open subset of $V$ as $V$ is flat over $X^{Y_\infty}_\eta$ (\cite[2.3.7]{ega4-2}), and the same holds for $V_m$. Thus, the integral morphism $\coprod_{m \in M} V_m\to V$ is surjective if and only if $\coprod_{m \in M} Y_\infty\times_{X^{Y_\infty}_\eta}V_m\to Y_\infty\times_{X^{Y_\infty}_\eta}V$ is surjective. On the other hand, it is tautological that the equivalence identifies the Cartesian coverings. Hence, the two sites are naturally equivalent.
	
	The identification of the structural sheaves by the equivalence of sites follows from the fact that $V$ is integrally closed in $Y_\infty\times_{X^{Y_\infty}_\eta}V$ for any object $V\to U$ of $\fal^\proet_{X^{Y_\infty}_\eta\to X}$ as $V$ is pro-\'etale over $X^{Y_\infty}_\eta$ (\ref{lem:relative-normal}).
\end{proof}

\begin{mycor}\label{cor:pre-perf-falsite}
	Under the assumptions in {\rm\ref{para:notation-open}} and with the same notation, let $V\to U$ be an object of $\fal_{Y_{\infty}\to X}^\proet$ such that $U^V$ is the spectrum of an $\ca{O}_K$-algebra which is almost pre-perfectoid, and let $V'\to U'$ be an object of $\fal_{V\to U}^\proet$ with $U'$ affine. Then, $U'^{V'}$ is the spectrum of an $\ca{O}_K$-algebra which is almost pre-perfectoid.
\end{mycor}
\begin{proof}
	It follows directly from \ref{cor:abhyankar} and \ref{lem:pre-perf-falsite}.
\end{proof}

\begin{mythm}\label{thm:acyclic}
	Under the assumptions in {\rm\ref{para:notation-open}} and with the same notation, let $V\to U$ be an object of $\fal_{Y_{\infty}\to X}^\proet$. Then, the following statements are equivalent:
	\begin{enumerate}
		\renewcommand{\labelenumi}{{\rm(\theenumi)}}
		\item The morphism $V\to U$ is Faltings acyclic.\label{item:thm-acyclic1}
		\item The scheme $U$ is affine and $U^V=\spec(A)$ is the spectrum of an $\ca{O}_K$-algebra $A$ which is almost pre-perfectoid.\label{item:thm-acyclic2}
	\end{enumerate} 
\end{mythm}
\begin{proof}
	(\ref{item:thm-acyclic2}) $\Rightarrow$ (\ref{item:thm-acyclic1}): Let $V'\to U$ be an object of $\fal^\proet_{X^{Y_\infty}_\eta\to X}$ whose image under the equivalence in \ref{cor:abhyankar} is isomorphic to $V\to U$. Then, $U^{V'}=\spec(A)$, $V'=U^{V'}_\eta$, and $\rr\Gamma(\fal^\proet_{V\to U},\falb/p^n\falb)=\rr\Gamma(\fal^\proet_{V'\to U},\falb/p^n\falb)$. The conclusion follows from \ref{prop:fal-coh}.
	
	(\ref{item:thm-acyclic1}) $\Rightarrow$ (\ref{item:thm-acyclic2}): Firstly, notice that the objects $V'\to U'$ of $\fal_{Y_{\infty}\to X}^\proet$ satisfying the condition (\ref{item:thm-acyclic2}) form a topological generating family by \ref{cor:abhyankar} and \ref{prop:pre-perf-base}. Let $p_1\in \ca{O}_K$ be a $p$-th root of $p$ up to a unit. Then, we see that the Frobenius induces an almost isomorphism $\falb/p_1\falb\to \falb/p\falb$ by evaluating these sheaves at the objects $V'\to U'$. The Frobenius also induces an almost isomorphism $A/p_1A\to A/pA$ by \ref{lem:acyclic-falsite}.(\ref{item:lem-acyclic-falsite1}), which shows that $A$ is almost pre-perfectoid.
\end{proof}

\section{Complements on Logarithmic Geometry}\label{sec:log-geo}

We briefly recall some notions and facts of logarithmic geometry which will be used in the rest of the paper. We refer to \cite{kato1989log,kato1994toric,gabber2004foundations,ogus2018log} for a systematic development of logarithmic geometry, and to \cite[\Luoma{2}.5]{abbes2016p} for a brief summary of the theory.

\begin{mypara}
	We only consider logarithmic structures in \'etale topology. More precisely, let $X$ be a scheme, $X_\et$ the \'etale site of $X$, $\ca{O}_{X_\et}$ the structure sheaf on $X_\et$, $\ca{O}_{X_\et}^\times$ the subsheaf of units of $\ca{O}_{X_\et}$. A logarithmic struture on $X$ is a homomorphism of sheaves of monoids $\alpha:\scr{M} \to \ca{O}_{X_\et}$ on $X_\et$ which induces an isomorphism $\alpha^{-1}(\ca{O}_{X_\et}^\times) \iso \ca{O}_{X_\et}^\times$. We denote by $(X,\scr{M})$ the associated logarithmic scheme (cf. \cite[\Luoma{2}.5.11]{abbes2016p}). 
\end{mypara}

\begin{mypara}\label{para:log-fs-bc}
	Let $(X,\scr{M})$ be a coherent log scheme (cf. \cite[\Luoma{2}.5.15]{abbes2016p}). Then, there is a maximal open subscheme $X^{\mrm{tr}}$ of $X$ on which $\scr{M}$ is trivial, and moreover it is functorial in $(X,\scr{M})$ (\cite[\Luoma{3}.1.2.8]{ogus2018log}). Let $(X,\scr{M}) \to (S,\scr{L}) \leftarrow (Y,\scr{N})$ be a diagram of fine and saturated log schemes (cf. \cite[\Luoma{2}.5.15]{abbes2016p}). Then, the fibred product is representable in the category of fine and saturated log schemes by $(Z,\scr{P})=(X,\scr{M}) \times_{(S,\scr{L})}^{\mrm{fs}}  (Y,\scr{N})$. We remark that $Z^{\mrm{tr}}= X^{\mrm{tr}} \times_{S^{\mrm{tr}}} Y^{\mrm{tr}}$, that $Z \to X \times_S Y$ is finite, and that $Z^{\mrm{tr}} \to Z$ is Cartesian over $X^{\mrm{tr}} \times_{S^{\mrm{tr}}} Y^{\mrm{tr}} \to X \times_S Y$ (\cite[\Luoma{3} 2.1.2, 2.1.6]{ogus2018log}). Moreover, if $X^{\mrm{tr}}=X$, then $Z=X \times_S Y$ (\cite[\Luoma{3}.2.1.3]{ogus2018log}).
\end{mypara}

\begin{mypara}\label{para:log-compact}
	For an open immersion $j:Y \to X$, we denote by $j_\et: Y_\et \to X_\et$ the morphism of their \'etale sites defined by the base change by $j$. Let $\scr{M}_{Y \to X}$ be the preimage of $j_{\et*} \ca{O}_{Y_\et}^\times$ under the natural map $\ca{O}_{X_\et}\to j_{\et*}\ca{O}_{Y_\et}$, and we endow $X$ with the logarithmic structure $\scr{M}_{Y \to X} \to \ca{O}_{X_\et}$, which is called the compactifying log structure associated to the open immersion $j$ (\cite[\Luoma{3}.1.6.1]{ogus2018log}). Sometimes we write $\scr{M}_{Y \to X}$ as $\scr{M}_{X}$ if $Y$ is clear in the context.
\end{mypara}

\begin{mypara}\label{para:log-reg}
	Let $(X,\scr{M})$ be a fine and saturated log scheme which is regular (\cite[2.1]{kato1994toric}, \cite[2.3]{niziol2006toric}). Then, $X$ is locally Noetherian and normal, and $X^{\mrm{tr}}$ is regular and dense in $X$ (\cite[4.1]{kato1994toric}). Moreover, there is a natural isomorphism $\scr{M} \iso \scr{M}_{X^{\mrm{tr}} \to X}$ (\cite[11.6]{kato1994toric}, \cite[2.6]{niziol2006toric}). We remark that if $X$ is a regular scheme with a strict normal crossings divisor $D$, then $(X,\scr{M}_{X\setminus D \to X})$ is fine, saturated and regular (\cite[\Luoma{3}.1.11.9]{ogus2018log}). 
	
	Let $f:(X,\scr{M}) \to (S,\scr{L})$ be a smooth (resp. saturated) morphism of fine and saturated log schemes (cf. \cite[\Luoma{2} 5.25, 5.18]{abbes2016p}). Then, $f$ remains smooth (resp. saturated) under the base change in the category of fine and saturated log schemes (\cite[\Luoma{4}.3.1.2, \Luoma{4}.3.1.11]{ogus2018log}, resp. \cite[\Luoma{3}.2.5.3]{ogus2018log}). We remark that if $f$ is smooth, then $f^{\mrm{tr}}:X^{\mrm{tr}}\to S^{\mrm{tr}}$ is a smooth morphism of schemes. If moreover $(S,\scr{L})$ is regular, then $(X,\scr{M})$ is also regular (\cite[\Luoma{4}.3.5.3]{ogus2018log}). We also remark that if $f$ is saturated, then for any fibred product in the category of fine and saturated log schemes $(Z,\scr{P})=(X,\scr{M}) \times_{(S,\scr{L})}^{\mrm{fs}}  (Y,\scr{N})$, we have $Z=X\times_S Y$ (\cite[\Luoma{2}.2.13]{tsuji2019saturated}).
\end{mypara}

\begin{mypara}\label{para:log-dvr}
	Let $K$ be a complete discrete valuation field with valuation ring $\ca{O}_K$, $k$ the residue field of $\ca{O}_K$, $\pi$ a uniformizer of $\ca{O}_K$. We set $\eta=\spec(K)$, $S= \spec(\ca{O}_K)$ and $s=\spec(k)$. Then, $(S,\scr{M}_{\eta \to S})$ is fine, saturated and regular, since $\bb{N} \to \Gamma(S,\scr{M}_{\eta \to S})$ sending $1$ to $\pi$ forms a chart of $(S,\scr{M}_{\eta \to S})$ (cf. \cite[\Luoma{2}.5.13, \Luoma{2}.6.1]{abbes2016p}). Recall that an open immersion $Y \to X$ of quasi-compact and separated schemes over $\eta \to S$ is strictly semi-stable (\cite[6.3]{dejong1996alt}) if and only if the following conditions are satisfied (\cite[6.4]{dejong1996alt}, \cite[17.5.3]{ega4-4}): 
	\begin{enumerate}
		\renewcommand{\theenumi}{\roman{enumi}}
		\renewcommand{\labelenumi}{{\rm(\theenumi)}}
		\item For each point $x$ of the generic fibre $X_\eta$,  there is an open neighborhood $U\subseteq X_\eta$ of $x$ and a smooth $K$-morphism 
		\begin{align}
			f:U\longrightarrow \spec(K[s_1,\dots,s_m])
		\end{align}
		such that $f$ maps $x$ to the point associated to the maximal ideal $(s_1,\dots,s_m)$ and that $U\setminus Y$ is the inverse image of the closed subset defined by $s_1\cdots s_m=0$.
		\item For each point $x$ of the special fibre $X_s$,  there is an open neighborhood $U\subseteq X$ of $x$ and a smooth $\ca{O}_K$-morphism 
		\begin{align}
			f:U\longrightarrow \spec(\ca{O}_K[t_1,\dots,t_n, s_1,\dots,s_m]/(\pi-t_1\cdots t_n))
		\end{align}
		such that $f$ maps $x$ to the point associated to the maximal ideal $(t_1,\dots,t_n, s_1,\dots,s_m)$ and that $U\setminus Y$ is the inverse image of the closed subset defined by $t_1\cdots t_n\cdot s_1\cdots s_m=0$.
	\end{enumerate}
	We call an open immersion $Y \to X$ of quasi-compact and separated schemes over $\eta \to S$ is semi-stable if for any point $x$ of $X$ there is an \'etale neighborhood $U$ of $x$ such that $Y\times_X U\to U$ is strictly semi-stable. In this case, $(X,\scr{M}_{Y \to X})$ is fine, saturated and regular which is smooth and saturated over $(S,\scr{M}_{\eta \to S})$, since for any point $x$ of $X$ there is an \'etale neighborhood $U$ of $x$ such that there exists a chart for the morphism $(U,\scr{M}_{Y\times_X U \to U})\to (S,\scr{M}_{\eta \to S})$ subordinate to the morphism $\bb{N} \to \bb{N}^n\oplus \bb{N}^m$ sending $1$ to $(1,\dots,1,0,\dots,0)$ such that the induced morphism $U\to S\times_{\bb{A}_{\bb{N}}}\bb{A}_{\bb{N}^n\oplus \bb{N}^m}$ is smooth (cf. \cite[\Luoma{4}.3.1.18]{ogus2018log}).
\end{mypara}

\begin{mypara}\label{para:generic-finite}
	Recall that a morphism of schemes $f:X\to S$ is called \emph{generically finite} if there exists a dense open subscheme $U$ of $S$ such that $f^{-1}(U)\to U$ is finite. We remark that for a morphism $f:X\to S$ of finite type between Noetherian schemes which maps generic points to generic points, $f$ is generically finite if and only if the residue field of any generic point $\eta$ of $X$ is a finite field extension of the residue field of $f(\eta)$ (\cite[\Luoma{2}.1.1.7]{gabber2014travaux}).
\end{mypara}

\begin{mypara}\label{para:notation-alter}
	Let $K$ be a complete discrete valuation field with valuation ring $\ca{O}_K$, $L$ an algebraically closed valuation field of height $1$ extension of $K$ with valuation ring $\ca{O}_L$, $\overline{K}$ the algebraic closure of $K$ in $L$.
	
	Consider the category $\scr{C}$ of open immersions between integral affine schemes $U\to T$ over $\spec(K)\to \spec(\ca{O}_K)$ under $\spec(L)\to \spec(\ca{O}_L)$ such that $T$ is of finite type over $\ca{O}_K$ and that $\spec(L)\to U$ is dominant. Let $\scr{C}_{\mrm{car}}$ be the full subcategory of $\scr{C}$ formed by those objects $U \to T$ Cartesian over $\spec(K) \to \spec(\ca{O}_K)$.
	\begin{align}
		\xymatrix@R=1pc{
			\spec(L)\ar[r]\ar[d]&\spec(\ca{O}_L)\ar[d]\\
			U=\spec(B)\ar[r]\ar[d]&T=\spec(A)\ar[d]\\
			\spec(K)\ar[r]&\spec(\ca{O}_K)
		}
	\end{align}
	We note that the objects of $\scr{C}$ are of the form $(U=\spec(B) \to T=\spec(A))$ where $A$ (resp. $B$) is a finitely generated $\ca{O}_K$-subalgebra of $\ca{O}_L$ (resp. $K$-subalgebra of $L$) with $A\subseteq B$ such that $\spec(B) \to \spec(A)$ is an open immersion. 
\end{mypara}

\begin{mylem}\label{thm:alteration-ok}
	With the notation in {\rm\ref{para:notation-alter}}, we have:
	\begin{enumerate}
		\renewcommand{\labelenumi}{{\rm(\theenumi)}}
		\item The category $\scr{C}$ is cofiltered, and the subcategory $\scr{C}_{\mrm{car}}$ is initial in $\scr{C}$.
		\item The morphism $\spec(L) \to \spec(\ca{O}_L)$ represents the cofiltered limit of morphisms $U \to T$ indexed by $\scr{C}$ in the category of morphisms of schemes {\rm(cf. \ref{para:E})}.
		\item There exists a directed inverse system $(U_\lambda \to T_\lambda)_{\lambda \in \Lambda}$ of objects of $\scr{C}_{\mrm{car}}$ over a directed inverse system $(\spec(K_\lambda) \to \spec(\ca{O}_{K_\lambda}))_{\lambda \in \Lambda}$ of objects of $\scr{C}_{\mrm{car}}$ such that $K_\lambda$ is a finite field extension of $K$ in $L$, that $\overline{K}=\bigcup_{\lambda \in \Lambda} K_\lambda$, that $U_\lambda \to T_\lambda$ is strictly semi-stable over $\spec(K_\lambda) \to \spec(\ca{O}_{K_\lambda})$ {\rm(\ref{para:log-dvr})}, and that $(U_\lambda \to T_\lambda)_{\lambda \in \Lambda}$ forms an initial full subcategory of $\scr{C}_{\mrm{car}}$.
	\end{enumerate}
\end{mylem}
\begin{proof}
	(1) For a diagram $(U_1\to T_1)\to (U_0 \to T_0) \leftarrow (U_2 \to T_2)$ in $\scr{C}$, let $T$ be the scheme theoretic image of $\spec(L) \to T_1\times_{T_0} T_2$ and let $U$ be the intersection of $U_1\times_{U_0} U_2$ with $T$. It is clear that $T$ is of finite type over $\ca{O}_K$ as $\ca{O}_K$ is Noetherian, that $U$ and $T$ are integral and affine, that $\spec(L) \to U$ is dominant, and that $\spec(L)\to T$ factors through $\spec(\ca{O}_L)$. Thus, $U \to T$ is an object of $\scr{C}$, which shows that $\scr{C}$ is cofiltered. For an object $(U=\spec(B)\to T=\spec(A))$ of $\scr{C}$, we write $\ca{O}_L$ as a filtered union of finitely generated $A$-subalgebras $A_i$. Let $\pi$ be a uniformizer of $K$. Notice that $L=\ca{O}_L[1/\pi]=\colim A_i[1/\pi]$ and that $\ho_{K\alg}(B,L)=\colim \ho_{K\alg}(B,A_i[1/\pi])$ by \cite[8.14.2.2]{ega4-3}. Thus, there exists an index $i$ such that $\spec(A_i[1/\pi]) \to \spec(A_i)$ is an object of $\scr{C}_{\mrm{car}}$ over $U \to T$. 
	
	(2) It follows immediately from the arguments above.
	
	(3) Consider the category $\scr{D}$ of morphisms of $\scr{C}_{\mrm{car}}$,
	\begin{align}
		\xymatrix{
			U' \ar[r]\ar[d] & T' \ar[d]\\
			\spec(K') \ar[r] & \spec(\ca{O}_{K'})
		}
	\end{align}
	such that $K'$ is a finite field extension of $K$. Similarly, this category is also cofiltered with limit of diagrams of schemes $(\spec(L) \to \spec(\ca{O}_L)) \to (\spec(\overline{K}) \to \spec(\ca{O}_{\overline{K}}))$. It suffices to show that the full subcategory of $\scr{D}$ formed by strictly semi-stable objects is initial. For any object $U \to T$ of $\scr{C}_{\mrm{car}}$, by de Jong's alteration theorem \cite[6.5]{dejong1996alt}, there exists a proper surjective and generically finite morphism $T' \to T$ of integral schemes such that $U'=U\times_T T' \to T'$ is strictly semi-stable over $\spec(K') \to \spec(\ca{O}_{K'})$ for a finite field extension $K\to K'$. Since $L$ is algebraically closed, the dominant morphism $\spec(L)\to U$ lifts to a dominant morphism $\spec(L)\to U'$ (\ref{para:generic-finite}), which further extends to a lifting $\spec(\ca{O}_L) \to T'$ of $\spec(\ca{O}_L) \to T$ by the valuative criterion. After replacing $T'$ by an affine open neighborhood of the image of the closed point of $\spec(\ca{O}_L)$, we obtain a strictly semi-stable object of $\scr{D}$ over $(U \to T)\to (\spec(K) \to \spec(\ca{O}_K))$, which completes the proof.
\end{proof}

\begin{mythm}[{\cite[\Luoma{10} 3.5, 3.7]{gabber2014travaux}}]\label{thm:alteration}
	Let $K$ be a complete discrete valuation field with valuation ring $\ca{O}_K$, $(Y\to X) \to (U \to T)$ a morphism of dominant open immersions over $\spec(K) \to \spec(\ca{O}_K)$ between irreducible $\ca{O}_K$-schemes of finite type such that $X\to T$ is proper surjective. Then, there exists a commutative diagram of dominant open immersions between irreducible $\ca{O}_K$-schemes of finite type
	\begin{align}\label{diam:alteration}
		\xymatrix{
			(Y'\to X') \ar[r]^-{(\beta^{\circ},\beta)}\ar[d]_-{(f'^{\circ}, f')}& (Y\to X)\ar[d]^-{(f^{\circ}, f)}\\
			(U'\to T') \ar[r]_-{(\alpha^{\circ},\alpha)}& (U\to T)
		}
	\end{align}
	satisfying the following conditions:
	\begin{enumerate}
		\renewcommand{\theenumi}{\roman{enumi}}
		\renewcommand{\labelenumi}{{\rm(\theenumi)}}
		\item We have $Y'=\beta^{-1}(Y)\cap f'^{-1}(U')$, i.e. $Y'\to X'$ is Cartesian over $U'\times_{U}Y \to T'\times_{T} X$ {\rm(cf. \ref{para:E})}.
		\item The morphism $(X',\scr{M}_{Y'\to X'})\to (T',\scr{M}_{U'\to T'})$ induced by $(f'^{\circ},f')$ is a smooth and saturated morphism of fine, saturated and regular log schemes.
		\item The morphisms $\alpha$ and $\beta$ are proper surjective and generically finite, and $f'$ is projective surjective.
	\end{enumerate}
\end{mythm}
\begin{proof}
	We may assume that $T$ is nonempty. Recall that $\spec(\ca{O}_K)$ is universally $\bb{Q}$-resolvable (\cite[\Luoma{10}.3.3]{gabber2014travaux}) by de Jong's alteration theorem \cite[6.5]{dejong1996alt}. Thus, $T$ is also universally $\bb{Q}$-resolvable by \cite[\Luoma{10} 3.5, 3.5.2]{gabber2014travaux} so that we can apply \cite[\Luoma{10}.3.5]{gabber2014travaux} to the proper surjective morphism $f$ and the nowhere dense closed subset $X\setminus Y$. Then, we obtain a commutative diagram of schemes
	\begin{align}
		\xymatrix{
			X'\ar[r]^-{\beta}\ar[d]_-{f'}& X\ar[d]^-{f}\\
			T'\ar[r]_-{\alpha}& T
		}
	\end{align}
	and dense open subsets $U'\subseteq T'$, $Y'=\beta^{-1}(Y)\cap f'^{-1}(U')\subseteq X'$ such that $(X',\scr{M}_{Y'\to X'})$ and $(T',\scr{M}_{U'\to T'})$ are fine, saturated and regular, that $(X',\scr{M}_{Y'\to X'})\to (T',\scr{M}_{U'\to T'})$ is smooth, that $\alpha, \beta$ are proper surjective and generically finite morphisms which map generic points to generic points, and that $f'$ is projective (since $f$ is proper, cf. \cite[\Luoma{10} 3.1.6, 3.1.7]{gabber2014travaux}). Since $X$ (resp. $T$) is irreducible and $X'$ (resp. $T'$) is a disjoint union of normal integral schemes (\ref{para:log-reg}), after firstly replacing $X'$ by an irreducible component and then replacing $T'$ by the irreducible component under $X'$, we may assume that $X'$ and $T'$ are irreducible. Then, $Y'\to U'$ is dominant (so that $f'$ is projective surjective), since it is smooth and $Y'$ is nonempty (\cite[2.3.4]{ega4-2}). We claim that $\alpha$ maps $U'$ into $U$. Indeed, if there exists a point $u\in U'$ with $\alpha(u)\notin U$, then $f'^{-1}(u)\cap Y'=\emptyset$. However, endowing $u$ with the trivial log structure, the log scheme $(u,\ca{O}_{u_{\et}}^\times)$ is fine, saturated and regular, and the fibred product in the category of fine and saturated log schemes
	\begin{align}
		(u,\ca{O}_{u_{\et}}^\times)\times_{(T',\scr{M}_{U' \to T'})}^{\mrm{fs}} (X',\scr{M}_{Y' \to X'})
	\end{align}
	is regular with underlying scheme $f'^{-1}(u)$ (\ref{para:log-reg}, \ref{para:log-fs-bc}). Thus, $f'^{-1}(u)\cap Y'$ is dense in $f'^{-1}(u)$, which contradicts the assumption that $f'^{-1}(u)\cap Y'=\emptyset$ since $f'$ is surjective. Thus, we obtain a diagram \eqref{diam:alteration} satisfying all the conditions except the saturatedness of $(X',\scr{M}_{Y'\to X'})\to (T',\scr{M}_{U'\to T'})$.
	
	To make $(X',\scr{M}_{Y'\to X'})\to (T',\scr{M}_{U'\to T'})$ saturated, we apply \cite[\Luoma{10}.3.7]{gabber2014travaux} to the morphism $(f'^{\circ}, f')$. We obtain a Cartesian morphism $(\gamma^{\circ}, \gamma):(U'' \to T'')\to (U' \to T') $ of dominant open immersions such that $(T'',\scr{M}_{U'' \to T''})$ is a fine, saturated and regular log scheme, that $\gamma$ is a proper surjective and generically finite morphism which maps generic points of $T''$ to the generic point of $T'$, and that the fibred product in the category of fine and saturated log schemes
	\begin{align}\label{eq:fib-log-prod}
		(T'',\scr{M}_{U'' \to T''})\times^{\mrm{fs}}_{(T',\scr{M}_{U' \to T'})} (X',\scr{M}_{Y' \to X'})
	\end{align}
	is saturated over $(T'',\scr{M}_{U'' \to T''})$. The fibred product \eqref{eq:fib-log-prod} is still smooth over $(T'',\scr{M}_{U'' \to T''})$, and thus it is regular  (\ref{para:log-reg}). Let $X''$ be the underlying scheme of it and let $Y''=(X'')^{\mrm{tr}}$. Then, the fibred product \eqref{eq:fib-log-prod} is isomorphic to $(X'',\scr{M}_{Y''\to X''})$ (\ref{para:log-reg}). Thus, we obtain a commutative diagram of dominant open immersions of schemes
	\begin{align}\label{diam:alteration2}
		\xymatrix{
			(Y''\to X'') \ar[r]^-{(\delta^{\circ},\delta)}\ar[d]_-{(f''^{\circ}, f'')}& (Y'\to X')\ar[d]^-{(f'^{\circ}, f')}\\
			(U''\to T'') \ar[r]_-{(\gamma^{\circ},\gamma)}& (U'\to T')
		}
	\end{align}
	Notice that $Y''= U'' \times_{U'} Y'$ and $X'' \to T'' \times_{T'} X'$ is finite, and that $Y''\to X''$ is Cartesian over $U'' \times_{U'} Y' \to T'' \times_{T'} X'$ (\ref{para:log-fs-bc}). Thus, we see that $Y''\to X''$ is Cartesian over $U'' \times_{U} Y \to T'' \times_{T} X$ and that $f''$ is projective. Since $T'$ (resp. $X'$) is irreducible and $T''$ (resp. $X''$) is a disjoint union of normal integral schemes (\ref{para:log-reg}), after firstly replacing $T''$ by an irreducible component and then replacing $X''$ by an irreducible component on which the restriction of $\delta^{\circ}$ is dominant, we may assume that $T''$ and $X''$ are irreducible. In particular, $\delta$ is generically finite and so is $\beta\circ\delta$ (\ref{para:generic-finite}), and again $Y''\to U''$ is dominant so that $f''$ is projective surjective.
\end{proof}

\begin{mylem}\label{lem:noetherian}
	Let $X$ be a scheme of finite type over a valuation ring $A$ of height $1$. Then, the underlying topological space of $X$ is Noetherian.
\end{mylem}
\begin{proof}
	Let $\eta$ and $s$ be the generic point and closed point of $\spec(A)$ respectively. Then, the generic fibre $X_\eta$ and the special fibre $X_s$ are both Noetherian. As a union of $X_\eta$ and $X_s$, the underlying topological space of $X$ is also Noetherian (\cite[\href{https://stacks.math.columbia.edu/tag/0053}{0053}]{stacks-project}).
\end{proof}

\begin{myprop}\label{prop:val-lim-ss}
	With the notation in {\rm\ref{para:notation-alter}} and {\rm\ref{thm:alteration-ok}}, let $Y\to X$ be a quasi-compact dominant open immersion over $\spec(L) \to \spec(\ca{O}_L)$ such that $X \to \spec(\ca{O}_L)$ is proper of finite presentation. Then, there exists a proper surjective $\ca{O}_L$-morphism of finite presentation $X' \to X$, an index $\lambda_1\in\Lambda$, and a directed inverse system of open immersions $(Y'_\lambda \to X'_\lambda)_{\lambda \geq \lambda_1}$ over $(U_\lambda\to T_\lambda)_{\lambda \geq \lambda_1}$ satisfying the following conditions for each $\lambda \geq \lambda_1$:
	\begin{enumerate}
		\renewcommand{\theenumi}{\roman{enumi}}
		\renewcommand{\labelenumi}{{\rm(\theenumi)}}
		\item We have $Y'=Y\times_X X'=\lim_{\lambda \geq \lambda_1} Y'_\lambda$ and $X'=\lim_{\lambda \geq \lambda_1} X'_\lambda$.
		\item The log scheme $(X'_\lambda,\scr{M}_{Y'_\lambda \to X'_\lambda})$ is fine, saturated and regular.
		\item The morphism $(X'_\lambda,\scr{M}_{Y'_\lambda \to X'_\lambda})\to (T_\lambda,\scr{M}_{U_\lambda\to T_\lambda})$ is smooth and saturated, and $X'_\lambda \to T_\lambda$ is projective.
		\item If moreover $Y= \spec(L) \times_{\spec(\ca{O}_L)} X$, then we can require that $Y'_\lambda= U_\lambda \times_{T_\lambda}X'_\lambda$.\label{prop:val-lim-ss-car}
	\end{enumerate}
\end{myprop}
\begin{proof}
	We follow closely the proof of \cite[5.2.19]{temkin2019logval}. Since the underlying topological space of $X$ is Noetherian by \ref{lem:noetherian}, each irreducible component $Z$ of $X$ admits a closed subscheme structure such that $Z\to X$ is of finite presentation (\cite[\href{https://stacks.math.columbia.edu/tag/01PH}{01PH}]{stacks-project}). After replacing $X$ by the disjoint union of its irreducible components, we may assume that $X$ is irreducible. Then, the generic fibre of $X\to \spec(\ca{O}_L)$ is also irreducible as an open subset of $X$. Using \cite[8.8.2, 8.10.5]{ega4-3}, there exists an index $\lambda_0 \in \Lambda$, a proper $T_{\lambda_0}$-scheme $X_{\lambda_0}$, and an open subscheme $Y_{\lambda_0}$ of $U_{\lambda_0}\times_{T_{\lambda_0}} X_{\lambda_0}$, such that $X=\spec(\ca{O}_L) \times_{T_{\lambda_0}} X_{\lambda_0}$ and that $Y=\spec(L) \times_{U_{\lambda_0}} Y_{\lambda_0}$. Let $\eta$ denote the generic point of $X$, $\eta_{\lambda_0}$ the image of $\eta$ under the morphism $X\to X_{\lambda_0}$, $Z_{\lambda_0}$ the scheme theoretic closure of $\eta_{\lambda_0}$ in $X_{\lambda_0}$. Notice that $\spec(\ca{O}_L)\times_{T_{\lambda_0}} Z_{\lambda_0}\to X$ is a surjective finitely presented closed immersion. After replacing $X$ by $\spec(\ca{O}_L)\times_{T_{\lambda_0}} Z_{\lambda_0}$ and replacing $X_{\lambda_0}$ by $Z_{\lambda_0}$, we may assume that $X\to X_{\lambda_0}$ is a dominant morphism of irreducible schemes. Since $T_{\lambda_0}$ is irreducible and $L$ is algebraically closed, the generic fibre of $f:X_{\lambda_0}\to T_{\lambda_0}$ is geometrically irreducible. In particular, if $\xi_{\lambda_0}$ (resp. $\eta_{\lambda_0}$) denotes the generic point of $T_{\lambda_0}$ (resp. $X_{\lambda_0}$), then $\eta=\spec(L)\times_{\xi_{\lambda_0}}\eta_{\lambda_0}$ (\cite[4.5.9]{ega4-2}). In the situation of (\ref{prop:val-lim-ss-car}), we can moreover assume that $Y_{\lambda_0}= U_{\lambda_0} \times_{T_{\lambda_0}} X_{\lambda_0}$.
	
	By \ref{thm:alteration}, there exists a commutative diagram of dominant open immersions of irreducible schemes,
	\begin{align}
		\xymatrix{
			(Y'_{\lambda_0}\to X'_{\lambda_0}) \ar[r]^-{(\beta^{\circ},\beta)}\ar[d]_-{(f'^{\circ}, f')}& (Y_{\lambda_0}\to X_{\lambda_0})\ar[d]^-{(f^{\circ}, f)}\\
			(U'_{\lambda_0}\to T'_{\lambda_0}) \ar[r]_-{(\alpha^{\circ},\alpha)}& (U_{\lambda_0}\to T_{\lambda_0})
		}
	\end{align}
	where $Y'_{\lambda_0}\to X'_{\lambda_0}$ is Cartesian over $U'_{\lambda_0}\times_{U_{\lambda_0}}Y_{\lambda_0} \to T'_{\lambda_0}\times_{T_{\lambda_0}} X_{\lambda_0}$, and where $(X'_{\lambda_0},\scr{M}_{Y'_{\lambda_0}\to X'_{\lambda_0}})\to (T'_{\lambda_0},\scr{M}_{U'_{\lambda_0}\to T'_{\lambda_0}})$ is a smooth and saturated morphism of fine, saturated and regular log schemes, and where $\alpha$ and $\beta$ are proper surjective and generically finite, and where $f'$ is projective surjective. We take a dominant morphism $\gamma^{\circ}:\spec(L)\to U_{\lambda_0}'$ which lifts $\spec(L)\to U_{\lambda_0}$ since $L$ is algebraically closed and $\alpha$ is generically finite, the morphism $\spec(\ca{O}_L) \to T_{\lambda_0}$ lifts to $\gamma:\spec(\ca{O}_L) \to T'_{\lambda_0}$ by the valuative criterion. We set $Y'=\spec(L)\times_{U'_{\lambda_0}}Y'_{\lambda_0}$ and $X'=\spec(\ca{O}_L)\times_{T'_{\lambda_0}} X'_{\lambda_0}$. It is clear that $Y'\to X'$ is Cartesian over $Y \to X$ by base change. Let $\xi'_{\lambda_0}$ (resp. $\eta'_{\lambda_0}$) be the generic points of $T'_{\lambda_0}$ (resp. $X'_{\lambda_0}$). Since the generic fibre of $f$ is geometrically irreducible, $\xi'_{\lambda_0}\times_{\xi_{\lambda_0}}\eta_{\lambda_0}$ is a single point and $\eta'_{\lambda_0}$ maps to it (\cite[4.5.9]{ega4-2}). Since $\spec(L)\times_{\xi_{\lambda_0}}\eta_{\lambda_0}$ is the generic point of $X$, we see that $X' \to X$ is proper surjective and of finite presentation. 
	It remains to construct $(Y'_\lambda \to X'_\lambda)_{\lambda \geq \lambda_1}$. 
	
	After replacing $T'_{\lambda_0}$ by an affine open neighborhood of the image of the closed point of $\spec(\ca{O}_L)$, lemma \ref{thm:alteration-ok} implies that there exists an index $\lambda_1 \geq \lambda_0$ such that the transition morphism $(U_{\lambda_1}\to T_{\lambda_1}) \to (U_{\lambda_0}\to T_{\lambda_0})$ factors through $(U'_{\lambda_0}\to T'_{\lambda_0})$. For each index $\lambda\geq \lambda_1$, consider the fibred product in the category of fine and saturated log schemes
	\begin{align}\label{eq:fs-log-fibprod}
		(X'_\lambda,\scr{M}_{Y'_\lambda \to X'_\lambda})=(T_{\lambda},\scr{M}_{U_{\lambda}\to T_{\lambda}}) \times_{(T'_{\lambda_0},\scr{M}_{U'_{\lambda_0}\to T'_{\lambda_0}})}^{\mrm{fs}} (X'_{\lambda_0},\scr{M}_{Y'_{\lambda_0}\to X'_{\lambda_0}}),
	\end{align}
	which is a fine, saturated and regular log scheme smooth and saturated over $(T_{\lambda},\scr{M}_{U_{\lambda}\to T_{\lambda}})$ (\ref{para:log-fs-bc}, \ref{para:log-reg}). Moreover, we have $Y'_\lambda= U_\lambda \times_{U'_{\lambda_0}} Y'_{\lambda_0}$, $X'_\lambda = T_\lambda \times_{T'_{\lambda_0}} X'_{\lambda_0}$, and in the situation of (\ref{prop:val-lim-ss-car}), $Y'_\lambda=U_\lambda \times_{T_\lambda}X'_\lambda$ by base change. Therefore, $(Y'_\lambda \to X'_\lambda)_{\lambda \geq \lambda_1}$ meets our requirements. 
\end{proof}

\section{Faltings' Main $p$-adic Comparison Theorem: the Absolute Case}\label{sec:abs-comp}

\begin{mylem}\label{lem:finite-etale-loc}
	Let $Y$ be a coherent scheme, $V$ a finite \'etale $Y$-scheme. Then, there exists a finite \'etale surjective morphism $Y'\to Y$ such that $Y'\times_Y V$ is isomorphic to a finite disjoint union of $Y'$.
\end{mylem}
\begin{proof}
	If $Y$ is connected, let $\overline{y}$ be a geometric point of $Y$, $\pi_1(Y,\overline{y})$ the fundamental group of $Y$ with base point $\overline{y}$. Then, $Y_\fet$ is equivalent to the category of finite $\pi_1(Y,\overline{y})$-sets so that the lemma holds (\cite[\href{https://stacks.math.columbia.edu/tag/0BND}{0BND}]{stacks-project}).
	
	In general, for any connected component $Z$ of $Y$, let $(Y_\lambda)_{\lambda\in \Lambda_Z}$ be the directed inverse system of all open and closed subschemes of $Y$ which contain $Z$ and whose transition morphisms are inclusions. Notice that $\lim_{\lambda\in \Lambda_Z} Y_\lambda$ is a closed subscheme of $Y$ with underlying topological space $Z$ by \cite[\href{https://stacks.math.columbia.edu/tag/04PL}{04PL}]{stacks-project} and \cite[8.2.9]{ega4-3}. We endow $Z$ with the closed subscheme structure of $\lim_{\lambda\in \Lambda_Z} Y_\lambda$. The first paragraph shows that there exists a finite \'etale surjective morphism $Z'\to Z$ such that $Z'\times_Y V=\coprod_{i=1}^r Z'$. Using \cite[8.8.2, 8.10.5]{ega4-3} and \cite[17.7.8]{ega4-4}, there exists an index $\lambda_0\in \Lambda_Z$, a finite \'etale surjective morphism $Y'_{\lambda_0}\to Y_{\lambda_0}$ and an isomorphism $Y'_{\lambda_0}\times_Y V = \coprod_{i=1}^r Y'_{\lambda_0}$. Notice that $Y'_{\lambda_0}$ is also finite \'etale over $Y$. Since $Z$ is an arbitrary connected component of $Y$, the conclusion follows from the quasi-compactness of $Y$.
\end{proof}

\begin{mylem}\label{lem:loc-sys-trans}
	Let $Y$ be a coherent scheme, $\rho:Y_\et \to Y_\fet$ the morphism of sites defined by the inclusion functor. Then, the functor $\rho^{-1}: \widetilde{Y_\fet}\to \widetilde{Y_\et}$ of the associated topoi induces an equivalence $\rho^{-1}:\locsys(Y_\fet) \to \locsys(Y_\et)$ between the categories of finite locally constant abelian sheaves with quasi-inverse $\rho_*$.
\end{mylem}
\begin{proof}
	Since any finite locally constant sheaf on $Y_\et$ (resp. $Y_\fet$) is representable by a finite \'etale $Y$-scheme by faithfully flat descent (cf. \cite[\href{https://stacks.math.columbia.edu/tag/03RV}{03RV}]{stacks-project}), the Yoneda embeddings induce a commutative diagram
	\begin{align}
		\xymatrix{
			\locsys(Y_\fet) \ar[d]_-{\rho^{-1}}\ar[r]&Y_\fet\ar[r]^-{h^\fet}\ar[d]&\widetilde{Y_\fet}\ar[d]\\
			\locsys(Y_\et)\ar[r]& Y_\et\ar[r]^-{h^\et}&\widetilde{Y_\et}
		}
	\end{align}
	where the horizontal arrows are fully faithful. In particular, $\rho^{-1}$ is fully faithful. For a finite locally constant abelian sheaf $\bb{F}$ on $Y_\et$, let $V$ be a finite \'etale $Y$-scheme representing $\bb{F}$ and let $h_V^\et$ (resp. $h_V^\fet$) be the representable sheaf of $V$ on $Y_\et$ (resp. $Y_\fet$). We have $\bb{F}=h_V^\et=\rho^{-1}h_V^\fet$ (\cite[\href{https://stacks.math.columbia.edu/tag/04D3}{04D3}]{stacks-project}). By \ref{lem:finite-etale-loc}, $h_V^\fet$ is finite locally constant. It is clear that the adjunction morphism $h_V^\fet\to \rho_*h_V^\et$ is an isomorphism, which shows that $h_V^\fet$ is an abelian sheaf. Thus, $\rho^{-1}$ is essentially surjective. Moreover, the argument also shows that $\rho_*$ induces a functor $\rho_*:\locsys(Y_\et) \to \locsys(Y_\fet)$ which is a quasi-inverse of $\rho^{-1}$.
\end{proof}

\begin{myprop}\label{prop:loc-sys-trans}
	With the notation in {\rm\ref{para:notation-psi-beta-sigma}}, the functors between the categories of finite locally constant abelian sheaves 
	\begin{align}
		\locsys(Y_\fet) \stackrel{\beta^{-1}}{\longrightarrow}\locsys(\fal_{Y \to X}^\et) \stackrel{\psi^{-1}}{\longrightarrow}\locsys(Y_\et) 
	\end{align}
	are equivalences with quasi-inverses $\beta_*$ and $\psi_*$ respectively.
\end{myprop}
\begin{proof}
	Notice that for any finite locally constant abelian sheaf $\bb{G}$ on $Y_\fet$, the canonical morphism $\beta^{-1}\bb{G}\to \psi_*\rho^{-1}\bb{G}$, which is induced by the adjunction $\id \to \psi_*\psi^{-1}$ and by the identity $\psi^{-1}\beta^{-1}=\rho^{-1}$, is an isomorphism by \ref{lem:loc-sys-trans} and the proof of \cite[\Luoma{6}.6.3.(\luoma{3})]{abbes2016p}. For a finite locally constant abelian sheaf $\bb{F}$ over $Y_\et$, we write $\bb{F}=\rho^{-1}\bb{G}$ by \ref{lem:loc-sys-trans}. Then, $\bb{F}=\psi^{-1}\beta^{-1}\bb{G}\iso\psi^{-1}\psi_*\rho^{-1}\bb{G}=\psi^{-1}\psi_*\bb{F}$, whose inverse is the adjunction map $\psi^{-1}\psi_*\bb{F} \to \bb{F}$ since the composition of $\psi^{-1}(\beta^{-1}\bb{G}) \to \psi^{-1}(\psi_*\psi^{-1})(\beta^{-1}\bb{G})=(\psi^{-1}\psi_*)\psi^{-1}(\beta^{-1}\bb{G})\to \psi^{-1}(\beta^{-1}\bb{G})$ is the identity. On the other hand, for a finite locally constant abelian sheaf $\bb{L}$ over $\fal_{Y \to X}^\et$, we claim that $\bb{L} \to \psi_*\psi^{-1}\bb{L}$ is an isomorphism. The problem is local on $\fal_{Y \to X}^\et$. Thus, we may assume that $\bb{L}$ is the constant sheaf with value $L$ where $L$ is a finite abelian group. Let $\underline{L}$ be the constant sheaf with value $L$ on $Y_\fet$. Then, $\bb{L}=\beta^{-1}\underline{L}$, and the isomorphism $\bb{L}=\beta^{-1}\underline{L}\iso \psi_*\rho^{-1}\underline{L}=\psi_*\psi^{-1}\bb{L}$ coincides with the adjunction map $\bb{L} \to \psi_*\psi^{-1}\bb{L}$. Therefore, $\psi^{-1}:\locsys(\fal_{Y \to X}^\et) \to \locsys(Y_\et) $ is an equivalence with quasi-inverse $\psi_*$. The conclusion follows from \ref{lem:loc-sys-trans}.
\end{proof}

\begin{mypara}\label{para:easy-comparison-mor-rel}
	Let $f:(Y' \to X')\to (Y \to X)$ be a morphism of morphisms between coherent schemes over $\spec(\bb{Q}_p)\to \spec(\bb{Z}_p)$. The base change by $f$ induces a commutative diagram of sites
	\begin{align}
		\xymatrix{
			Y'_\et\ar[d]_-{f_{\et}}\ar[r]^-{\psi'} & \fal_{Y' \to X'}^\et\ar[d]^-{f_{\fal}}\\
			Y_\et\ar[r]^-{\psi} & \fal_{Y \to X}^\et
		}
	\end{align}
	Let $\bb{F}'$ be a finite locally constant abelian sheaf on $Y'_\et$. Remark that the sheaf $\falb$ on $\fal_{Y \to X}^\et$ is flat over $\bb{Z}$. Consider the natural morphisms in the derived category $\dd(\falb\module_{\fal_{Y \to X}^\et})$,
	\begin{align}\label{eq:fal-rel-comp}
		\xymatrix{
			(\rr\psi_*\rr f_{\et *} \bb{F}')\otimes^{\dl}_{\bb{Z}}\falb & (\rr f_{\fal *} \psi'_*\bb{F}') \otimes^{\dl}_{\bb{Z}}\falb \ar[l]_-{\alpha_1}\ar[r]^-{\alpha_2}& \rr f_{\fal *} (\psi'_*\bb{F}'\otimes_{\bb{Z}}\falb'),
		}
	\end{align}
	where $\alpha_1$ is induced by the canonical morphism $\psi'_*\bb{F}' \to \rr\psi'_*\bb{F}'$, and $\alpha_2$ is the canonical morphism.
\end{mypara}

\begin{mypara}\label{para:easy-comparison-mor-abs}
	We keep the notation in \ref{para:easy-comparison-mor-rel} and assume that $X$ is the spectrum of an absolutely integrally closed valuation ring $A$ and that $Y$ is a quasi-compact open subscheme of $X$. Applying the functor $\rr\Gamma(Y \to X,-)$ on \eqref{eq:fal-rel-comp}, we obtain the natural morphisms in the derived category $\dd(A\module)$ by \ref{lem:fal-localtopos},
	\begin{align}\label{eq:fal-abs-comp}
		\xymatrix{
			\rr\Gamma (Y'_\et, \bb{F}')\otimes^{\dl}_{\bb{Z}}A & \rr\Gamma (\fal_{Y' \to X'}^\et, \psi'_*\bb{F}') \otimes^{\dl}_{\bb{Z}}A \ar[l]_-{\alpha_1}\ar[r]^-{\alpha_2}& \rr\Gamma (\fal_{Y' \to X'}^\et, \psi'_*\bb{F}'\otimes_{\bb{Z}}\falb').
		}
	\end{align}
\end{mypara}

\begin{mydefn}[{\cite[4.8.13, 5.7.3]{abbes2020suite}}]\label{defn:easy-fal-comp-mor}
	Under the assumptions in \ref{para:easy-comparison-mor-rel} (resp. \ref{para:easy-comparison-mor-abs}) and with the same notation, if $\alpha_1$ is an isomorphism (for instance, if the canonical morphism $\psi'_*\bb{F}' \to \rr\psi'_*\bb{F}'$ is an isomorphism), then we call the canonical morphism
	\begin{align}
		&\alpha_2\circ\alpha_1^{-1}: (\rr\psi_*\rr f_{\et *} \bb{F}')\otimes^{\dl}_{\bb{Z}}\falb \longrightarrow \rr f_{\fal *} (\psi'_*\bb{F}'\otimes_{\bb{Z}}\falb')\label{eq:easy-fal-rel-comp}\\
		\trm{(resp. }
		&\alpha_2\circ\alpha_1^{-1}: \rr\Gamma (Y'_\et, \bb{F}')\otimes^{\dl}_{\bb{Z}}A \longrightarrow \rr\Gamma (\fal_{Y' \to X'}^\et, \psi'_*\bb{F}'\otimes_{\bb{Z}}\falb') \trm{)}\label{eq:easy-fal-abs-comp}
	\end{align}
	the \emph{relative} (resp. \emph{absolute}) \emph{Faltings' comparison morphism} associated to $f:(Y' \to X')\to (Y \to X)$ and $\bb{F}'$. In this case, we say that \emph{Faltings' comparison morphisms exist}.
\end{mydefn}

\begin{mythm}[{\cite[Cor.6.9]{achinger2017wildram}, cf. \cite[4.4.2]{abbes2020suite}}]\label{thm:achinger}
	Let $\ca{O}_K$ be a strictly Henselian discrete valuation ring with fraction field $K$ of characteristic $0$ and residue field of characteristic $p$. We fix an algebraic closure $\overline{K}$ of $K$. Let $X$ be an $\ca{O}_K$-scheme of finite type, $\bb{F}$ a finite locally constant abelian sheaf on $X_{\overline{K},\et}$, $\psi:X_{\overline{K},\et} \to \fal_{X_{\overline{K}} \to X}^\et$ the morphism of sites defined in {\rm\ref{para:notation-psi-beta-sigma}}. Then, the canonical morphism $\psi_*\bb{F} \to \rr\psi_*\bb{F}$ is an isomorphism.
\end{mythm}

\begin{mycor}\label{cor:achinger}
	Let $\ca{O}_K$ be a strictly Henselian discrete valuation ring with fraction field $K$ of characteristic $0$ and residue field of characteristic $p$. We fix an algebraic closure $\overline{K}$ of $K$. Let $X$ be a coherent $\ca{O}_{\overline{K}}$-scheme, $Y=\spec(\overline{K})\times_{\spec(\ca{O}_{\overline{K}})} X$, $\bb{F}$ a finite locally constant abelian sheaf on $Y_\et$, $\psi:Y_\et \to \fal_{Y \to X}^\et$ the morphism of sites defined in {\rm\ref{para:notation-psi-beta-sigma}}. Then, the canonical morphism $\psi_*\bb{F} \to \rr\psi_*\bb{F}$ is an isomorphism.
\end{mycor}

We emphasize that we don't need any finiteness condition of $X$ over $\ca{O}_{\overline{K}}$ in \ref{cor:achinger}. In fact, one can replace $\ca{O}_{\overline{K}}$ by $\overline{\bb{Z}_p}$ without loss of generality, where $\overline{\bb{Z}_p}$ is the integral closure of $\bb{Z}_p$ in an algebraic closure of $\bb{Q}_p$. We keep working over $\ca{O}_{\overline{K}}$ only for the continuation of our usage of notation.

\begin{proof}[Proof of \ref{cor:achinger}]
	We take a directed inverse system $(X_\lambda \to \spec(\ca{O}_{K_\lambda}))_{\lambda \in \Lambda}$ of morphisms of finite type of schemes by Noetherian approximation, such that $K_\lambda$ is a finite field extension of $K$ and $\overline{K}=\bigcup_{\lambda\in\Lambda}K_\lambda$, and that the transition morphisms $X_{\lambda'} \to X_\lambda$ are affine and $X=\lim_{\lambda\in\Lambda} X_\lambda$ (cf. \cite[\href{https://stacks.math.columbia.edu/tag/09MV}{09MV}]{stacks-project}). For each $\lambda\in \Lambda$, we set $Y_\lambda=\spec(\overline{K}) \times_{\spec(\ca{O}_{K_\lambda})}X_\lambda$. Notice that $Y=\lim Y_\lambda$. Then, there exists an index $\lambda_0\in \Lambda$ and a finite locally constant abelian sheaf $\bb{F}_{\lambda_0}$ on $Y_{\lambda_0,\et}$ such that $\bb{F}$ is the pullback of $\bb{F}_{\lambda_0}$ by $Y_\et \to Y_{\lambda_0,\et}$ (cf. \cite[\href{https://stacks.math.columbia.edu/tag/09YU}{09YU}]{stacks-project}). Let $\bb{F}_\lambda$ be the pullback of $\bb{F}_{\lambda_0}$ by $Y_{\lambda,\et} \to Y_{\lambda_0,\et}$ for each $\lambda\geq \lambda_0$. Notice that $\ca{O}_{K_\lambda}$ also satisfies the conditions in \ref{thm:achinger}. Let $\psi_\lambda:Y_{\lambda,\et} \to \fal_{Y_\lambda \to X_\lambda}^\et$ be the morphism of sites defined in {\rm\ref{para:notation-psi-beta-sigma}}, $\varphi_\lambda: \fal_{Y \to X}^\et\to \fal_{Y_\lambda \to X_\lambda}^\et$ the morphism of sites defined by the transition morphism. Then, we have $\rr^q\psi_{\lambda*}\bb{F}_\lambda=0$ for each integer $q>0$ by \ref{thm:achinger}, and moreover 
	\begin{align}
		\rr^q\psi_* \bb{F}=\colim_{\lambda\geq \lambda_0} \varphi_\lambda^{-1} \rr^q\psi_{\lambda*}\bb{F}_\lambda=0
	\end{align}
	by \ref{prop:limit-fal-sites}, \cite[\Luoma{7}.5.6]{sga4-2} and \cite[\Luoma{6}.8.7.3]{sga4-2} whose conditions are satisfied because each object in each concerned site is quasi-compact.
\end{proof}

\begin{mylem}\label{lem:loc-sys-trans-bc}
	With the notation in {\rm\ref{para:easy-comparison-mor-rel}}, let $\bb{F}$ be a finite locally constant abelian sheaf on $Y_\et$. Then, the canonical morphism $f_\fal^{-1}\psi_*\bb{F}\to \psi'_*f_\et^{-1}\bb{F}$ is an isomorphism.
\end{mylem}
\begin{proof}
	The base change morphism $f_\fal^{-1}\psi_*\bb{F}\to \psi'_*f_\et^{-1}\bb{F}$ is the composition of the adjunction morphisms (\cite[\Luoma{17}.2.1.3]{sga4-3})
	\begin{align}
		f_\fal^{-1}\psi_*\bb{F}\to \psi'_*\psi'^{-1}(f_\fal^{-1}\psi_*\bb{F})=\psi'_*f_\et^{-1}(\psi^{-1}\psi_*\bb{F})\to \psi'_*f_\et^{-1}\bb{F}
	\end{align}
	which are both isomorphisms by \ref{prop:loc-sys-trans}.
\end{proof}

\begin{mypara}\label{para:notation-falcomp-rel}
	Let $K$ be a complete discrete valuation field of characteristic $0$ with valuation ring $\ca{O}_K$ whose residue field $k$ is algebraically closed (a condition required by \cite[4.1.3, 5.1.3]{abbes2020suite}) of characteristic $p>0$, $\overline{K}$ an algebraic closure of $K$, $\ca{O}_{\overline{K}}$ the integral closure of $\ca{O}_K$ in $\overline{K}$, $\eta= \spec (K)$, $\overline{\eta}=\spec(\overline{K})$, $S=\spec (\ca{O}_K)$, $\overline{S}=\spec(\ca{O}_{\overline{K}})$, $s=\spec (k)$. Remark that $\overline{K}$ is a pre-perfectoid field with valuation ring $\ca{O}_{\overline{K}}$ so we are also in the situation {\rm\ref{para:notation}}.
\end{mypara}

\begin{mypara}\label{para:notation-log}
	With the notation in {\rm\ref{para:notation-falcomp-rel}}, let $X$ be an $S$-scheme, $Y$ an open subscheme of the generic fibre $X_\eta$. We simply denote by $\scr{M}_X$ the compactifying log structure $\scr{M}_{X_\eta \to X}$ (\ref{para:log-compact}). Following \cite[\Luoma{3}.4.7]{abbes2016p}, we say that $Y\to X$ is \emph{adequate} over $\eta\to S$ if the following conditions are satisfied:
	\begin{enumerate}
		\renewcommand{\theenumi}{\roman{enumi}}
		\renewcommand{\labelenumi}{{\rm(\theenumi)}}
		\item $X$ is of finite type over $S$;
		\item Any point of the special fibre $X_s$ admits an \'etale neighborhood $U$ such that $U_\eta\to \eta$ is smooth and that $U_\eta\setminus Y$ is the support of a strict normal crossings divisor on $U_\eta$;
		\item $(X,\scr{M}_{Y \to X})$ is a fine log scheme and the structure morphism $(X,\scr{M}_{Y \to X}) \to (S,\scr{M}_{S})$ is smooth and saturated.
	\end{enumerate}
	In this case, $(X,\scr{M}_{Y \to X}) \to (S,\scr{M}_{S})$ is adequate in the sense of \cite[\Luoma{3}.4.7]{abbes2016p}. We remark that for any adequate $(S,\scr{M}_{S})$-log scheme $(X,\scr{M})$, $X^{\mrm{tr}}\to X$ is adequate over $\eta\to S$ and $(X,\scr{M})=(X,\scr{M}_{X^{\mrm{tr}}\to X})$ (cf. \ref{para:log-reg}, \ref{para:log-dvr}).
	Note that if $Y \to X$ is semi-stable over $\eta\to S$ then it is adequate (cf. \ref{para:log-dvr}).
\end{mypara}

\begin{mypara}\label{para:abbes-gros}
	We recall the statement of Faltings' main $p$-adic comparison theorem following Abbes-Gros \cite{abbes2020suite}. We take the notation and assumptions in {\rm\ref{para:notation-falcomp-rel}}. Firstly, recall that for any adequate open immersion of schemes $X^\circ \to X$ over $\eta\to S$ and any finite locally constant abelian sheaf $\bb{F}$ on $X^\circ_{\overline{\eta},\et}$, the canonical morphism $\psi_*\bb{F} \to \rr\psi_*\bb{F}$ is an isomorphism, where $\psi:X^\circ_{\overline{\eta},\et} \to \fal_{X_{\overline{\eta}}^\circ \to X}^\et$ is the morphism of sites defined in \ref{para:notation-psi-beta-sigma} (\cite[4.4.2]{abbes2020suite}).
	
	Let $(X'^{\triangleright}\to X') \to (X^\circ \to X)$ be a morphism of adequate open immersions of schemes over $\eta \to S$ such that $X' \to X$ is projective and that the induced morphism $(X',\scr{M}_{X'^{\triangleright} \to X'})\to (X,\scr{M}_{X^\circ \to X})$ is smooth and saturated. Let $Y'=\overline{\eta}\times_{\eta} X'^{\triangleright}$, $Y=\overline{\eta}\times_{\eta} X^\circ$, $f:(Y' \to X') \to (Y\to X)$ the natural morphism, $\bb{F}'$ a finite locally constant abelian sheaf on $Y'_\et$. By the first paragraph, we have the relative Faltings' comparison morphism associated to $f$ and $\bb{F}'$ \eqref{eq:easy-fal-rel-comp},
	\begin{align}
		(\rr\psi_*\rr f_{\et *} \bb{F}')\otimes^{\dl}_{\bb{Z}}\falb \longrightarrow \rr f_{\fal *} (\psi'_*\bb{F}'\otimes_{\bb{Z}}\falb').
	\end{align}
	Remark that under our assumption, the sheaf $\rr^q f_{\et *} \bb{F}'$ on $Y_\et$ is finite locally constant for each integer $q$ (\cite[2.2.25]{abbes2020suite}).
\end{mypara}

\begin{mythm}[{\cite[Thm.6, page 266]{faltings2002almost}, \cite[5.7.3]{abbes2020suite}}]\label{thm:abess-gros}
	Under the assumptions in {\rm\ref{para:abbes-gros}} and with the same notation, the relative Faltings' comparison morphism associated to $f$ and $\bb{F}'$ is an almost isomorphism in the derived category $\dd(\ca{O}_{\overline{K}}\module_{\fal_{Y \to X}^\et})$ {\rm(\ref{para:almost-derived})}, and it induces an almost isomorphism 
	\begin{align}
		(\psi_*\rr^q f_{\et *} \bb{F}')\otimes_{\bb{Z}}\falb \longrightarrow \rr^q f_{\fal *} (\psi'_*\bb{F}'\otimes_{\bb{Z}}\falb')
	\end{align}
	of $\ca{O}_{\overline{K}}$-modules for each integer $q$.
\end{mythm}

\begin{myprop}\label{prop:fal-comp-hypercov}
	With the natation in {\rm\ref{para:notation-falcomp-rel}}, let $A$ be an absolutely integrally closed valuation ring of height $1$ extension of $\ca{O}_{\overline{K}}$, $X$ a proper $A$-scheme of finite presentation, $Y= \spec(A[1/p]) \times_{\spec(A)} X$, $\bb{F}$ a finite locally constant abelian sheaf on $Y_\et$. Then, there exists a proper surjective morphism $X' \to X$ of finite presentation such that the relative and absolute Faltings' comparison morphisms associated to $f':(Y' \to X') \to (\spec(A[1/p])\to \spec(A))$ and $\bb{F}'$ (which exist by {\rm\ref{cor:achinger}}) are almost isomorphisms, where $Y'=Y\times_X X'$ and $\bb{F}'$ is the pullback of $\bb{F}$ on $Y'_\et$.
\end{myprop}
\begin{proof}
	Since the underlying topological space of $X$ is Noetherian by \ref{lem:noetherian}, each irreducible component $Z$ of $X$ admits a closed subscheme structure such that $Z\to X$ is of finite presentation (\cite[\href{https://stacks.math.columbia.edu/tag/01PH}{01PH}]{stacks-project}). After replacing $X$ by the disjoint union of its irreducible components, we may assume that $X$ is irreducible. If $Y$ is empty, then we take $X'=X$ and thus the relative (resp. absolute) Faltings' comparison morphism associated to $f'$ and $\bb{F}'$ is an isomorphism between zero objects. If $Y$ is not empty, then we are in the situation of \ref{prop:val-lim-ss}.(\ref{prop:val-lim-ss-car}) by taking $\ca{O}_L=A$. With the notation in \ref{prop:val-lim-ss}, we check that the morphism $X' \to X$ meets our requirements. We set $\eta_\lambda=\spec(K_\lambda)$, $S_\lambda=\spec(\ca{O}_{K_\lambda})$, $T_{\lambda,\overline{\eta_\lambda}} =\overline{\eta}\times_{\eta_\lambda} U_\lambda$, $X'_{\lambda,\overline{\eta_\lambda}}=\overline{\eta}\times_{\eta_\lambda} Y'_\lambda$, and denote by $f'_\lambda:(X'_{\lambda,\overline{\eta_\lambda}} \to X'_\lambda) \to (T_{\lambda,\overline{\eta_\lambda}} \to T_\lambda)$ the natural morphism. We obtain a commutative diagram
	\begin{align}
		\xymatrix{
			\fal_{Y' \to X'}^\et \ar[rrr]^-{g_{\lambda, \fal}} \ar[ddd]_-{f'_{\fal}} &&& \fal_{X'_{\lambda,\overline{\eta_\lambda}} \to X'_\lambda}^\et \ar[ddd]^-{f'_{\lambda,\fal}}\\
			&Y'_\et \ar[r]^-{g_{\lambda,\et}}\ar[d]_-{f'_\et} \ar[ul]_-{\psi'} & X'_{\lambda,\overline{\eta_\lambda},\et}\ar[d]^-{f'_{\lambda,\et}} \ar[ur]^-{\psi_\lambda'}& \\
			&\spec(A[1/p])_\et \ar[r]^-{h_{\lambda,\et}} \ar[dl]_-{\psi} & T_{\lambda,\overline{\eta_\lambda},\et} \ar[dr]^-{\psi_\lambda}& \\
			\fal_{\spec(A[1/p]) \to \spec(A)}^\et \ar[rrr]^-{h_{\lambda,\fal}} &&& \fal_{T_{\lambda,\overline{\eta_\lambda}} \to T_\lambda}^\et
		}
	\end{align}
	Firstly notice that the site $Y'_\et$ (resp. $\spec(A[1/p])_\et$) is the limit of the sites $X'_{\lambda,\overline{\eta_\lambda},\et}$ (resp. $T_{\lambda,\overline{\eta_\lambda},\et}$) and the site $\fal_{Y' \to X'}^\et$ (resp. $\fal_{\spec(A[1/p]) \to \spec(A)}^\et$) is the limit of the sites $\fal_{X'_{\lambda,\overline{\eta_\lambda}} \to X'_\lambda}^\et$ (resp. $\fal_{T_{\lambda,\overline{\eta_\lambda}} \to T_\lambda}^\et$) (\cite[\Luoma{7}.5.6]{sga4-2} and \ref{prop:limit-fal-sites}). There exists an index $\lambda_0\in \Lambda$ and a finite locally constant abelian sheaf $\bb{F}'_{\lambda_0}$ on $X'_{\lambda_0,\overline{\eta_{\lambda_0}},\et}$ such that $\bb{F}'$ is the pullback of $\bb{F}'_{\lambda_0}$ by $Y'_\et \to X'_{\lambda_0,\overline{\eta_{\lambda_0}},\et}$ (cf. \cite[\href{https://stacks.math.columbia.edu/tag/09YU}{09YU}]{stacks-project}). Let $\bb{F}'_\lambda$ be the pullback of $\bb{F}'_{\lambda_0}$ by $X'_{\lambda,\overline{\eta_{\lambda}},\et} \to X'_{\lambda_0,\overline{\eta_{\lambda_0}},\et}$ for each $\lambda\geq \lambda_0$. We also have $\falb'=\colim g_{\lambda,\fal}^{-1}\falb'$ (resp. $\falb=\colim h_{\lambda,\fal}^{-1}\falb$) by \ref{prop:limit-fal-sites}. According to \cite[\Luoma{6}.8.7.3]{sga4-2}, whose conditions are satisfied because each object in each concerned site is quasi-compact, there are canonical isomorphisms for each integer $q$,
	\begin{align}
		(\rr^q (\psi\circ f'_\et)_*\bb{F}') \otimes_{\bb{Z}} \falb &\iso \colim h_{\lambda,\fal}^{-1}((\rr^q (\psi_{\lambda}\circ f'_{\lambda,\et})_*\bb{F}'_{\lambda}) \otimes_{\bb{Z}}\falb),\label{eq:10.9.5}\\
		\rr^q f'_{\fal *}(\psi'_*\bb{F}'\otimes_{\bb{Z}} \falb') & \iso \colim h_{\lambda,\fal}^{-1} \rr^q f'_{\lambda, \fal *}(\psi'_{\lambda *}\bb{F}'_{\lambda}\otimes_{\bb{Z}} \falb').\label{eq:10.9.6}
	\end{align}
	On the other hand, $(X'_\lambda,\scr{M}_{X'_\lambda}) \to (T_\lambda,\scr{M}_{T_\lambda})$ is a smooth and saturated morphism of adequate $(S_\lambda,\scr{M}_{S_\lambda})$-log schemes with $X'_\lambda \to T_\lambda$ projective for each $\lambda\in \Lambda$ by construction. Thus, we are in the situation of \ref{thm:abess-gros}, which implies that the relative Faltings' comparison morphism associated to $f'_\lambda$ and $\bb{F}'_\lambda$,
	\begin{align}
		(\rr^q (\psi_{\lambda}\circ f'_{\lambda,\et})_*\bb{F}'_{\lambda}) \otimes_{\bb{Z}}\falb\longrightarrow  \rr^q f'_{\lambda, \fal *}(\psi'_{\lambda*}\bb{F}'_{\lambda}\otimes_{\bb{Z}} \falb')
	\end{align} 
	is an almost isomorphism for each $\lambda\geq \lambda_0$. Combining with \eqref{eq:10.9.5} and \eqref{eq:10.9.6}, we see that the relative Faltings' comparison morphism associated to $f'$ and $\bb{F}'$, 
	\begin{align}\label{eq:10.9.3}
		\rr\psi_*(\rr f'_{\et *}\bb{F}') \otimes_{\bb{Z}}^{\dl} \falb \longrightarrow\rr f'_{\fal *}(\psi'_*\bb{F}'\otimes_{\bb{Z}} \falb'),
	\end{align}
	is an almost isomorphism (and thus so is the absolute one).
\end{proof}

\begin{mycor}\label{cor:fal-comp-hypercov}
	Under the assumptions in {\rm\ref{prop:fal-comp-hypercov}} and with the same notation, there exists a proper hypercovering $X_\bullet\to X$ of coherent schemes {\rm(\cite[\href{https://stacks.math.columbia.edu/tag/0DHI}{0DHI}]{stacks-project})} such that for each degree $n$, the relative and absolute Faltings' comparison morphisms associated to $f_n:(Y_n \to X_n) \to (\spec(A[1/p])\to \spec(A))$ and $\bb{F}_n$ (which exist by {\rm\ref{cor:achinger}}) are almost isomorphisms, where $Y_n=Y\times_X X_n$ and $\bb{F}_n$ is the pullback of $\bb{F}$ by the natural morphism $Y_{n,\et} \to Y_\et$. In particular, $Y_\bullet\to Y$ is a proper hypercovering and $X_\bullet^{Y_\bullet}\to X^Y$ is a hypercovering in $\falh_{\overline{\eta}\to \overline{S}}$.
\end{mycor}
\begin{proof}
	Let $\scr{C}$ be the category of proper $A$-schemes of finite presentation endowed with the pretopology formed by families of morphisms $\{f_i:X_i \to X\}_{i \in I}$ with $I$ finite and $X=\bigcup_{i \in I} f_i(X_i)$. Consider the functor $u^+:\scr{C} \to \falh_{\spec(A[1/p]) \to \spec(A)}$ sending $X$ to $X^{Y}$ where $Y=\spec(A[1/p])\times_{\spec(A)} X$. It is well-defined by \ref{lem:relative-normal}.(\ref{lem:relative-normal-norm}) and commutes with fibred products by \ref{lem:relative-normal-limit} and continuous by \ref{prop:normal-v-cov}. Lemma \ref{prop:fal-comp-hypercov} allows us to take a hypercovering $X_\bullet \to X$ in $\scr{C}$ meeting our requirement by \cite[\href{https://stacks.math.columbia.edu/tag/094K}{094K} and \href{https://stacks.math.columbia.edu/tag/0DB1}{0DB1}]{stacks-project}. We see that $Y_\bullet\to Y$ is a proper hypercovering and that $X_\bullet^{Y_\bullet}\to X^Y$ is a hypercovering in $\falh_{\overline{\eta}\to \overline{S}}$ by the properties of $u^+$ (\cite[\href{https://stacks.math.columbia.edu/tag/0DAY}{0DAY}]{stacks-project}).
\end{proof}

\begin{mylem}\label{lem:fal-comp-reduce-val}
	Let $\overline{\bb{Z}_p}$ be the integral closure of $\bb{Z}_p$ in an algebraic closure of $\bb{Q}_p$, $A$ a $\overline{\bb{Z}_p}$-algebra which is an absolutely integrally closed valuation ring, $X$ a proper $A$-scheme of finite presentation, $Y= \spec(A[1/p]) \times_{\spec(A)} X$, $\bb{F}$ a finite locally constant abelian sheaf on $Y_\et$. Let $A'=((A/\cap_{n>0}p^nA)_{\sqrt{pA}})^\wedge$ ($p$-adic completion), $X'=X_{A'}$, $Y'=Y_{A'}$, $\bb{F}'$ the pullback of $\bb{F}$ on $Y'_\et$. Then, the following statements are equivalent:
	\begin{enumerate}
		\renewcommand{\labelenumi}{{\rm(\theenumi)}}
		\item The absolute Faltings' comparison morphism associated to $f:(Y\to X) \to (\spec(A[1/p]) \to \spec(A))$ and $\bb{F}$ (which exists by {\rm\ref{cor:achinger}}) is an almost isomorphism.
		\item The absolute Faltings' comparison morphism associated to $f':(Y'\to X') \to (\spec(A'[1/p]) \to \spec(A'))$ and $\bb{F}'$ (which exists by {\rm\ref{cor:achinger}}) is an almost isomorphism.
	\end{enumerate}
\end{mylem}
\begin{proof}
	If $p$ is zero (resp. invertible) in $A$, then the absolute Faltings' comparison morphisms are both isomorphisms between zero objects, since $Y$ and $Y'$ are empty (resp. the abelian sheaves $\bb{F}$ and $\bb{F}'$ are zero after inverting $p$). Thus, we may assume that $p$ is a nonzero element of the maximal ideal of $A$. Notice that $\cap_{n>0}p^nA$ is the maximal prime ideal of $A$ not containing $p$ and that $\sqrt{pA}$ is the minimal prime ideal of $A$ containing $p$ (\ref{para:notation-val}). Thus, $(A/\cap_{n>0}p^nA)_{\sqrt{pA}}$ is an absolutely integrally closed valuation ring of height $1$ extension of $\overline{\bb{Z}_p}$ (\ref{para:notation-val}) and thus so is its $p$-adic completion $A'$. 
	
	We denote by $u:(Y' \to X') \to (Y \to X)$ the natural morphism. We have $\bb{F}'=u_\et^{-1}\bb{F}$. The natural morphisms in \eqref{eq:fal-abs-comp} induce a commutative diagram
	\begin{align}
		\xymatrix{
			\rr\Gamma (Y_\et, \bb{F})\otimes^{\dl}_{\bb{Z}}A\ar[d]^-{\gamma_1} & \rr\Gamma (\fal_{Y \to X}^\et, \psi_*\bb{F}) \otimes^{\dl}_{\bb{Z}}A \ar[l]_-{\alpha_1}\ar[r]^-{\alpha_2}\ar[d]^-{\gamma_2}& \rr\Gamma (\fal_{Y \to X}^\et, \psi_*\bb{F}\otimes_{\bb{Z}}\falb)\ar[d]^-{\gamma_3}\\
			\rr\Gamma (Y'_\et, \bb{F}')\otimes^{\dl}_{\bb{Z}}A' & \rr\Gamma (\fal_{Y' \to X'}^\et, \psi'_*\bb{F}') \otimes^{\dl}_{\bb{Z}}A' \ar[l]_-{\alpha_1'}\ar[r]^-{\alpha_2'}& \rr\Gamma (\fal_{Y' \to X'}^\et, \psi'_*\bb{F}'\otimes_{\bb{Z}}\falb')
		}
	\end{align}
	where $\gamma_1$ is induced by the canonical morphism $\bb{F} \to \rr u_{\et *}u_{\et}^{-1}\bb{F}$, and $\gamma_2$ (resp. $\gamma_3$) is induced by the composition of $\psi_*\bb{F} \to \rr u_{\fal *} u_{\fal}^{-1}\psi_*\bb{F} \to \rr u_{\fal *}\psi'_*u_{\et}^{-1}\bb{F}$ (resp. and by the canonical morphism $\falb \to \rr u_{\fal *}\falb'$). Since $\alpha_1$ and $\alpha_1'$ are isomorphisms by \ref{cor:achinger}, it suffices to show that $\gamma_1$ and $\gamma_3$ are almost isomorphisms.
	
	Since $A/\cap_{n>0}p^nA\to (A/\cap_{n>0}p^nA)_{\sqrt{pA}}$ is injective whose cokernel is killed by $\sqrt{pA}$ (\ref{lem:tech1}), the morphism $A \to A'$ induces an almost isomorphism $A/p^nA \to A'/p^nA'$ for each $n$. Then, for any torsion abelian group $M$, the natural morphism $M\otimes_{\bb{Z}} A \to M \otimes_{\bb{Z}} A'$ is an almost isomorphism. Therefore, $\gamma_1$ is an almost isomorphism by the proper base change theorem over the strictly Henselian local ring $A[1/p]$ (\cite[\Luoma{12} 5.5, 5.4]{sga4-3}). For $\gamma_3$, it suffices to show that the canonical morphism $\psi_*\bb{F}\otimes \falb \to \rr u_{\fal*} (\psi'_*\bb{F}'\otimes \falb')$ is an almost isomorphism. The problem is local on $\fal_{Y \to X}^\et$, thus we may assume that $\psi_*\bb{F}$ is the constant sheaf with value $\bb{Z}/p^n\bb{Z}$ by \ref{prop:loc-sys-trans}. Then, $\psi'_*\bb{F}'$ is also the constant sheaf with value $\bb{Z}/p^n\bb{Z}$ by \ref{lem:loc-sys-trans-bc}. Let $V \to U$ be an object of $\fal_{Y \to X}^\proet$ such that $U^V=\spec (R)$ is the spectrum of an $\overline{\bb{Z}_p}$-algebra $R$ which is almost pre-perfectoid. Since the almost isomorphisms $R/p^n \to (R\otimes_A A')/p^n$ ($n\geq 1$) induces an almost isomorphism of the $p$-adic completions $\widehat{R}\to R\widehat{\otimes}_A A'$, the $\overline{\bb{Z}_p}$-algebra $R\otimes_A A'$ is still almost pre-perfectoid (\ref{defn:pre-alg}). The pullback of $V \to U$ in $\fal_{Y' \to X'}^\proet$ is the object $V_{A'} \to U_{A'}$ and $U_{A'}^{V_{A'}}$ is the spectrum of the integral closure $R'$ of $R\otimes_A A'$ in $R\otimes_A A'[1/p]$. Since $R\otimes_A A'$ is almost pre-perfectoid, $R'$ is also almost pre-perfectoid and the morphism $(R\otimes_A A')/p^n \to R'/p^n$ is an almost isomorphism by \ref{cor:pre-perf-intclos}. Therefore, the morphism $\falb/p^n\falb \to \rr u_{\fal *}(\falb'/p^n\falb')$ is an almost isomorphism by \ref{cor:nu-falb}, \ref{prop:coh-vanish-preperf} and \ref{thm:falh-proet}.
\end{proof}

\begin{mythm}\label{thm:fal-comp-abs}
	Let $\overline{\bb{Z}_p}$ be the integral closure of $\bb{Z}_p$ in an algebraic closure of $\bb{Q}_p$, $A$ a $\overline{\bb{Z}_p}$-algebra which is an absolutely integrally closed valuation ring, $X$ a proper $A$-scheme of finite presentation, $Y= \spec(A[1/p]) \times_{\spec(A)} X$, $\bb{F}$ a finite locally constant abelian sheaf on $Y_\et$. Then, the absolute Faltings' comparison morphism associated to $f:(Y \to X) \to (\spec(A[1/p])\to \spec(A))$ and $\bb{F}$ \eqref{eq:easy-fal-abs-comp} (which exists by {\rm\ref{cor:achinger}}),
	\begin{align}\label{eq:fal-comp-abs}
		\rr\Gamma (Y_\et, \bb{F})\otimes^{\dl}_{\bb{Z}}A \longrightarrow \rr\Gamma (\fal_{Y \to X}^\et, \psi_*\bb{F}\otimes_{\bb{Z}}\falb),
	\end{align}
	is an almost isomorphism in $\dd(\overline{\bb{Z}_p}\module)$ {\rm(\ref{para:almost-derived})}.
\end{mythm}
\begin{proof}
	Let $K$ be the $p$-adic completion of the maximal unramified extension of $\bb{Q}_p$.	By \ref{lem:fal-comp-reduce-val}, we may assume that $A$ is a valuation ring of height $1$ extension of $\ca{O}_{\overline{K}}$. Let $X_\bullet \to X$ be the proper hypercovering of coherent schemes constructed in \ref{cor:fal-comp-hypercov}. For each degree $n$ the canonical morphisms \eqref{eq:easy-fal-abs-comp}
	\begin{align}\label{eq:10.18.2}
		\rr\Gamma(Y_{n,\et}, \bb{F}_n) \otimes_{\bb{Z}}^{\dl} A \longleftarrow \rr\Gamma(\fal_{Y_n \to X_n}^\et,\psi_{n*}\bb{F}_n) \otimes_{\bb{Z}}^{\dl} A \longrightarrow \rr\Gamma(\fal_{Y_n \to X_n}^\et,\psi_{n*}\bb{F}_n\otimes_{\bb{Z}} \falb)
	\end{align}
	are an isomorphism and an almost isomorphism, where $\bb{F}_n$ is the pullback of $\bb{F}$ by the natural morphism $Y_{n,\et} \to Y_\et$. Consider the commutative diagram 
	\begin{align}
		\xymatrix{
			\rr\Gamma(Y_{\et}, \bb{F})\otimes_{\bb{Z}}^{\dl} A \ar[d]& \rr\Gamma(\fal_{Y \to X}^\et, \psi_{ *}\bb{F})\otimes_{\bb{Z}}^{\dl} A  \ar[r]^-{\alpha_2}\ar[d]\ar[l]_-{\alpha_1}& \rr\Gamma(\fal_{Y \to X}^\et, \psi_{ *}\bb{F}\otimes_{\bb{Z}}\falb)\ar[d]\\
			\rr\Gamma(Y_{\bullet,\et}, \bb{F}_\bullet)\otimes_{\bb{Z}}^{\dl} A & \rr\Gamma(\fal_{Y_{\bullet} \to X_\bullet}^\et, \psi_{\bullet *}\bb{F}_\bullet)\otimes_{\bb{Z}}^{\dl} A \ar[l]_-{\alpha_{1\bullet}}\ar[r]^-{\alpha_{2\bullet}}& \rr\Gamma(\fal_{Y_{\bullet} \to X_\bullet}^\et, \psi_{\bullet *}\bb{F}_\bullet\otimes_{\bb{Z}}\falb_\bullet)
		}
	\end{align}
	where $\bb{F}_\bullet=(\bb{F}_n)_{[n]\in\ob(\Delta)}$ with the notation in \ref{para:presheafonE}. By the functorial spectral sequence of simplicial sites (\cite[\href{https://stacks.math.columbia.edu/tag/09WJ}{09WJ}]{stacks-project}), we deduce from \eqref{eq:10.18.2} that $\alpha_{1\bullet}$ is an isomorphism and $\alpha_{2\bullet}$ is an almost isomorphism. Since $\alpha_1$ is an isomorphism by \ref{cor:achinger}, it remains to show that the left vertical arrow is an isomorphism and the right vertical arrow is an almost isomorphism. 
	
	We denote by $a:\fal_{Y_{\bullet} \to X_\bullet}^\et\to \fal_{Y \to X}^\et$ the augmentation of simplicial site and by $a_n:\fal_{Y_{n} \to X_n}^\et\to \fal_{Y \to X}^\et$ the natural morphism of sites. Notice that $a^{-1}\psi_*\bb{F}=(a_n^{-1}\psi_*\bb{F})_{[n]\in\ob(\Delta)}=(\psi_{n *}\bb{F}_n)_{[n]\in\ob(\Delta)}=\psi_{\bullet *}\bb{F}_\bullet$ by \ref{lem:loc-sys-trans-bc} (\cite[\href{https://stacks.math.columbia.edu/tag/0D70}{0D70}]{stacks-project}).
	Since $X_\bullet^{Y_\bullet} \to X^Y$ forms a hypercovering in $\falh_{\overline{\eta}\to \overline{S}}$, the right vertical arrow is an almost isomorphism by \ref{prop:loc-sys-trans} and \ref{cor:resolution}. Finally, the left vertical arrow is an isomorphism by the cohomological descent for \'etale cohomology \cite[\href{https://stacks.math.columbia.edu/tag/0DHL}{0DHL}]{stacks-project}.
\end{proof}

\section{Faltings' Main $p$-adic Comparison Theorem: the Relative Case for More General Coefficients}\label{sec:rel-comp}

\begin{mypara}\label{para:notation-psi-beta-2}
	Let $Y \to X$ be a morphism of coherent schemes such that $Y \to X^Y$ is an open immersion. We obtain from  \ref{para:mor-sigma-psi-N}, \ref{par:epsilon-mor} and \ref{para:notation-psi-beta-sigma} a commutative diagram of sites
	\begin{align}\label{diag:beta-psi-2}
		\xymatrix{
			(\schqcqs_{/Y})_{\mrm{v}}\ar[r]^-{a}\ar[d]_-{\Psi} &Y_\et\ar[d]^-{\psi}\ar[dr]^-{\rho}&\\
			\falh_{Y \to X^Y}\ar[r]^-{\varepsilon} & \fal_{Y \to X}^\et \ar[r]^-{\beta}& Y_{\fet}
		}
	\end{align}
	where $a: (\schqcqs_{/Y})_{\mrm{v}} \to Y_\et$ and $\rho: Y_\et \to Y_\fet$ are defined by the inclusion functors.
\end{mypara}

\begin{mylem}\label{lem:loc-sys-trans-2}
	With the notation in {\rm\ref{para:notation-psi-beta-2}}, for any finite locally constant abelian sheaf $\bb{F}$ on $Y_\et$, the canonical morphism $\varepsilon^{-1}\psi_*\bb{F}\to \Psi_*a^{-1}\bb{F}$ is an isomorphism.
\end{mylem}
\begin{proof}
	The base change morphism $\varepsilon^{-1}\psi_*\bb{F}\to \Psi_*a^{-1}\bb{F}$ is the composition of the adjunction morphisms (\cite[\Luoma{17}.2.1.3]{sga4-3})
	\begin{align}
		\varepsilon^{-1}\psi_*\bb{F}\to \Psi_*\Psi^{-1}(\varepsilon^{-1}\psi_*\bb{F})=\Psi_*a^{-1}(\psi^{-1}\psi_*\bb{F})\to \Psi_*a^{-1}\bb{F}
	\end{align}
	which are both isomorphisms by \ref{prop:Psi-const-equal}.(\ref{prop:Psi-const-equal-const}) and \ref{prop:loc-sys-trans}.
\end{proof}

\begin{mypara}\label{para:mor-rel-notation}
	We fix an algebraic closure $\overline{\bb{Q}_p}$ of the $p$-adic number field $\bb{Q}_p$ and we denote by $\overline{\bb{Z}_p}$ the integral closure of $\bb{Z}_p$ in $\overline{\bb{Q}_p}$. We set $\eta= \spec (\bb{Q}_p)$, $\overline{\eta}=\spec(\overline{\bb{Q}_p})$, $S=\spec (\bb{Z}_p)$, $\overline{S}=\spec(\overline{\bb{Z}_p})$. Remark that $\overline{\bb{Q}_p}$ is a pre-perfectoid field with valuation ring $\overline{\bb{Z}_p}$ so we are also in the situation {\rm\ref{para:notation}}. Let $f:(Y'\to X')\to (Y\to X)$ be a Cartesian morphism of morphisms of coherent schemes with a Cartesian morphism $(Y\to X^Y)\to (\overline{\eta} \to \overline{S})$ (then, $Y'\to X'^{Y'}$ is Cartesian over $\overline{\eta} \to \overline{S}$ by \ref{lem:relative-normal}.(\ref{lem:relative-normal-norm})). Thus, $X^Y$ and $X'^{Y'}$ are objects of $\falh_{\overline{\eta} \to \overline{S}}$. Consider the following commutative diagram of sites associated to $f$.
	\begin{align}\label{diam:sites}
		\xymatrix{
			Y'_\et \ar[d]_-{f_\et} \ar@/^2pc/[rrr]^-{\psi'}& (\schqcqs_{/Y'})_{\mrm{v}} \ar[l]_-{a'} \ar[r]^-{\Psi'} \ar[d]_-{f_{\mrm{v}}}& \falh_{Y' \to X'^{Y'}}\ar[r]^-{\varepsilon'} \ar[d]^-{f_{\falh}}& \fal_{Y' \to X'}^\et\ar[d]^-{f_\fal}\\
			Y_\et \ar@/_2pc/[rrr]^-{\psi}& (\schqcqs_{/Y})_{\mrm{v}} \ar[l]_-{a} \ar[r]^-{\Psi} & \falh_{Y \to X^{Y}}\ar[r]^-{\varepsilon} & \fal_{Y \to X}^\et
		}
	\end{align}
\end{mypara}

\begin{mypara}\label{para:comparison-mor-rel-bc}
	Following \ref{para:mor-rel-notation}, let $g: (\widetilde{Y}\to \widetilde{X})\to (Y\to X)$ be a morphisms of coherent schemes such that $\widetilde{Y}\to \widetilde{X}^{\widetilde{Y}}$ is also Cartesian over $\overline{\eta} \to \overline{S}$. We denote by $g':(\widetilde{Y'}\to \widetilde{X'})\to (Y'\to X')$ the base change of $g$ by $f$, and denote by $\widetilde{f}: (\widetilde{Y'}\to \widetilde{X'})\to (\widetilde{Y}\to \widetilde{X})$ the natural morphism which is Cartesian by base change. Thus, $\widetilde{X}^{\widetilde{Y}}$ and $\widetilde{X'}^{\widetilde{Y'}}$ are also objects of $\falh_{\overline{\eta} \to \overline{S}}$. We write the diagram \eqref{diam:sites} associated $\widetilde{f}$ equipping all labels with tildes.
\end{mypara}

\begin{mylem}\label{lem:comparison-mor-rel}
	With the notation in {\rm\ref{para:mor-rel-notation}} and  {\rm\ref{para:comparison-mor-rel-bc}}, let $\bb{F}'$ be a finite locally constant abelian sheaf on $Y'_\et$ and we set $\scr{F}'=\Psi'_*a'^{-1}\bb{F}'$. Let $\widetilde{X}$ be an object of $\falh_{Y\to X^Y}$, $\widetilde{Y}=\overline{\eta}\times_{\overline{S}} \widetilde{X}$, $\widetilde{\bb{F}'}=g'^{-1}_{\et}\bb{F}'$, $q$ an integer.
	\begin{enumerate}
		\renewcommand{\labelenumi}{{\rm(\theenumi)}}
		\item The sheaf $\rr^q f_{\falh *}\scr{F}'$ on $\falh_{Y \to X^{Y}}$ is canonically isomorphic to the sheaf associated to the presheaf $\widetilde{X}\mapsto H^q_{\et}(\widetilde{Y'},\widetilde{\bb{F}'})$.\label{lem:comparison-mor-rel-1}
		\item The sheaf $\rr^q f_{\falh *}(\scr{F}'\otimes_{\bb{Z}}\falhb')$ on $\falh_{Y \to X^{Y}}$ is canonically almost isomorphic to the sheaf associated to the presheaf $\widetilde{X}\mapsto H^q(\fal_{\widetilde{Y'}\to \widetilde{X'}}^\et,\widetilde{\psi'}_*\widetilde{\bb{F}'}\otimes_{\bb{Z}}\falb')$.
		\item The canonical morphism $(\rr^q f_{\falh *}\scr{F}')\otimes_{\bb{Z}}\falhb\to (\rr^q f_{\falh *}\scr{F}'\otimes_{\bb{Z}}\falhb')$ is compatible with the canonical morphisms $H^q_{\et}(\widetilde{Y'},\widetilde{\bb{F}'})\otimes_{\bb{Z}} R \stackrel{\alpha_1}{\longleftarrow} H^q(\fal_{\widetilde{Y'}\to \widetilde{X'}}^\et,\widetilde{\psi'}_*\widetilde{\bb{F}'})\otimes_{\bb{Z}} R\stackrel{\alpha_2}{\longrightarrow} H^q(\fal_{\widetilde{Y'}\to \widetilde{X'}}^\et,\widetilde{\psi'}_*\widetilde{\bb{F}'}\otimes_{\bb{Z}}\falb')$, where $R=\falb(\widetilde{Y}\to \widetilde{X})$ {\rm(cf.\ref{eq:fal-abs-comp})}.
	\end{enumerate}
\end{mylem}
\begin{proof}
	Let $\widetilde{\scr{F}'}$ be the restriction of $\scr{F}'$ on $\falh_{\widetilde{Y'}\to \widetilde{X'}^{\widetilde{Y'}}}$. We have $\widetilde{\scr{F}'}=\widetilde{\Psi'}_*\widetilde{a'}^{-1}\widetilde{\bb{F}'}$. We set $\widetilde{\bb{L}'}=\widetilde{\psi'}_*\widetilde{\bb{F}'}$ which is a finite locally constant abelian sheaf on $\fal_{\widetilde{Y'} \to \widetilde{X'}}^\et$ by \ref{prop:loc-sys-trans}. Notice that the canonical morphisms $\widetilde{\psi'}^{-1}\widetilde{\bb{L}'}\to \widetilde{\bb{F}'}$ and $\widetilde{\varepsilon'}^{-1}\widetilde{\bb{L}'} \to \widetilde{\scr{F}'}$ are isomorphisms by \ref{prop:loc-sys-trans} and \ref{lem:loc-sys-trans-2} respectively.
	
	(1) It follows from the canonical isomorphisms
	\begin{align}
		H^q(I_{\widetilde{Y'}\to \widetilde{X'}^{\widetilde{Y'}}},\widetilde{\varepsilon'}^{-1}\widetilde{\bb{L}'})\stackrel{\gamma_1}{\longrightarrow} H^q_{\mrm{v}}(\widetilde{Y'},\widetilde{\Psi'}^{-1}\widetilde{\varepsilon'}^{-1}\widetilde{\bb{L}'})= H^q_{\mrm{v}}(\widetilde{Y'},\widetilde{a'}^{-1}\widetilde{\psi'}^{-1}\widetilde{\bb{L}'}) \stackrel{\gamma_2}{\longleftarrow} H^q_{\et}(\widetilde{Y'},\widetilde{\psi'}^{-1}\widetilde{\bb{L}'}),
	\end{align}
	where $\gamma_1$ is induced by the canonical isomorphism $\widetilde{\varepsilon'}^{-1}\widetilde{\bb{L}'}\iso \rr\widetilde{\Psi'}_*\widetilde{\Psi'}^{-1}\widetilde{\varepsilon'}^{-1}\widetilde{\bb{L}'}$ (\ref{prop:Psi-const-equal}.(\ref{prop:Psi-const-equal-const})), and $\gamma_2$ is induced by the canonical isomorphism $\widetilde{\psi'}^{-1}\widetilde{\bb{L}'}\to \rr\widetilde{a'}_*\widetilde{a'}^{-1}\widetilde{\psi'}^{-1}\widetilde{\bb{L}'}$ (\ref{cor:v-etale-coh}).
	
	(2) It follows from the canonical almost isomorphism 
	\begin{align}
		\gamma_3: H^q(\fal_{\widetilde{Y'}\to \widetilde{X'}}^\et,\widetilde{\bb{L}'}\otimes_{\bb{Z}}\falb')\longrightarrow H^q(I_{\widetilde{Y'}\to \widetilde{X'}^{\widetilde{Y'}}},\widetilde{\varepsilon'}^{-1}\widetilde{\bb{L}'}\otimes \falhb')
	\end{align}
	which is induced by the canonical almost isomorphism $\widetilde{\bb{L}'}\otimes_{\bb{Z}}\falb'\to \rr\widetilde{\varepsilon'}_*(\widetilde{\varepsilon'}^{-1}\widetilde{\bb{L}'}\otimes \falhb')$ (\ref{cor:falh-et}).
	
	(3) Consider the following diagram
	\begin{align}\label{eq:11.5.3}
		\xymatrix{
			H^q_{\et}(\widetilde{Y'},\widetilde{\psi'}^{-1}\widetilde{\bb{L}'})\otimes  R \ar[d]_-{\gamma_2\otimes\id_R}^-{\wr}&  H^q(\fal_{\widetilde{Y'}\to \widetilde{X'}}^\et,\widetilde{\bb{L}'})\otimes R\ar[l]_-{\alpha_1}\ar[r]^-{\alpha_2}\ar[d]& H^q(\fal_{\widetilde{Y'}\to \widetilde{X'}}^\et,\widetilde{\bb{L}'}\otimes_{\bb{Z}}\falb')\ar[d]^-{\gamma_3}\\
			H^q_{\mrm{v}}(\widetilde{Y'},\widetilde{\Psi'}^{-1}\widetilde{\varepsilon'}^{-1}\widetilde{\bb{L}'})\otimes R &  H^q(I_{\widetilde{Y'}\to \widetilde{X'}^{\widetilde{Y'}}},\widetilde{\varepsilon'}^{-1}\widetilde{\bb{L}'})\otimes R\ar[l]_-{\sim}^-{\gamma_1\otimes \id_R}\ar[r]& H^q(I_{\widetilde{Y'}\to \widetilde{X'}^{\widetilde{Y'}}},\widetilde{\varepsilon'}^{-1}\widetilde{\bb{L}'}\otimes \falhb')
		}
	\end{align}
	where the unlabelled vertical arrow is induced by the canonical morphism $\widetilde{\bb{L}'}\to \rr\widetilde{\varepsilon'}_*\widetilde{\varepsilon'}^{-1}\widetilde{\bb{L}'}$, and the unlabelled horizontal arrow is the canonical morphism which induces $(\rr^q f_{\falh *}\scr{F}')\otimes_{\bb{Z}}\falhb\to \rr^q f_{\falh *}(\scr{F}'\otimes_{\bb{Z}}\falhb')$ on $\falh_{Y \to X^{Y}}$ by sheafification. It is clear that the diagram \eqref{eq:11.5.3} is commutative, which completes the proof.
\end{proof}

\begin{mypara}\label{para:comparison-mor-rel}
	We remark that \ref{lem:comparison-mor-rel} gives a new definition of the relative (resp. absolute) Faltings' comparison morphism without using \ref{cor:achinger}. Following \ref{para:mor-rel-notation}, let $\bb{F}'$ be a finite locally constant abelian sheaf on $Y'_\et$ and we set $\scr{F}'=\Psi'_*a'^{-1}\bb{F}'$. We set $\bb{L}'=\psi'_*\bb{F}'$, which is a finite locally constant abelian sheaf on $\fal_{Y'\to X'}^\et$ by \ref{prop:loc-sys-trans}. Remark that the canonical morphisms $\psi'^{-1}\bb{L}'\to \bb{F}'$ and $\varepsilon'^{-1}\bb{L}'\to \scr{F}'$ are isomorphisms by \ref{prop:loc-sys-trans} and \ref{lem:loc-sys-trans-2} respectively. We also remark that $\falb$, $\falhb$ are flat over $\bb{Z}$. The canonical morphisms in the derived category $\dd(\falb\module_{\fal_{Y \to X}^\et})$ (cf. \ref{eq:fal-rel-comp}),
		\begin{align}
			\xymatrix{
				(\rr\psi_*\rr f_{\et *} \psi'^{-1}\bb{L}')\otimes^{\dl}_{\bb{Z}}\falb & (\rr f_{\fal *} \bb{L}') \otimes^{\dl}_{\bb{Z}}\falb \ar[l]_-{\alpha_1}\ar[r]^-{\alpha_2}& \rr f_{\fal *} (\bb{L}'\otimes_{\bb{Z}}\falb'),
			}
		\end{align}
		fit into the following commutative diagram
		\begin{align}\label{diam:comp-rel}
			\xymatrix{
				\rr\psi_*(\rr f_{\et *}\psi'^{-1}\bb{L}') \otimes_{\bb{Z}}^{\dl} \falb \ar[d]^-{\wr}_-{\alpha_3}&  (\rr f_{\fal *} \bb{L}') \otimes^{\dl}_{\bb{Z}}\falb \ar[l]_-{\alpha_1}\ar[r]^-{\alpha_2}\ar[d] &\rr f_{\fal *}(\bb{L}'\otimes_{\bb{Z}}\falb')\ar[d]^-{\alpha_4}\\
				\rr\psi_* (\rr a_* \rr f_{\mrm{v}*}\Psi'^{-1}\varepsilon'^{-1}\bb{L}')\otimes_{\bb{Z}}^{\dl} \falb & \rr \varepsilon_*(\rr f_{\falh *}\varepsilon'^{-1}\bb{L}') \otimes_{\bb{Z}}^{\dl} \falb \ar[l]_-{\sim}^-{\alpha_5} \ar[r]_-{\alpha_6}&\rr \varepsilon_* \rr f_{\falh *}(\varepsilon'^{-1}\bb{L}'\otimes_{\bb{Z}} \falhb')
			}
		\end{align} 
		\begin{enumerate}
			\renewcommand{\labelenumi}{{\rm(\theenumi)}}
			\item The morphism $\alpha_3$ is induced by the canonical isomorphism $\psi'^{-1}\bb{L}'\to \rr a'_*a'^{-1} (\psi'^{-1}\bb{L}')$ by \ref{cor:v-etale-coh}, and thus $\alpha_3$ is an isomorphism.
			\item The morphism $\alpha_5$ is induced by the canonical isomorphism $\varepsilon'^{-1}\bb{L}' \to \rr \Psi'_*\Psi'^{-1}\varepsilon'^{-1}\bb{L}'$ by \ref{prop:Psi-const-equal}, and thus $\alpha_5$ is an isomorphism.
			\item The unlabelled arrow is induced by the canonical morphism $\bb{L}\to \rr\varepsilon'_*\varepsilon'^{-1}\bb{L}$.
			\item The morphism $\alpha_4$ is induced by the canonical almost isomorphism $\bb{L}'\otimes_{\bb{Z}} \falb' \to \rr \varepsilon'_* (\varepsilon'^{-1}\bb{L}'\otimes_{\bb{Z}} \falhb')$ by \ref{cor:falh-et}, and thus $\alpha_4$ is an almost isomorphism.
			\item The morphism $\alpha_6$ is the composition of 
			\begin{align}
				\rr \varepsilon_*(\rr f_{\falh *}\varepsilon'^{-1}\bb{L}') \otimes_{\bb{Z}}^{\dl} \falb &\longrightarrow \rr \varepsilon_*((\rr f_{\falh *}\varepsilon'^{-1}\bb{L}') \otimes_{\bb{Z}}^{\dl} \falhb)\label{eq:alpha41}\\ \trm{ with }
				\rr \varepsilon_*((\rr f_{\falh *}\varepsilon'^{-1}\bb{L}') \otimes_{\bb{Z}}^{\dl} \falhb) &\longrightarrow \rr \varepsilon_* \rr f_{\falh *}(\varepsilon'^{-1}\bb{L}'\otimes_{\bb{Z}} \falhb').\label{eq:alpha42}
			\end{align}
		\end{enumerate}
		In conclusion, the arrows $\alpha_3$, $\alpha_5$, $\alpha_6$ and $\alpha_4$ induce an arrow
		\begin{align}\label{eq:11.6.5}
			\alpha_4^{-1}\circ\alpha_6\circ \alpha_5^{-1}\circ\alpha_3: \rr\psi_*(\rr f_{\et *}\bb{F}') \otimes_{\bb{Z}}^{\dl} \falb \longrightarrow  \rr f_{\fal *}(\psi'_*\bb{F}'\otimes_{\bb{Z}}\falb')
		\end{align}
		in the derived category of almost $\overline{\bb{Z}_p}$-modules on $\fal_{Y \to X}^\et$ (\ref{para:almost-derived}). Remark that we don't assume that $\alpha_1$ is an isomorphism here. We also call \eqref{eq:11.6.5} the \emph{relative Faltings' comparison morphism}. Indeed, if $\alpha_1$ is an isomorphism, then the relative Faltings' comparison morphism \eqref{eq:easy-fal-rel-comp} induces \eqref{eq:11.6.5} in $\dd(\overline{\bb{Z}_p}^\al\module)$ due to the commutativity of the diagram \eqref{diam:comp-rel}.
		
		If $X$ is the spectrum of an absolutely integrally closed valuation ring $A$ and if $Y=\overline{\eta}\times_{\overline{S}} X$, then applying the functor $\rr\Gamma(Y \to X,-)$ on \eqref{diam:comp-rel} we obtain the natural morphisms in the derived category $\dd(A\module)$ by \ref{lem:fal-localtopos},
		\begin{align}\label{diam:comp-abs}
			\xymatrix{
				\rr\Gamma(Y'_\et, \psi'^{-1}\bb{L}') \otimes_{\bb{Z}}^{\dl} A \ar[d]^-{\wr}_-{\alpha_3}&\rr\Gamma(\fal_{Y' \to X'}^\et,\bb{L}')\otimes_{\bb{Z}}^{\dl} A\ar[l]_-{\alpha_1}\ar[r]^-{\alpha_2}\ar[d]&\rr\Gamma(\fal_{Y' \to X'}^\et,\bb{L}'\otimes_{\bb{Z}} \falb')\ar[d]^-{\alpha_4}\\
				\rr\Gamma((\schqcqs_{/Y'})_{\mrm{v}},\Psi'^{-1}\varepsilon'^{-1}\bb{L}')\otimes_{\bb{Z}}^{\dl} A & \rr \Gamma(\falh_{Y' \to X'^{Y'}},\varepsilon'^{-1}\bb{L}') \otimes_{\bb{Z}}^{\dl} A \ar[l]_-{\sim}^-{\alpha_5} \ar[r]_-{\alpha_6}&\rr \Gamma(\falh_{Y' \to X'^{Y'}},\varepsilon'^{-1}\bb{L}'\otimes_{\bb{Z}} \falhb')
			}
		\end{align} 
		The arrows $\alpha_3$, $\alpha_5$, $\alpha_6$ and $\alpha_4$ induce an arrow
		\begin{align}\label{eq:11.6.7}
			\alpha_4^{-1}\circ\alpha_6\circ \alpha_5^{-1}\circ\alpha_3: \rr\Gamma(Y'_\et, \bb{F}') \otimes_{\bb{Z}}^{\dl} A \longrightarrow  \rr\Gamma(\fal_{Y' \to X'}^\et,\psi'_*\bb{F}'\otimes_{\bb{Z}} \falb')
		\end{align}
		in the derived category $\dd(\overline{\bb{Z}_p}^\al\module)$ of almost $\overline{\bb{Z}_p}$-modules (\ref{para:almost-derived}). We also call \eqref{eq:11.6.7} the \emph{absolute Faltings' comparison morphism}.
\end{mypara}

\begin{mylem}\label{lem:hard-comp-rel-loc}
	With the notation in {\rm\ref{para:mor-rel-notation}}, let $\bb{F}'$ be a finite locally constant abelian sheaf on $Y'_\et$ and we set $\scr{F}'=\Psi'_*a'^{-1}\bb{F}'$. Assume that $X' \to X$ is proper of finite presentation. Then, the canonical morphism
	\begin{align}
		(\rr f_{\falh *}\scr{F}')\otimes_{\bb{Z}}^\dl \falhb \longrightarrow \rr f_{\falh *}(\scr{F}'\otimes_{\bb{Z}} \falhb')
	\end{align}
	is an almost isomorphism.
\end{mylem}
\begin{proof}
	Following \ref{lem:comparison-mor-rel}, consider the following presheaves on $\falh_{Y \to X^{Y}}$ for each integer $q$:
	\begin{align}
		\ca{H}_1^q&: \widetilde{X}\longmapsto  H^q_{\et}(\widetilde{Y'},\widetilde{\bb{F}'})\otimes_{\bb{Z}} \falb(\widetilde{Y}\to \widetilde{X}),\\
		\ca{H}_2^q&: \widetilde{X} \longmapsto  H^q(\fal_{\widetilde{Y'}\to \widetilde{X'}}^\et,\widetilde{\psi'}_*\widetilde{\bb{F}'})\otimes_{\bb{Z}} \falb(\widetilde{Y}\to \widetilde{X}),\\
		\ca{H}_3^q&:  \widetilde{X} \longmapsto  H^q(\fal_{\widetilde{Y'}\to \widetilde{X'}}^\et,\widetilde{\psi'}_*\widetilde{\bb{F}'}\otimes_{\bb{Z}}\falb'),
	\end{align}
	They satisfy the limit-preserving condition \ref{prop:v-site-stalk}.(\ref{prop:v-site-stalk-limit}) by \ref{prop:limit-fal-sites}, \cite[\Luoma{7}.5.6]{sga4-2} and \cite[\Luoma{6} 8.5.9, 8.7.3]{sga4-2}. Moreover, if $\widetilde{X}=\spec(A)$ where $A$ is an absolutely integrally closed valuation ring with $p$ nonzero in $A$, then the canonical morphisms
	\begin{align}
		\ca{H}_1^q(\spec(A))\leftarrow \ca{H}_2^q(\spec(A)) \to \ca{H}_3^q(\spec(A))
	\end{align}
	are an isomorphism and an almost isomorphism by \ref{thm:fal-comp-abs}. Thus, the canonical morphisms $\ca{H}_1^q\leftarrow \ca{H}_2^q \to \ca{H}_3^q$ induce an isomorphism and an almost isomorphism of their sheafifications by \ref{prop:v-site-stalk}. The conclusion follows from \ref{lem:comparison-mor-rel}.
\end{proof}

\begin{mylem}\label{lem:finite-fp}
	Let $Y\to X$ be an open immersion of coherent schemes, $Y' \to Y$ a finite morphism of finite presentation. Then, there exists a finite morphism $X'\to X$ of finite presentation whose base change by $Y\to X$ is $Y'\to Y$.
\end{mylem}
\begin{proof}
	Firstly, assume that $X$ is Noetherian. We have $Y'=Y\times_X X^Y$ by \ref{lem:relative-normal}.(\ref{lem:relative-normal-norm}). We write $X^Y=\spec_X(\ca{A})$ where $\ca{A}$ is an integral quasi-coherent $\ca{O}_X$-algebra on $X$, and we write $\ca{A}$ as a filtered colimit of its finite quasi-coherent $\ca{O}_X$-subalgebras $\ca{A}=\colim \ca{A}_\alpha$ (\cite[\href{https://stacks.math.columbia.edu/tag/0817}{0817}]{stacks-project}). Let $\ca{B}_\alpha$ be the restriction of $\ca{A}_\alpha$ to $Y$. Then, $\ca{B}=\colim \ca{B}_\alpha$ is a filtered colimit of finite quasi-coherent $\ca{O}_Y$-algebras with injective transition morphisms. Since $Y'=\spec_Y(\ca{B})$ is finite over $Y$, there exists an index $\alpha_0$ such that $Y'=\spec_Y(\ca{B}_{\alpha_0})$. Therefore, $X'=\spec_X(\ca{A}_{\alpha_0})$ meets our requirements.
	
	In general, we write $X$ as a cofiltered limit of coherent schemes of finite type over $\bb{Z}$ with affine transition morphisms $X=\lim_{\lambda\in \Lambda} X_\lambda$ (\cite[\href{https://stacks.math.columbia.edu/tag/01ZA}{01ZA}]{stacks-project}). Since $Y\to X$ is an open immersion of finite presentation, using \cite[8.8.2, 8.10.5]{ega4-3} there exists an index $\lambda_0\in \Lambda$, an open immersion $Y_{\lambda_0}\to X_{\lambda_0}$ and a finite morphism $Y'_{\lambda_0}\to Y_{\lambda_0}$ such that the base change of the morphisms $Y'_{\lambda_0}\to Y_{\lambda_0} \to X_{\lambda_0}$ by $X\to X_{\lambda_0}$ are the morphisms $Y'\to Y \to X$. By the first paragraph, there exists a finite morphism $X'_{\lambda_0}\to X_{\lambda_0}$ of finite presentation such that $Y'_{\lambda_0}=Y_{\lambda_0}\times_{X_{\lambda_0}} X'_{\lambda_0}$. We see that the base change $X'\to X$ of $X'_{\lambda_0}\to X_{\lambda_0}$ by $X\to X_{\lambda_0}$ meets our requirements.
\end{proof}

\begin{mylem}\label{lem:hard-comp-rel-finite}
	With the notation in {\rm\ref{para:mor-rel-notation}}, let $g:Y''\to Y'$ be a finite morphism of finite presentation, $\bb{F}''$ a finite locally constant abelian sheaf on $Y''_\et$ and we set $\scr{F}'=\Psi'_*a'^{-1}(g_{\et*}\bb{F}'')$. Assume that $X' \to X$ is proper of finite presentation. Then, the canonical morphism
	\begin{align}
		(\rr f_{\falh *}\scr{F}')\otimes_{\bb{Z}}^\dl \falhb \longrightarrow \rr f_{\falh *}(\scr{F}'\otimes_{\bb{Z}} \falhb')
	\end{align}
	is an almost isomorphism.
\end{mylem}
\begin{proof}
	There exists a Cartesian morphism $g:(Y''\to X'')\to (Y'\to X^{Y}\times_X X')$ of open immersions of coherent schemes such that $X'' \to X^{Y}\times_X X'$ is finite and of finite presentation by \ref{lem:finite-fp}. Consider the diagram \eqref{diam:sites} associated to $g$:
	\begin{align}
			\xymatrix{
				Y''_\et \ar[d]_-{g_\et} & (\schqcqs_{/Y''})_{\mrm{v}} \ar[l]_-{a''} \ar[r]^-{\Psi''} \ar[d]_-{g_{\mrm{v}}}& \falh_{Y'' \to X''^{Y''}}\ar[d]_-{g_\falh}\\
				Y'_\et & (\schqcqs_{/Y'})_{\mrm{v}} \ar[l]_-{a'} \ar[r]^-{\Psi'} & \falh_{Y' \to X'^{Y'}}
			}
	\end{align}
	We set $\scr{G}''=\Psi''_*a''^{-1}\bb{F}''$. The base change morphism $a'^{-1}g_{\et *}\to g_{\mrm{v}*}a''^{-1}$ induces a canonical isomorphism $\scr{F}'\iso g_{\falh*}\scr{G}''$ by \ref{cor:v-etale-bc}. Moreover, the canonical morphism $g_{\falh *}\scr{G}'' \to \rr g_{\falh *} \scr{G}''$ is an isomorphism by \ref{lem:comparison-mor-rel}.(\ref{lem:comparison-mor-rel-1}) and \ref{prop:v-site-stalk}, since $g:Y''\to Y'$ is finite (\cite[\Luoma{8}.5.6]{sga4-2}). By applying \ref{lem:hard-comp-rel-loc} to $g$ and $\bb{F}''$, the canonical morphism
	\begin{align}\label{eq:11.9.2}
		(\rr g_{\falh *}\scr{G}'')\otimes_{\bb{Z}}^\dl \falhb' \longrightarrow \rr g_{\falh *}(\scr{G}''\otimes_{\bb{Z}} \falhb'')
	\end{align}
	is an almost isomorphism. Let $h$ be the composition of $(Y''\to X'')\to (Y'\to X^{Y}\times_X X')\to (Y\to X^Y)$. Note that $X''\to X^Y$ is also proper of finite presentation. By applying \ref{lem:hard-comp-rel-loc} to $h$ and $\bb{F}''$, the canonical morphism
	\begin{align}\label{eq:11.9.3}
		(\rr h_{\falh *}\scr{G}'')\otimes_{\bb{Z}}^\dl \falhb \longrightarrow \rr h_{\falh *}(\scr{G}''\otimes_{\bb{Z}} \falhb'')
	\end{align}
	is an almost isomorphism. It is clear that $h_{\falh}=f_{\falh}\circ g_{\falh}$. The conclusion follows from the canonical isomorphism $\scr{F}'\to \rr g_{\falh *}\scr{G}''$ and the canonical almost isomorphisms \eqref{eq:11.9.2} and \eqref{eq:11.9.3}.
\end{proof}

\begin{mylem}\label{lem:hard-comp-rel-constr}
	With the notation in {\rm\ref{para:mor-rel-notation}}, let $\ca{F}'$ be a constructible abelian sheaf on $Y'_\et$ and we set $\scr{F}'=\Psi'_*a'^{-1}\ca{F}'$. Assume that $X' \to X$ is proper of finite presentation. Then, the canonical morphism
	\begin{align}
		(\rr f_{\falh *}\scr{F}')\otimes_{\bb{Z}}^\dl \falhb \longrightarrow \rr f_{\falh *}(\scr{F}'\otimes_{\bb{Z}} \falhb')
	\end{align}
	is an almost isomorphism.
\end{mylem}
\begin{proof}
	We prove by induction on an integer $q$ that the canonical morphism $(\rr^q f_{\falh *}\scr{F}')\otimes_{\bb{Z}} \falhb \to \rr^q f_{\falh *}(\scr{F}'\otimes_{\bb{Z}} \falhb')$ is an almost isomorphism. It holds trivially for each $q\leq -1$. Notice that there exists a finite morphism $g:Y''\to Y'$ of finite presentation, a finite locally constant abelian sheaf $\bb{F}''$ on $Y''_\et$ and an injective morphism $\ca{F}'\to g_{\et*}\bb{F}''$ by \cite[\href{https://stacks.math.columbia.edu/tag/09Z7}{09Z7}]{stacks-project} (cf. \cite[\Luoma{9}.2.14]{sga4-3}). Let $\ca{G}'$ be the quotient of $\ca{F}'\to g_{\et*}\bb{F}''$, which is also a constructible abelian sheaf on $Y'_\et$ since $g_{\et*}\bb{F}''$ is so (\cite[\href{https://stacks.math.columbia.edu/tag/095R}{095R}, \href{https://stacks.math.columbia.edu/tag/03RZ}{03RZ}]{stacks-project}). The exact sequence $0\to \ca{F}' \to g_{\et*}\bb{F}'' \to \ca{G}'\to 0$ induces an exact sequence by \ref{prop:Psi-const-equal}.(\ref{prop:Psi-const-equal-torsion}), 
	\begin{align}
		\xymatrix{
			0\ar[r]& \Psi'_*a'^{-1}\ca{F}'\ar[r]&\Psi'_*a'^{-1}(g_{\et*}\bb{F}'')\ar[r]&\Psi'_*a'^{-1}\ca{G}'\ar[r]&0.
		}
	\end{align}
	We set $\scr{H}'=\Psi'_*a'^{-1}(g_{\et*}\bb{F}'')$ and $\scr{G}'=\Psi'_*a'^{-1}\ca{G}'$. Then, we obtain a morphism of long exact sequences
	\begin{align}
		\small
		\xymatrix{
			(\rr^{q-1} f_{\falh *}\scr{H}')\otimes \falhb \ar[d]^-{\gamma_1} \ar[r] & (\rr^{q-1} f_{\falh *}\scr{G}')\otimes \falhb \ar[d]^-{\gamma_2} \ar[r] &(\rr^{q} f_{\falh *}\scr{F}')\otimes \falhb \ar[d]^-{\gamma_3} \ar[r] &(\rr^{q} f_{\falh *}\scr{H}')\otimes \falhb \ar[d]^-{\gamma_4} \ar[r] &(\rr^{q} f_{\falh *}\scr{G}')\otimes \falhb \ar[d]^-{\gamma_5}  \\ 
			\rr^{q-1} f_{\falh *}(\scr{H}'\otimes \falhb')\ar[r] & \rr^{q-1} f_{\falh *}(\scr{G}'\otimes \falhb')\ar[r] & \rr^{q} f_{\falh *}(\scr{F}'\otimes \falhb')\ar[r] & \rr^{q} f_{\falh *}(\scr{H}'\otimes \falhb')\ar[r] & \rr^{q} f_{\falh *}(\scr{G}'\otimes \falhb')
		}
	\end{align}
	Notice that $\gamma_1$ and $\gamma_2$ are almost isomorphisms by induction, and that $\gamma_4$ is an almost isomorphism by \ref{lem:hard-comp-rel-finite}. Thus, applying the 5-lemma (\cite[\href{https://stacks.math.columbia.edu/tag/05QA}{05QA}]{stacks-project}) in the abelian category of almost $\overline{\bb{Z}_p}$-modules over $\falh_{Y \to X^{Y}}$, we see that $\gamma_3$ is almost injective. Since $\ca{F}'$ is an arbitrary constructible abelian sheaf, the morphism $\gamma_5$ is also almost injective. Thus, $\gamma_3$ is an almost isomorphism.
\end{proof}

\begin{mythm}\label{thm:hard-comp-rel}
	With the notation in {\rm\ref{para:mor-rel-notation}}, let $\ca{F}'$ be a torsion abelian sheaf on $Y'_\et$ and we set $\scr{F}'=\Psi'_*a'^{-1}\ca{F}'$. Assume that $X' \to X$ is proper of finite presentation. Then, the canonical morphism
	\begin{align}
		(\rr f_{\falh *}\scr{F}')\otimes_{\bb{Z}}^\dl \falhb \longrightarrow \rr f_{\falh *}(\scr{F}'\otimes_{\bb{Z}} \falhb')
	\end{align}
	is an almost isomorphism in the derived category $\dd(\overline{\bb{Z}_p}\module_{\falh_{Y \to X^Y}})$ {\rm(\ref{para:almost-derived})}.
\end{mythm}
\begin{proof}
	We write $\ca{F}'$ as a filtered colimit of constructible abelian sheaves $\ca{F}'=\colim_{\lambda\in \Lambda}\ca{F}'_\lambda$ (\cite[\href{https://stacks.math.columbia.edu/tag/03SA}{03SA}]{stacks-project}, cf. \cite[\Luoma{9}.2.7.2]{sga4-3}). We set $\scr{F}'_\lambda=\Psi'_*a'^{-1}\ca{F}_\lambda$. We have $\scr{F}'=\colim_{\lambda\in \Lambda}\scr{F}'_\lambda$ by \cite[\Luoma{6}.5.1]{sga4-2} whose conditions are satisfied since each object in each concerned site is quasi-compact. Moreover, for each integer $q$, we have
	\begin{align}
		(\rr^q f_{\falh *}\scr{F}')\otimes_{\bb{Z}} \falhb=&\colim_{\lambda\in \Lambda} (\rr^q f_{\falh *}\scr{F}'_\lambda)\otimes_{\bb{Z}} \falhb,\\
		\rr^q f_{\falh *}(\scr{F}'\otimes_{\bb{Z}} \falhb')=& \colim_{\lambda\in \Lambda} \rr^q f_{\falh *}(\scr{F}'_\lambda\otimes_{\bb{Z}} \falhb').
	\end{align}
	The conclusion follows from \ref{lem:hard-comp-rel-constr}.
\end{proof}

\begin{mylem}\label{lem:falh-presheaf}
	With the notation in {\rm\ref{para:mor-rel-notation}} and  {\rm\ref{para:comparison-mor-rel-bc}}, let $\ca{F}'$ be a torsion abelian sheaf on $Y'_\et$, $\ca{H}=\rr f_{\et *}\ca{F}'$, and we set $\scr{F}'=\Psi'_*a'^{-1}\ca{F}'$, $\scr{H}=\rr\Psi_*a^{-1}\ca{H}$. Let $\widetilde{X}$ be an object of $\falh_{Y\to X^Y}$, $\widetilde{Y}=\overline{\eta}\times_{\overline{S}} \widetilde{X}$, $\widetilde{\ca{F}'}=g'^{-1}_{\et}\ca{F}'$.
	\begin{enumerate}
		\renewcommand{\labelenumi}{{\rm(\theenumi)}}
		\item The sheaf $\rr^q f_{\falh *}\scr{F}'$ is canonically isomorphic to the presheaf $\widetilde{X}\mapsto H^q_\et(\widetilde{Y'}, \widetilde{\ca{F}'})$ for each integer $q$.
		\item If $Y'\to Y$ is proper, then there exists a canonical isomorphism $\scr{H}\iso \rr f_{\falh *}\scr{F}' $.\label{lem:falh-presheaf-2}
	\end{enumerate}
\end{mylem}
\begin{proof}
	Note that the canonical morphism $\scr{F}'\to \rr\Psi'_*a'^{-1}\ca{F}'$ is an isomorphism by \ref{prop:Psi-const-equal}.(\ref{prop:Psi-const-equal-torsion}). Thus, $\rr f_{\falh *}\scr{F}'=\rr(\Psi \circ f_{\mrm{v}})_*a'^{-1}\ca{F}'$, whose $q$-th cohomology is the sheaf associated to the presheaf $\widetilde{X}\mapsto H^q_{\mrm{v}}(\widetilde{Y'}, \widetilde{a'}^{-1}\widetilde{\ca{F}'})=H^q_\et(\widetilde{Y'}, \widetilde{\ca{F}'})$ by \ref{cor:v-etale-coh}, and thus (1) follows. If $Y'\to Y$ is proper, then the base change morphism $a^{-1}\rr f_{\et *}\to \rr f_{\mrm{v}*}a'^{-1}$ induces an isomorphism $a^{-1} \ca{H}\iso \rr f_{\mrm{v}*}a'^{-1}\ca{F}'$ by \ref{cor:v-etale-bc}, and thus (2) follows.
\end{proof}

\begin{mythm}\label{thm:fal-comp-rel}
	With the notation in {\rm\ref{para:mor-rel-notation}}, let $\bb{F}'$ be a finite locally constant abelian sheaf on $Y'_\et$. Assume that 
	\begin{enumerate}
		\renewcommand{\theenumi}{\roman{enumi}}
		\renewcommand{\labelenumi}{{\rm(\theenumi)}}
		\item the morphism $X' \to X$ is proper of finite presentation, and that
		\item the sheaf $\rr^q f_{\et*} \bb{F}'$ is finite locally constant for each integer $q$ and nonzero for finitely many $q$, and that\label{thm:fal-comp-rel-loc}
		\item we have $\rr^q\psi_*\bb{H}=0$ (resp. $\rr^q\psi'_*\bb{H}=0$) for any finite locally constant abelian sheaf $\bb{H}$ on $Y_\et$ (resp. $Y'_\et$) and any integer $q>0$.\label{thm:fal-comp-rel-achinger}
	\end{enumerate}
	 Then, the relative Faltings' comparison morphism associated to $f$ and $\bb{F}'$ \eqref{eq:easy-fal-rel-comp} (which exists by {\rm(\ref{thm:fal-comp-rel-achinger})}) is an almost isomorphism in the derived category $\dd(\overline{\bb{Z}_p}\module_{\fal_{Y \to X}^\et})$ {\rm(\ref{para:almost-derived})}, and it induces an almost isomorphism 
	\begin{align}
		(\psi_*\rr^q f_{\et *} \bb{F}')\otimes_{\bb{Z}}\falb \longrightarrow \rr^q f_{\fal *} (\psi'_*\bb{F}'\otimes_{\bb{Z}}\falb')
	\end{align}
	of $\overline{\bb{Z}_p}$-modules for each integer $q$.
\end{mythm}
\begin{proof}
	We follow the discussion of \ref{para:comparison-mor-rel} and set $\scr{F}'=\Psi'_*a'^{-1}\bb{F}'$. The canonical morphism \eqref{eq:alpha42}
	\begin{align}
		\rr \varepsilon_*((\rr f_{\falh *}\scr{F}') \otimes_{\bb{Z}}^{\dl} \falhb) \longrightarrow \rr \varepsilon_* \rr f_{\falh *}(\scr{F}'\otimes_{\bb{Z}} \falhb')
	\end{align}
	is an almost isomorphism by \ref{lem:hard-comp-rel-loc}. It remains to show that the canonical morphism \eqref{eq:alpha41}
	\begin{align}\label{eq:11.13.3}
		\rr \varepsilon_*(\rr f_{\falh *}\scr{F}') \otimes_{\bb{Z}}^{\dl} \falb \longrightarrow \rr \varepsilon_*((\rr f_{\falh *}\scr{F}') \otimes_{\bb{Z}}^{\dl} \falhb)
	\end{align}
	is also an almost isomorphism. With the notation in \ref{lem:falh-presheaf} by taking $\ca{F}'=\bb{F}'$, the complex $\ca{H}$ is a bounded complex whose cohomologies $H^q(\ca{H})$ are finite locally constant abelian sheaves by condition (\ref{thm:fal-comp-rel-loc}). Consider the commutative diagram \eqref{diag:beta-psi-2},
	\begin{align}\label{eq:11.13.4}
		\xymatrix{
			(\schqcqs_{/Y})_{\mrm{v}}\ar[r]^-{a}\ar[d]_-{\Psi} &Y_\et\ar[d]^-{\psi}\\
			\falh_{Y \to X^Y}\ar[r]^-{\varepsilon} & \fal_{Y \to X}^\et
		}
	\end{align} 
	We set $\ca{L}=\rr\psi_*\ca{H}$. Then, $H^q(\ca{L})=\psi_*H^q(\ca{H})$ by Cartan-Leray spectral sequence and condition (\ref{thm:fal-comp-rel-achinger}). Hence, $\ca{L}$ is a bounded complex of abelian sheaves whose cohomologies are finite locally constant by \ref{prop:loc-sys-trans} so that the canonical morphism 
	\begin{align}\label{eq:11.13.6}
		\ca{L} \otimes_{\bb{Z}}^{\dl} \falb \longrightarrow \rr \varepsilon_*(\varepsilon^{-1}\ca{L} \otimes_{\bb{Z}}^{\dl} \falhb)
	\end{align}
	is an almost isomorphism by \ref{rem:falh-et}.
	
	On the other hand, $H^q(\scr{H})=\Psi_*a^{-1}H^q(\ca{H})$ by Cartan-Leray spectral sequence and \ref{prop:Psi-const-equal}.(\ref{prop:Psi-const-equal-torsion}). Thus, the base change morphism $\varepsilon^{-1}\rr\psi_*\to \rr\Psi_*a^{-1}$ induces an isomorphism $\varepsilon^{-1}\ca{L}\iso \scr{H}$ by \ref{lem:loc-sys-trans-2}. Moreover, the canonical morphism $\ca{L}\to \rr\varepsilon_*\varepsilon^{-1}\ca{L}=\rr\varepsilon_*\scr{H}=\rr\psi_*\rr a_*a^{-1}\ca{H}$ is an isomorphism by \ref{cor:v-etale-coh}. Thus, the canonical morphism
	\begin{align}\label{eq:11.13.7}
		(\rr\varepsilon_*\varepsilon^{-1}\ca{L}) \otimes_{\bb{Z}}^{\dl} \falb \longrightarrow \rr \varepsilon_*(\varepsilon^{-1}\ca{L} \otimes_{\bb{Z}} \falhb)
	\end{align}
	is an almost isomorphism by \eqref{eq:11.13.6}. In conclusion, \eqref{eq:11.13.3} is an almost isomorphism by \eqref{eq:11.13.7} and by the canonical isomorphisms $\varepsilon^{-1}\ca{L}\iso \scr{H}\iso \rr f_{\falh *}\scr{F}'$.
\end{proof}

\begin{myrem}\label{rem:fal-comp-rel}
	We give two concrete situations where the conditions in \ref{thm:fal-comp-rel} are satisfied:
	\begin{enumerate}
		\renewcommand{\labelenumi}{{\rm(\theenumi)}}
		\item Let $\overline{\bb{Z}_p}$ be the integral closure of $\bb{Z}_p$ in an algebraic closure $\overline{\bb{Q}_p}$ of $\bb{Q}_p$, $X'\to X$ a proper and finitely presented morphism of coherent $\overline{\bb{Z}_p}$-schemes, $Y'\to Y$ the base change of $X'\to X$ by $\spec(\overline{\bb{Q}_p})\to \spec(\overline{\bb{Z}_p})$. Assume that $Y'\to Y$ is smooth. Then, the condition (\ref{thm:fal-comp-rel-loc}) is guaranteed by \cite[\Luoma{16}.2.2 and \Luoma{17}.5.2.8.1]{sga4-3}, and the condition (\ref{thm:fal-comp-rel-achinger}) is guaranteed by \ref{cor:achinger}.
		\item Let $\ca{O}_K$ be a strictly Henselian discrete valuation ring with fraction field $K$ of characteristic $0$ and residue field of characteristic $p$, $\overline{K}$ an algebraic closure of $K$, $X'\to X$ a proper morphism of $\ca{O}_K$-schemes of finite type, $Y'\to Y$ the base change of $X'\to X$ by $\spec(\overline{K})\to \spec(\ca{O}_K)$. Assume that $Y'\to Y$ is smooth. Then, the condition (\ref{thm:fal-comp-rel-loc}) is guaranteed by \cite[\Luoma{16}.2.2 and \Luoma{17}.5.2.8.1]{sga4-3}, and the condition (\ref{thm:fal-comp-rel-achinger}) is guaranteed by \ref{thm:achinger}.
	\end{enumerate}
\end{myrem}

\section{A Local Version of the Relative Hodge-Tate Filtration}\label{sec:loc-hodge-tate-fil}

\begin{mypara}\label{para:notation-ab}
	Let $K$ be a complete discrete valuation field of characteristic $0$ with valuation ring $\ca{O}_K$ whose residue field is algebraically closed (a condition required by \cite[4.1.3, 5.1.3]{abbes2020suite}) of characteristic $p>0$, $\overline{K}$ an algebraic closure of $K$, $\ca{O}_{\overline{K}}$ the integral closure of $\ca{O}_K$ in $\overline{K}$, $\eta= \spec (K)$, $\overline{\eta}=\spec(\overline{K})$, $S=\spec (\ca{O}_K)$, $\overline{S}=\spec(\ca{O}_{\overline{K}})$. Let $(f,g):(X'^{\triangleright}\to X') \to (X^\circ \to X)$ be a morphism of adequate open immersions of schemes (\ref{para:notation-log}) over $\eta \to S$ such that $g:X' \to X$ is projective and that the induced morphism $(X',\scr{M}_{X'^{\triangleright} \to X'})\to (X,\scr{M}_{X^\circ \to X})$ is smooth and saturated. We are in fact in the situation \ref{para:abbes-gros} but with a slightly different notation.
	
	Consider the morphisms of sites defined in \ref{para:psi-sigma-nu-breve}:
	\begin{align}\label{eq:4.1.1}
		\xymatrix{
			X'^{\triangleright,\bb{N}}_{\overline{\eta},\et}\ar[r]^-{\breve{f}_{\overline{\eta},\et}}&
			X^{\circ,\bb{N}}_{\overline{\eta},\et}\ar[r]^-{\breve{\psi}}& \fal_{X^\circ_{\overline{\eta}}\to X}^{\et,\bb{N}}\ar[r]^-{\breve{\sigma}}& X^{\bb{N}}_\et.
		}
	\end{align}
	We consider $\breve{\sigma}$ as a morphism of ringed sites $\breve{\sigma}: (\fal_{X^\circ_{\overline{\eta}}\to X}^{\et,\bb{N}},\breve{\falb})\to (X^{\bb{N}}_\et,\breve{\ca{O}}_{X_\et})$,
	and $\breve{f}_{\overline{\eta},\et}$ as a morphism of ringed sites $\breve{f}_{\overline{\eta},\et}:(X'^{\triangleright,\bb{N}}_{\overline{\eta},\et},\breve{\bb{Z}}_p)\to (X^{\circ,\bb{N}}_{\overline{\eta},\et},\breve{\bb{Z}}_p)$, where $\breve{\bb{Z}}_p=(\bb{Z}/p^{n}\bb{Z})_{n\geq 0}$.
\end{mypara}

\begin{mythm}[{\cite[6.7.5, 6.7.10, 6.7.13]{abbes2020suite}}]\label{thm:rel-hodge-tate}
	Under the assumptions in {\rm\ref{para:notation-ab}} and with the same notation, for each integer $n\geq 0$, there is a canonical $G_K$-equivariant finite decreasing filtration $(\mrm{Fil}^q)_{q\in \bb{Z}}$ on the $\breve{\falb}_{\bb{Q}}$-module $\breve{\psi}_*(\rr^n\breve{f}_{\overline{\eta},\et*}(\breve{\bb{Z}}_p))\otimes_{\bb{Z}_p}\breve{\falb}_{\bb{Q}}$ (cf. {\rm\ref{para:isogeny}}) and a canonical $G_K$-equivariant isomorphism for each $q\in \bb{Z}$,
	\begin{align}
		\mrm{Gr}^q(\breve{\psi}_*(\rr^n\breve{f}_{\overline{\eta},\et*}(\breve{\bb{Z}}_p))\otimes_{\breve{\bb{Z}}_p}\breve{\falb}_{\bb{Q}})\cong\breve{\sigma}^*(\rr^qg_*(\Omega^{n-q}_{(X',\scr{M}_{X'})/(X,\scr{M}_X)})\otimes_{\ca{O}_{X_\et}} \breve{\ca{O}}_{X_\et}(q-n))_{\bb{Q}},
	\end{align}
	where $\mrm{Gr}^q$ denotes the graded piece $\mrm{Fil}^q/\mrm{Fil}^{q+1}$.
\end{mythm}
We call this filtration the \emph{relative Hodge-Tate filtration} of the morphism $(f,g):(X'^{\triangleright}\to X')\to (X^\circ\to X)$.

\begin{myrem}\label{rem:rel-hodge-tate}
	We keep the notation and assumptions in {\rm\ref{para:notation-ab}} and {\rm\ref{thm:rel-hodge-tate}}.
	\begin{enumerate}
		\renewcommand{\labelenumi}{{\rm(\theenumi)}}
		\item If we set $\rr^n\breve{f}_{\overline{\eta},\et*}(\breve{\bb{Z}}_p)=\breve{\bb{L}}^{(n)}=(\bb{L}_{k}^{(n)})_{k\geq 0}$ where $\bb{L}_{k}^{(n)}=\rr^nf_{\overline{\eta},\et*}(\bb{Z}/p^k\bb{Z})$, then $\bb{L}_k^{(n)}$ is a finite locally constant abelian sheaf on $X^{\circ}_{\overline{\eta},\et}$ (\cite[2.2.25]{abbes2020suite}), and the inverse system $(\bb{L}_{k}^{(n)})_{k\geq 0}$ is Artin-Rees $p$-adic (\cite[10.1.18.(\luoma{3})]{fulei2015etale}).\label{item:rem-rel-hodge-tate-et}
		\item The $\ca{O}_X$-module $\ca{M}^{q,n-q}=\rr^qg_*(\Omega^{n-q}_{(X',\scr{M}_{X'})/(X,\scr{M}_X)})$ is coherent and its restriction to $X_\eta$ is locally free (\cite[7.2]{illusie2005riemannhilbert}, cf. the proof of \cite[6.7.13]{abbes2020suite}).\label{item:rem-rel-hodge-tate-diff}
	\end{enumerate}
\end{myrem}

\begin{mypara}\label{para:notation-final}
	Under the assumptions in {\rm\ref{para:notation-ab}}, {\rm\ref{thm:rel-hodge-tate}}, {\rm\ref{rem:rel-hodge-tate}} and with the same notation, assume further that $X=\spec(R)$ is affine. We remark that $R$ is $p$-torsion free (\cite[\Luoma{2}.6.3.(\luoma{2})]{abbes2016p}).
	Let $V\to U$ be an object of $\fal_{X^\circ_{\overline{\eta}}\to X}^\proet$ satisfying the following conditions:
	\begin{enumerate}
		\renewcommand{\labelenumi}{{\rm(\theenumi)}}
		\item The morphism $V\to U$ is Faltings acyclic (cf. \ref{defn:faltings-acyclic},  \ref{thm:acyclic}).
		\item For any integers $n\geq 0$ and $k\geq 0$, the pullback $\bb{L}_{k}^{(n)}|_{V_\et}$ is constant with value $H^n_{\et,k}$ .
	\end{enumerate}
	We denote by $A$ the $\ca{O}_K$-algebra $\falb(V\to U)$ (i.e. $U^V=\spec(A)$), and we set $H^n_{\et}=\lim_{k\to \infty}H^n_{\et,k}$.
\end{mypara}

\begin{myrem}\label{rem:const-value-et-coh}
	Let $\overline{x}$ be a geometric point of $V$. Then, there is a natural isomorphism $H^n_{\et,k}\cong H^n_{\et}(X'^{\triangleright}_{\overline{x}},\bb{Z}/p^k\bb{Z})$ (\cite[2.2.25]{abbes2020suite}). We remark that $H^n_{\et}(X'^{\triangleright}_{\overline{x}},\bb{Z}_p)= \lim_{k\to \infty} H^n_{\et}(X'^{\triangleright}_{\overline{x}},\bb{Z}/p^k\bb{Z})$ is a finitely generated $\bb{Z}_p$-module (thus so is $H^n_{\et}$), and that the morphism of inverse systems of abelian groups
	\begin{align}\label{eq:6.4.2}
		(H^n_{\et}(X'^{\triangleright}_{\overline{x}},\bb{Z}_p)/p^kH^n_{\et}(X'^{\triangleright}_{\overline{x}},\bb{Z}_p))_{k\geq 0}\longrightarrow (H^n_{\et,k})_{k\geq 0}
	\end{align}
	is an Artin-Rees isomorphism, by which we mean its kernel and cokernel are Artin-Rees zero (cf. \ref{rem:rel-hodge-tate}.(\ref{item:rem-rel-hodge-tate-et}) and \cite[10.1.4]{fulei2015etale}).
\end{myrem}

\begin{mypara}\label{para:notation-final-galois}
	Following \ref{para:notation-final}, notice that $U^\circ=X^\circ\times_X U$ is affine and geometrically normal over $K$ and that $V$ is affine and normal, since $X^\circ$ is smooth over $K$. We assume further that the following conditions hold: 
	\begin{enumerate}
		\setcounter{enumi}{2}
		\renewcommand{\labelenumi}{{\rm(\theenumi)}}
		\item The scheme $V$ is integral and lies over a connected component $U^{\circ}_{\star}$ of $U^\circ$.
		\item The function field $\ca{L}$ of $V$ is a Galois extension of the function field $\ca{K}$ of $U^{\circ}_{\star}$ with Galois group $\Gamma$.
	\end{enumerate}
	Let $U^{\circ}_{\overline{\eta},\star}$ be the connected component of $U^{\circ}_{\overline{\eta}}$ over which $V$ lies. Its function field is the composite $\overline{K}\ca{K}$ of $\overline{K}$ and $\ca{K}$ in $\ca{L}$, which is Galois over $\ca{K}$ whose Galois group identifies with the closed subgroup $G_L$ of $G_K$ where $L$ is the algebraic closure of $K$ in $\ca{K}$. We denote by $\Delta$ the Galois group of $\ca{L}$ over $\overline{K}\ca{K}$. It is clear that $G_L=\Gamma/\Delta$.
	\begin{align}
		\xymatrix{
			V\ar[d]^-{\Delta}\ar@/_2pc/[dd]_-{\Gamma}&\\
			U^{\circ}_{\overline{\eta},\star}\ar[r]\ar[d]^-{G_L}&\spec(\overline{K})\ar[d]^-{G_K}\\
			U^{\circ}_{\star}\ar[r]&\spec(K)
		}
	\end{align}
	
	Since $V$ is the integral closure of $U^{\circ}_{\star}$ in $\ca{L}$, the canonical homomorphism of groups
	\begin{align}
		\mrm{Aut}_{U^{\circ}_{\star}}(V)^\oppo\longrightarrow \mrm{Aut}_{\ca{K}}(\ca{L})=\Gamma
	\end{align}
	is an isomorphism. In particular, $\Gamma$ acts naturally on $V$ on the right. For $\gamma\in \Gamma$ with image $u\in G_K$, we denote by $f_\gamma:V\to V$ the right action of $\gamma$ on $V$, and for any $K$-scheme $Z$, we denote by $f_u:Z_{\overline{\eta}}\to Z_{\overline{\eta}}$ the base change of the automorphism of $\overline{\eta}$ induced by $u$. There is a commutative diagram 
	\begin{align}\label{diam:6.5.1}
		\xymatrix{
			V\ar[r]^-{f_\gamma}\ar[d]& V \ar[d]\\
			U^{\circ}_{\overline{\eta}}\ar[r]^-{f_u}& U^{\circ}_{\overline{\eta}}
		}
	\end{align}
	The natural isomorphism (induced by the base change)
	\begin{align}
		f_{\gamma,\et}^{-1}(\bb{L}_{k}^{(n)}|_{V_\et})\iso\bb{L}_{k}^{(n)}|_{V_\et}
	\end{align}
	defines a natural action of $\Gamma$ on $H^n_{\et,k}$, and thus an action on $H^n_{\et}=\lim_k H^n_{\et,k}$. On the other hand, $\Gamma$ also acts naturally on $A$ as $\spec(A)$ is the integral closure of $U$ in $\ca{L}$, and thus acts on the Tate twist $A(i)$ via the map $\Gamma\to G_K$.
\end{mypara}

\begin{mythm}\label{thm:main}
	Under the assumptions in {\rm\ref{para:notation-final}} and with the same notation, for each integer $n\geq 0$, there is a canonical finite decreasing filtration $(\mrm{fil}^q)_{q\in\bb{Z}}$ on $H^n_\et\otimes_{\bb{Z}_p} \widehat{A}[1/p]$ and a canonical isomorphism for each $q\in\bb{Z}$,
	\begin{align}\label{eq:6.6.1}
		\mrm{gr}^q(H^n_\et\otimes_{\bb{Z}_p} \widehat{A}[\frac{1}{p}])\cong H^q(X',\Omega^{n-q}_{(X',\scr{M}_{X'})/(X,\scr{M}_X)})\otimes_R \widehat{A}[\frac{1}{p}](q-n),
	\end{align}
	where $\mrm{gr}^q$ denotes the graded piece $\mrm{fil}^q/\mrm{fil}^{q+1}$. Moreover, under the assumptions in {\rm\ref{para:notation-final-galois}} and with the same notation, the filtration $(\mrm{fil}^q)_{q\in\bb{Z}}$ and the isomorphisms \eqref{eq:6.6.1} are $\Gamma$-equivariant.
\end{mythm}
\begin{proof}
	We set $Y=X^\circ_{\overline{\eta}}$. We start from the filtration of \ref{thm:rel-hodge-tate}. Consider the natural morphism of ringed sites \eqref{eq:breve-nu} $\breve{\nu}:(\fal_{Y\to X}^{\proet,\bb{N}},\breve{\falb})\to (\fal_{Y\to X}^{\et,\bb{N}},\breve{\falb})$. We obtain a filtration $\breve{\nu}^*\mrm{Fil}^q$ on $\breve{\nu}^*(\breve{\psi}_*\breve{\bb{L}}^{(n)}\otimes_{\breve{\bb{Z}}_p}\breve{\falb}_{\bb{Q}})$ with graded pieces $\breve{\nu}^*\mrm{Gr}^q=\breve{\nu}^*\breve{\sigma}^*\breve{\ca{M}}^{q,n-q}(q-n)_{\bb{Q}}$ (as $\breve{\nu}^*=\breve{\nu}^{-1}$ is exact, cf. \eqref{eq:5.3.6}). We apply the derived functor $\rr\Gamma(\fal_{V\to U}^{\proet,\bb{N}},-)$ to these modules.
	
	Consider the commutative diagram  
	\begin{align}
		\xymatrix{
		V_\et\ar[d]^-{j}\ar[r]^-{\psi} & \fal_{V\to U}^{\et} \ar[d]^-{j}& \fal_{V\to U}^{\proet}\ar[l]_-{\nu}\ar[d]^-{j}\\
		Y_\et\ar[r]^-{\psi} & \fal_{Y\to X}^{\et} & \fal_{Y\to X}^{\proet}\ar[l]_-{\nu}
		}
	\end{align}
	Since the canonical morphism $j^{-1}\psi_*\bb{L}^{(n)}_k\to \psi_* j^{-1}\bb{L}^{(n)}_k$ is an isomorphism by \ref{lem:loc-sys-trans-bc}, $j^{-1}\psi_*\bb{L}^{(n)}_k$ is constant with value $H^n_{\et,k}$ by \ref{prop:loc-sys-trans}. Thus, the restriction of $\breve{\nu}^{-1}\breve{\psi}_*\breve{\bb{L}}^{(n)}$ to $\fal_{V\to U}^{\proet,\bb{N}}$ is the inverse system of constant sheaves $(H^n_{\et,k})_{k\geq 0}$. Therefore, the canonical morphism
	\begin{align}\label{eq:6.6.2}
		H^n_{\et,k}\otimes_{\bb{Z}_p} A \longrightarrow \rr\Gamma(\fal_{V\to U}^{\proet}, \nu^*(\psi_*\bb{L}_k^{(n)}\otimes_{\bb{Z}_p}\falb))
	\end{align}
	is an almost isomorphism, since $V\to U$ is Faltings acyclic and $H^n_{\et,k}$ is a finite abelian group. Notice that for any integer $q\geq 0$ there exists a canonical exact sequence (\cite[\Luoma{6}.7.10]{abbes2016p})
	\begin{align}
		0\to \rr^1\lim_{k\to \infty} H^{q-1}(\fal^\proet_{V\to U},\nu^*(\psi_*\bb{L}_k^{(n)}\otimes_{\bb{Z}_p}\falb))\to& H^q(\fal^{\proet,\bb{N}}_{V\to U},\breve{\nu}^*(\breve{\psi}_*\breve{\bb{L}}^{(n)}\otimes_{\breve{\bb{Z}}_p}\breve{\falb}))\\
		&\to \lim_{k\to \infty} H^{q}(\fal^\proet_{V\to U},\nu^*(\psi_*\bb{L}_k^{(n)}\otimes_{\bb{Z}_p}\falb))\to 0.\nonumber
	\end{align}
	Since the inverse system $(H^n_{\et,k}\otimes_{\bb{Z}_p} A)_{k\geq 0}$ is Artin-Rees $p$-adic, $\rr^1\lim_{k\to \infty} H^0(\fal^\proet_{V\to U},\nu^*(\psi_*\bb{L}_k^{(n)}\otimes_{\bb{Z}_p}\falb))$ is almost zero by the almost isomorphism \eqref{eq:6.6.2}. Moreover, we deduce that $\lim_{k\to\infty} (H^n_{\et,k}\otimes_{\bb{Z}_p} A)\to H^0(\fal^{\proet,\bb{N}}_{V\to U},\breve{\nu}^*(\breve{\psi}_*\breve{\bb{L}}^{(n)}\otimes_{\breve{\bb{Z}}_p}\breve{\falb}))$ is an almost isomorphism and that $H^q(\fal^{\proet,\bb{N}}_{V\to U},\breve{\nu}^*(\breve{\psi}_*\breve{\bb{L}}^{(n)}\otimes_{\breve{\bb{Z}}_p}\breve{\falb}))$ is almost zero for $q>0$. Since $H^n_{\et}$ is a finitely generated $\bb{Z}_p$-module (\ref{rem:const-value-et-coh}), we have $\lim_{k\to\infty} (H^n_{\et,k}\otimes_{\bb{Z}_p} A)=H^n_{\et}\otimes_{\bb{Z}_p} \widehat{A}$. Therefore, the canonical morphism
	\begin{align}
		H^n_{\et}\otimes_{\bb{Z}_p} \widehat{A} \longrightarrow \rr\Gamma(\fal_{V\to U}^{\proet,\bb{N}},\breve{\nu}^*(\breve{\psi}_*\breve{\bb{L}}^{(n)}\otimes_{\breve{\bb{Z}}_p}\breve{\falb}))
	\end{align}
	is an almost isomorphism. Inverting $p$, we obtain a canonical isomorphism
	\begin{align}\label{eq:6.6.7}
		H^n_{\et}\otimes_{\bb{Z}_p} \widehat{A}[\frac{1}{p}] \iso \rr\Gamma(\fal_{V\to U}^{\proet,\bb{N}},\breve{\nu}^*(\breve{\psi}_*\breve{\bb{L}}^{(n)}\otimes_{\breve{\bb{Z}}_p}\breve{\falb}))[\frac{1}{p}].
	\end{align}
	By taking $H^0(\fal_{V\to U}^{\proet,\bb{N}},\breve{\nu}^*\mrm{Fil}^q)$, the filtration $\breve{\nu}^*\mrm{Fil}^q$ on $\breve{\nu}^*(\breve{\psi}_*\breve{\bb{L}}^{(n)}\otimes_{\breve{\bb{Z}}_p}\breve{\falb}_{\bb{Q}})$ induces a canonical filtration $\mrm{fil}^q$ on $H^n_\et(X'^{\triangleright}_{\overline{x}},\bb{Z}_p)\otimes_{\bb{Z}_p} \widehat{A}[1/p]$ (cf. \ref{para:isogeny}).
	
	On the other hand, recall that $\ca{M}^{q,n-q}$ is the coherent $\ca{O}_X$-module associated to the coherent $R$-module $M^{q,n-q}=H^q(X',\Omega^{n-q}_{(X',\scr{M}_{X'})/(X,\scr{M}_X)})$ and that $M^{q,n-q}[1/p]$ is a projective $R[1/p]$-module (\ref{rem:rel-hodge-tate}.(\ref{item:rem-rel-hodge-tate-diff})). By \ref{prop:acyclic-diff}, we see that the canonical morphism
	\begin{align}\label{eq:6.6.8}
		H^q(X',\Omega^{n-q}_{(X',\scr{M}_{X'})/(X,\scr{M}_X)})\otimes_R \widehat{A}[\frac{1}{p}]\longrightarrow \rr\Gamma(\fal_{V\to U}^{\proet,\bb{N}},\breve{\nu}^*\breve{\sigma}^*\breve{\ca{M}}^{q,n-q})[\frac{1}{p}]
	\end{align}
	is an isomorphism. Thus, the canonical isomorphisms for the graded pieces $\breve{\nu}^*\mrm{Gr}^q=\breve{\nu}^*\breve{\sigma}^*\breve{\ca{M}}^{q,n-q}(q-n)_{\bb{Q}}$ induce canonical isomorphisms \eqref{eq:6.6.1}. This completes the proof of the first statement.
	
	For the $\Gamma$-equivariance, let $\gamma\in \Gamma$ with image $u\in G_K$. We obtain from \ref{para:notation-final-galois} a commutative diagram
	\begin{align}
		\xymatrix{
			(V\to U)\ar[d]\ar[rr]^-{(f_\gamma,\id_U)} && (V\to U)\ar[d]\\
			(Y\to X)\ar[rr]^-{(f_u,\id_X)}&& (Y\to X)
		}
	\end{align}
	where the vertical arrows are the same pro-\'etale morphism. It induces a commutative diagram of fibred ringed sites over $\bb{N}$,
	\begin{align}
		\xymatrix{
			\fal^{\proet,\bb{N}}_{V\to U}\ar[r]^-{f_\gamma}\ar[d]_-{j}& \fal^{\proet,\bb{N}}_{V\to U}\ar[d]^-{j}\\
			\fal^{\proet,\bb{N}}_{Y\to X}\ar[r]^-{f_u}& \fal^{\proet,\bb{N}}_{Y\to X}
		}
	\end{align}
	where the vertical morphisms induce the same localization morphism of the associated topoi (\cite[\Luoma{3}.7.9]{abbes2016p}).	Recall that the $G_K$-actions on the $\breve{\falb}$-modules $\breve{\nu}^*(\breve{\psi}_*\breve{\bb{L}}^{(n)}\otimes_{\breve{\bb{Z}}_p}\breve{\falb})$, $\breve{\nu}^*\breve{\sigma}^*\breve{\ca{M}}^{q,n-q}(q-n)$ on $\fal^{\proet,\bb{N}}_{Y\to X}$ define isomorphisms
	\begin{align}
		\breve{\nu}^*(\breve{\psi}_*\breve{\bb{L}}^{(n)}\otimes_{\breve{\bb{Z}}_p}\breve{\falb})&\iso f_u^{-1}(\breve{\nu}^*(\breve{\psi}_*\breve{\bb{L}}^{(n)}\otimes_{\breve{\bb{Z}}_p}\breve{\falb})),\label{eq:6.7.10}\\
		\breve{\nu}^*\breve{\sigma}^*\breve{\ca{M}}^{q,n-q}(q-n)&\iso f_u^{-1}(\breve{\nu}^*\breve{\sigma}^*\breve{\ca{M}}^{q,n-q}(q-n))\label{eq:6.7.11}.
	\end{align}
	By \ref{thm:rel-hodge-tate}, up to isogenies, they are respectively compatible with the relative Hodge-Tate filtration $\breve{\nu}^*\mrm{Fil}^q$ and compatible with the canonical isomorphisms for the graded pieces $\breve{\nu}^*\mrm{Gr}^q$. Passing to localization by $j^{-1}$, we obtain from the isomorphisms \eqref{eq:6.7.10} and \eqref{eq:6.7.11} the isomorphisms
	\begin{align}
		j^{-1}(\breve{\nu}^*(\breve{\psi}_*\breve{\bb{L}}^{(n)}\otimes_{\breve{\bb{Z}}_p}\breve{\falb}))&\iso j^{-1}f_u^{-1}(\breve{\nu}^*(\breve{\psi}_*\breve{\bb{L}}^{(n)}\otimes_{\breve{\bb{Z}}_p}\breve{\falb}))\iso f_\gamma^{-1}j^{-1}(\breve{\nu}^*(\breve{\psi}_*\breve{\bb{L}}^{(n)}\otimes_{\breve{\bb{Z}}_p}\breve{\falb})),\label{eq:6.6.13}\\
		j^{-1}(\breve{\nu}^*\breve{\sigma}^*\breve{\ca{M}}^{q,n-q}(q-n))&\iso j^{-1}f_u^{-1}(\breve{\nu}^*\breve{\sigma}^*\breve{\ca{M}}^{q,n-q}(q-n))\iso f_\gamma^{-1}j^{-1}(\breve{\nu}^*\breve{\sigma}^*\breve{\ca{M}}^{q,n-q}(q-n)).\label{eq:6.6.14}
	\end{align}
	Applying the derived functor $\rr\Gamma(\fal_{V\to U}^{\proet,\bb{N}},-)[1/p]$ and combining with the canonical isomorphisms \eqref{eq:6.6.7} and \eqref{eq:6.6.8}, we obtain automorphisms of $H^n_{\et}\otimes_{\bb{Z}_p} \widehat{A}[1/p]$ and of $H^q(X',\Omega^{n-q}_{(X',\scr{M}_{X'})/(X,\scr{M}_X)})\otimes_R \widehat{A}[1/p](q-n)$, which are exactly the semi-linear actions of $\gamma\in \Gamma$ on these $\widehat{A}[1/p]$-modules defined in \ref{para:notation-final-galois} by going through the definitions. Thus, we see that the actions of $\Gamma$ are compatible with the relative Hodge-Tate filtration and the canonical isomorphisms for the graded pieces, which completes the proof.
\end{proof}

\begin{myrem}\label{rem:ab-assumption}
	The arguments for \ref{thm:main} does not make use of the assumption that the residue field of $K$ is algebraically closed.
\end{myrem}
\begin{myrem}\label{rem:ab-question}
	In \ref{para:notation-final-galois}, we take $U$ to be an \'etale neighborhood of a point of the special fibre of $X$ which is affine and admits an adequate chart in the sense of \cite[\Luoma{3}.4.4]{abbes2016p} (cf. \cite[\Luoma{3}.4.7]{abbes2016p}), and take $V$ to be the inverse limit of the normalized universal cover $(V_i)$ of $U^\circ_{\star}$ at $\overline{x}$ (cf. \cite[\Luoma{5}.7]{sga1}, \cite[\Luoma{6}.9.8]{abbes2016p}). We set $U^{V_i}=\spec(R_i)$ and $\overline{R}=\colim(R_i)$. Then, $\Gamma=\pi_1(U^\circ_{\star},\overline{x})$ and $A=\overline{R}$. We obtain from the adequate chart finitely many nonzero divisors $f_1,\dots,f_r$ of $\Gamma(U_\eta,\ca{O}_{U_\eta})$ such that the divisor $D=\sum_{i=1}^r \mrm{div}(f_i)$ has support $U_\eta\setminus U^\circ$ and that at each strict henselization of $U_\eta$ those elements $f_i$ contained in the maximal ideal form a subset of a regular system of parameters (cf. \cite[4.2.2.(\luoma{2})]{abbes2020suite}). Then, $A$ is almost pre-perfectoid and admits compatible $n$-th power roots of $f_i$ (\cite[\Luoma{2}.9.10]{abbes2016p}). Hence, $V\to U$ is Faltings acyclic by \ref{thm:acyclic}, and thus Theorem \ref{thm:main} holds in this setting, which gives a local version of the relative Hodge-Tate filtration answering the question of Abbes-Gros raised in the first version of \cite{abbes2020suite} (cf. \cite[1.2.3]{abbes2020suite}).
\end{myrem}

\nocite{abbes2010rigid}
\nocite{beilinson2012derham}
\nocite{achinger2015kpi1}
\nocite{dubois1981complex}
\nocite{tsuji2019saturated}
\nocite{voevodsky1996htop}
\nocite{deligne1974hodge3}
\nocite{huber2014h-diff}
\nocite{rydh2010v}
\nocite{dejong1997alt}
\bibliographystyle{myalpha}%or amsplain
\bibliography{bibli}

\begin{thebibliography}{EGA \rm{IV}$_{4}$}

\bibitem[Abb10]{abbes2010rigid}
Ahmed Abbes.
\newblock {\em \'{E}l\'{e}ments de g\'{e}om\'{e}trie rigide. {V}olume {I}},
  volume 286 of {\em Progress in Mathematics}.
\newblock Birkh\"{a}user/Springer Basel AG, Basel, 2010.
\newblock Construction et \'{e}tude g\'{e}om\'{e}trique des espaces rigides.
  [Construction and geometric study of rigid spaces], With a preface by Michel
  Raynaud.

\bibitem[Ach15]{achinger2015kpi1}
Piotr Achinger.
\newblock {$K(\pi,1)$}-neighborhoods and comparison theorems.
\newblock {\em Compos. Math.}, 151(10):1945--1964, 2015.

\bibitem[Ach17]{achinger2017wildram}
Piotr Achinger.
\newblock Wild ramification and {$K(\pi, 1)$} spaces.
\newblock {\em Invent. Math.}, 210(2):453--499, 2017.

\bibitem[AG20]{abbes2020suite}
Ahmed Abbes and Michel Gros.
\newblock Les suites spectrales de {H}odge-{T}ate.
\newblock \url{https://arxiv.org/abs/2003.04714v2}, 2020.

\bibitem[AGT16]{abbes2016p}
Ahmed Abbes, Michel Gros, and Takeshi Tsuji.
\newblock {\em The {$p$}-adic {S}impson correspondence}, volume 193 of {\em
  Annals of Mathematics Studies}.
\newblock Princeton University Press, Princeton, NJ, 2016.

\bibitem[ALPT19]{temkin2019logval}
Karim Adiprasito, Gaku Liu, Igor Pak, and Michael Temkin.
\newblock Log smoothness and polystability over valuation rings.
\newblock \url{https://arxiv.org/abs/1806.09168v3}, 2019.

\bibitem[And18]{andre2018abhyankar}
Yves Andr\'{e}.
\newblock Le lemme d'{A}bhyankar perfectoide.
\newblock {\em Publ. Math. Inst. Hautes \'{E}tudes Sci.}, 127:1--70, 2018.

\bibitem[Bei12]{beilinson2012derham}
Alexander Beilinson.
\newblock {$p$}-adic periods and derived de {R}ham cohomology.
\newblock {\em J. Amer. Math. Soc.}, 25(3):715--738, 2012.

\bibitem[BM20]{bhattmathew2020arc}
Bhargav Bhatt and Akhil Mathew.
\newblock The arc-topology.
\newblock \url{https://arxiv.org/abs/1807.04725v4}, 2020.

\bibitem[Bou06]{bourbaki2006commalg5-7}
Nicolas Bourbaki.
\newblock {\em Alg\`ebre commutative. Chapitres 5 \`a 7.}
\newblock Springer-Verlag, Berlin Heidelberg, 2006.
\newblock Edition originale publiée par Herman, Paris, 1975.

\bibitem[Bou07]{bourbaki2007top5-10}
Nicolas Bourbaki.
\newblock {\em Topologie g\'en\'erale. Chapitres 5 \`a 7.}
\newblock Springer-Verlag, Berlin Heidelberg, 2007.
\newblock Edition originale publiée par Herrman, Paris, 1974.

\bibitem[BS15]{bhattscholze2015proet}
Bhargav Bhatt and Peter Scholze.
\newblock The pro-\'{e}tale topology for schemes.
\newblock {\em Ast\'{e}risque}, (369):99--201, 2015.

\bibitem[BS17]{bhattscholze2017witt}
Bhargav Bhatt and Peter Scholze.
\newblock Projectivity of the {W}itt vector affine {G}rassmannian.
\newblock {\em Invent. Math.}, 209(2):329--423, 2017.

\bibitem[BS19]{bhattscholze2019prisms}
Bhargav Bhatt and Peter Scholze.
\newblock Prisms and prismatic cohomology.
\newblock \url{https://arxiv.org/abs/1905.08229v4}, 2019.

\bibitem[BST17]{bhatt2017gabber}
Bhargav Bhatt, Karl Schwede, and Shunsuke Takagi.
\newblock The weak ordinarity conjecture and {$F$}-singularities.
\newblock In {\em Higher dimensional algebraic geometry---in honour of
  {P}rofessor {Y}ujiro {K}awamata's sixtieth birthday}, volume~74 of {\em Adv.
  Stud. Pure Math.}, pages 11--39. Math. Soc. Japan, Tokyo, 2017.

\bibitem[CS17]{caraiani2017generic}
Ana Caraiani and Peter Scholze.
\newblock On the generic part of the cohomology of compact unitary {S}himura
  varieties.
\newblock {\em Ann. of Math. (2)}, 186(3):649--766, 2017.

\bibitem[CS19]{cesnaviciusscholze2019purity}
Kestutis Cesnavicius and Peter Scholze.
\newblock Purity for flat cohomology.
\newblock \url{https://arxiv.org/abs/1912.10932v2}, 2019.

\bibitem[DB81]{dubois1981complex}
Philippe Du~Bois.
\newblock Complexe de de {R}ham filtr\'{e} d'une vari\'{e}t\'{e} singuli\`ere.
\newblock {\em Bull. Soc. Math. France}, 109(1):41--81, 1981.

\bibitem[Del74]{deligne1974hodge3}
Pierre Deligne.
\newblock Th\'{e}orie de {H}odge. {III}.
\newblock {\em Inst. Hautes \'{E}tudes Sci. Publ. Math.}, (44):5--77, 1974.

\bibitem[dJ96]{dejong1996alt}
Aise~Johan de~Jong.
\newblock Smoothness, semi-stability and alterations.
\newblock {\em Inst. Hautes \'{E}tudes Sci. Publ. Math.}, (83):51--93, 1996.

\bibitem[dJ97]{dejong1997alt}
Aise~Johan de~Jong.
\newblock Families of curves and alterations.
\newblock {\em Ann. Inst. Fourier (Grenoble)}, 47(2):599--621, 1997.

\bibitem[EGA \rm{II}]{ega2}
Alexander Grothendieck.
\newblock \'{E}l\'{e}ments de g\'{e}om\'{e}trie alg\'{e}brique. {II}. \'{E}tude
  globale \'{e}l\'{e}mentaire de quelques classes de morphismes.
\newblock {\em Inst. Hautes \'{E}tudes Sci. Publ. Math.}, (8):222, 1961.

\bibitem[EGA \rm{IV}$_{2}$]{ega4-2}
Alexander Grothendieck.
\newblock \'{E}l\'{e}ments de g\'{e}om\'{e}trie alg\'{e}brique. {IV}. \'{E}tude
  locale des sch\'{e}mas et des morphismes de sch\'{e}mas. {II}.
\newblock {\em Inst. Hautes \'{E}tudes Sci. Publ. Math.}, (24):231, 1965.

\bibitem[EGA \rm{IV}$_{3}$]{ega4-3}
Alexander Grothendieck.
\newblock \'{E}l\'{e}ments de g\'{e}om\'{e}trie alg\'{e}brique. {IV}. \'{E}tude
  locale des sch\'{e}mas et des morphismes de sch\'{e}mas. {III}.
\newblock {\em Inst. Hautes \'{E}tudes Sci. Publ. Math.}, (28):255, 1966.

\bibitem[EGA \rm{IV}$_{4}$]{ega4-4}
Alexander Grothendieck.
\newblock \'{E}l\'{e}ments de g\'{e}om\'{e}trie alg\'{e}brique. {IV}. \'{E}tude
  locale des sch\'{e}mas et des morphismes de sch\'{e}mas {IV}.
\newblock {\em Inst. Hautes \'{E}tudes Sci. Publ. Math.}, (32):361, 1967.

\bibitem[Fal88]{faltings1988p}
Gerd Faltings.
\newblock {$p$}-adic {H}odge theory.
\newblock {\em J. Amer. Math. Soc.}, 1(1):255--299, 1988.

\bibitem[Fal02]{faltings2002almost}
Gerd Faltings.
\newblock Almost \'{e}tale extensions.
\newblock {\em Ast\'{e}risque}, (279):185--270, 2002.
\newblock Cohomologies $p$-adiques et applications arithm\'{e}tiques, II.

\bibitem[Fu15]{fulei2015etale}
Lei Fu.
\newblock {\em Etale cohomology theory}, volume~14 of {\em Nankai Tracts in
  Mathematics}.
\newblock World Scientific Publishing Co. Pte. Ltd., Hackensack, NJ, revised
  edition, 2015.

\bibitem[GR03]{gabber2003almost}
Ofer Gabber and Lorenzo Ramero.
\newblock {\em Almost ring theory}, volume 1800 of {\em Lecture Notes in
  Mathematics}.
\newblock Springer-Verlag, Berlin, 2003.

\bibitem[GR04]{gabber2004foundations}
Ofer Gabber and Lorenzo Ramero.
\newblock Foundations for almost ring theory -- release 7.5.
\newblock \url{https://arxiv.org/abs/math/0409584v13}, 2004.

\bibitem[HJ14]{huber2014h-diff}
Annette Huber and Clemens J\"{o}rder.
\newblock Differential forms in the h-topology.
\newblock {\em Algebr. Geom.}, 1(4):449--478, 2014.

\bibitem[IKN05]{illusie2005riemannhilbert}
Luc Illusie, Kazuya Kato, and Chikara Nakayama.
\newblock Quasi-unipotent logarithmic {R}iemann-{H}ilbert correspondences.
\newblock {\em J. Math. Sci. Univ. Tokyo}, 12(1):1--66, 2005.

\bibitem[ILO14]{gabber2014travaux}
Luc Illusie, Yves Laszlo, and Fabrice Orgogozo, editors.
\newblock {\em Travaux de {G}abber sur l'uniformisation locale et la
  cohomologie \'{e}tale des sch\'{e}mas quasi-excellents}.
\newblock Soci\'{e}t\'{e} Math\'{e}matique de France, Paris, 2014.
\newblock S\'{e}minaire \`a l'\'{E}cole Polytechnique 2006--2008. [Seminar of
  the Polytechnic School 2006--2008], With the collaboration of
  Fr\'{e}d\'{e}ric D\'{e}glise, Alban Moreau, Vincent Pilloni, Michel Raynaud,
  Jo\"{e}l Riou, Beno\^{\i}t Stroh, Michael Temkin and Weizhe Zheng,
  Ast\'{e}risque No. 363-364 (2014) (2014).

\bibitem[Kat89]{kato1989log}
Kazuya Kato.
\newblock Logarithmic structures of {F}ontaine-{I}llusie.
\newblock In {\em Algebraic analysis, geometry, and number theory ({B}altimore,
  {MD}, 1988)}, pages 191--224. Johns Hopkins Univ. Press, Baltimore, MD, 1989.

\bibitem[Kat94]{kato1994toric}
Kazuya Kato.
\newblock Toric singularities.
\newblock {\em Amer. J. Math.}, 116(5):1073--1099, 1994.

\bibitem[Ker16]{kerz2016transfinite}
Moritz Kerz.
\newblock Transfinite limits in topos theory.
\newblock {\em Theory Appl. Categ.}, 31:Paper No. 7, 175--200, 2016.

\bibitem[Niz06]{niziol2006toric}
Wies{\l}awa Nizio{\l}.
\newblock Toric singularities: log-blow-ups and global resolutions.
\newblock {\em J. Algebraic Geom.}, 15(1):1--29, 2006.

\bibitem[Ogu18]{ogus2018log}
Arthur Ogus.
\newblock {\em Lectures on logarithmic algebraic geometry}, volume 178 of {\em
  Cambridge Studies in Advanced Mathematics}.
\newblock Cambridge University Press, Cambridge, 2018.

\bibitem[Ryd10]{rydh2010v}
David Rydh.
\newblock Submersions and effective descent of \'{e}tale morphisms.
\newblock {\em Bull. Soc. Math. France}, 138(2):181--230, 2010.

\bibitem[Sch12]{scholze2012perfectoid}
Peter Scholze.
\newblock Perfectoid spaces.
\newblock {\em Publ. Math. Inst. Hautes \'{E}tudes Sci.}, 116:245--313, 2012.

\bibitem[Sch13a]{scholze2013hodge}
Peter Scholze.
\newblock {$p$}-adic {H}odge theory for rigid-analytic varieties.
\newblock {\em Forum Math. Pi}, 1:e1, 77, 2013.

\bibitem[Sch13b]{scholze2013perfsurv}
Peter Scholze.
\newblock Perfectoid spaces: a survey.
\newblock In {\em Current developments in mathematics 2012}, pages 193--227.
  Int. Press, Somerville, MA, 2013.

\bibitem[Sch16]{scholze2016erratum}
Peter Scholze.
\newblock {$p$}-adic {H}odge theory for rigid-analytic varieties---corrigendum
  [{MR}3090230].
\newblock {\em Forum Math. Pi}, 4:e6, 4, 2016.

\bibitem[Sch21]{scholze2021diamond}
Peter Scholze.
\newblock Etale cohomology of diamonds.
\newblock \url{https://arxiv.org/abs/1709.07343v2}, 2021.

\bibitem[SGA 1]{sga1}
{\em Rev\^{e}tements \'{e}tales et groupe fondamental ({SGA} 1)}, volume~3 of
  {\em Documents Math\'{e}matiques (Paris) [Mathematical Documents (Paris)]}.
\newblock Soci\'{e}t\'{e} Math\'{e}matique de France, Paris, 2003.
\newblock S\'{e}minaire de g\'{e}om\'{e}trie alg\'{e}brique du Bois Marie
  1960--61. [Algebraic Geometry Seminar of Bois Marie 1960-61], Directed by A.
  Grothendieck, With two papers by M. Raynaud, Updated and annotated reprint of
  the 1971 original [Lecture Notes in Math., 224, Springer, Berlin; MR0354651
  (50 \#7129)].

\bibitem[SGA 4$_{\text{\rm I}}$]{sga4-1}
{\em Th\'{e}orie des topos et cohomologie \'{e}tale des sch\'{e}mas. {T}ome 1:
  {T}h\'{e}orie des topos}.
\newblock Lecture Notes in Mathematics, Vol. 269. Springer-Verlag, Berlin-New
  York, 1972.
\newblock S\'{e}minaire de G\'{e}om\'{e}trie Alg\'{e}brique du Bois-Marie
  1963--1964 (SGA 4), Dirig\'{e} par M. Artin, A. Grothendieck, et J. L.
  Verdier. Avec la collaboration de N. Bourbaki, P. Deligne et B. Saint-Donat.

\bibitem[SGA 4$_{\text{\rm II}}$]{sga4-2}
{\em Th\'{e}orie des topos et cohomologie \'{e}tale des sch\'{e}mas. {T}ome 2}.
\newblock Lecture Notes in Mathematics, Vol. 270. Springer-Verlag, Berlin-New
  York, 1972.
\newblock S\'{e}minaire de G\'{e}om\'{e}trie Alg\'{e}brique du Bois-Marie
  1963--1964 (SGA 4), Dirig\'{e} par M. Artin, A. Grothendieck et J. L.
  Verdier. Avec la collaboration de N. Bourbaki, P. Deligne et B. Saint-Donat.

\bibitem[SGA 4$_{\text{\rm III}}$]{sga4-3}
{\em Th\'{e}orie des topos et cohomologie \'{e}tale des sch\'{e}mas. {T}ome 3}.
\newblock Lecture Notes in Mathematics, Vol. 305. Springer-Verlag, Berlin-New
  York, 1973.
\newblock S\'{e}minaire de G\'{e}om\'{e}trie Alg\'{e}brique du Bois-Marie
  1963--1964 (SGA 4), Dirig\'{e} par M. Artin, A. Grothendieck et J. L.
  Verdier. Avec la collaboration de P. Deligne et B. Saint-Donat.

\bibitem[{Sta}22]{stacks-project}
The {Stacks project authors}.
\newblock The stacks project.
\newblock \url{https://stacks.math.columbia.edu}, 2022.

\bibitem[Tat67]{tate1967p}
John Tate.
\newblock {$p$}-divisible groups.
\newblock In {\em Proc. {C}onf. {L}ocal {F}ields ({D}riebergen, 1966)}, pages
  158--183. Springer, Berlin, 1967.

\bibitem[Tsu99]{tsuji1999p}
Takeshi Tsuji.
\newblock {$p$}-adic \'{e}tale cohomology and crystalline cohomology in the
  semi-stable reduction case.
\newblock {\em Invent. Math.}, 137(2):233--411, 1999.

\bibitem[Tsu02]{tsuji2002semi}
Takeshi Tsuji.
\newblock Semi-stable conjecture of {F}ontaine-{J}annsen: a survey.
\newblock Number 279, pages 323--370. 2002.
\newblock Cohomologies $p$-adiques et applications arithm\'{e}tiques, II.

\bibitem[Tsu19]{tsuji2019saturated}
Takeshi Tsuji.
\newblock Saturated morphisms of logarithmic schemes.
\newblock {\em Tunis. J. Math.}, 1(2):185--220, 2019.

\bibitem[Voe96]{voevodsky1996htop}
Vladimir Voevodsky.
\newblock Homology of schemes.
\newblock {\em Selecta Math. (N.S.)}, 2(1):111--153, 1996.

\bibitem[Xu22]{xu2022higgs}
Daxin Xu.
\newblock Parallel transport for {H}iggs bundles over $p$-adic curves.
\newblock \url{https://arxiv.org/abs/2201.06697v1}, 2022.

\end{thebibliography}

%\begin{thebibliography}{1}
%\bibitem{kuiper1976proj} \textemdash, Topological conjugacy of real projective transformations, \textit{Topology} \textbf{15} no.1 (1976), 13--22.
%\end{thebibliography}

\end{document}